\documentclass[a4paper,reqno]{amsart}

\usepackage[T1]{fontenc}
\usepackage[utf8]{inputenc}
\usepackage{lmodern}
\usepackage[english]{babel}
\usepackage[margin=2.0cm]{geometry}

\usepackage{amsmath,amsthm,amssymb,mathrsfs,yhmath,stmaryrd}

\usepackage{mathtools}
\usepackage{amsfonts}
\usepackage{enumerate}
\usepackage{enumitem}
\usepackage{hyperref}
\usepackage{colonequals,bbm}
\DeclareRobustCommand\longtwoheadrightarrow
	{\relbar\joinrel\twoheadrightarrow}
\DeclareRobustCommand\longhookrightarrow
	{\lhook\joinrel\longrightarrow}

\usepackage{tikz}
\usetikzlibrary{cd}

\numberwithin{equation}{section}

\allowdisplaybreaks

\makeatletter
\renewcommand{\tocsection}[3]{%
	\indentlabel{\@ifnotempty{#2}{\bfseries\ignorespaces#1 #2\quad}}\bfseries#3}
\renewcommand{\tocsubsection}[3]{%
	\indentlabel{\@ifnotempty{#2}{\ignorespaces#1 #2\quad}}#3}

\newcommand\@dotsep{4.5}
\def\@tocline#1#2#3#4#5#6#7{\relax
	\ifnum #1>\c@tocdepth 
	\else
	\par \addpenalty\@secpenalty\addvspace{#2}%
	\begingroup \hyphenpenalty\@M
	\@ifempty{#4}{%
		\@tempdima\csname r@tocindent\number#1\endcsname\relax
	}{%
		\@tempdima#4\relax
	}%
	\parindent\z@ \leftskip#3\relax \advance\leftskip\@tempdima\relax
	\rightskip\@pnumwidth plus1em \parfillskip-\@pnumwidth
	#5\leavevmode\hskip-\@tempdima{#6}\nobreak
	\leaders\hbox{$\m@th\mkern \@dotsep mu\hbox{.}\mkern \@dotsep mu$}\hfill
	\nobreak
	\hbox to\@pnumwidth{\@tocpagenum{\ifnum#1=1\bfseries\fi#7}}\par
	\nobreak
	\endgroup
	\fi}
\AtBeginDocument{%
\expandafter\renewcommand\csname r@tocindent0\endcsname{0pt}
}
\def\l@subsection{\@tocline{2}{0pt}{2.5pc}{5pc}{}}
\makeatother

\makeatletter
\@namedef{subjclassname@2020}{\textup{2020} Mathematics Subject Classification}
\makeatother

\newcommand{\N}{\mathbb{N}}

\newcommand{\R}{\mathbb{R}}

\newcommand{\Hq}{\mathbb{H}}
\newcommand{\bk}{\mathbbm{k}}

\newcommand{\Mod}{\operatorname{Mod}}
\newcommand{\AMod}{{}_A\!\Mod}
\newcommand{\ModA}{\Mod_A}
\newcommand{\AModA}{{}_A\!\Mod_A}
\newcommand{\Proj}{\operatorname{Proj}}
\newcommand{\AProj}{{}_A\!\Proj}
\newcommand{\AFlat}{{}_A\!\operatorname{Flat}}
\newcommand{\AFG}{{}_A\!\operatorname{FGen}}
\newcommand{\AFGP}{{}_A\!\operatorname{FGP}}
\newcommand{\FlatA}{\operatorname{Flat}_A}
\newcommand{\CalcSA}{\operatorname{Calc}^S_A}
\newcommand{\CalcA}{\operatorname{Calc}_A}

\newcommand{\Diff}{\operatorname{Diff}}
\newcommand{\SDiff}{\mathrm{S}\!\Diff}
\newcommand{\NDiff}{\mathrm{N}\!\Diff}

\newcommand{\Hom}{\operatorname{Hom}}
\newcommand{\AHom}{{}_A\!\Hom}

\newcommand{\Aoplus}{{\,}_A\!\oplus}
\newcommand{\oplusA}{\oplus_A}
\newcommand{\AoplusA}{{\,}_A\!\oplus_A}

\newcommand{\DO}{\mathcal{D}}
\newcommand{\SDO}{\mathrm{S}\DO}
\newcommand{\NDO}{\mathrm{N}\DO}

\newcommand{\Eq}{\operatorname{Eq}}
\DeclareMathOperator*{\colim}{co{\lim}}

\DeclareMathOperator*{\coker}{co{\ker}}

\newcommand{\smb}{\star}
\newcommand{\smbr}{\star}

\newcommand{\Ext}{\operatorname{Ext}}
\newcommand{\Tor}{\operatorname{Tor}}

\newcommand{\Sym}{\operatorname{Sym}}

\newcommand{\pj}{\operatorname{proj}}
\newcommand{\inc}{\operatorname{incl}}

\newcommand{\smooth}[1]{\mathcal{C}^{\infty}\!{({#1})}}

\newcommand{\spn}[1]{\left\langle {#1} \right\rangle}
\newcommand{\id}{\mathrm{id}}

\newcommand{\im}{\mathrm{Im}}

\newcommand{\ev}{\operatorname{ev}}
\newcommand{\op}{\mathrm{op}}

\newcommand{\vsfd}{-15pt}	

\theoremstyle{definition}

\newtheorem{defi}{Definition}[section]
\newtheorem{eg}[defi]{Example}

\theoremstyle{plain}

\newtheorem{theo}[defi]{Theorem}
\newtheorem{prop}[defi]{Proposition}
\newtheorem{cor}[defi]{Corollary}
\newtheorem{lemma}[defi]{Lemma}

\newtheorem*{theo*}{Theorem}
\newtheorem*{prop*}{Proposition}
\newtheorem*{cor*}{Corollary}
\newtheorem*{lemma*}{Lemma}

\newtheorem{innertheo}{Theorem}
\newenvironment{ntheo}[1]
	{\renewcommand\theinnertheo{#1}\innertheo}
	{\endinnerprop}
\newtheorem{innerprop}{Proposition}
\newenvironment{nprop}[1]
	{\renewcommand\theinnerprop{#1}\innerprop}
	{\endinnerprop}

\theoremstyle{remark}
\newtheorem{rmk}[defi]{Remark}

\begin{document}

\title{Jet Functors in Noncommutative Geometry}
\author{Keegan J.~Flood, Mauro Mantegazza, Henrik Winther}
\address{Faculty of Mathematics and Computer Science\\
  UniDistance Suisse\\
  Schinerstrasse 18\\
  3900 Brig\\
  Switzerland}
  \email{keegan.flood@unidistance.ch}
\address{Department of Mathematics and Physics\\
  Charles University\\
  Sokolovsk\'{a} 49/83\\
  186 75 Prague 8\\
  Czech Republic}
  \email{mauro.mantegazza.uni@gmail.com}
\address{Department of Mathematics and Statistics\\
  UiT - The Arctic University of Norway\\
  Hansine Hansens veg 18\\
  N-9019 Tromsø\\
  Norway}
  \email{henrik.winther@uit.no}



\begin{abstract}
	In this article we construct three infinite families of endofunctors $J_d^{(n)}$, $J_d^{[n]}$, and $J_d^n$ on the category of left $A$-modules, where $A$ is a unital associative algebra over a commutative ring $\bk$, equipped with an exterior algebra $\Omega^\bullet_d$.
	We prove that these functors generalize the corresponding classical notions of nonholonomic, semiholonomic, and holonomic jet functors, respectively.
	Our functors come equipped with natural transformations from the identity functor to the corresponding jet functors, which play the r\^{o}les of the classical prolongation maps.
	This allows us to define the notion of linear differential operators with respect to $\Omega^{\bullet}_d$.
	We show that if $\Omega^1_d$ is flat as a right $A$-module, the semiholonomic jet functor satisfies the semiholonomic jet exact sequence $0 \rightarrow \bigotimes^n_A \Omega^1_d \rightarrow J^{[n]}_d\rightarrow J^{[n-1]}_d \rightarrow 0$.
	Moreover, we construct a functor of symmetric (in a suitable noncommutative sense) forms $S^n_d$ associated to $\Omega^\bullet_d$, and proceed to introduce the corresponding noncommutative analogue of the Spencer $\delta$-complex.
	We give necessary and sufficient conditions under which the holonomic jet functor $J_d^n$ satisfies the (holonomic) jet exact sequence, $0\rightarrow S^n_d \rightarrow J_d^n \rightarrow J_d^{n-1} \rightarrow 0$.
	In particular, for $n=1$ the sequence is always exact, for $n=2$ it is exact for $\Omega^1_d$ flat as a right $A$-module, and for $n\ge 3$, it is sufficient to have $\Omega^1_d$, $\Omega^2_d$, and $\Omega^3_d$ flat as right $A$-modules and the vanishing of the Spencer $\delta$-cohomology $H^{\bullet,2}_{\delta_d}$.
\end{abstract}

\maketitle

\tableofcontents

\section{Introduction}
Jet constructions are a fundamental part of differential geometry.
Indeed, under very permissive assumptions, geometric objects in differential geometry \cite[Chapter V]{NaturalOperations} are precisely sections of associated bundles to a representation of a jet group, the jet order of which is constrained by the dimension of the fiber and base manifold.
Moreover, every natural operation mapping between geometric objects, taking smoothly parametrized families to smoothly parametrized families, can be seen to depend only on the infinite jet of the section, and in a neighborhood of a fixed section, only on a finite order jet.
When the operation is also sufficiently regular, the jet order dependence can be bounded, and the mapping is locally a finite order differential operator (cf.\ \cite[§19]{NaturalOperations}).
These results show that jet constructions and operations are enough to recover differential geometry.
Jet constructions are also naturally coordinate-free and functorial, and these properties make them attractive for generalizations beyond the classical setting, see for example \cite{Schreiber}.

The duality between geometry and algebra leads to the point of view that differential geometry can be realized as the study of a particular class of commutative associative algebras.
A natural generalization is to drop the commutativity assumption.
Therefore, we will approach noncommutative differential geometry as the generalization of the commutative algebra $\smooth{M}$ of smooth functions on a manifold $M$ to an arbitrary (associative, unital) algebra $A$ over a commutative unital ring $\bk$.
Constructions should then be impartial with respect to the commutativity of the underlying algebra $A$, which may be commutative or not, but should reproduce classical geometry in the case that it is.
There are several approaches following these principles, e.g.\ \cite{Connes,Manin1988,Drinfeld1987}.
In this work we will be most closely aligned with the “quantum group” approach (\'{a} la \cite{BeggsMajid}) in that we work with exterior algebras and first order differential calculi, although we will require no further structure beyond that.
The irreducible quantum flag manifolds, equipped with the Heckenberger-Kolb first order differential calculus \cite{HeckenbergerKolb}, form a prominent class of examples in the quantum group setting.

Thus, the motivation for this work is to construct a natural notion of jet functor, and subsequently of differential operator, which treats, as special cases, mildly noncommutative geometries such as supergeometry (with the Deligne sign convention \cite[§6]{deligne1999sign}), as well as the aforementioned quantum flag manifolds.
That is to say, an infinite family of endofunctors $J^n_d$ on the category of left modules $E$ over a possibly noncommutative algebra $A$, equipped with an exterior algebra $\Omega^\bullet_d$ for $A$, i.e.\ a differential graded algebra with $\Omega^0_d = A$ and which is generated by $A$ and $dA$, cf.\ Definition \ref{def:extalgebra}.
We will generalize the prolongation map, or “universal differential operator” to this setting.
This allows us to define finite order differential operators between (left) $A$-modules, and this can be seen as a way of lifting the differentiable structure on $A$, given by $d$, to its modules $E$.
This procedure will be extended to the projective limit $J^\infty_d$, which lets us, in principle, generalize any natural operation arising in the framework described above.
In this paper we only work with left $A$-modules (or bimodules where the relevant $A$-action operates on the left), therefore our functors can more specifically be referred to as left jet functors.
By duality, all the constructions we present and all the results we prove can also be realized on right $A$-modules, allowing to define an analogous theory of right jet functors.

We offer several possible definitions of jet functor (cf.\ Lemma \ref{lemma:Jchar}), which are equivalent under sufficient assumptions on the exterior algebra.
In particular, all of them generalize the classical notion, in light of \cite{Quillen,Spencer,Goldschmidt}.
Our choice of fundamental definition is the one that requires no assumptions to build the noncommutative generalization of the Spencer $\delta$-complex.
We provide an explicit isomorphism between our construction and the classical notion of jets in the case that the exterior algebra $\Omega^\bullet_d$ is the classical exterior algebra $\Omega^\bullet(M)$ equipped with the de Rham differential $d_{dR}$.
From the viewpoint of differential geometry, our construction can then be interpreted as treating H.\ Goldschmidt's $\rho$ operators (see below) as fundamental.
Goldschmidt derives these operators in a context where jets are already defined.
In contrast, our method does not presuppose the existence of the higher order jet modules, but rather we explicitly construct the $\rho$ maps and obtain the (holonomic) jet modules as their kernels.
To the best of our knowledge, this is a novel approach.

For decades, there has been considerable interest \cite{verbovetsky1997,laksov2000,Sardanashvily2003jets,majid2023quantum} in genuinely noncommutative jets.
\cite{verbovetsky1997} and \cite{Sardanashvily2003jets}, both take differential operators as the primary notion and define jets in terms of spaces of differential operators.
Their notion of differential operator is limited in that a certain degree of commutativity is intrinsic to their definition, and hence it cannot capture the full noncommutative picture.
Rather than working with an exterior algebra, \cite{laksov2000} instead work in the setting of $A$-bimodules equipped with “balanced derivations” and, in that context, they develop a notion analogous to our notion of semiholonomic jets.
However, their theory does not allow them to proceed beyond semiholonomic jets without very strong conditions, including the $A$-bimodule being free and finitely generated, and its basis satisfying a commutativity property.
Most recently, \cite{majid2023quantum} formulated a notion of jet for which their $1$-jet $A$-bimodule decomposes as $A\oplus\Omega^1$ in the category of $A$-bimodules, by construction.
For jets of order strictly higher than $1$, they require their first order differential calculus to come equipped with a bimodule connection which satisfies a number of conditions, including torsion-freeness (and flatness for jets higher than $2$).
Furthermore, they require a flat bimodule connection on any $A$-bimodule to which they wish to apply their construction.
Bimodule connections come equipped with a braiding map, hence restricting to objects equipped with bimodule connections is a commutativity condition.
In contrast to the previous notions of jet, ours requires no commutativity assumptions, and the constructions require no additional conditions beyond the data of an exterior algebra.
Most of the classical properties are then recovered by imposing mild homological conditions on the exterior algebra.

\subsection{Generalizing classical properties of jets.}
In order to construct our jet functors, we need to describe some important properties of jets from differential geometry.
It is well known that given a vector bundle $E\rightarrow M$, we have the following short exact sequence (cf.\ \cite[Proposition 1.1, p.~2]{Quillen}) involving the jet bundles and the bundle of $E$-valued symmetric forms
\begin{equation}
	0\longrightarrow S^n(M)\otimes E \longrightarrow J^n E \longrightarrow J^{n-1}E \longrightarrow 0.
\end{equation}
Via the Serre-Swan theorem (cf.\ \cite[§12.33, p.~191]{nestruev2020smooth}), we can equivalently consider the short exact sequence relating the (finitely generated projective) $\smooth{M}$-modules of their global sections.
We thus obtain the following short exact sequence containing jets $\Gamma(M,J^n E)$ of sections of $E\rightarrow M$, and $E$-valued symmetric forms
\begin{equation}
	\label{eq:kjetses}
	0\longrightarrow \Sym^n(\Omega^1(M))\otimes_{\smooth{M}} \Gamma(M,E) \longrightarrow \Gamma(M,J^n E) \longrightarrow \Gamma(M,J^{n-1}E) \longrightarrow 0.
\end{equation}
We also have the prolongation maps $j^n\colon \Gamma(M,E) \rightarrow \Gamma(M,J^nE)$, which takes a section to its $n$-jet.
The prolongation map is also known as the universal differential operator, because it realizes the jet bundle as the universal object for differential operators from $E$ to any target bundle $F$.
This is also a coordinate-free way to define differential operators, i.e.\ the map $\Delta$ is a differential operator if there exists a \textit{bundle map}, $\widetilde \Delta$, such that the following diagram commutes.
\begin{equation}
\begin{tikzcd}\label{classicaldifferentialoperator}
	\Gamma(M,J^n E) \arrow[dr, "\widetilde\Delta"] & \\
	\Gamma(M,E) \arrow[r,"\Delta"] \arrow[u,"j^n"] & \Gamma(M,F) 
\end{tikzcd}
\end{equation}
Connections are particular instances of differential operators of order $1$.
Classically, given a fiber bundle $E$, the lift of a connection on $E$ as a differential operator produces a bijective correspondence between said connections and sections of the first jet projection (cf.\ \cite[§17.1]{NaturalOperations}).

In order to motivate our results, we will need some structural results about classical jets.
First, we have the following lemma attributed to \cite{Quillen}.
\begin{lemma*}[{\cite[Lemma 1.2.1, p.~184]{Spencer}}]
	Let $E$ be a vector bundle.
	Then the following holds
	\begin{equation}
		J^{n+m}E = J^n (J^m E) \cap J^{n-1} (J^{m+1}E) \subset J^{(n+m)}E,
	\end{equation}
	where $J^{(n+m)}E$ denotes the (classical) nonholonomic jet bundle of order $n+m$.
\end{lemma*}
Secondly, we have the following proposition.
\begin{prop*}[{\cite[Proposition 3, p.~432]{Goldschmidt}}]
	Let $E$ be a vector bundle.
	There exists a unique differential operator of order $1$ corresponding to the bundle map
	\begin{equation}
		\rho\colon J^1(J^n E) \longrightarrow C^1_n E,
	\end{equation}
	whose \emph{symbol} (i.e.\ the restriction of the lift $\rho$ to $T^\ast \otimes J^n E$) is the natural projection $T^\ast \otimes J^n E \rightarrow C^1_n E$ and the sequence
	\begin{equation}
		0 \longrightarrow J^{n+1}E \longrightarrow J^1 (J^n E) \longrightarrow C^1_n E \longrightarrow 0
	\end{equation}
	is exact.
\end{prop*}
Here $C^1_n E$ is the space $(T^\ast \otimes J^n E)/\delta(\Sym^{n+1}(T^\ast)\otimes E)$, and $\delta$ is the Spencer differential, cf.\ \cite[p.~188]{Spencer}.

Furthermore, we note the following proposition.
\begin{prop*}[{\cite[Proposition 4, p.~433]{Goldschmidt}}]
There is a unique differential operator
\begin{equation}
		\lambda\colon J^1(J^n E) \longrightarrow T^\ast\otimes J^{n-1} E
	\end{equation}
	of order $1$, such that:
\begin{enumerate}
\item $J^n E\subseteq \ker(\lambda)$;
\item the symbol of $\lambda$ is the projection $\pi^{n-1}$ from $T^\ast\otimes J^n E$ to $T^\ast \otimes J^{n-1}E$.
\end{enumerate}
\end{prop*}

\subsection{The noncommutative generalization}
From now on we shall consider an algebra $A$ over a unital commutative ring $\bk$ in the following sense
\begin{defi}
A \emph{(unital associative) algebra $A$ over a commutative unital ring $\bk$} is a monoid in the monoidal category of $\bk$-modules equipped with the tensor product $\otimes$ over $\bk$.
More explicitly, it is a $\bk$-module (hence bimodule) $A$ endowed with a multiplication $\cdot \colon A\otimes A\to A$ and a unity $1\in A$ such that they satisfy the associativity and unity axioms (cf.\ \cite[§VII.3, p.~170]{MacLane}).
\end{defi}
This definition of algebra yields a unital (not necessarily commutative) ring structure on $A$.
The algebra structure can be equivalently encoded in a unital ring homomorphism $\bk\to A$, where the image of $\bk$ is contained in the center of $A$.
When $\bk$ is a field, this notion of algebra coincides with the one used in \cite{BeggsMajid}, and the results that we use from that source generalize with analogous proofs to our notion of algebra.

This algebra will be equipped with extra data encoding differential properties in the form of a first order differential calculus.
\begin{defi}\label{def:differential calculus}
A \emph{first order differential calculus} $\Omega^1_d$ over an algebra $A$ is an $A$-bimodule $\Omega^1_d$ equipped with a \emph{differential}, that is a $\bk$-linear map $d\colon A\rightarrow \Omega^1_d$ such that
\begin{enumerate}
\item\label{def:differential calculus:1} (Leibniz rule) for all $a,b\in A$, we have
\begin{equation}
	d(ab) = a db + (da) b.
\end{equation}
\item\label{def:differential calculus:2} (Surjectivity condition) $AdA=\Omega^1_d$, i.e.\ $\Omega^1_d$ is generated as a left module by $dA$, the image of $d$.
\end{enumerate}
\end{defi}
The notion of first order differential calculus was introduced in \cite{woronowicz1989}.

Up until §\ref{s:quantumsymmform}, all of our results will depend only on the first order differential calculus.
After §\ref{s:quantumsymmform} we require more structure on the algebra $A$, namely an exterior algebra $\Omega^\bullet_d$ (cf.\ Definition \ref{def:extalgebra}).
Note that the first grade of an exterior algebra is a first order differential calculus.

\subsubsection{Notation}\label{sss:notation}
Let $\AMod$ and $\ModA$ denote the categories of left $A$-modules and right $A$-modules, respectively.
The category of $(A,B)$-bimodules will be denoted $\AMod_B$.
As previously mentioned, we denote the tensor product over $\bk$ by $\otimes$, whereas we denote the tensor product over $A$ by $\otimes_A$.
The $\Hom$ sets for the categories $\Mod$ and $\AMod$ will be denoted by $\Hom$ and $\AHom$, respectively.
In the event that a direct sum of modules could be interpreted in several of these categories, we will denote the direct sum in $\AMod$ by $\Aoplus$, and similarly we will write $\oplusA$ or $\AoplusA$ for the direct sums in $\ModA$ or $\AModA$, respectively.
The full subcategories of $\AMod$ of flat, projective, and finitely generated left $A$-modules will be denoted by $\AFlat$, $\AProj$, and $\AFG$, respectively.
The intersection of $\AFG$ and $\AProj$ will be denoted by $\AFGP$.
We also denote by $\CalcA$ the category whose objects are first order differential calculi over $A$, and whose morphisms are bimodule maps that form a commutative triangle with the respective differentials.

\subsubsection{Generalization of symmetric forms}
The generalization of symmetric tensors from classical differential geometry to the noncommutative setting is essential to us.
Given an exterior algebra $\Omega^{\bullet}_d$, we define the tensor functors $S^0_d= \Omega^0_d=\id_{\AMod}$, $S^1_d= \Omega^1_d \colonequals \Omega^1_d \otimes_A -$, and $\Omega^2_d \colonequals \Omega^2_d \otimes_A -$, cf.\ §\ref{ss:extalg}.
Be aware that by abuse of notation we denote by $\Omega^1_d$ and $\Omega^2_d$ both the bimodules and their corresponding tensor functors.
Now by induction, the inclusion $\iota^n_{\wedge}\colon S^n_d\longhookrightarrow \Omega^1_d\circ S^{n-1}_d$ is defined by the kernel of the following composition
\begin{equation}
\Omega^1_d\circ S^{n-1}_d
\xrightarrow{\hphantom{AA}\Omega^1_d(\iota^{n-1}_{\wedge})\hphantom{AA}} \Omega^1_d\circ\Omega^1_d\circ S^{n-2}_d\xrightarrow{\hphantom{AA}\wedge_{S^{n-2}_d}\hphantom{AA}}
\Omega^2_d\circ S^{n-2}_d,
\end{equation}
where $\wedge_{S^{n-2}_d}$ denotes the $S^{n-2}_d$ component of the $\wedge$ as a natural transformation, i.e.\ $\wedge\otimes_A \id_{S^{n-2}_d}$.
We call $S^{\bullet}_d\colonequals \bigoplus_{n\ge 0} S^n_d$ the \emph{functor of symmetric forms}.

When no confusion arises, we denote $S^n_d (A)$ and $S^{\bullet}_d (A)$ also by $S^n_d$ and $S^{\bullet}_d$, respectively.
The latter will be termed \emph{bimodule of symmetric forms}
\begin{rmk}
Suppose $\Omega^1_d$ is free and finitely-generated as a left $A$-module, i.e.\ \emph{parallelizable}.
Given a basis $\theta_1,\dots,\theta_m$, we can define the partial derivative operator $\partial_i\colon A \rightarrow A$, by $da = \sum_i \partial_i(a)\theta_i$.

Furthermore, suppose $\Omega^1_d$ has a basis of exact forms $dx_1,\dots,dx_m$, i.e.\ it generalizes the cotangent bundle over a local chart.
	Then we have that
	\begin{equation}
		d^2a = \sum_{ij} \partial_i(\partial_j(a)) dx_i \wedge dx_j = 0.
	\end{equation}
	Which implies that $\sum_{ij} \partial_i(\partial_j(a)) dx_i \otimes_A dx_j \in S^2_d$.
	This shows that the symmetric forms $S^n_d$ generalize the classical property that partial derivatives commute in differential geometry.
\end{rmk}

\subsection{Main results}
In this article, we construct the following three families of functors:
\begin{itemize}
	\item The nonholonomic jet functors $J^{(n)}_d\colon \AMod \rightarrow \AMod$
	\item The semiholonomic jet functors $J^{[n]}_d\colon \AMod \rightarrow \AMod$
	\item The holonomic jet functors $J^n_d\colon \AMod \rightarrow \AMod$
\end{itemize}
In particular we have $J^{(0)}_d= J^{[0]}_d=J^0_d = \id_{\AMod}$, and $J^{(1)}_d= J^{[1]}_d=J^1_d $.
These functors come equipped with natural transformations
\begin{align}
	j^{(n)}_d\colon \id_{\AMod} \longrightarrow J^{(n)}_d,
	& \hfill &
	j^{[n]}_d\colon \id_{\AMod} \longrightarrow J^{[n]}_d,
	& \hfill &
	j^n_d\colon \id_{\AMod} \longrightarrow J^n_d,
\end{align}
which are respectively called the nonholonomic, semiholonomic, and holonomic jet prolongation maps.
We also have the natural transformations,
\begin{align}
	\pi^{(n,n-1; m)}_d\colon J^{(n)}_d\longrightarrow J^{(n-1)}_d,
	& \hfill &
	\pi^{[n,n-1]}_d\colon J^{[n]}_d\longrightarrow J^{[n-1]}_d,
	& \hfill &
	\pi^{n,n-1}_d\colon J^n_d \longrightarrow J^{n-1}_d,
\end{align}
respectively called the nonholonomic, semiholonomic, and holonomic jet projections.
Here, the index $m$ in the nonholonomic jet projection, ranging from $1$ to $n$, is the projection in position $m$ (cf.\ Definition \ref{def:otherprojections}).
\begin{theo*}[Stability. Proposition \ref{prop:Jnhstable}, Proposition \ref{prop:Jshstable}, and Proposition \ref{prop:Jhstable}]
Let $A$ be a $\bk$-algebra.
\begin{enumerate}
\item Let $\Omega^1_d$ be a first order differential calculus for $A$.
If $\Omega^1_d$ is either in $\AFlat$, $\AProj$, $\AFG$, or $\AFGP$, then $J^{(n)}_d$ preserves the subcategory $\AFlat$, $\AProj$, $\AFG$, or $\AFGP$, respectively;
\item Let $\Omega^1_d$ be a first order differential calculus for $A$ which is flat in $\ModA$.
If $\Omega^1_d$ is either in $\AFlat$, $\AProj$, $\AFG$, or $\AFGP$, then $J^{[n]}_d$ preserves the subcategory $\AFlat$, $\AProj$, $\AFG$, or $\AFGP$, respectively;
\item Let $\Omega^\bullet_d$ be an exterior algebra for $A$ such that the $n$-jet sequences are exact and $S^m_d$ is either in $\AFlat$, $\AProj$, or $\AFGP$ for all $1\le m \le n$, then $J^{n}_d$ preserves the subcategory $\AFlat$, $\AProj$, or $\AFGP$, respectively.
\end{enumerate}
\end{theo*}
This generalizes the fact that the jet bundle of a vector bundle is itself a vector bundle in differential geometry.
We show that our notion of $1$-jet generalizes the classical correspondence between connections and splittings of the $1$-jet short exact sequence.
\begin{nprop}{\ref{prop:connexionsplits}}[Left connections and splittings of the $1$-jet sequence]
Let $A$ be a $\bk$-algebra and $E$ be in $\AMod$.
There is a bijective correspondence between left connections on $E$ and left $A$-linear splittings of the $1$-jet short exact sequence corresponding to $E$.
\end{nprop}

We define the $1$-jet functor $J^1_d\colon \AMod\to \AMod$ as a tensor functor by a suitable $A$-bimodule $J^1_d A$ constructed in §\ref{ss:1jetA}.
The iterated application of the $1$-jet functor leads to the notion of nonholonomic jet functors
\begin{equation}
J^{(n)}_d
=(J^1_d)^{\circ n}
=J^1_d A\otimes_A \cdots \otimes_A J^1_d A\otimes_A -\colon \AMod\longrightarrow \AMod.
\end{equation}

We build the semiholonomic jets $J^{[n]}_d$ from the nonholonomic jets as the equalizer of the nonholonomic jet projections $\pi^{(n,n-1; m)}_d$.
In the following theorem we provide further equivalent characterizations using the map $\widetilde{\DH}^I_d$ (cf.\ §\ref{ss:shj}) obtained as the difference between two consecutive jet projections.
\begin{ntheo}{\ref{theo:sjchar}}[Characterization of the semiholonomic jet functor]
Let $A$ be a $\bk$-algebra endowed with a first order differential calculus $\Omega^1_d$, then for all $n\ge 0$, we have
\begin{equation}
J^{[n]}_d
=\bigcap_{m=1}^{n-1}\ker\left(J^{(n-m-1)}_d \widetilde{\DH}^I_{d,J^{(m-1)}_d} \right).
\end{equation}

Furthermore, if we assume $\Omega^1_d$ is flat in $\ModA$, then for $n\ge 2$ the following subfunctors of $J^{(n)}_d$ coincide
\begin{enumerate}
\item $J^{[n]}_d$;
\item The kernel of
\begin{equation}
\begin{tikzcd}
\widetilde{\DH}^I_{d,J^{[n-2]}_d}\circ J^1_d(l^{[n-1]}_d)\colon J^{1}_d\circ J^{[n-1]}_d\ar[r]& \Omega^1_d\circ J^{[n-2]}_d,
\end{tikzcd}
\end{equation}
where $l^{[n-1]}_d\colon J^{[n-1]}_d\hookrightarrow J^1_d\circ J^{[n-2]}_d$ denotes the natural inclusion ($l^{[1]}_d=\id_{J^1_d}$ and $l^{[0]}_d=\id_{A}$).
\item The intersection (pullback) of $J^{[h]}_d\circ J^{(n-h)}_d$ and $J^{(n-m)}_d \circ J^{[m]}_d$ in $J^{(n)}_d$ for $m,h\ge 0$ such that $h+m>n$.
\end{enumerate}
\end{ntheo}
The map $\widetilde{\DH}^I_d$ is a generalization of the map $\lambda$ from \cite[Proposition 4, p.~433]{Goldschmidt}.
It is also a first order differential operator, a notion which will be introduced in §\ref{s:firstorderdifferentialoperators}.
We also get the following short exact sequence, generalizing the corresponding result from the setting of classical differential geometry.
\begin{ntheo}{\ref{theo:Jshes}}[Semiholonomic jet short exact sequence]
Let $A$ be a $\bk$-algebra endowed with a first order differential calculus $\Omega^1_d$ which is flat in $\ModA$.
For $n\ge 1$, the following is a short exact sequence
\begin{equation}
\begin{tikzcd}[column sep=40pt]
	0\ar[r]& \bigotimes^n_A \Omega^1_d \ar[r,hookrightarrow,"\iota^{[n]}_d"]& J^{[n]}_d\ar[r,twoheadrightarrow,"\pi^{[n,n-1]}_d"]& J^{[n-1]}_d\ar[r]&0,
\end{tikzcd}
\end{equation}
where $\bigotimes^n_A \Omega^1_d\colon \AMod \rightarrow \AMod$ is the functor defined by $E \mapsto (\bigotimes^n_A \Omega^1_d) \otimes_A E$.
\end{ntheo}

In §\ref{ss:Spencer}, we introduce the noncommutative Spencer $\delta$-complex and its cohomology, which will allow us to state the conditions required for the exactness of the holonomic jet sequence.
\begin{ntheo}{\ref{theo:higherwolves}}[Holonomic jet exact sequence]
Let $A$ be a $\bk$-algebra endowed with an exterior algebra $\Omega^{\bullet}_d$ such that $\Omega^1_d$, $\Omega^2_d$, and $\Omega^3_d$ are flat in $\ModA$.
For $n\ge 1$, if the Spencer $\delta$-cohomology $H^{m,2}_{\delta_d}$ vanishes, for all $1\le m< n-2$, then the following sequence is exact,
\begin{equation}
\begin{tikzcd}
0\ar[r]&S^n_d\ar[r,hookrightarrow,"\iota^n_{d}"]& J^{n}_d\ar[r,"\pi^{n,n-1}_{d}"]&J^{n-1}_d\ar[r,"\dh^{n-1}_d"]&H^{n-2,2}_{\delta_d}.
\end{tikzcd}
\end{equation}
Therefore, if $H^{n-2,2}_{\delta_d}=0$ we obtain a short exact sequence
\begin{equation}
\begin{tikzcd}
0\ar[r]&S^n_d\ar[r,hookrightarrow,"\iota^n_{d}"]& J^{n}_d\ar[r,two heads,"\pi^{n,n-1}_{d}"]&J^{n-1}_d\ar[r]&0.
\end{tikzcd}
\end{equation}
\end{ntheo}

Finally, we show that our construction recovers, as a particular case, the theory of jets in differential geometry.
\begin{ntheo}{\ref{theo:classicalnonsemi} and \ref{theo:classicalnjet}}
	Let $A=\smooth{M}$ for a smooth manifold $M$, let $\Omega^{\bullet}_d = \Omega^{\bullet}(M)$ equipped with the de Rham differential $d$, and let $E$ be the space of smooth sections of a vector bundle.
	Then all the classical nonholonomic, semiholonomic, and holonomic $n$-jet bundles of $E$ are isomorphic to $J^{(n)}_dE$, $J^{[n]}_dE$, and $J^n_d E$ in $\AMod$, respectively, and the prolongation maps and jet projections are compatible with the isomorphisms.
\end{ntheo}

The holonomic jet functors together with its prolongation map allows us to define the notion of differential operator.
We give a characterization of linear differential operators of order at most $1$ in §\ref{s:firstorderdifferentialoperators} and some results on the category of linear differential operators of any order in §\ref{s:differentialoperators}.

\subsection{Further work and future developments}
Jet bundles provide the framework for the Lagrangian and Hamiltonian formalism in physics \cite{sardanashvily1994five}.
Thus, a necessary step for bringing these formalisms to the noncommutative setting is to obtain an appropriate definition of jet functor, which we aim to provide with this paper.

On the side of pure mathematics, the geometric theory of differential equations is formulated in the context of jet bundles.
Natural differential equations, including integrability conditions for geometric structures, are formulated via the jet bundles corresponding to the objects underlying said structures.
For example, the Nijenhuis tensor obstructing integrability of an almost complex structure, the exterior differential obstructing integrability of an almost symplectic form, or the Ricci tensor of a pseudo-Riemannian metric determining the Einstein property, all provide natural equations.
Hence, this article represents a step towards the intrinsic formulation of fundamentally noncommutative geometric integrability conditions.

The categorical approach to PDE, initially considered in \cite{Vinogradov}, can also be treated via the jet comonad and its corresponding Eilenberg-Moore category of coalgebras \cite{Marvan1}.
Our functorial approach is amenable to generalizing the latter.
We will treat this in a forthcoming article on infinite order differential operators and the jet comonad.
There exists another approach to PDE arising from algebraic geometric techniques, namely that of $\mathcal{D}$-modules and, more generally, $\mathcal{D}$-geometry (cf.\ \cite{beilinson2004chiral}).
To determine the relationship between the noncommutative generalizations of these notions would be an important structural result.

\subsection*{Acknowledgments}
We would like to thank R\'{e}amonn \'{O} Buachalla and Jan Slov\'{a}k for useful discussions.
Further, we gratefully acknowledge the support of the Czech Science Foundation via the project (GA\v{C}R) No.\ GX19-28628X.

\section{The $1$-jet functor}\label{s:1jetfunctor}
\subsection{Universal $1$-jet module}\label{ss:universal1jet}
Consider an algebra $A$ over $\bk$, endowed with the \emph{universal first order differential calculus} $\Omega^1_u$ (cf.\ \cite[Proposition 1.5, p.~5]{BeggsMajid}).
By definition of $\Omega^1_u$, we have the following short exact sequence.
\begin{equation}\label{es:diffses}
\begin{tikzcd}
0 \arrow[r] &\Omega^1_u \arrow[r]& A\otimes A \arrow[r,"\cdot"]& A \arrow[r]&0.
\end{tikzcd}
\end{equation}
Furthermore, it splits in the category $\ModA$, with right splitting provided by $a\mapsto 1\otimes a$.
Since any additive functor preserves split exact sequences, we apply the functor $-\otimes_A E\colon \AModA\to\AMod$ to \eqref{es:diffses}, where $E$ is a left $A$-module, to obtain the following short exact sequence in $\AMod$
\begin{equation}\label{es:jetses}
\begin{tikzcd}
0 \arrow[r] &\Omega_u^1 \otimes_{A}E \arrow[r]& J_u^1E \arrow[r,"\pi^{1,0}_{u,E}"]& E \arrow[r]&0.
\end{tikzcd}
\end{equation}
Here $J_u^1E\colonequals A\otimes A\otimes_A E\cong A\otimes E$.

\begin{defi}
	We define the {\em universal $1$-jet prolongation} $j_u^1\colon E\to J^1_u E=A\otimes E$ by the mapping $e$ to $1\otimes e$.
\end{defi}

\begin{rmk}
	If $E$ is a projective $A$-module, then $\Ext^1_A(E,-)$ vanishes, implying that there exists a unique extension via $E$ up to extension isomorphism, and that it is a split extension.
	This implies that there is a unique left $A$-module $J^1_u E$ satisfying \eqref{es:jetses}.
	Thus $J^1_u E=A\otimes E\cong E\Aoplus (\Omega^1_u\otimes_A E)$.
\end{rmk}
\begin{rmk}
The universal $1$-jet bimodule $J^1_u A$ is the free $A$-bimodule $A\otimes A$ generated by a single element.
\end{rmk}

\subsection{The $1$-jet module for non-universal first order differential calculi}
\label{ss:1jetA}
First, we make the following observation.
\begin{prop}
	Morphisms in $\CalcA$ are epimorphisms when considered in $\AModA$.
	Moreover, when they exist, morphisms between any two objects are unique.
\end{prop}
\begin{proof}
	Consider the morphism $\varphi$.
	We have the following commutative diagram.
	\begin{equation}\label{calcamorphism}
	\begin{tikzcd}
	&A\arrow[dl,"d"']\arrow[rd,"d'"]&\\
	\Omega^1_{d}\arrow[rr,rightarrow,"\varphi"]& &\Omega^1_{d'}\\[\vsfd]
	\end{tikzcd}
	\end{equation}
	Let $\sum_i a_i d' b_i \in \Omega^1_{d}$.
	By commutativity, $\varphi(d b_i)= d' b_i$ for all $i$, which gives uniqueness.
	Then, $\varphi(\sum_i a_i d b_i ) =\sum_i a_i \varphi(d b_i) = \sum_i a_i d' b_i$, showing the map is an epimorphism.
\end{proof}
Consider a first order differential calculus $\Omega^1_d$.
This can always be realized as a quotient of the universal first order differential calculus (cf.\ \cite[Proposition 1.5, p.~5]{BeggsMajid}), in particular, $\Omega^1_u$ is the initial object in $\CalcA$.
Explicitly,
\begin{equation}\label{eq:suruni}
\begin{tikzcd}
A\arrow[d,"d_u"']\arrow[rd,"d"]\\
\Omega^1_u\arrow[r,twoheadrightarrow,"p_d"]&\Omega^1_d\\[\vsfd]
\sum_i a_i\otimes b_i\arrow[r,mapsto]& \sum_i a_i db_i
\end{tikzcd}
\end{equation}
\begin{rmk}\label{rmk:pdother}
Since $p_d$ acts on elements $\sum_i a_i\otimes b_i\in \Omega^1_u$, we also have $p_d(\sum_i a_i\otimes b_i)=-\sum_i (da_i) b_i$.
\end{rmk}
Let $N_d\colonequals\ker(p_d)=\{\sum_i a_i\otimes b_i\in A\otimes A|\sum_i a_i b_i=0,\sum_i a_i d b_i=0\}\subseteq \Omega^1_u\subseteq A\otimes A$ be the kernel of the quotient projection.
We denote its inclusion by $\iota_{N_d}$.
Note that $N_d$ is an object in $\AModA$.

Below, we will construct the $1$-jet module $J^1_d A$ for the $A$-module $A$ so that it satisfies the jet short exact sequence
\begin{equation}\label{es:jetdA}
\begin{tikzcd}
0 \arrow[r] &\Omega_{d}^1 \arrow[r,hookrightarrow,"\iota^1_d"]& J_d^1 A \arrow[r,twoheadrightarrow,"\pi^{1,0}_{d}"]& A \arrow[r]&0.
\end{tikzcd}
\end{equation}
Since $A$ is projective in $\AMod$, \eqref{es:diffses} and \eqref{es:jetdA} split in $\AMod$.
This implies that $p_d$ and $\id_A$ extend to a map of extensions from \eqref{es:diffses} to \eqref{es:jetdA}.
For every such map, we obtain the following diagram, where columns and rows are exact
\begin{equation}\label{eq:1jetd}
\begin{tikzcd}
&N_d \ar[d,hookrightarrow,"\iota_{N_d}"] & \ker(\widehat{p}_d)\ar[d,hookrightarrow] & 0\ar[d]&\\
0 \arrow[r] &\Omega^1_{u}\arrow[d,"p_d",twoheadrightarrow]\arrow[r,hookrightarrow,"\iota^1_u"]& J^1_u A\colonequals A\otimes A\ar[d,"\widehat{p}_d"] \arrow[r,twoheadrightarrow,"\pi^{1,0}_u"]& A\arrow[d,equals] \arrow[r]&0\\
0 \arrow[r]& \Omega^1_{d}\ar[d]\arrow[r,hookrightarrow,"\iota^1_d"]& J^1_d A \ar[d]\arrow[r,twoheadrightarrow,"\pi^{1,0}_d"]& A\ar[d]\arrow[r]&0\\
&0& \coker(\widehat{p}_d)& 0 &
\end{tikzcd}
\end{equation}
By applying the snake lemma,
we deduce that $\ker(\widehat{p}_d)=N_d$ and that $\coker(\widehat{p}_d)=0$.
It follows that $\widehat{p}_d$ is an epimorphism, and thus we can give an explicit description of the $1$-jet module as
\begin{equation}
J^1_d A\colonequals (A\otimes A)/N_d.
\end{equation}
We then define the maps $\pi^{1,0}_d$ and $\iota^1_d$ to be the unique maps that commute in \eqref{eq:1jetd}.
Explicitly,
\begin{align}
\pi^{1,0}_d([a\otimes b])=ab,
&\hfill&
\iota^1_d\left(\sum_i a_i db_i\right)=\sum[a_i\otimes b_i-a_i b_i\otimes 1].
\end{align}
The quotient $J^1_d A$ inherits a natural $A$-bimodule structure since $N_d$ is a subbimodule.
Moreover, with the right action induced by the quotient, all maps in \eqref{eq:1jetd} are $A$-bilinear.
\subsubsection{Induced splitting and $1$-jet prolongation}\label{sss:Splitting}
The differential $d$ induces a left splitting of the sequence \eqref{es:diffses} in the category $\AMod$, via the map
\begin{align}
\widetilde d_u\colon J^1_u A \longrightarrow& \ \Omega^1_u,
&\hfill&
a\otimes b\longmapsto ad_u b.
\end{align}
This splitting descends to a canonical splitting of the sequence \eqref{es:jetdA} in $\AMod$.
Consider now the precomposition of $p_d\circ\widetilde{d}_u$ with the inclusion of $N_d$ in $J^1_u A$.
This yields $p_d\circ\widetilde{d}_u\circ(\iota^1_u\circ \iota_{N_d})=p_d\circ \iota_{N_d}=0$, thus $p_d\circ\widetilde{d}_u$ factors through $\widehat{p}_d$ via a map
\begin{equation}
	\begin{tikzcd}
		\widetilde{d}\colon J^1_d A \ar[r]& \Omega^1_d.
	\end{tikzcd}
\end{equation}

Next, we will consider the noncommutative generalization of the classical notion of $1$-jet prolongation.
\begin{defi}[$1$-jet prolongation]\label{def:1-jet prolongation}
	Given a first order differential calculus $\Omega^1_d$, we define the {\em $1$-jet prolongation map} $j^1_d\colon A \rightarrow J^1_d A$ by
\begin{equation}
j^1_d\colonequals \widehat{p}_d\circ j^1_u.
\end{equation}
\end{defi}
As the composition of two right $A$-linear morphisms, $j^1_d$ is evidently right $A$-linear.
It is not left $A$-linear in general, e.g.\ in the universal case.
\begin{rmk}
	As we will see (cf.\ Remark \ref{rmk:classicsplitaction}), classically, the left action of $A$ on $J^1_d A$ corresponds to the standard point-wise multiplication on sections of a vector bundle, whereas the right action is that for which $j^1_d$ is $A$-linear.
	The latter action has been remarked upon before, see \cite[p.~945]{Crainic}.
\end{rmk}
\begin{rmk}
Notice that $\widetilde{d}\circ j^1_d=\widetilde{d}\circ \widehat{p}_d\circ j^1_u=p_d\circ \widetilde{d}_u\circ j^1_u=p_d\circ d_u=d$.
	In §\ref{s:firstorderdifferentialoperators} we will see that $\widetilde{d}$ and $\widetilde{d}_u$ can be seen as the lifts of the differential operators $d$ and $d_u$, respectively, to their corresponding $1$-jet modules.
\end{rmk}

\begin{prop}
	The $1$-jet prolongation map provides a right splitting of \eqref{es:jetdA} in $\ModA$.
\end{prop}
\begin{proof}
	The map $j^1_u$ is right $A$-linear and $\widehat{p}_d$ is an $A$-bimodule map.
	Moreover, we have
	\begin{equation}		
		\pi^{1,0}_d \circ j^1_d
		=\pi^{1,0}_d \circ \widehat{p}_d\circ j^1_u
		=\pi^{1,0}_u \circ j^1_u
		=\id_A.
	\end{equation}
\end{proof}
We now describe the right $A$-linear left split corresponding to the right split $j^1_d$.
By the properties of the biproduct \cite[§2]{MacLane}, this split can be obtained from $\id_{J^1_d A}-j^1_d\circ \pi^{1,0}_d\colon J^1_d A\to J^1_d A$.
In fact, since composing this map with $\pi^{1,0}_d$ gives zero, the map factors through the kernel of $\pi^{1,0}_d$, i.e.\ $\Omega^1_d$, giving us
\begin{align}\label{def:rho}
\rho_d \colon J^1_d A \longrightarrow & \ \Omega^1_d,
&\hfill&
[a\otimes b] \longmapsto -(da)b.
\end{align}
The maps $\rho_d$ and $\rho_u$ are related, as we have
\begin{equation}
\iota^1_d\circ \rho_d\circ \widehat{p}_d
=\widehat{p}_d-j^1_d\circ \pi^{1,0}_d\circ \widehat{p}_d
=\widehat{p}_d-j^1_d\circ \pi^{1,0}_u
=\widehat{p}_d-\widehat{p}_d\circ j^1_u\circ \pi^{1,0}_u
=\widehat{p}_d\circ \iota^1_u\circ \rho_u
=\iota^1_d\circ p_d\circ \rho_u,
\end{equation}
whence
\begin{equation}
\rho_d\circ \widehat{p}_d
= p_d\circ \rho_u.
\end{equation}
\begin{rmk}\label{rmk:rhodo}
The map $\rho_d$ is right $A$-linear, but not necessarily left $A$-linear.
In fact, it behaves like a differential, as it satisfies the following Leibniz-like property (cf.\ Remark \ref{rmk:rhodif})
\begin{equation}
\rho_d(\lambda [a\otimes b])
=-d(\lambda a)b
=-(d\lambda)ab-\lambda (da)b
=\left(-(d\lambda)\pi^{1,0}_d+\lambda \rho_d\right)[a\otimes b].
\end{equation}
It is instead left $A$-linear if and only if $d=0$.
\end{rmk}

Notice that \eqref{es:jetdA} splits canonically both in $\AMod$ and in $\ModA$, but in general it does not split in $\AModA$.
\begin{rmk}\label{rmk:semidirectjet}
By using the splitting induced by $\widetilde{d}$, we can explicitly write the $1$-jet module as $A\ltimes' \Omega^1_d$, regarded as an $A$-module as $A\Aoplus \Omega^1_d$, with the right action defined by $(a,\alpha)\smbr f=(af,\alpha f+a df)$ for all $f\in A$.
The isomorphism is given by
\begin{align}
J^1_d A\longrightarrow A\ltimes' \Omega^1_d,
&\hfill&
[x\otimes y]\longmapsto (xy,xdy).
\end{align}

Analogously, we can use the splitting in $\Mod_A$, obtaining the bimodule $A\ltimes \Omega^1_d$, with left and right $A$-actions induced by the isomorphism
\begin{align}
J^1_d A\longrightarrow A\ltimes \Omega^1_d,
&\hfill&
[x\otimes y]\longmapsto (xy,-(dx)y).
\end{align}
Explicitly, the right action is the component-wise right action, and the left action is given by $f\smb(a,\alpha)=(fa,f\alpha-(df)a)$.
\end{rmk}

\subsection{Jet module on a left $A$-module}
\label{ss:jetgenericE}
The short exact sequence \eqref{es:jetses} shows that $\Omega^1_u\otimes_A E$ can be identified with the kernel of the scalar action $A\otimes E\to E$.
Now define $p_{d,E}\colonequals p_d\otimes_A \id_E$.
By Remark \ref{rmk:pdother} we have $p_{d,E}(\sum_i a_i\otimes e_i)=-\sum_i d a_i\otimes_A e_i$.
For later usage, we define
\begin{equation}\label{eq:definitionNdE}
N_{d}(E)\colonequals \ker(p_{d,E})
=\left\{\sum_i a_i\otimes e_i\in A\otimes E\middle|\sum_i a_i e_i=0,\sum_i d a_i \otimes_A e_i=0\right\}
\subseteq \Omega^1_u\otimes_A E
\subseteq A\otimes E.
\end{equation}

Next, we apply the functor $-\otimes_A E$ to \eqref{es:jetdA}.
Since $A$ is flat in $\ModA$, we have $\Tor^A_1(A,E)=0$, whereby we deduce the short exact sequence
\begin{equation}\label{es:jetd}
\begin{tikzcd}
0 \arrow[r] &\Omega_{d}^1 \otimes_{A}E \arrow[r,"\iota^1_{d,E}"]& J_d^1E \arrow[r,"\pi^{1,0}_{d,E}"]& E \arrow[r]&0,
\end{tikzcd}
\end{equation}
where $J_d^1E\colonequals J_d^1A\otimes_A E$, $\pi^{1,0}_{d,E}\colonequals\pi^{1,0}_{d}\otimes_A \id_{E}$, and $\iota^1_{d,E}\colonequals\iota^1_{d}\otimes_A \id_{E}$.
Since $j^1_d$ is right $A$-linear, and provides a right splitting for \eqref{es:jetdA} in $\ModA$, we define the $1$-prolongation $j^1_{d,E}\colon E\to J^1E$ to be $j^1_{d}\otimes_A \id_{E}$, which is a $\bk$-linear map and a right inverse for $\pi^{1,0}_{d,E}$.

We can also describe $J^1_d E$ in terms of $J^1_u E$ by applying the functor $-\otimes_A E$ to diagram \eqref{eq:1jetd} to obtain
\begin{equation}\label{diag:1jetfromuni}
\begin{tikzcd}
0 \arrow[r] &\Omega_{u}^1 \otimes_{A}E \arrow[r,"\iota^1_{u,E}"]\arrow[d,twoheadrightarrow,"p_{d,E}"]& J^1_u E \arrow[r,"\pi^{1,0}_{u,E}"]\arrow[d,twoheadrightarrow,"\widehat{p}_d\otimes_A \id_E"]&[20pt] E \arrow[r]\arrow[d,equal]&0\\
0 \arrow[r] &\Omega_{d}^1 \otimes_{A}E \arrow[r,"\iota^1_{d,E}"]& J^1_d E \arrow[r,"\pi^{1,0}_{d,E}"]& E \arrow[r]&0
\end{tikzcd}
\end{equation}
Since the maps $p_d$ and $\widehat{p}_d$ are surjective, and $-\otimes_A E$ is a right exact functor, the vertical maps are epimorphisms.
From the snake lemma, we have $N_{d}(E)=\ker(p_{d,E})\cong\ker(\widehat{p}_d\otimes_A\id_E)$, so $N_{d}(E)$ is also a submodule of $J^1_u E$.
Thus we have the following isomorphism
\begin{equation}\label{eq:J1dE=J1dA/NdE}
J^1_d E \cong J^1_u E/N_{d}(E).
\end{equation}

\begin{rmk}\label{rmk:flpj}
By right exactness of $-\otimes_A E$, in general $N_{d}(E)$ coincides with the image of $N_d\otimes E$ in $\Omega^1_u\otimes_A E$.
If $\Tor^1_A(\Omega^1_d, E)=0$ (e.g.\ if $E$ is flat in $\AMod$ or $\Omega^1_d$ is flat in $\ModA$), then $N_{d}(E)=\ker(p_d\otimes_A\id_E)=\ker(p_d)\otimes_A E=N_d\otimes_A E$.
Moreover, if $E$ is a projective left $A$-module, then up to isomorphism there is a unique $J^1_dE$ satisfying \eqref{es:jetd}, and the latter splits in $\AMod$.
\end{rmk}
\begin{rmk}
Due to Remark \ref{rmk:semidirectjet}, we have $J^1_d E= E\ltimes(\Omega^1_d \otimes_A E)$, where the left action is given by
\begin{equation}
a\smb(e_0,\alpha\otimes_A e_1)
=(ae_0,-(da)\otimes_A e_0+a\alpha\otimes_A e_1).
\end{equation}
\end{rmk}

\begin{prop}\label{prop:Nd-inclusion}
	Given a morphism $\varphi\colon \Omega^1_{d}\rightarrow \Omega^1_{d'} $ in $\CalcA$, then we have $N_d(E) \subseteq N_{d'}(E)$.
	Moreover, $\ker(\varphi\otimes_A \id_E)=N_{d'}(E)/N_d(E)$, and in particular $\ker(\varphi)=N_{d'}/N_d$.
\end{prop}
\begin{proof}
	Consider the following diagram.
	\begin{equation}\label{diag:butterfly}
	\begin{tikzcd}
		0 \arrow[r] &N_d(E) \arrow[r,hook]\arrow[d,dashed]& \Omega^1_u \otimes_A E \arrow[r,twoheadrightarrow,"p_{d,E}"]\arrow[d,equals]& \Omega^1_d\otimes_A E \arrow[r]\arrow[d,twoheadrightarrow,"\varphi\otimes_A \id_E"]&0\\
		0 \arrow[r] &N_{d'}(E) \arrow[r,hook]& \Omega^1_u \otimes_A E \arrow[r,two heads,"p_{d',E}"]& \Omega^1_{d'}\otimes_A E \arrow[r]&0
	\end{tikzcd}
\end{equation}
	By definition, $\varphi$ is $A$-bilinear and compatible with the differentials $d$ and $d'$.
	By these two properties, together with the surjectivity axiom of $\Omega^1_u$ and \eqref{eq:suruni}, we obtain $\varphi\circ p_d=p_{d'}$.
	By applying the functor $-\otimes_A E$ to this equality, we show that the right square commutes.
	The kernel universal property implies the existence and uniqueness of the leftmost vertical arrow, commuting in the left square.
	The snake lemma yields that this map is a monomorphism (in particular it is the inclusion map) and that $\ker(\varphi\otimes_A\id_E)=N_{d'}(E)/N_d(E)$.
	Finally, we obtain $\ker(\varphi)$ by setting $E=A$.
\end{proof}

\subsection{Functoriality}\label{ss:funct1j}
By definition, the $1$-jet module of a left $A$-module $E$ is obtained as the tensor $J^1_d A\otimes_A E$.
Given a left $A$-linear map $\phi\colon E\to F$, we can build a map $J^1_d \phi\colon J^1_d E\to J^1_d F$ defined by $\id_{J^1_{d} A}\otimes_A \phi$.
We can thus see the construction taking left $A$-modules to their $1$-jet as a functor
\begin{equation}
J^1_d\colonequals J^1_d A\otimes_A -\colon \AMod\longrightarrow \AMod.
\end{equation}
\begin{rmk}\label{rmk:exactnessJ1functor}
The functor $J^1_d$ is always right exact.
If $J^1_d A$ is flat in $\ModA$, then the functor $J^1_d$ is exact.
For example, if $\bk$ is a field or more generally if $A$ is a flat $\bk$-module, the functor $J^1_u$ is exact, since it can be viewed as the composition of exact functors $A\otimes_A -$ and $A\otimes -$.
\end{rmk}
Similarly, we can consider the functor $\Omega^1_d\otimes_A -\colon \AMod\longrightarrow\AMod$, also denoted $\Omega^1_d$.
In particular, we can write $\Omega^1_d (E)$ to simplify the notation for $\Omega^1_d \otimes_A E$.
For the same reason as in Remark \ref{rmk:exactnessJ1functor}, this functor is right exact, and it is exact precisely when $\Omega^1_d$ is flat in $\ModA$.

The families of $\pi^{1,0}_{d,E}$ and $\iota^1_{d,E}$ for $E$ in $\AMod$ provide two natural transformations, $\iota^1_d$ and $\pi^{1,0}_d$, respectively.
Analogously, we can consider another natural transformation, corresponding to the $1$-jet prolongation, defined for every $E$, as the $\bk$-linear map
\begin{align}\label{eq:1jp}
j^1_{d,E}\colonequals j^1_d\otimes_A \id_{E}\colon E\longrightarrow J^1_d E,
&\hfill&
e\longmapsto {[1\otimes 1]}\otimes_A e={[1\otimes e]}.
\end{align}
By definition, it forms a natural transformation $j^1_d\colon \id_{\AMod}\to J^1_d$, but seen as objects in the category of functors $\AMod\to \Mod$.

\begin{defi}
We call $\pi^{1,0}_d$ the \emph{$1$-jet projection} and $j^1_d$ the \emph{$1$-jet prolongation}.
\end{defi}
\begin{rmk}\label{rmk:jpg}
Since $J^1_d E$ is a quotient of $J^1_u E=A\otimes E$, where $j^1_d$ maps $e$ to $[1\otimes e]$, it follows that $Aj^1_d(E)=A\widehat{p}_d (1\otimes E)=J^1_d E$.
\end{rmk}
Furthermore, the $1$-jet prolongation is a section of $\pi^{1,0}_d$, since for every object $\pi^{1,0}_{d,E}\circ j^1_{d,E}=\id_{E}$, i.e.
\begin{equation}\label{eq:jsecpi}
\pi^{1,0}_{d}\circ j^1_{d}=\id,
\end{equation}
as natural transformations.

By the tensorial nature of the construction, if $E$ is an $(A,B)$-bimodule, $J^1_d E$ and $\Omega^1_d(E)$ will inherit a natural structure of $(A,B)$-bimodule.
We can thus define the functors $J^1_d,\Omega^1_d\colon\AMod_B\to \AMod_B$.
The functor on $\AMod_B$ and the functor on $\AMod$ are compatible with the forgetful functor, in the sense that the following diagrams commute
\begin{equation}\label{diag:compforgetbi}
\begin{tikzcd}
\AMod_B\ar[r,"J^1_d"]\ar[d]&\AMod_B\ar[d]&&\AMod_B\ar[r,"\Omega^1_d"]\ar[d]&\AMod_B\ar[d]\\
\AMod\ar[r,"J^1_d"]&\AMod&&\AMod\ar[r,"\Omega^1_d"]&\AMod
\end{tikzcd}
\end{equation}
We will see that generally, all constructions we do on $\AMod$ lift to categories of bimodules compatibly with the forgetful functor.
\begin{rmk}\label{rmk:jbim}
If $E$ is an $(A,B)$-bimodule for a $\bk$-algebra $B$, then $j^1_{d,E}$ is also right $B$-linear.
\end{rmk}

The results obtained in this section lead to the following.
\begin{prop}\label{prop:1jses}
The following is a short exact sequence in the category of additive endofunctors on $\AMod$.
\begin{equation}\label{es:natJ1}
\begin{tikzcd}
0 \arrow[r] &\Omega_{d}^1 \arrow[r,hookrightarrow,"\iota^1_{d}"]& J_d^1 \arrow[r,twoheadrightarrow,"\pi^{1,0}_{d}"]& \id_{(\AMod)} \arrow[r]&0.
\end{tikzcd}
\end{equation}
This sequence splits in the category $\AMod\to \Mod$ with canonical splitting $j^1_d$.
Furthermore, \eqref{es:natJ1} yields a short exact sequence of additive endofunctors on $\AMod_B$, which splits via $j^1_d$ in $\AMod_B\to \Mod_B$.
\end{prop}

As a consequence, the left splitting $\rho_d$ (cf.\ \eqref{def:rho}) induces a natural left splitting in the category of functors $\AMod\to\Mod$
\begin{equation}\label{eq:rhon}
\rho_{d,E}\colonequals\rho_d\otimes_A \id_E\colon J^1_d E\longrightarrow \Omega^1_d(E),
\end{equation}
which lifts to the category of functors $\AMod_B\to \Mod_B$ compatibly to the forgetful functor $\AMod_B\to \AMod$.
\begin{rmk}\label{rmk:proldop}
The jet prolongation satisfies in each component the following Leibniz rule (cf.\ Proposition \ref{prop:1diffop})
\begin{equation}
j^1_{d,E}(a e)
=[1\otimes ae]
=[a\otimes e]-[a\otimes e]+[1\otimes ae]
=-\rho_{d,E}[a\otimes e]+aj^1_{d,E}(e)
=da\otimes_A e+aj^1_{d,E} e.
\end{equation}
Cf.\ \cite[p.~945]{Crainic} (note that here the authors use a nonstandard sign convention for the embedding of forms into jets).
\end{rmk}
\subsection{Flatness, projectivity, and exactness}
We start with some preliminary results concerning extensions by flat and projective modules.
\begin{lemma}\label{lemma:2-3}
In the abelian categories $\AMod$ and $\ModA$ consider the following short exact sequence
\begin{equation}\label{es:2-3}
\begin{tikzcd}
0\ar[r]& M\ar[r,hookrightarrow]& N\ar[r,twoheadrightarrow]& Q\ar[r]& 0.
\end{tikzcd}
\end{equation}
\begin{enumerate}
\item Suppose $Q$ is flat, then $M$ is flat if and only if $N$ is flat.
\item Suppose $Q$ is projective, then $M$ is projective if and only if $N$ is projective.
\item Suppose $Q$ and $M$ are finitely generated, then so is $N$.
\end{enumerate}
\end{lemma}
\begin{proof}
We prove the result in $\AMod$.
The proof for $\ModA$ is identical.
\begin{enumerate}
\item Given a right $A$-module $E$, applying the functor $E\otimes_A -$ to the short exact sequence \eqref{es:2-3} yields the long exact sequence
\begin{equation}
\begin{tikzcd}
\Tor^A_{k+1}(E,Q)\ar[r,"\delta"]& \Tor^A_k(E, M)\ar[r]& \Tor^A_k(E, N)\ar[r,twoheadrightarrow]& \Tor^A_k(E, Q)\ar[r]& 0.
\end{tikzcd}
\end{equation}
Since $Q$ is flat in $\AMod$, $\Tor^A_k(E,Q)=0$ for $k>0$.
Thus, this long exact sequence provides an isomorphism $\Tor^A_k(E, M)\cong \Tor^A_k(E, N)$.
In particular, $M$ is flat if and only if $N$ is flat.
\item If $Q$ is projective, then \eqref{es:2-3} splits in $\AMod$.
Therefore $N\cong M\Aoplus Q$.
A direct sum is projective if and only if each summand is projective (cf.\ \cite[Proposition 2.1, p.~6]{CE}).
Thence follows the statement.
\item Since $Q$ and $M$ are finitely generated, we have epimorphisms $A^q\twoheadrightarrow Q$ and $A^m\twoheadrightarrow M$ for $q,m\in\N$.
Proceeding as in the proof of the horseshoe lemma (cf.\ \cite[Horseshoe Lemma 2.2.9, p.~37]{Weibel}), we know that there is an epimorphism $A^{m+q}\twoheadrightarrow N$, proving that $N$ is finitely generated.
\qedhere
\end{enumerate}
\end{proof}
\begin{rmk}
Observe that an analogue of Lemma \ref{lemma:2-3}, wherein one replaces ``flat'' (resp.\ ``projective'') with ``free'', does not hold.

Moreover, if one assumes $M$ and $N$ to be flat (resp.\ projective), $Q$ is not necessarily flat (resp.\ projective).
To see this let $A=\bk[x]$, then, in the sequence
\begin{equation}
\begin{tikzcd}
0\ar[r]& \bk[x]\ar[r,hookrightarrow,"x"]& \bk[x]\ar[r,twoheadrightarrow]& \bk[x]/(x)\cong \bk \ar[r]& 0,
\end{tikzcd}
\end{equation}
the module $\bk[x]$ is free, and hence flat and projective, but $\bk$ is neither projective nor flat as a $\bk[x]$-module, since the injective map $x\colon \bk[x]\to \bk[x]$, multiplication by $x$, becomes the zero map $0\colon \bk\to\bk$.
\end{rmk}
\begin{lemma}\label{lemma:flpr}
Let $A$, $B$, and $C$ be $\bk$-algebras.
Let $M$ be an $(A,B)$-bimodule and let $N$ be a $(B,C)$-bimodule.
\begin{enumerate}
	\item If $M$ is in $\AMod_B$ and flat in $\AMod$ (respectively, in $\Mod_B$), and $N$ is in ${}_B\!\Mod_C$ and flat in ${}_B\!\Mod$ (respectively, in $\Mod_C$), then $M\otimes_B N$ is flat in $\AMod$ (respectively, in $\Mod_C$);
	\item If $M$ is in $\AMod_B$ and projective in $\AMod$ (respectively, in $\Mod_B$), and $N$ is in ${}_B\!\Mod_C$ and projective in ${}_B\!\Mod$ (respectively, in $\Mod_C$), then $M\otimes_B N$ is projective in $\AMod$ (respectively, in $\Mod_C$);
	\item If $M$ is in $\AMod_B$ and finitely generated in $\AMod$ (respectively, in $\Mod_B$), and $N$ is in ${}_B\!\Mod_C$ and finitely generated in ${}_B\!\Mod$ (respectively, in $\Mod_C$), then $M\otimes_B N$ is finitely generated in $\AMod$ (respectively, in $\Mod_C$).
\end{enumerate}
\end{lemma}
\begin{proof}\ 
\begin{enumerate}
\item The functor $M\otimes_B N\otimes_C -$ is the composition of two exact functors and so is the functor $-\otimes_A M\otimes_B N$.
\item Cf.\ \cite[Proposition 5.3, p.~28]{CE};
\item Since $N$ is finitely generated, there is a epimorphism $B^n\twoheadrightarrow N$ in ${}_B\!\Mod$.
Applying the right exact functor $M\otimes_B -$, we obtain the epimorphism $M^n\cong M\otimes_B B^n\twoheadrightarrow N$ in $\AMod$.
Since there is an $n$ such that $A^m\twoheadrightarrow M$ is an epimorphism in $\AMod$, we also have an epimorphism $A^{nm}\twoheadrightarrow N$, implying that $N$ is finitely generated in $\AMod$.
\qedhere
\end{enumerate}
\end{proof}
\begin{rmk}
In the previous lemma, if one or more among $A$, $B$, or $C$ is $\bk$, we obtain further results for left or right modules, as $\AMod_{\bk}=\AMod$ and ${}_{\bk}\!\ModA=\ModA$.
\end{rmk}
\begin{cor}\label{cor:uniflat}
If $A$ is flat (resp.~projective) in $\Mod$, then $\Omega^1_u$ is flat (resp.\ projective) in $\AMod$ and $\ModA$.
\end{cor}
\begin{proof}
By Lemma \ref{lemma:flpr}, $A\otimes A$ is flat (resp.\ projective), since $A$ is flat (resp.\ projective) as a $\bk$-module and free as a right and left $A$-module.
It follows that $\Omega^1_u$ is projective by applying Lemma \ref{lemma:2-3} to the sequence \eqref{es:diffses}.
\end{proof}
For later reference, we also mention the following lemma.
\begin{lemma}\label{lemma:flprsum}
Let $A$ be a $\bk$-algebra and $I$ a set, and let $M_i$ be in $\AMod$ (respectively, $\ModA$) for all $i\in I$.
\begin{enumerate}
\item The module $\bigoplus_{i\in I}M_i$ is flat if and only if $M_i$ is flat for all $i\in I$;
\item The module $\bigoplus_{i\in I}M_i$ is projective if and only if $M_i$ is projective for all $i\in I$.
\end{enumerate}
\end{lemma}
\begin{proof}\ 
\begin{enumerate}
\item The $\Tor^A_n$ functors preserves direct sums \cite[Corollary 2.6.11, p.~55]{Weibel}, and a direct sum is zero if and only if each summand is zero.
\item Cf.\ \cite[Proposition 2.1, p.~6]{CE}.\qedhere
\end{enumerate}
\end{proof}

We can now educe the following proposition.
\begin{prop}\label{prop:OmJstab}
Consider the functors
\begin{align}
\Omega^1_d\colon\AMod\longrightarrow \AMod,
&\hfill &
J^1_d\colon\AMod\longrightarrow \AMod.
\end{align}
\begin{enumerate}
\item\label{prop:OmJstab:1} If $\Omega^1_d$ is in $\AFlat$, then the functors $\Omega^1_d$ and $J^1_d$ preserve the subcategory $\AFlat$;
\item\label{prop:OmJstab:2} If $\Omega^1_d$ is in $\AProj$, then the functors $\Omega^1_d$ and $J^1_d$ preserve the subcategory $\AProj$;
\item\label{prop:OmJstab:3} If $\Omega^1_d$ is in $\AFG$, then the functors $\Omega^1_d$ and $J^1_d$ preserve the subcategory $\AFG$;
\item\label{prop:OmJstab:4} If $\Omega^1_d$ is in $\AFGP$, then the functors $\Omega^1_d$ and $J^1_d$ preserve the subcategory $\AFGP$.
\end{enumerate}
\end{prop}
\begin{proof}
The proof is identical for the first three cases, so we prove it for \eqref{prop:OmJstab:1}.
The statement for $\Omega^1_d$ follows from Lemma \ref{lemma:flpr}, and the statement for $J^1_d$ follows from Lemma \ref{lemma:2-3} applied to the $1$-jet short exact sequence \eqref{es:jetd}.
The statement \eqref{prop:OmJstab:4} is obtained by combining \eqref{prop:OmJstab:2} and \eqref{prop:OmJstab:3}.
\end{proof}

As an immediate consequence of the definition of the $1$-jet functor, we obtain the following result concerning exactness.
\begin{lemma}\label{lemma:derfun}
The following properties hold
\begin{enumerate}
\item\label{lemma:derfun:1} $\Omega^1_d$ and $J^1_d$ are left adjoints and hence preserve all colimits;
\item\label{lemma:derfun:2} the left derived functors of $\Omega^1_d$ and $J^1_d$ are $L_k\Omega^1_d=\Tor^A_k(\Omega^1_d,\cdot)$ and $L_k J^1_d=\Tor^A_k(J^1_d A,\cdot)$;
\item\label{lemma:derfun:3} $\iota^1_d$ induces a natural isomorphism $L_k\Omega^1_d\to L_k J^1_d$ for $k>0$.
\end{enumerate}
\end{lemma}
\begin{proof} \ 
\begin{enumerate}
\item The two functors are of the form $M\otimes_A -$ for an $A$-bimodule $M$, which has $\AHom(M,-)$ as a right adjoint.
The image of this functor is in $\AMod$, since for a left $A$-module $N$ we can define the left action of $A$ on $\AHom(M,N)$ by precomposition with the right action of $A$ on $M$.
\item Follows directly from their definitions as tensor functors.
\item By applying the functor $-\otimes_A E$ to \eqref{es:jetdA}, we obtain the corresponding long exact sequence, which in particular contains the exact sequences
\begin{equation}
\begin{tikzcd}
\Tor^A_{k+1}(A,E)\ar[r,"\delta"] &\Tor^A_k(\Omega^1_d,E)\ar[r,"\iota^1_d"] &\Tor^A_k(J^1_d A,E)\ar[r,"\pi_d^{1,0}"] &\Tor^A_k(A,E).
\end{tikzcd}
\end{equation}
Since $A$ is $A$-flat, $\Tor^A_k(A,E)=0$ for all $k>0$, and thus the central map $\Tor^A_k(\iota^1_d,E)$ is an isomorphism.
Naturality follows from that of $\iota^1_d$.\qedhere
\end{enumerate}
\end{proof}

\begin{cor}\label{cor:Jex}
The following are equivalent
\begin{enumerate}
\item $\Omega^1_d$ is an exact functor;
\item $J^1_d$ is an exact functor;
\item $\Omega^1_d$ is a flat right $A$-module;
\item $J^1_dA$ is a flat right $A$-module;
\item $\Tor_1^A(\Omega^1_d,E)=0$ for all $E$ in $\AMod$;
\item $\Tor_1^A(J^1_d A,E)=0$ for all $E$ in $\AMod$.
\end{enumerate}
\end{cor}

Under a mild regularity condition on $A$, we can explicitly compute $\Tor^A_1(\Omega^1_d,E)$ from $N_d$ by the following.
\begin{prop}
Let $A$ be flat as a $\bk$-module (e.g.\ if $\bk$ is a field).
Let $E$ be a left $A$-module, then $\Tor^A_1(\Omega^1_d,E)\cong\ker(\iota_{N_d}\otimes_A \id_E)$.
\end{prop}
\begin{proof}
Consider the short exact sequence appearing as the left column in \eqref{eq:1jetd}.
Applying the functor $-\otimes_A E$ to it, we obtain the long exact sequence
\begin{equation}
\begin{tikzcd}
\Tor^A_1(\Omega^1_u,E)\ar[r]& \Tor^A_1(\Omega^1_d,E)\ar[r]& N_d\otimes_A E\ar[r,"\iota_{N_d}\otimes_A \id_E"]&[10pt] \Omega^1_u\otimes_A E\ar[r,twoheadrightarrow,"p_d\otimes_A \id_E"]&[10pt] \Omega^1_d\otimes_A E \ar[r]& 0.
\end{tikzcd}
\end{equation}
The statement follows from the fact that $\Tor^A_1(\Omega^1_u,E)=0$ by Corollary \ref{cor:uniflat}.
\end{proof}

\section{Explicit examples}
\subsection{Classical manifolds}\label{ss:cman}
Let $M$ be a smooth manifold and consider the case $\bk=\R$ and $A=\smooth{M}$ with the exterior de Rham differential $d\colon \smooth{M}\to \Omega^1(M)$, regarded as a first order differential calculus.
Let $N\to M$ be a (finite rank) vector bundle, and let $E=\Gamma(M,N)$ be its $\smooth{M}$-module of sections.
By the Serre-Swan theorem, $E$ is a (finitely generated) projective $\smooth{M}$-module (cf.\ \cite[§12.32, p.~189]{nestruev2020smooth}).
By Remark \ref{rmk:flpj} there is a unique left $\smooth{M}$-module satisfying the first jet exact sequence, and we will now realize it as the module of sections of the classical jet bundle $J^1 N$.
The universal first order differential calculus is given by 
\begin{equation}
\Omega^1_u=\left\{\sum_i f_i\otimes g_i\in \smooth{M}\otimes\smooth{M}\middle|\sum_i f_i g_i=0\right\},
\end{equation}
equipped with the universal differential defined for all $f\in\smooth{M}$ as $d_u f=1\otimes f-f\otimes 1$.

Consider the left $\smooth{M}$-linear map, in the version provided by Remark \ref{rmk:pdother}.
\begin{align}
p_{d,E}\colon \Omega^1_u(E)\longrightarrow \Omega^1(M,N),
&\hfill&
\sum_i f_i\otimes \sigma_i\in\Omega^1_u(E)\subseteq \smooth{M}\otimes\Gamma(M,N)\longmapsto -\sum_i df_i\otimes \sigma_i.
\end{align}
The last tensor is written according to the standard notation of differential geometry, although in our notation it would be over the algebra $\smooth{M}$.

We know that classically the $\smooth{M}$-module of global sections of the jet bundle $J^1 N$ satisfies \eqref{eq:kjetses} for $n=1$.
By the Serre-Swan theorem, this sequence splits, and thus we can find a $\smooth{M}$-linear map $\widehat{p}_{d,E}$ such that it commutes in the diagram \eqref{diag:1jetfromuni}.
Once we find any such map, the construction in §\ref{ss:jetgenericE} yields the isomorphism between the classical jet $\smooth{M}$-module and ours.
In order to build the map $\widehat{p}_{d,E} \colon J^1_u E\to \Gamma (M,J^1 N)$ we use the $\R$-linear splitting provided by $j^1_u$ and $\rho_u$ on the universal jet exact sequence, and the splitting given by the classical $1$-jet prolongation $j^1\colon f\mapsto [f]^1$ on the classical jet exact sequence.
We define $\widehat{p}_{d,E}$ as follows
\begin{align}
\widehat{p}_{d,E}
=j^1\circ \pi^{1,0}_u+\iota_d\circ p_d\circ \rho_u\colon J^1_u E\longrightarrow \Gamma(M,J^1 N),
&\hfill&
f\otimes \sigma\longmapsto [f\sigma]^1-(df)\otimes \sigma.
\end{align}
\begin{rmk}\label{rmk:convention}
In our framework, we adopt the sign convention of \cite[Proposition 1.5, p.~5]{BeggsMajid} for the universal differential, which realizes the standard embedding of $1$-forms into sections of the $1$-jet bundle from differential geometry (e.g.\ \cite{NaturalOperations}).
In principle one could adopt the convention of \cite[III,§10.10]{bourbaki2007algebre} instead, obtaining the opposite sign convention used in \cite{Crainic}.
\end{rmk}
We now have $j^1(f\sigma)=fj^1(\sigma)+df\otimes \sigma$ (cf.\ \cite[p.~945]{Crainic}, keeping in mind their nonstandard sign convention).
This implies that the map $\widehat{p}_d$ is left $\smooth{M}$-linear.

By construction, this map makes the following diagram commute.
\begin{equation}
\begin{tikzcd}
0 \arrow[r] &\Omega_{u}^1(E) \arrow[r,"\iota^1_{u,E}"]\arrow[d,twoheadrightarrow,"p_{d,E}"]& J^1_uE \arrow[r,"\pi^{1,0}_{u,E}"]\arrow[d,twoheadrightarrow,"\widehat{p}_{d,E}"]& E \arrow[r]\arrow[d,equal]&0\\
0 \arrow[r] &\Omega^1(M,N) \arrow[r]& \Gamma(M,J^1 N) \arrow[r,"\pi^{1,0}"]& \Gamma(M,N) \arrow[r]&0
\end{tikzcd}
\end{equation}
Using the snake lemma, $\widehat{p}_d$ is surjective and its kernel coincides with the kernel of $p_d$ in $J^1_u E$.
The explicit isomorphism with the classical jet is thus given by
\begin{align}\label{eq:classicaliso}
[f\otimes \sigma]\in J^1_d E\longmapsto [f\sigma]^1-df\otimes \sigma\in \Gamma(M,J^1 N).
\end{align}
Notice also that, by construction, the classical $1$-jet prolongation $j^1\colon \Gamma(M,N)\to \Gamma(M,J^1 N)$ corresponds with the one given in Definition \ref{def:1-jet prolongation}, as
\begin{equation}
j^1_d (\sigma)
=\widehat{p}_d \circ j^1_u(\sigma)
=\widehat{p}_d(1\otimes \sigma)
=[\sigma]^1.
\end{equation}

In particular, on the trivial bundle $N=M\times \R$, we have that $\Gamma(M,J^1 N)=J^1_d \smooth{M}=J^1_d A$.
We can now use the epimorphism $\widehat{p}_d$ to compute the right action induced on the module $\Gamma(M,J^1(M\times \R))$.
Let $f\otimes g\in J^1_u \smooth{M}=\smooth{M}\otimes \smooth{M}$, and let $a\in \smooth{M}$.
\begin{equation}
\widehat{p}_d([f\otimes g]a)
=\widehat{p}_d[f\otimes ga]
=[fga]^1-(df) ga
=a[fg]^1+(da)fg-(df)ga
=\widehat{p}_d([f\otimes g])a+da\otimes \pi^{1,0}\circ \widehat{p}_d[f\otimes g]
\end{equation}
This implies that the right action of $g\in\smooth{M}$ on a generic $1$-jet of the form $[f]^1+\alpha$, has the form
\begin{equation}
([f]^1+\alpha)\cdot g
=[f]^1 g+\alpha g-(dg) f.
\end{equation}
\begin{rmk}\label{rmk:classicsplitaction}
This is the action that makes the jet prolongation $\smooth{M}$-linear (cf.\ \cite[p.~945]{Crainic}).
This action in general is different from the standard point-wise action of $\smooth{M}$ on sections of a fiber bundle.
Our interpretation thus displays an intrinsic noncommutative nature in the notion of jet bundle.
\end{rmk}
We summarize this example with the following theorem.
\begin{theo}\label{theo:classical1jet}
	Let $M$ be a smooth manifold, $A=\smooth{M}$ its algebra of smooth functions, and $E=\Gamma(M,N)$ the space of smooth sections of a vector bundle $N\rightarrow M$.
	Then $J^1_dE \simeq \Gamma(M,J^1N)$ in $\AMod$, and the isomorphism takes the prolongation $j^1_d(e)$ of a section $e$ to its jet class $[e]^1\in \Gamma(M,J^1N)$.
\end{theo}
From now on, when we need to treat jets in classical differential geometry, we will treat our first jet functor as equal to the classical one.

In order to compute $\ker(p_d)$ we use $\widetilde{d}_u$ to describe the universal first order differential calculus $\Omega^1_u=\widetilde{d}_u(\smooth{M}\otimes\smooth{M})=\spn{f\otimes g-fg\otimes 1|f,g\in\smooth{M}}_{\R}$.
Notice that for all $p\in M$, we have
\begin{equation}
\sum_i (f_i\otimes g_i-f_ig_i\otimes 1)
=\sum_i (f_i\otimes (g_i-g_i(p))-f_i(g_i-g_i(p))\otimes 1),
\end{equation}
and now $g_i-g_i(p)\in\mathfrak{m}_p$.
This implies that for all $p\in M$, the restriction of $\widetilde{d}_u$ to $\smooth{M}\otimes\mathfrak{m}_p$ is surjective.
It is also injective, hence an isomorphism, as we can build a left inverse.
Consider the evaluation in the second tensorial component composed with the multiplication
\begin{align}
\ev_p^2\colon \smooth{M}\otimes\smooth{M}\longrightarrow \smooth{M},
&\hfill&
a\otimes b\longmapsto a\cdot b(p).
\end{align}
The left inverse of $\widetilde{d}_u$ is then $\id_{\smooth{M}\otimes \smooth{M}}-ev^2_p\otimes 1\colon \Omega^1_u\to \smooth{M}\otimes \mathfrak{m}_p$.
This realizes for all $p\in M$ a left $\smooth{M}$-isomorphism
\begin{equation}
\Omega^1_u\cong\smooth{M}\otimes \mathfrak{m}_p.
\end{equation}

If we now take a generic element $\sum_i (f_i\otimes g_i-f_i g_i\otimes 1)$ with $f_i,g_i\in\smooth{M}$ and $g_i(p)=0$, then this is in $\ker(p_d)$ if and only if
\begin{equation}
0=p_d\left(\sum_i (f_i\otimes g_i-f_i g_i\otimes 1)\right)
=\sum_i f_i dg_i.
\end{equation}

\begin{rmk}
	Recall that $N_d = \ker(\widehat{p}_d) $.
	In local coordinates $x^1,\ldots,x^n$, and for a smooth function $f$, we have $d_u f-\sum_{i=1}^n\frac{\partial f}{\partial x^i}d_u x^i \in N_d \subset \Omega^1_u$.
	Thus, $N_d$ can be seen as encoding information about the differential relationships between algebra elements, similarly to the Cartan distribution on the classical jet bundle.
\end{rmk}

\subsection{The infinitesimal first order differential calculus at $0$}
\label{ss:K2}
We will now build an example that will be useful as a pathological first order differential calculus.

Let $\bk$ be a field and $A=\bk[t]/(t^2)$, which as a $\bk$-module is $A=\bk+\bk t\cong\bk^2$.
Consider $\bk[0]\colonequals \bk[t]/[t]$, which as a $\bk$-module is isomorphic to $\bk$, and as an $A$-module, $p(t)\in A$ acts by multiplication with $p(0)$.
Let $\Omega^1_d=\bk[0]$ and define the differential $d$ by setting $dt=1$.
We can think of this first order differential calculus as assigning to each polynomial its formal derivative at $0$.

Now we will build the $1$-jet, starting from the universal jet exact sequence.
We identify $J_u A=A\otimes A$ with the algebra $\bk[x,y]/(x^2,y^2)$, by the isomorphism $p(t)\otimes q(t)\mapsto p(x)q(y)$.
The action of $a(t),b(t)\in A$ on $p(x,y)\in J^1_u A$ is given by $a(t)p(x,y)b(t)=a(x)p(x,y)b(y)$.
The projection $\pi^{1,0}_u$, that is the multiplication on $A$, can be identified with the map
\begin{align}
\bk[x,y]/(x^2,y^2)\longrightarrow \bk[t]/(t^2),
&\hfill&
p(x,y)\longmapsto p(t,t).
\end{align}
We can thus realize $\Omega^1_u=\ker(\pi^{1,0}_u)=(y-x,xy)$ as an ideal of $\bk[x,y]/(x^2,y^2)$.
Recalling the definition of $\Omega^1_u\subset A\otimes A$, we see that $d_u\colon A\to \Omega^1_u$ corresponds to the map $d_u t=y-x$.
The two generators of $\Omega^1_u\subset A\otimes A$ correspond to $d_u t$ and $td_u t$, respectively.
Hence, the ideal $\Omega^1_u$ is principal and generated by $d_u t=y-x$.

The first order jet exact sequence for the universal first order differential calculus is thus
\begin{equation}\label{es:1jK2}
\begin{tikzcd}
0\ar[r]& (d_u t=x-y) \ar[r,hookrightarrow]& \bk[x,y]/(x^2,y^2) \ar[r,twoheadrightarrow]&A \ar[r]&0.
\end{tikzcd}
\end{equation}

We can now compute $p_d\colon \Omega^1_u\to\Omega^1_d$, defined by the mapping
\begin{equation}
p(t)d_u t\longmapsto p(t)\cdot dt=p(0).
\end{equation}
It follows that $N_d=\ker(p_d)=(td_u t)=(xy)$.
Finally, we obtain
\begin{equation}
J^1_d A
=J^1_u A/N_d
\cong\bk[x,y]/(x^2,y^2,xy).
\end{equation}
which as a $\bk$-module is isomorphic to $\bk+\bk x+ \bk y\cong \bk^3$.
The first jet exact sequence is thus
\begin{equation}
\begin{tikzcd}[row sep=tiny]
0 \ar[r] &\bk[0] \arrow[r,hookrightarrow,"\iota^1_d"]& \bk[x,y]/(x^2,y^2,xy) \arrow[r,twoheadrightarrow,"\pi^{1,0}_d"]& \bk[t]/(t^2) \arrow[r]&0\\
&a\arrow[r,|->]& a(y-x)&\\
&&p(x,y)\arrow[r,|->]& p(t,t)&
\end{tikzcd}
\end{equation}

The universal prolongation is given by
\begin{align}
j^1_d\colon A\longrightarrow \bk[x,y]/(x^2,y^2,xy),
&\hfill&
p(t)\longmapsto p(y).
\end{align}

The reasons why this first order differential calculus is pathological is due to the following result.
\begin{lemma}\label{lemma:K2}
$\Tor^A_n(\bk[0],\bk[0])=\bk[0]$, thus in particular $\bk[0]$ is neither flat in $\AMod$ nor in $\ModA$.
\end{lemma}
\begin{proof}
The $\Tor^A_{\bullet}$ functor can be computed from a free resolution of $\bk[0]$, i.e.\ an exact sequence
\begin{equation}
\cdots \to F_2\xrightarrow{\delta_2} F_1\xrightarrow{\delta_1} F_0\xrightarrow{\hphantom{\delta_0}} \bk[0]\to 0
\end{equation}
such that $F_n$ is free in $\AMod$ for all $n\ge 0$.
Choose the following
\begin{equation}\label{es:fresK2}
\cdots \to A\xrightarrow{t\cdot}A\xrightarrow{t\cdot}A\xrightarrow{\ev_0}\bk[0]\to 0,
\end{equation}
where $\ev_0$ is the evaluation at $0$.
Notice that it is a complex, since $t^2=0$.
Moreover, it is exact because $\ev_0$ is surjective, its kernel is $(t)$, and $\ker(t\cdot)=\im(t\cdot)=(t)\subset A$.

We can now compute $\Tor^A_{\bullet}(\bk[0],\bk[0])$ as the homology of the complex obtained by applying $\bk[0]\otimes_A -$ to $F_{\bullet}\to 0$, i.e.\ \eqref{es:fresK2} where $\bk[0]$ is substituted with $0$.
Observe that $\bk[0]\otimes_A A=\bk[0]$ and
\begin{equation}
\id_{\bk[0]}\otimes_A t\cdot
=\id_{\bk[0]}\cdot t\otimes_A\id_A
=0\otimes_A \id_A
=0.
\end{equation}
Therefore, the desired complex is
\begin{equation}
\bk[0]\otimes_A F_{\bullet}=\cdots \xrightarrow{0}\bk[0]\xrightarrow{0}\bk[0]\xrightarrow{0}\bk[0]\to 0.
\end{equation}
It follows that
\begin{equation}
\Tor^A_n(\bk[0],\bk[0])
=H_n(\bk[0]\otimes_A F_{\bullet})
=\bk[0].
\end{equation}

From $\Tor^A_1(\bk[0],\bk[0])=\bk[0]\neq 0$ it follows that $\bk[0]$ is not flat in either $\AMod$ or $\ModA$.
\end{proof}

As a consequence, we obtain the following result
\begin{prop}
$\Tor^A_n(J^1_d A,\Omega^1_d)=\Tor^A_n(\Omega^1_d A,\Omega^1_d)=\bk[0]$ for all $n> 0$, and the functors $J^1_d$ and $\Omega^1_d$ are not exact.
\end{prop}
\begin{proof}
It follows from Lemma \ref{lemma:derfun}.
\end{proof}

\section{First order differential operators}\label{s:firstorderdifferentialoperators}
\subsection{Linear differential operators of order at most $1$}

\begin{defi}
\label{firstorderdifferentialoperators}	
Let $E,F\in \AMod$.
A $\bk$-linear map $\Delta\colon E \rightarrow F$ is called a \textit{linear differential operator} of order at most $1$ with respect to $\Omega^1_d$, if it factors through the prolongation operator $j^{1}_d$, i.e.\ there exists an $A$-linear map $\widetilde \Delta \in \AHom(J^1_d E,F)$ such that the following diagram commutes:
	\begin{equation}
	\begin{tikzcd}\label{universalfirstorderdifferentialoperator}
		J_d^1E \arrow[dr, "\widetilde\Delta"] & \\
		E \arrow[r,"\Delta"] \arrow[u,"j_d^1"] & F 
	\end{tikzcd}
	\end{equation}
	If $\Delta$ is a morphism in $\Mod$ but not in $\AMod$, we say that $\Delta$ is a differential operator of order $1$ with respect to the first order differential calculus $\Omega^1_d$.
\end{defi}	

\subsection{Differential operators of order zero}
\label{sec:DOzero}
By the convention that $J^0_d E = E$, with prolongation map the identity map, differential operators $\Delta$ of order $0$ from $E$ to $F$ are equivalently morphisms in $\AMod$, and $\widetilde \Delta = \Delta$.
\begin{prop}\label{prop:zeroorder}
	For any $\bk$-algebra $A$, differential operators $A \rightarrow A$ of order zero are precisely right multiplications by elements $f \in A$, i.e.\ $R_f\colon a \mapsto a f$.
	Hence, the left multiplications by an element $f\in A$ is a differential operator of order zero if and only if the element $f$ is central.
\end{prop}
\begin{proof}
	By associativity of the algebra multiplication, right multiplication is left $A$-linear, and hence a differential operator of order zero.
	Vice versa, given a differential operator  $f\colon A\to A$ of order zero, we have $f(a)=af(1)=R_{f(1)}(a)$ for all $a\in A$. 
\end{proof}
\subsection{Differential operators for the universal first order differential calculus}
We can characterize all linear differential operators for the universal first order differential calculus with the following proposition.
\begin{prop}\label{prop:1diffuni}
Every $\bk$-linear map $f\colon E\to F$ is a differential operator of order at most $1$ for the universal first order differential calculus, and the unique lift is given by
\begin{align}
\widetilde{f}\colon A\otimes E\longrightarrow F,
&\hfill&
a\otimes e\longmapsto af(e).
\end{align}
\end{prop}
\begin{proof}
Straightforward computation.
\end{proof}
The universal first order differential calculus thus imposes no further condition on a $\bk$-linear map being a first order differential operator.

\subsection{Differential operators of first order and connections for general calculi}
\begin{prop}\label{prop:uniquelift}
	Let $\Delta\colon E \rightarrow F$ be a differential operator of order at most one.
	Then the lift $\widetilde \Delta\colon J^1_d E \rightarrow F$ described in Definition \ref{firstorderdifferentialoperators} is unique.
\end{prop}
\begin{proof}
	The proof follows from the $A$-linearity of $\widetilde{\Delta}$, as we have
	\begin{equation}\label{eq:uniquefirstorderlift}
	\widetilde{\Delta}[a\otimes b]
	=a\widetilde{\Delta}[1\otimes b]
	=a\widetilde{\Delta}\circ j^1_d (b)
	=a\Delta (b).
	\end{equation}
\end{proof}
\begin{cor}
	Linear differential operators $E\to F$ of order at most $1$ with respect to the first order differential calculus $\Omega^1_d$ are in bijective correspondence with $A$-linear maps $A \otimes E \rightarrow F$ which vanish on the submodule $N_d(E)$.
\end{cor}
\begin{proof}
	By Definition and Proposition \ref{prop:uniquelift} there is a bijective correspondence between $A$-linear maps $J^1_d E\to F$ and differential operators $E\to F$.
	Since $J^1_d E=J^1_u E/N_d(E)=(A\otimes E)/N_d(E)$, cf.\ \eqref{eq:J1dE=J1dA/NdE}, by the quotient universal property, there is a bijective correspondence between linear maps $J^1_d E\to F$ and $A$-linear maps $A\otimes E\to F$ that vanish on $N_d(E)$.
\end{proof}

\begin{prop}\label{1storderwrtd}
	Let $E$ and $F$ be left $A$-modules.
	Then a $\bk$-linear map $\Delta\colon E \rightarrow F$ is a differential operator of order at most $1$ with respect to $d$ if and only if it satisfies
	\begin{equation}\label{eq:1storderwrtd}
		\sum_i n_i \Delta(e_i) = 0
	\end{equation}
	for all elements $\sum_i n_i \otimes e_i \in N_{d}(E)$.
\end{prop}
\begin{proof}
	The map $\Delta$ is a differential operator of order at most $1$ if and only if the universal lift $\widetilde \Delta$ on $J^1_uE$ remains well-defined when we descend to $J^1_dE$.
	The classes in the latter jet module are defined up to $N_{d}(E)$.
	Thus, for all $\sum_i n_i\otimes e_i\in N_d(E)$, we must have
	\begin{equation}
		\Delta(e) = \widetilde\Delta([1 \otimes e]) = \widetilde\Delta([1 \otimes e + \sum_i n_i \otimes e_i]) = \Delta(e) + \sum_i n_i \Delta(e_i),
	\end{equation}
	which holds if and only if \eqref{eq:1storderwrtd} is satisfied.
\end{proof}
\begin{rmk}\label{rmk:tensorNddiffop}
Condition \eqref{eq:1storderwrtd} is satisfied in particular if $\sum_i n_i \Delta(m_i e) = 0$ for all $\sum_i n_i\otimes m_i\in N_d$ and $e\in E$ (cf.\ Remark \ref{rmk:flpj}).
The two conditions are equivalent if $E$ is flat in $\AMod$ or $\Omega^1_d$ is flat in $\ModA$.
\end{rmk}

Finally, we show that certain classes of maps which one would expect to be differential operators really are.
\begin{defi}
\label{defi:connection}
	A \emph{(left) connection} \cite[Definition 3.18]{BeggsMajid} with respect to the first order differential calculus $\Omega^1_d$ on a left $A$-module $E$ is a $\bk$-linear map
	\begin{equation}
		E \longrightarrow \Omega^1_d \otimes_A E,
	\end{equation}
	satisfying the following identity for $f\in A$ and $e\in E$
	\begin{equation}
		\nabla(f e) = df \otimes_A e + f \nabla e.
	\end{equation}
\end{defi}

\begin{prop}\label{prop:1diffop}
	We have the following.
	\begin{enumerate}
		\item\label{prop:1diffop:1} The $1$-jet prolongation $j^1_{d}$ is a (natural) differential operator of order at most $1$ and each component $j^1_{d,E}$ is a differential operator of order $0$ if and only if $\Omega^1_d(E)=0$.
		
		Thus, it is a natural differential operator of order zero if and only if $\Omega^1_d=0$.
		\item\label{prop:1diffop:2} A connection $\nabla\colon E \rightarrow \Omega^1_d \otimes_A E$ is a differential operator of order at most $1$ with respect to $\Omega^1_d$ and it is a differential operator of order $0$ if and only if $\Omega^1_d(E)=0$.
		\item $d$ is a differential operator of order at most $1$ with respect to $\Omega^1_d$ and it is a differential operator of order $0$ if and only if $\Omega^1_d=0$.
		\item Partial derivatives for a parallelizable first order differential calculus are differential operators of order at most $1$ and one, hence all of them, are differential operators of order $0$ if and only if $\Omega^1_d=0$.
	\end{enumerate}
\end{prop}
\begin{proof}\
\begin{enumerate}
	\item It is straightforward since $\widetilde{j}^1_{d,E}=\id_E$ lifts $j^1_{d,E}$.
	The lift is natural in $E$, as $\id$ lifts $j^1_d$ as natural transformations.
	
	The differential operator $j^1_{d,E}$ is of order $0$ if and only if it is $A$-linear, i.e.\ for all $a\in A$, $e\in E$, we have
	\begin{equation}\label{eq:jlinearity}
	0
	=j^1_{d,E}(ae)-aj^1_{d,E}(e)
	=da\otimes_A e
	\end{equation}
	by Remark \ref{rmk:proldop}.
	This equation is equivalent to $A\otimes E=\ker(p_d\circ \rho_{u,E})$, which implies by \eqref{eq:definitionNdE} that $N_d(E)=\Omega^1_u(E)\cap \ker(p_d\circ\rho_{u,E})=\Omega^1_u$, and hence $\Omega^1_d(E)=0$.
	Vice versa, if $\Omega^1_d(E)=0$, then $p_d\circ \rho_{d,E}=0$, so $j^1{d,E}$ is linear by \eqref{eq:jlinearity}.
	\item We use the condition from Proposition \ref{1storderwrtd}.
	Then, for all $n_i\otimes e_i\in N_d(E)$, we have 
		\begin{equation}
			\sum_i n_i\nabla e_i
			= \nabla\left(\sum_i n_i e_i\right)- \sum_i d n_i\otimes e_i
			= 0,
		\end{equation}
	as the first summand vanishes because $N_d(E) \subseteq \Omega^1_u(E)$, and the second by definition of $N_d(E)$.
	
	In order to prove the last part of this point, we proceed as in \eqref{prop:1diffop:1} by noticing that the differential operator $\nabla$ is of order $0$ if and only if for all $a\in A$, $e\in E$, we have
	\begin{equation}
	0
	=\nabla(ae)-a\nabla(e)
	=da\otimes_A e.
	\end{equation}
	\item Follows from the \eqref{prop:1diffop:2}.
	\item let $\theta_1,\dots,\theta_n$ be a basis for $\Omega^1_d$.
	Then we have $da = \sum_i \alpha_i \theta_i$.
	Define $\partial_i(a) = \alpha_i$.
	We need to show that $\partial_i$ vanishes on $N_d$ when extended to $J^1_uA$.
	
	The expression $\sum_j n_jdm_j = 0$ can be written as $\sum_{i,j} n_j \partial_i(m_j)\theta_i=0$.
	Since the $\theta_i$'s are linearly independent in $\Omega^1_d$, each coefficient in this sum vanishes independently.
	Thus $\sum_{j}n_j\partial_i(m_j) = \widetilde\partial_i (\sum_j n_j \otimes m_j) = 0$.
	
	For the last statement, if $\Omega^1_d=0$, then $\partial_i=0$ for all $i$, so they are in particular $A$-linear, and hence differential operators of order $0$.
	Vice versa, notice that
	\begin{equation}
	0
	=d1
	=\sum_i \partial_i(1) \theta_i
	\end{equation}
	and since $\{\theta_i\}$ is a basis, this forces $\partial_i(1)=0$ for all $i$.
	Now, if $\partial_k$ is $A$-linear for some $k$, then for all $a\in A$ we have $\partial_k(a)=a\partial_k(1)=0$, so $\partial_k=0$.
	This implies that $AdA$ is generated as a left $A$-module by $\{\theta_i|i\neq k\}$, contradicting the surjectivity condition of a differential calculus, cf.\ Definition \ref{def:differential calculus}.\eqref{def:differential calculus:2}.
	\qedhere
\end{enumerate}
\end{proof}
Another classical operator that we expect to be a differential operator is the exterior covariant derivative.
We will define a noncommutative generalization (cf.\ Definition \ref{defi:exterior_covariant_derivative}) and prove that it is in fact a differential operator of order at most one (cf.\ Proposition \ref{prop:excoddiffop}).

We can now generalize another classical result by the following.
\begin{prop}[Connections and splittings of the $1$-jet sequence]\label{prop:connexionsplits}
Let $A$ be a $\bk$-algebra and $E$ be in $\AMod$.
There is a bijective correspondence between connections on $E$ and left $A$-linear splittings of the $1$-jet short exact sequence at $E$.
\end{prop}
\begin{proof}
Let $\nabla\colon E\to \Omega^1_d(E)$ be a connection.
By Proposition \ref{prop:1diffop} and Proposition \ref{prop:uniquelift}, there is a unique $A$-linear lift $\widetilde{\nabla}\colon J^1_d E\to \Omega^1_d (E)$ such that $\widetilde{\nabla}\circ j^1_{d,E}=\nabla$.
The lifting is such that for all $[a\otimes e]\in J^1_d E$, $\widetilde{\nabla}[a\otimes e]=a\nabla e$.
All elements of $\Omega^1_d(E)$ can be written as $\sum_i da_i\otimes_A e_i$, and for such elements we have
\begin{equation}
\widetilde{\nabla}\circ \iota^1_{d,E}\left(\sum_i da_i\otimes_A e_i\right)
=\sum_i\widetilde{\nabla}\left([1\otimes a_i e_i] -[a_i\otimes e_i]\right)
=\sum_i \nabla(a_i e_i)-\sum_i a_i\nabla e_i
=\sum_i da_i\otimes_A e_i.
\end{equation}
This implies that $\widetilde{\nabla}\circ \iota^1_{d,E}=\id_{\Omega^1(E)}$, and thus $\widetilde{\nabla}$ provides a splitting for the $1$-jet exact sequence.

Vice versa, let $\widetilde{\varrho}\colon J^1_d E\to \Omega^1_d (E)$ be a left $A$-linear left split of the $1$-jet exact sequence.
Consider $\varrho\colonequals \widetilde{\varrho}\circ j^1_{d,E}$.
Remark \ref{rmk:proldop}, left $A$-linearity of $\widetilde{\varrho}$, and the splitting property of $\widetilde{\varrho}$, give us the Leibniz rule
\begin{equation}
\varrho(ae)
=\widetilde{\varrho}(j^1_{d,E}(ae))
=\widetilde{\varrho}\left(da\otimes_A e + aj^1_{d,E}(e)\right)
=\widetilde{\varrho}\circ \iota^1_{d,E} (da\otimes e) + a\widetilde{\varrho}\circ j^1_{d,E}(e)
=da\otimes_A e + a\varrho(e).
\end{equation}
The two constructions are inverse to each other.
\end{proof}
\begin{rmk}
A left splitting for a sequence is equivalent to a right one, so connections on $E$ are in bijective correspondence with maps $E\to J^1_d E$ that are sections of $\pi^{1,0}_d$.
This equivalence holds in classical differential geometry (cf.\ \cite[§17.1]{NaturalOperations}) as well as in synthetic differential geometry (cf.\ \cite[p.~88]{kock2010synthetic}).

Therefore, a connection on $E$ can equivalently be defined as a section of the map $\pi^{1,0}_d\colon J^1_d E\to E$.
This approach has been adopted before in the context of classical differential geometry, e.g.\ \cite[CHAPTER IV, §9, p.~84]{Palais}.
\end{rmk}

When $A$ is noncommutative, the left action of an element $a\in A$ need not be a differential operator of order $0$, because this action is not (left) $A$-linear as a map $A \rightarrow A$ when $a$ is non-central.
	However, these maps can still be differential operators of higher order.
	For example, left multiplication by any $f \in A$ is a differential operator with respect to the universal first order differential calculus by Proposition \ref{prop:1diffuni}.
	Hence, we will characterize when the multiplication by a non-central element is a differential operator for a general first order differential calculus.

\begin{lemma}
\label{lemma:fterminalcalc}
	Let $f\in A$, then the left multiplication by $f$ on $A$ is a differential operator of order at most $1$ with respect to the differential $d\colon A\to \Omega^1_d$ if and only if $N_d \subseteq N_f \colonequals \ker(n \otimes m \mapsto nfm) \subset \Omega^1_{u}$.
	
	In this case, the left multiplication operator $L_f\colon e \mapsto f e$ is a differential operator of order at most one with respect to $d$ for any left $A$-module $E$.
\end{lemma}
\begin{proof}
	We have $L_f(e) = fe$, and for $\sum_i n_i \otimes m_ie \in N_d\otimes_A E$ we have $\widetilde L_f(\sum_i n_i \otimes m_i e) = (\sum_i n_ifm_i)e = 0$.
	Thus, the statement holds by Proposition \ref{1storderwrtd} and Remark \ref{rmk:tensorNddiffop}.
\end{proof}
\begin{rmk}
	Note that $N_f$ is generally nontrivial, as it contains the elements $af \otimes b - a \otimes fb \in \Omega^1_u$, where $a,b \in A$.
\end{rmk}

We will now provide a recipe to construct calculi for an algebra $A$ so that the left multiplication operators corresponding to elements of a given subset of $A$ are differential operators of order at most $1$.
\begin{defi}
\label{defi:Sterminal}
Let $A$ be a $\bk$-algebra and consider a subset $S\subseteq A$.
We define the \emph{$S$-terminal first order differential calculus} as $\Omega^1_S\colonequals \Omega^1_u/N_S$ where $N_S\colonequals \bigcap_{f\in S}N_f$ for $N_f$ as in Lemma \ref{lemma:fterminalcalc}.
We define the $S$-terminal differential as $d_S\colonequals p_S\circ d_u\colon A\to \Omega^1_S$, where $p_S\colon \Omega^1_u\twoheadrightarrow \Omega^1_S$ is the quotient projection.

For simplicity, if $S=\{f\}$, we will instead call this calculus the $f$-terminal first order differential calculus.
By abuse of notation, in this case we will just denote the differential and the $f$-terminal first order differential calculus as $d_f\colon A\to \Omega^1_f$, and we write $p_f$ for the quotient projection.
Finally, we define $\CalcSA$ to be the full subcategory of $\CalcA$ whose objects are first order differential calculi for which left multiplication operators corresponding to elements of $S$ are differential operators of order at most $1$.
\end{defi}
The name of this object is justified by the following.
\begin{prop}
\label{prop:Sterminalcalc}
Let $S\subseteq A$, then $d_S\colon A\to \Omega^1_S$ is the terminal object in $\CalcSA$.
In other words, for any other calculus $d\colon A\to \Omega^1_d$ for which the left multiplication by any element of $S$ is a differential operator of order at most $1$, there is a unique $A$-bilinear epimorphism $p^d_S\colon \Omega^1_d\twoheadrightarrow \Omega^1_S$ such that
\begin{equation}
p^d_S\circ p_d=p_S.
\end{equation}
It follows that $\ker(p^d_S)\cong N_S/N_d$ and $d_S=p^d_S\circ d$.
\end{prop}
\begin{proof}
Being defined as a quotient of the universal calculus, $d_S\colon A\to \Omega^1_S$ is a differential calculus.
The Leibniz rule follows from $A$-bilinearity of $p_S$ and surjectivity condition from the surjectivity of $p_S$.
This calculus satisfies the desired property by Lemma \ref{lemma:fterminalcalc}, as for all $f\in S$ we have $N_S\subseteq N_f$.

We now prove the universal property using again Lemma \ref{lemma:fterminalcalc}.
We have that $N_d\subseteq N_f$ for all $f\in S$ and thus $N_d\subseteq N_S$.
By the cokernel universal property we have a unique map $p^d_S\colon \Omega^1_d\to \Omega^1_f$ commuting in the following diagram
\begin{equation}
\begin{tikzcd}
0 \arrow[r]&N_d\ar[r,hookrightarrow] \ar[d,hookrightarrow]&\Omega^1_{u}\arrow[r,"p_d",twoheadrightarrow]\arrow[d,equals]& \Omega^1_d\arrow[r] \arrow[d,"p^d_S"]&0\\
0 \arrow[r]&N_S\ar[r,hookrightarrow] &\Omega^1_{u}\arrow[r,"p_S",twoheadrightarrow]& \Omega^1_S\arrow[r]&0
\end{tikzcd}
\end{equation}
Applying the snake lemma elicits the surjectivity of the map $\Omega^1_d\to \Omega^1_f$, and further yields the short exact sequence
\begin{equation}
\begin{tikzcd}
0 \arrow[r]&N_S/N_d\ar[r,hookrightarrow] &\Omega^1_{d}\arrow[r,twoheadrightarrow]& \Omega^1_S\arrow[r]&0.
\end{tikzcd}
\end{equation}
\end{proof}

We can explicitly describe the $f$-terminal first order differential calculus $\Omega^1_f\colonequals\Omega^1_u/N_f$ by noticing that
\begin{align}
\Omega^1_u\longrightarrow A,
&\hfill&
m\otimes n\longmapsto mfn
\end{align}
is well-defined and $A$-bilinear map, with kernel $N_f$.
Hence, we can identify $\Omega^1_f$ with the image of this map; map which we can thus call $p_f$.
Thus $\Omega^1_f$ can be identified with an $A$-subbimodule of $A$, which is contained in the two sided principal ideal generated by $f$, and it forms a first order differential calculus with differential
\begin{equation}
d_f\colonequals p_f\circ d_u =[f,\cdot]\colon A\longrightarrow \Omega^1_f.
\end{equation}
We can thus explicitly describe $\Omega^1_f$ as $A \cdot d_f(A)=A[f,A]=\{a[f,b]|a,b\in A\}$ by the surjectivity condition.
Similarly we can write $\Omega^1_f=[f,A]A$.

\section{Nonholonomic and semiholonomic jet functors}\label{s:nonsemi}
\subsection{Nonholonomic jet functors}
\label{ss:nhj}
Now that we have the $1$-jet functor, we can iterate it in order to obtain the so-called nonholonomic $n$-jet for all $n$.
\begin{defi}\label{def:nonholjetfunctor}
	We term the functor $J^{(n)}_d\colonequals (J^1_d)^{\circ n}=J^1_d \circ\dots\circ J^1_d=(J^1_d A)^{\otimes_A n}\otimes_A - \colon \AMod\to\AMod$ the {\em nonholonomic $n$-jet functor}.
\end{defi}
\begin{rmk}\label{rmk:nhJexact}
The functor $J^{(n)}_d$ is always right exact, but in general it is not left exact for $n>0$.
However, being the iterated composition of the functor $J^1_d$, we have that it is exact if $J^1_d$ is exact.
Hence, by Corollary \ref{cor:Jex}, the functor $J^{(n)}_d$ is exact for all $n\ge 0$ if $\Omega^1_d$ is flat in $\ModA$.
\end{rmk}

Applying the jet short exact sequence of functors to $J^{(n-1)}_d$ yields another short exact sequence
\begin{equation}\label{es:nhjets}
\begin{tikzcd}
0\ar[r]&\Omega^1_d\circ J^{(n-1)}_d\ar[r,hookrightarrow,"\iota^1_{d,J^{(n-1)}_d}"]&[30pt]J^{(n)}_d\ar[r,twoheadrightarrow,"\pi^{1,0}_{d,J^{(n-1)}_d}"]&[30pt]J^{(n-1)}_d\ar[r]&0.
\end{tikzcd}
\end{equation}
\begin{defi}
We call \eqref{es:nhjets} the \emph{nonholonomic $n$-jet sequence}.
\end{defi}
For short, we will write $\iota^{(n)}_d\colonequals \iota^1_{d,J^{(n-1)}_d}$, and for $0\le m\le n$, we may consider the composition
\begin{equation}\label{eq:nhpink}
\pi^{(n,m)}_d\colonequals \pi^{1,0}_{d,J^{(m)}_d}\circ \pi^{1,0}_{d,J^{(m+1)}_d}\circ\cdots \circ \pi^{1,0}_{d,J^{(n-2)}_d}\circ \pi^{1,0}_{d,J^{(n-1)}_d}
\colon J^{(n)}_d\longrightarrow J^{(m)}_d.
\end{equation}
By definition, we have
\begin{equation}
\pi^{(n,m)}_d=\pi^{(n-m,0)}_{d,J^{(m)}_d}.
\end{equation}

Now consider the $1$-jet natural prolongation \eqref{eq:1jp}, and apply it on the image of the functor $J^{(n-1)}_d$, that is $j^1_{d,J^{(n-1)}_d}\colon J^{(n-1)}_d\to J^{(n)}_d$.
By \eqref{eq:jsecpi}, it is a section of $\pi^{1,0}_{d,J^{(n-1)}_d}$, and hence it is a $\bk$-linear natural monomorphism.
We now define the following composition
\begin{equation}\label{eq:nhjp}
j^{(n)}_d\colonequals j^1_{d,J^{(n-1)}_d}\circ j^1_{d,J^{(n-2)}_d}\circ\cdots\circ j^1_{d,J^{1}_d}\circ j^1_d\colon \id_{\AMod}\longrightarrow J^{(n)}_d.
\end{equation}
In particular, $j^{(0)}_d=\id$.
\begin{defi}
We call this composition the \emph{nonholonomic $n$-jet prolongation}.
Explicitly, for a left $A$-module $E$, it is defined as
\begin{align}\label{eq:jexpl}
j^{(n)}_{d,E}\colon E\longrightarrow J^{(n)}_d E,
&\hfill&
e\longmapsto \underbrace{[1\otimes 1]\otimes_A\cdots \otimes_A[1\otimes 1]}_{n-\mathrm{times}}\otimes_A e
=[1\otimes 1]^{\otimes_A n}\otimes_A e.
\end{align}
\end{defi}
Being the composition of sections, $j^{(n)}_d$ is itself a section, corresponding to the natural projection $\pi^{(n,0)}_d$.
Moreover, for all $m\ge 0$ we obtain the following decomposition
\begin{equation}\label{eq:nhjpdec}
j^{(n)}_d=j^{(n-m)}_{d,J^{(m)}_d}\circ j^{(m)}_d.
\end{equation}
Using these observations we further obtain
\begin{equation}\label{eq:jpi}
\pi^{(n,m)}_d\circ j^{(n)}_d
=\pi^{(n-m,0)}_{d,J^{(m)}_d}\circ j^{(n-m)}_{d,J^{(m)}_d}\circ j^{(m)}_d
=j^{(m)}_d.
\end{equation}
\begin{rmk}[Nonholonomic jet functors on bimodules]\label{rmk:nJbi}
The functor $J^{(n)}_d$ can also be defined on a category $\AMod_B$, by iterating the $1$-jet functor here defined.
With the same principle, we can compatibly lift the whole \eqref{es:nhjets}.
By Proposition \ref{prop:1jses}, we obtain a lift to the category of endofunctors on $\AMod_B$.
The lift is compatible with the forgetful functor.

The nonholonomic prolongation lifts similarly.
By Remark \ref{rmk:jbim}, we can see $j^{(n)}_d\colon \id_{\AMod_B}\to J^{(n)}_d$ as a natural transformation between functors of type $\AMod_B\to \Mod_B$, whereby $j^{(n)}_d$ restricts to a natural monomorphism of right $B$-linear maps.
\end{rmk}

The short exact sequence \eqref{es:nhjets} and the jet prolongations induce a canonical $\bk$-linear decomposition of the jet module.
First, we define the following natural retraction corresponding to the right $A$-linear right splitting of \eqref{es:nhjets} given by $j^1_{d,J^{(n-1)}_d}$ (cf.\ \eqref{def:rho}), that is
\begin{equation}\label{eq:defrhon}
\rho^n_d\colonequals
\rho_{d,J^{(n-1)}_d}
=\id_{J^{(n)}_d}-j^{1}_{d,J^{(n-1)}_d}\circ \pi^{(n,n-1)}_d
\colon J^{(n)}_d \longrightarrow \Omega^1_d \circ J^{(n-1)}_d.
\end{equation}
The image is in $ \Omega^1_d\circ J^{(n-1)}_d$ because composing $\rho^n_d$ with $\pi^{(n,n-1)}_d$ gives
\begin{equation}
\begin{split}
\pi^{(n,n-1)}_{d}\circ \rho^n_d
&= \pi^{(n,n-1)}_{d} \circ \left(\id_{J^{(n)}_d}-j^1_{d,J^{(n-1)}_d}\circ\pi^{(n,n-1)}_{d}\right)\\
&=\pi^{(n,n-1)}_{d}-\pi^{1,0}_{d,J^{(n-1)}_d} \circ j^1_{d,J^{(n-1)}_d}\circ\pi^{(n,n-1)}_{d}\\
&=\pi^{(n,n-1)}_{d}-\pi^{(n,n-1)}_{d}\\
&=0.
\end{split}
\end{equation}
Therefore, by \eqref{es:nhjets}, $\rho^n_d$ factors through $\Omega^1_d\circ J^{(n-1)}_d$.
By construction of $\rho^n_d$ and the properties of the biproduct,
\begin{align}\label{eq:shsplit}
&\rho^n_d\circ \iota^{(n)}_{d}=\id_{\Omega^1_d\circ J^{(n-1)}_d} && \rho^n_d\circ j^1_{d,J^{(n-1)}_d}=0\nonumber\\
&\pi^{(n,n-1)}_d\circ \iota^{(n)}_{d}=0 && \pi^{(n,n-1)}_d\circ j^1_{d,J^{(n-1)}_d}=\id_{J^{(n-1)}_d}\\
&j^1_{d,J^{(n-1)}_d}\circ \pi^{(n,n-1)}_d + \iota^{(n)}_{d} \circ \rho^n_d = \id_{J^{(n)}_d}\nonumber
\end{align}
\begin{rmk}
The natural epimorphism $\rho^n_d\colon J^{(n)}_d \to \Omega^1_d \circ J^{(n-1)}_d$ is $\bk$-linear, being the composition of $\bk$-linear maps.
For the same reason, the corresponding $\rho^n_d$ on $\AMod_B$ is a right $B$-linear natural transformation.
\end{rmk}
This leads to the following decomposition result.
\begin{theo}[Nonholonomic jet decomposition]\label{theo:nhJdec}
Given a $\bk$-algebra $A$ endowed with a first order differential calculus $\Omega^1_d$, and $E$ in $\AMod$, we can write $J^{(n)}_d E$ as
\begin{equation}
J^{(n)}_d E=j^{(n)}_d(E)\oplus \bigoplus_{m=0}^{n-1}j^{(n-m)}_d\left(\Omega^1_d \circ J^{(m)}_d E\right).
\end{equation}
Therefore, we can uniquely decompose an element $\xi\in J^{(n)}_d E$ as
\begin{equation}\label{eq:nhjdec}
\xi=j^{(n)}_d(\xi^0)+\sum_{m=1}^n j^{(n-m)}_d\left(\sum_{i_m}\alpha^m_{i_m}\otimes_A \xi^m_{i_m}\right).
\end{equation}
for a unique $\xi^0\in E$, and unique $\sum_{i_m}\alpha^m_{i_m}\otimes_A \xi^m_{i_m}\in \Omega^1_d\otimes_A J^{(m-1)}_d E$ for each $m>0$.
In particular,
\begin{align}
\xi^0=\pi^{(n,0)}_d(\xi),
&\hfill &
\sum_{i_m}\alpha^m_{i_m}\otimes_A \xi^m_{i_m}=\rho^{m}_{d,E}\left(\pi^{(n,m)}_{d,E}(\xi)\right),
\end{align}
where $\alpha^m$ can be chosen to be exact for all $m$.
\end{theo}
\begin{proof}
We will prove this by induction on $n$.
For $n=0$ the statement is trivially true, as $\xi^0=\xi$ and $J^0_d E=j^{(0)}_{d,E}(E)$.
Now let $n>0$, and assume the statement is true for $n-1$.
Let $\xi\in J^{(n)}_d E$, then $\pi^{(n,n-1)}_{d,E}(\xi)\in J^{(n-1)}_d E$, so we can uniquely decompose it as
\begin{equation}
\begin{split}
\pi^{(n,n-1)}_{d,E}(\xi)
&=j^{(n-1)}_{d,E}(\pi^{(n-1,0)}_{d,E}(\pi^{(n,n-1)}_{d,E}\xi))+\sum_{m=1}^{n-1} j^{(n-m)}_{d,\Omega^1_d\circ J^{(m-1)}_d E}\left(\rho^m_{d,E}\circ \pi^{(n-1,m)}_{d,E}(\pi^{(n,n-1)}_{d,E}\xi)\right)\\
&=j^{(n-1)}_d (\pi^{(n,0)}_{d,E}(\xi))+\sum_{m=1}^{n-1} j^{(n-m)}_{d,\Omega^1_d\circ J^{(m-1)}_d E}\left(\rho^m_{d,E}\circ \pi^{(n,m)}_{d,E}(\xi)\right).
\end{split}
\end{equation}
Now consider $\rho^n_{d,E}(\xi)=\xi-j^1_{d,J^{(n-1)}_d E}(\pi^{(n,n-1)}_{d,E}(\xi))$, which is an element in $\Omega^1\circ J^{(n-1)}_d E$.
By definition, we have
\begin{equation}
\begin{split}
\xi
&=j^1_{d,J^{(n-1)}_d E}(\pi^{(n,n-1)}_{d,E}(\xi))+\rho^n_{d,E}(\xi)\\
&=j^1_{d,J^{(n-1)}_d E}\left(
j^{(n-1)}_d (\pi^{(n,0)}_{d,E}(\xi))+\sum_{m=1}^{n-1} j^{(n-m)}_{d,\Omega^1_d\circ J^{(m-1)}_d E}\left(\rho^m_{d,E}\circ \pi^{(n,m)}_{d,E}(\xi)\right)
\right)+\rho^n_{d,E}(\xi)\\
&=j^{(n)}_{d,E}(\pi^{(n,0)}_{d,E}\xi)+\sum_{m=1}^{n-1} j^{(n-m)}_{d,\Omega^1_d\circ J^{(m-1)}_d E}\left(\rho^m_{d,E}\circ \pi^{(n,m)}_{d,E}(\xi)\right).
\end{split}
\end{equation}

For uniqueness, consider any other decomposition \eqref{eq:nhjdec}, and apply $\rho^n_d$ to both terms.
This gives
\begin{equation}
\rho^n_d\xi
=\rho^n_d\left(j^{(n)}_d(\xi^0)+\sum_{m=1}^n j^{(n-m)}_d\left(\sum_{i_m}\alpha^m_{i_m}\otimes_A \xi^m_{i_m}\right)\right)
=0+\rho^n_d\left(\sum_{i_m}\alpha^n_{i_n}\otimes_A \xi^n_{i_n}\right)
=\sum_{i_m}\alpha^n_{i_n}\otimes_A \xi^n_{i_n},
\end{equation}
which means the term in $\Omega^1_d\circ J^{(n-1)}_d E$ is fully determined.
Now, if we consider the difference $\xi-\sum_{i_m}\alpha^n_{i_n}\otimes_A \xi^n_{i_n}$, this is the image of $\pi^{(n,n-1)}_{d,E}\xi$ via the monomorphism $j^1_{d,J^{(n-1)}_d}$, by definition of $\rho^n_d$, and hence, by inductive hypothesis, it must coincide with the unique decomposition of $\pi^{(n,n-1)}_{d,E}\xi$ given above.
This proves uniqueness.

The maps used in this decomposition are all $\bk$-linear.
Therefore, the above decomposition of the functor $J^{(n)}_d$ is $\bk$-linear.

Finally, every element of the form $\alpha\otimes_A \xi\in\Omega^1\circ J^{(m-1)}_d E$ satisfies $\alpha=\sum_i (dx_i)y_i$ by the surjectivity axiom, hence
\begin{equation}
\alpha\otimes_A \xi
=\sum_i (dx_i)y_i\otimes_A \xi
=\sum_i dx_i\otimes_A y_i\xi,
\end{equation}
as desired.
\end{proof}
\begin{rmk}
The choice of $\alpha^m_{i_m}$ and $\xi^m_{i_m}$ is not unique.

Moreover, if the nonholonomic functor is lifted to $\AMod_B$, then this decomposition preserves the right $B$-module structure.
\end{rmk}

Notice that $\pi^{(n,n-1)}_d=\pi^{1,0}_{d,J^{(n-1)}_d}$ is not the only projection $J^{(n)}_d\to J^{(n-1)}_d$
\begin{defi}\label{def:otherprojections}
	For all $1\le m\le n$, we have the natural epimorphisms
	\begin{equation}
	\begin{tikzcd}
		\pi^{(n,n-1;m)}_d
		=J^{(n-m)}_d \pi^{1,0}_{d,J^{(m-1)}_d}
		\colon J^{(n)}_d\ar[r,twoheadrightarrow]& J^{(n-1)}_d,
	\end{tikzcd}
	\end{equation}
	which will be called the {\em nonholonomic $n$-jet projection in position $m$}.
\end{defi}
\begin{rmk}\label{rmk:projexpl}
Given a left $A$-module $E$, we can write the projections $\pi^{(n,n-1;m)}_d$ explicitly by mapping an element $[a_n\otimes b_n]\otimes_A \cdots \otimes_A [a_1\otimes b_1]\otimes_A e_0\in J^{(n)}_d E$ to the following element of $J^{(n-1)}_d E$:
\begin{equation}
\begin{split}
&[a_n\otimes b_n]\otimes_A \cdots\otimes_A [a_{m+1}\otimes b_{m+1}a_m b_m]\otimes_A[a_{m-1}\otimes b_{m-1}]\otimes_A\cdots \otimes_A [a_1\otimes b_1]\otimes_A e_0\\
&\qquad=[a_n\otimes b_n]\otimes_A \cdots\otimes_A [a_{m+1}\otimes b_{m+1}]\otimes_A[a_m b_m a_{m-1}\otimes b_{m-1}]\otimes_A\cdots \otimes_A [a_1\otimes b_1]\otimes_A e_0.
\end{split}
\end{equation}
\end{rmk}
The projections commute, in the sense that for all $1\le h < m\le n$, we have
\begin{equation}\label{eq:picomm}
\pi^{(n-1,n-2;m-1)}_d\circ\pi^{(n,n-1;h)}_d=\pi^{(n-1,n-2;h)}_d\circ\pi^{(n,n-1;m)}_d.
\end{equation}

If $J^1_d$ is exact (cf.\ Corollary \ref{cor:Jex}), then for all $1\le m \le n$ we have the short exact sequence
\begin{equation}\label{es:nhjetsal}
\begin{tikzcd}
0\ar[r]&J^{(n-m)}_d\circ \Omega^1_d\circ J^{(m-1)}_d\ar[r,hookrightarrow,"\iota^{(n,n-1;m)}_d"]&[40pt]J^{(n)}_d\ar[r,twoheadrightarrow,"\pi^{(n,n-1;m)}_d"]&[40pt]J^{(n-1)}_d\ar[r]&0,
\end{tikzcd}
\end{equation}
where we define $\iota^{(n,n-1;m)}_d=J^{(n-m)}_d \iota^1_{d,J^{(m-1)}_d}$.

If $J^1_d$ is not exact, \eqref{es:nhjetsal} is still right exact, but not necessarily left exact, as evinced by the following.
\begin{eg}\label{eg:K2}
Consider the example presented in §\ref{ss:K2}.
We have that the right exact sequence of functors
\begin{equation}
\begin{tikzcd}
J^{1}_d\circ \Omega^1_d \ar[r,"\iota^{(2,1;1)}_d"]&[40pt]J^{(2)}_d\ar[r,twoheadrightarrow,"\pi^{(2,1;1)}_d"]&[40pt]J^{1}_d\ar[r]&0
\end{tikzcd}
\end{equation}
is not left exact.
In fact, if we apply it to the module $\Omega^1_d$, we can show that $\iota^{(2,1;1)}_d\otimes_A\id_{\Omega^1_d}$ is not injective.
We first apply the functor $\Omega^1_d \otimes_A -\otimes_A \Omega^1_d=\bk[0]\otimes_A -\otimes_A \bk[0]$ to the $1$-jet exact sequence \eqref{es:1jK2}.
The result is
\begin{equation}
\begin{tikzcd}
\bk[0]\ar[r,"0"]&\bk[0]\ar[r,equals]&\bk[0]\ar[r]&0.
\end{tikzcd}
\end{equation}
The left and right terms can be computed using Lemma \ref{lemma:K2}, as $\bk[0]\otimes_A \bk[0]=\Tor^A_0(\bk[0],\bk[0])=\bk[0]$.
The computation of the central term follows from right exactness if we prove that the leftmost morphism is the zero map.
This morphism maps $1\in\bk[0]$ in $\bk[0]\otimes_A J^1_d A\otimes_A \bk[0]$ to
\begin{equation}
1\otimes_A (y-x)\otimes_A 1
=1\otimes_A 1\cdot t \otimes_A 1-1\otimes_A t\cdot 1\otimes_A 1
=1\otimes_A 1 \otimes_A 0-0\otimes_A 1\otimes_A 1
=0.
\end{equation}
Using the naturality of \eqref{es:natJ1} applied to the monomorphism $\iota^1_{d,\Omega^1_d}=\iota^1_d\otimes_A \id_{\Omega^1_d}$, we obtain the following commutative diagram
\begin{equation}
\begin{tikzcd}
\Omega^1_d\circ \Omega^1_d(\Omega^1_d)\ar[r,hookrightarrow,"\iota^1_{d,\Omega^1_d( \Omega^1_d)}"]\ar[d,"\Omega^1_d(\iota^1_{d,\Omega^1_d})"']&[40pt] J^1_d\circ \Omega^1_d(\Omega^1_d)\ar[d,"J^1_d(\iota^1_{d,\Omega^1_d})"]\\
\Omega^1_d\circ J^1_d(\Omega^1_d)\ar[r,hookrightarrow,"\iota^1_{d,J^1_d(\Omega^1_d)}"]&J^{(2)}_d(\Omega^1_d)
\end{tikzcd}
\end{equation}
The two horizontal maps are injective, since $\iota^1_d$ is a natural monomorphism, and the left vertical map is $\id_{\Omega^1_d}\otimes_A \iota^1_d\otimes_A \id_{\Omega^1_d}=0$.
It follows that $\Omega^1_d\circ \Omega^1_d(\Omega^1_d)=\bk[0]$ injects into the kernel of $J^1_d(\iota^1_{d,\Omega^1_d})$, which therefore is not injective.
\end{eg}
\subsubsection{Stability}
We now see that the nonholonomic jet functor preserves some important subcategories under a given regularity of $\Omega^1_d$.
\begin{prop}\label{prop:Jnhstable}
Let $A$ be a $\bk$-algebra endowed with a first order differential calculus $\Omega^1_d$ and consider the functor
\begin{align}
J^{(n)}_d\colon \AMod\longrightarrow\AMod.
\end{align}
\begin{enumerate}
\item If $\Omega^1_d$ is in $\AFlat$, then $J^{(n)}_d$ preserves the subcategory $\AFlat$;
\item If $\Omega^1_d$ is in $\AProj$, then $J^{(n)}_d$ preserves the subcategory $\AProj$;
\item If $\Omega^1_d$ is in $\AFG$, then $J^{(n)}_d$ preserves the subcategory $\AFG$;
\item If $\Omega^1_d$ is in $\AFGP$, then $J^{(n)}_d$ preserves the subcategory $\AFGP$.
\end{enumerate}
\end{prop}
\begin{proof}
For $n=0$ there is nothing to prove.
The case $n=1$ is true for all points by Proposition \ref{prop:OmJstab}.
Every other case follows by iterating the functor $J^1_d$.
\end{proof}

\subsection{The tensor algebra}
\label{ss:T}
We introduce the following object
\begin{defi}[Tensor algebra]
Let $A$ be a $\bk$-algebra equipped with a first order differential calculus $\Omega^1_d$.
We define the functors $T^n_d=(\Omega^1_d)^{\circ n}=\underbrace{\Omega^1_d\circ \cdots \circ \Omega^1_d}_{n-\textrm{times}}\colon \AMod\longrightarrow \AMod$, explicitly
\begin{equation}
E\longmapsto
\underbrace{\Omega^1_d\otimes_A \cdots \otimes_A \Omega^1_d}_{n-\textrm{times}}\otimes_A E.
\end{equation}
In particular $T^0_d=\id_{\AMod}$ and $T^1_d=\Omega^1_d$.

We define the \emph{tensor algebra functor} with respect to $\Omega^1_d$ to be $T^{\bullet}_d=\bigoplus_{n\ge 0} T^n_d$.

In particular, we call $T^\bullet_d(A)$ the \emph{tensor algebra} with respect to $\Omega^1_d$.
\end{defi}
\begin{rmk}
The bimodule $T^\bullet_d(A)$ is an associative algebra with $\otimes_A$ as multiplication.
\end{rmk}
When no confusion arises, we will denote the $A$-bimodules $T^k_d(A)$ and $T^{\bullet}_d(A)$ also by $T^k_d$ and $T^{\bullet}_d$ respectively.
\begin{eg}[Classical tensor algebra]
	In the example of §\ref{ss:cman}, the tensor algebra corresponds to the classical notion of tensor algebra, i.e.\ $T^{\bullet}_d(\smooth{M})=\Gamma(M,T^{\bullet}M)$.
\end{eg}
\begin{eg}[Infinitesimal algebra]
In the example of §\ref{ss:K2}, since $\bk[0]\otimes_A \bk[0]\cong\bk[0]$, we have $T^n_d\cong\Omega^1_d\cong\bk[0]$ for all $n>0$, and thus $T^\bullet_d=A\oplus\bigoplus_{n>0} \bk[0]$, where $\otimes_A\colon \bk[0]\otimes_A \bk[0]\to\bk[0]$ is an isomorphism.
\end{eg}

From $\iota^1_d\colon\Omega^1_d\to J^1_d$, we obtain a mapping
\begin{equation}\label{eq:iotaTn}
\iota_{T^n_d}\colon T^n_d\longrightarrow J^{(n)}_d
\end{equation}
and by the naturality of $\iota^1_d$, we have
\begin{equation}\label{eq:iotaTnequ}
J^{(n-1)}_d(\iota^1_{d})\circ J^{(n-2)}_d(\iota^1_{d,T^1_d})\circ \cdots \circ J^{1}_d(\iota^1_{d,T^{n-2}_d})\circ \iota^1_{d,T^{n-1}_d}
=\iota_{T^n_d}
=\iota^1_{d,J^{(n-1)}_d}\circ T^1_d(\iota^1_{d,J^{(n-2)}_d})\circ\cdots \circ T^{n-2}_d(\iota^1_{d,J^1_d})\circ J^{(n-1)}_d(\iota^1_{d})
\end{equation}
It follows that for all $1\le m\le n$ we can write
\begin{equation}\label{eq:iotaTnequ2}
\begin{split}
\iota_{T^n_d}
&=J^{(n-m)}_d\left(\iota^1_{d,J^{(m-1)}_d}\right)\circ J^{(n-m)}_d\left(\Omega^1_d(\iota_{T^{m-1}_d})\right)\circ (\iota_{T^{n-m}_d})_{T^{m}_d}\\
\iota_{T^n_d}
&=J^{(n-m)}_d\left(\iota^1_{d,J^{(m-1)}_d}\right)\circ (\iota_{T^{n-m}_d})_{J^{(m-1)}_d}\circ T^{n-m+1}_d\left(\iota_{T^{m-1}_d}\right).
\end{split}
\end{equation}
\begin{rmk}\label{rmk:iotaTninj}
If $\Omega^1_d$ is flat in $\ModA$, the map $\iota_{T^n_d}$ is a monomorphism (cf.\ Corollary \ref{cor:Jex}).
\end{rmk}
\begin{rmk}\label{rmk:TnEstd}
Let $E$ be in $\AMod$, then every element of $T^n_d(E)$ can be written as 
\begin{equation}
\sum_i dx^n_i\otimes_A dx^{n-1}_i\otimes_A \cdots\otimes_A dx^1_i\otimes_A e^0_i,
\end{equation}
where $x^j_i\in A$ and $e^0_i\in E$.
This holds because we can shift the scalars to the right by writing $adb=d(ab)-(da)b$.
\end{rmk}
\begin{rmk}[Tensor algebra on bimodules]
Using Proposition \ref{prop:1jses}, we can lift $T^n_d$ to an endofunctor on $\AMod_B$, which is compatible with the forgetful functor in the sense of \eqref{diag:compforgetbi}.
In this context, the map $\iota_{T^n_d}$ is $(A,B)$-bilinear.
\end{rmk}
\subsubsection{Stability and exactness}
We have the following result, analogous to Proposition \ref{prop:Jnhstable}.
\begin{prop}\label{prop:Tstable}
Let $A$ be a $\bk$-algebra endowed with a first order differential calculus $\Omega^1_d$ and consider the functors
\begin{align}
T^n_d\colon \AMod\longrightarrow\AMod,
&\hfill &
T^{\bullet}_d\colon \AMod\longrightarrow\AMod.
\end{align}
\begin{enumerate}
\item\label{prop:Tstable:1} If $\Omega^1_d$ is in $\AFlat$, then the functors $T^n_d$ and $T^{\bullet}_d$ preserve the subcategory $\AFlat$;
\item\label{prop:Tstable:2} If $\Omega^1_d$ is in $\AProj$, then the functors $T^n_d$ and $T^{\bullet}_d$ preserve the subcategory $\AProj$;
\item\label{prop:Tstable:3} If $\Omega^1_d$ is in $\AFG$, then the functor $T^n_d$ preserves the subcategory $\AFG$;
\item\label{prop:Tstable:4} If $\Omega^1_d$ is in $\AFGP$, then the functor $T^n_d$ preserves the subcategory $\AFGP$.
\end{enumerate}
\end{prop}
\begin{proof}
For each point, the proof for $T^n_d$ follows from the fact that under the given conditions, $\Omega^1_d$ preserves $\AFlat$, $\AProj$, $\AFG$, and $\AFGP$ (cf.\ Proposition \ref{prop:OmJstab}).
For the first two points, the statement for $T^\bullet_d$ follows from the fact that direct sums of flat and projective left $A$-modules are, respectively, flat and projective left $A$-modules (cf.\ Lemma \ref{lemma:flprsum}).
\end{proof}

\begin{prop}\label{prop:Texact}
The functor $T^n_d$ is always right exact, and if $\Omega^1_d$ is flat in $\AMod$, then $T^n_d$ is exact.
\end{prop}
\begin{proof}
Being defined as composition of the right exact functor $\Omega^1_d$, it is right exact.
For the same reason, if $\Omega^1_d$ is flat in $\ModA$ (cf.\ Corollary \ref{cor:Jex}), then $T^n_d$ is exact.
\end{proof}
\subsection{Semiholonomic jet functor}
\label{ss:shj}
Paralleling the construction of jets in the classical case \cite{libermann1997}, we now construct a noncommutative analogue of the semiholonomic jet functor.
\begin{defi}
\label{def:shj}
We define the {\em semiholonomic $n$-jet functor}, denoted $J^{[n]}_d$, to be the equalizer
\begin{equation}
J^{[n]}_d\colonequals \Eq\left(\pi^{(n,n-1;m)}_d \middle| 1\le m\le n \right).
\end{equation}
We denote by $\iota_{J^{[n]}_d}\colon J^{[n]}_d\hookrightarrow J^{(n)}_d$ the corresponding inclusion.
\end{defi}
\begin{rmk}\label{rmk:lowsemihol}
Note that $J^{[0]}_d=\id_{\AMod}$ and $J^{[1]}_d=J^1_d$.
\end{rmk}
\begin{prop}\label{prop:piwd}
Each natural transformation $\pi^{(n,n-1;m)}_d\circ \iota_{J^{[n]}_d}\colon J^{[n]}_d\to J^{(n-1)}_d$ factors (uniquely) through $J^{[n-1]}_d$, for all $1\le m\le n$.
\end{prop}
\begin{proof}
All projections $\pi^{(n,n-1;m)}_d$ coincide when restricted to $J^{[n]}_d$, so it is enough to prove it for $m=1$.
In order to apply the universal property of the equalizer, we need to show that for all $1\le h\le n-1$ the compositions $\pi^{(n-1,n-2;h)}_d\circ \pi^{(n,n-1;1)}_d\circ \iota_{J^{[n]}_d}$ coincide.
By \eqref{eq:picomm}, on $J^{[n]}_d$ we have
\begin{equation}
\pi^{(n-1,n-2;h)}_d\circ \pi^{(n,n-1;1)}_d\circ\iota_{J^{[n]}_d}
=\pi^{(n-1,n-2;1)}_d\circ \pi^{(n,n-1;h+1)}_d\circ \iota_{J^{[n]}_d}
=\pi^{(n-1,n-2;1)}_d\circ \pi^{(n,n-1;1)}_d\circ \iota_{J^{[n]}_d}.
\end{equation}
\end{proof}
As a consequence of Proposition \ref{prop:piwd}, all maps $\pi^{(n,n-1;m)}_d\circ \iota_{J^{[n]}_d}$ factor through $J^{[n-1]}_d$ as a unique natural transformation, which we denote by
\begin{equation}\label{eq:semiholpi}
\pi^{[n,n-1]}_d\colon J^{[n]}_d\longrightarrow J^{[n-1]}_d.
\end{equation}
\begin{defi}
We term the map \eqref{eq:semiholpi} the \emph{semiholonomic $n$-jet projection}.
Their composition will give for all $0\le m\le n$ a map
\begin{equation}
\pi^{[n,m]}_d\colonequals \pi^{[m+1,m]}_d\circ \pi^{[m+2,m+1]}_d\circ \cdots \circ \pi^{[n-1,n-2]}_d\circ\pi^{[n,n-1]}_d\colon J^{[n]}_d\longrightarrow J^{[m]}_d.
\end{equation}
\end{defi}

Consider the map $\iota_{T^n_d}\colon T^n_d\to J^{(n)}_d$ defined in \eqref{eq:iotaTn}.
By \eqref{eq:iotaTnequ2}, for all $0\le m\le n$, we have
\begin{equation}\label{eq:factoriotaTn}
\begin{split}
\pi^{(n,n-1;m)}_d\circ\iota_{T^n_d}
&=J^{(n-m)}_d \pi^{1,0}_{d,J^{(m-1)}_d}\circ J^{(n-m)}_d\left(\iota^1_{d,J^{(m-1)}_d}\right)\circ J^{(n-m)}_d\left(\Omega^1_d(\iota_{T^{m-1}_d})\right)\circ (\iota_{T^{n-m}_d})_{T^{m}_d}\\
&=J^{(n-m)}_d\left( \pi^{1,0}_{d,J^{(m-1)}_d}\circ \iota^1_{d,J^{(m-1)}_d}\right)\circ J^{(n-m)}_d\left(\Omega^1_d(\iota_{T^{m-1}_d})\right)\circ (\iota_{T^{n-m}_d})_{T^{m}_d}\\
&=0
\end{split}
\end{equation}
It follows that $\iota_{T^n_d}$ factors uniquely through $\iota_{J^{[n]}_d}$.
\begin{defi}
We define $\iota^{[n]}_d$ to be the unique factorization of $\iota_{T^n_d}$ through $\iota_{J^{[n]}_d}$.
\end{defi}
\begin{prop}\label{prop:complexinclusionseminon}
The following is a map in the category of complexes of functors of type $\AMod\to \AMod$
\begin{equation}\label{es:shjci}
\begin{tikzcd}[column sep=40pt]
0\ar[r]&T^{n}_d\ar[d,"\Omega^1_d(\iota_{T^{n-1}_d})"']\ar[r,"\iota^{[n]}_d"]&J^{[n]}_d\ar[d,hookrightarrow,"\iota_{J^{[n]}_d}"]\ar[r,"\pi^{[n,n-1]}_d"]&J^{[n-1]}_d\ar[r]\ar[d,hookrightarrow, "\iota_{J^{[n-1]}_d}"]&0\\
0\ar[r]&\Omega^1_d \circ J^{(n-1)}_d\ar[r,hookrightarrow,"\iota^{(n)}_d"]&J^{(n)}_d\ar[r,twoheadrightarrow,"\pi^{(n,n-1)}_d"]& J^{(n-1)}_d\ar[r]&0
\end{tikzcd}
\end{equation}
The vertical arrows become inclusions if $\Omega^1_d$ is flat in $\ModA$.
\end{prop}
\begin{proof}
The right square commutes by definition of $\pi^{[n,n-1]}_d$, and the left square commutes by definition of $\iota^{[n]}_d$, since $\iota^{(n)}_d\circ \Omega^1_d(\iota_{T^{n-1}_d})=\iota_{T^n_d}$ (cf.\ \eqref{eq:iotaTnequ}).
From \eqref{eq:factoriotaTn}, we know that $\pi^{[n,n-1]}_d\circ \iota^{[n]}_d=0$.

If $\Omega^1_d$ is flat in $\ModA$, then $\Omega^1_d(\iota_{T^{n-1}_d})$ is injective (cf.\ Remark \ref{rmk:iotaTninj} and Corollary \ref{cor:Jex}).
\end{proof}
\begin{defi}
We call the top row of \eqref{es:shjci} the \emph{semiholonomic $n$-jet sequence}.
\end{defi}
Consider now the nonholonomic jet prolongation $j^{(n)}_d$ defined in \eqref{eq:nhjp}.
We have the following result
\begin{prop}\label{prop:nhjtosh}
The nonholonomic jet prolongation factors (uniquely) through the inclusion of the semiholonomic jet subfunctor, that is
\begin{equation}
j^{(n)}_d=\iota_{J^{[n]}_d}\circ j^{[n]}_d.
\end{equation}
\end{prop}
\begin{proof}
Given the explicit formulae for the prolongation (cf.\ \eqref{eq:jexpl}) and the projections (cf.\ Remark \ref{rmk:projexpl}), one can straightforwardly prove the proposition object-wise.
We will prove it instead via categorical methods.

Let $1\le m\le n$ and consider the projection $\pi^{(n,n-1;m)}_d=J^{(n-m)}_d\pi^{1,0}_{d,J^{(m-1)}_d}$.
By \eqref{eq:nhjpdec}, we can decompose the prolongation as
\begin{equation}
j^{(n)}_d
=j^{(n-m)}_{d,J^{(m)}_d}\circ j^{(m)}_d.
\end{equation}
Thus, combining the diagram encoding \eqref{eq:jpi} with the naturality diagram of $j^{(n-m)}_{d}$ with respect to the natural transformation $\pi^{(m,m-1)}_d\colon J^{(m)}_d\to J^{(m-1)}_d$, we obtain
\begin{equation}
\begin{tikzcd}[column sep=50pt,row sep=30pt]
\id_{\AMod}\ar[r,hookrightarrow,"j^{(m)}_d"']\ar[rr,hookrightarrow,bend left=15pt,"j^{(n)}_d"]\ar[d,equals]&J^{(m)}_d\ar[r,hookrightarrow,"j^{(n-m)}_{d,J^{(m)}_d}"']\ar[d,twoheadrightarrow,"\pi^{(m,m-1)}_d"']&J^{(n)}_d\ar[d,twoheadrightarrow,"\pi^{(n,n-1;m)}_d"]\\
\id_{\AMod}\ar[r,hookrightarrow,"j^{(m-1)}_d"]\ar[rr,hookrightarrow,bend right=15pt,"j^{(n-1)}_d"']&J^{(m-1)}_d\ar[r,hookrightarrow,"j^{(n-m)}_{d,J^{(m-1)}_d}"]&J^{(n-1)}_d
\end{tikzcd}
\end{equation}
For all $m$, the diagram shows that $\pi^{(n,n-1;m)}_d\circ j^{(n)}_d=j^{(n-1)}_d$, whence we conclude that $j^{(n)}_d$ factors through the equalizer of all the $\pi^{(n,n-1;m)}_d$.
\end{proof}
\begin{defi}[Semiholonomic jet prolongation]
We call the map $j^{[n]}_d\colon \id_{\AMod}\to J^{[n]}_d$ described in Proposition \ref{prop:nhjtosh}, the \emph{semiholonomic $n$-jet prolongation}.
\end{defi}
\begin{rmk}
The semiholonomic $n$-jet prolongation is a section of $\pi^{[n,n-1]}_d$.
More generally, for all $0\le m\le n$, from \eqref{eq:jpi}, and the fact that $\iota_{J^{[n]}_d}$ is injective, we get
\begin{equation}
\pi^{[n,m]}_d\circ j^{[n]}_d
=j^{[m]}_d.
\end{equation}
\end{rmk}

We will now prove the following technical lemma, which is a category theoretical fact that will be used frequently further on.
\begin{lemma}\label{lemma:pb}
In an abelian category, consider a commutative square as follows
\begin{equation}
\begin{tikzcd}
M\ar[r,"f"]\ar[d,"g"]&N\ar[d,hookrightarrow,"h"]\\
M'\ar[r,"f'"]&N'
\end{tikzcd}
\end{equation}
and assume $h$ is a monomorphism, then $\ker(f)$ is the pullback of the kernel of $f'$ along $g$ via the universal morphisms.
Moreover, if $g$ is a monomorphism, then $\ker(f)=\ker(f')\cap M$.
\end{lemma}
\begin{proof}
We compute the kernels of $f$ and $f'$ and observe that we have a unique map $k$ by the universal property of the kernel.
\begin{equation}
\begin{tikzcd}
K\ar[d,dashed,"k"]\ar[r,hookrightarrow,"\iota"]&M\ar[r,"f"]\ar[d,"g"]&N\ar[d,hookrightarrow,"h"]\\
K'\ar[r,hookrightarrow,"\iota'"]&M'\ar[r,"f'"]&N'
\end{tikzcd}
\end{equation}
We will now prove that the left square is a pullback square by showing the universal property.
Let $x\colon L\to M$ and $y\colon L\to K'$ be such that $\iota'\circ y=g\circ x$.
Composing this equality with $f'$ gives
\begin{equation}
h\circ f\circ x
=f'\circ g\circ x
=f'\circ \iota'\circ y
=0.
\end{equation}
Since $h$ is a mono, this implies $f\circ x=0$ and hence the unique factorization of $x$ through $K$ as $x=\iota\circ z$.
Notice that $k\circ z=y$, since composing with the mono $\iota'$ we get
\begin{equation}\label{eq:iotaprimestuff}
\iota'\circ k\circ z
=g\circ \iota\circ z
=g\circ x
=\iota' \circ y.
\end{equation}

Moreover, if $g$ is a mono, then $\iota\circ k=g\circ\iota$ is also a mono, implying $k$ mono.
\end{proof}

Let us now define the following natural transformation $\widetilde{\DH}^I_d$ as the unique map
\begin{equation}\label{eq:DHI}
\begin{tikzcd}
\widetilde{\DH}^{I}_d
\colon J^{(2)}_d\ar[r]& \Omega^{1}_d
&
\textrm{such that}
&
\iota^1_d\circ \widetilde{\DH}^I_d
=J^1_d(\pi^{1,0}_d)-\pi^{1,0}_{d,J^1_d}.
\end{tikzcd}
\end{equation}
\begin{rmk}\label{rmk:DHIeq}
Notice that $J^1_d(\pi^{1,0}_d)-\pi^{1,0}_{d,J^1_d}$ coincides with $\pi^{(2,1;1)}_d-\pi^{(2,1;2)}_d\colon J^{(2)}_d\to J^{(1)}_d$, and they factor through $\iota^1_d\colon\Omega^1_d\hookrightarrow J^1_d$ because from \eqref{eq:picomm} we get
\begin{equation}
\pi^{1,0}_d\circ \left(\pi^{(2,1;1)}_d-\pi^{(2,1;2)}_d \right)
=\pi^{1,0}_d\circ \pi^{(2,1;1)}_d-\pi^{1,0}_d\circ \pi^{(2,1;2)}_d
=0.
\end{equation}
\begin{rmk}
This map is a generalization of $\lambda$ appearing in \cite[Proposition 4, p.~433]{Goldschmidt}.
\end{rmk}
We can describe this map explicitly on a left $A$-module $E$ by
\begin{align}
\widetilde{\DH}^{I}_{d,E}= J^1_d (\pi^{1,0}_{d,E})-\pi^{1,0}_{d,J^1_d E}
\colon J^{(2)}_d E\longrightarrow \Omega^{1}_d E,
&\hfill&
[a_1\otimes b_1]\otimes_A[a_2\otimes b_2]\otimes_A e \longmapsto a_1 d(b_1a_2)b_2 \otimes_A e.
\end{align}
We can more implicitly write $\widetilde{\DH}^I_{d,E}$ as a sort of commutativity condition between $\pi^{1,0}_d$ and $\rho_d$, i.e.
\begin{equation}\label{DHIcom}
\widetilde{\DH}^I_{d,E}
=\Omega^1_d(\pi^{1,0}_{d,E})\circ \rho_{d,J^1_d E}-\rho_{d,E}\circ \pi^{1,0}_{d,J^1_d E}.
\end{equation}
\end{rmk}
\begin{rmk}\label{rmk:rhodif}
The map $\rho_d$ (cf.\ \eqref{def:rho}) is a first order differential operator with lift $-\widetilde{\DH}^I_d\colon J^{(2)}_d A\to \Omega^1_d$.
The same happens for the induced natural transformation $\rho_d$ (cf.\ \eqref{eq:rhon}).
In particular $\rho_d=-\DH^I_d\colon J^1_d\to \Omega^1_d$.
\end{rmk}

The following theorem gives equivalent descriptions of $J^{[n]}_d$.

\begin{theo}[Characterization of the semiholonomic jet functor]\label{theo:sjchar}
Let $A$ be a $\bk$-algebra endowed with a first order differential calculus $\Omega^1_d$.
Then for all $n\ge 0$, we have
\begin{equation}\label{theo:sjchar:3}
J^{[n]}_d=\bigcap_{m=1}^{n-1}\ker\left(J^{(n-m-1)}_d \widetilde{\DH}^I_{d,J^{(m-1)}_d} \right).
\end{equation}

Furthermore, if $\Omega^1_d$ is flat in $\ModA$, then for $n\ge 2$ the following subfunctors of $J^{(n)}_d$ coincide
\begin{enumerate}
\item\label{theo:sjchar:1} $J^{[n]}_d$;
\item\label{theo:sjchar:2} The kernel of
\begin{equation}\label{eq:feq}
\begin{tikzcd}
\widetilde{\DH}^I_{d,J^{[n-2]}_d}\circ J^1_d(l^{[n-1]}_d)\colon J^{1}_d\circ J^{[n-1]}_d\ar[r]& \Omega^1_d\circ J^{[n-2]}_d,
\end{tikzcd}
\end{equation}
where $l^{[n-1]}_d\colon J^{[n-1]}_d\hookrightarrow J^1_d\circ J^{[n-2]}_d$ denotes the natural inclusion ($l^{[1]}_d=\id_{J^1_d}$ and $l^{[0]}_d=\id_{A}$), and
\begin{equation}\label{eq:decsjn}
\begin{tikzcd}
\iota_{J^{[n]}_d}=J^{(n-1)}_d(l^{[1]}_d)\circ\cdots\circ J^1_d(l^{[n-1]}_d)\circ l^{[n]}_d\colon J^{[n]}_d\ar[r]&J^{(n)}_d.
\end{tikzcd}
\end{equation}
\item\label{theo:sjchar:4} The intersection (pullback) of $J^{[h]}_d\circ J^{(n-h)}_d$ and $J^{(n-m)}_d \circ J^{[m]}_d$ in $J^{(n)}_d$ for $m,h\ge 0$ such that $h+m>n$.
\end{enumerate}
\end{theo}
\begin{proof}
The equality \eqref{theo:sjchar:3} holds for $n=0$ and $n=1$, as in these cases both equalizer and intersection of kernels describe an empty condition.
We prove the cases $n\ge 2$ from the fact that the maps $\pi^{(n,n-1;m)}_d$ and $\pi^{(n,n-1;m+1)}_d$ for $1\le m\le n-1$ coincide if and only if their difference vanish, and $\pi^{(n,n-1;m)}_d-\pi^{(n,n-1;m-1)}_d=J^{(n-m)}_d\widetilde{\DH}^{I}_{d,J^{(m-1)}_d}$.
The statement follows because a finite collection of morphisms coincide if and only if their consecutive differences vanishes.

For the remaining cases, we proceed by induction.
All conditions coincide for $n=2$, since $\ker(\widetilde{\DH}^I_d)=\Eq(\pi^{(2,1;1)}_d,\pi^{(2,1;2)}_d)$ (cf.\ Remark \ref{rmk:DHIeq}).
Now let $n>2$ and assume the statement true for all $m<n$.
\eqref{theo:sjchar:2}$\subseteq$\eqref{theo:sjchar:1}: By using \eqref{theo:sjchar:3}, we need to prove that $J^{(n-m-1)}_d \widetilde{\DH}_{d,J^{(m-1)}_d}$ vanishes on \eqref{theo:sjchar:2} for $1\le m\le n-1$.
For $1\le m\le n-2$ it follows by inductive hypothesis, as the subfunctor \eqref{theo:sjchar:2} is contained in $J^1_d\circ J^{[n-1]}_d$.
Therefore, for $1\le m\le n-2$ we have
\begin{equation}
J^{(n-m-1)}_d \widetilde{\DH}_{d,J^{(m-1)}_d}\circ J^1_d(\iota_{J^{[n]}_d})
=J^1_d\left(J^{(n-m-2)}_d \widetilde{\DH}_{d,J^{(m-1)}_d}\circ \iota_{J^{[n]}_d} \right),
\end{equation}
which vanishes by inductive hypothesis.

If instead $m=n-1$, then we have to show that $\widetilde{\DH}^{I}_{d,J^{(n-2)}_d}$ vanishes on $J^{[n]}_d$.
By naturality of $\widetilde{\DH}^I_d$ and $\iota_{J^{[n-2]}_d}$, we obtain the right-hand square in the following commutative diagram
\begin{equation}\label{eq:edhInat}
\begin{tikzcd}[column sep=40pt]
J^1_d\circ J^{[n-1]}_d \ar[dr,hookrightarrow,"J^1_d (\iota_{J^{[n-1]}_d})"']\ar[r,hookrightarrow,"J^1_d (l^{[n-1]}_d)"]&J^{(2)}_d\circ J^{[n-2]}_d\ar[r,"\widetilde{\DH}^{I}_{d,J^{[n-2]}_d}"]\ar[d,hookrightarrow,"J^{(2)}_d \iota_{J^{[n-2]}_d}"]&\Omega^{1}_d\circ J^{[n-2]}_d\ar[d,hookrightarrow,"\Omega^1_d (\iota_{J^{[n-2]}_d})"]\\
&J^{(n)}_d \ar[r,"\widetilde{\DH}^{I}_{d,J^{(n-2)}_d}"']&\Omega^1_d \circ J^{(n-2)}_d
\end{tikzcd}
\end{equation}
From this diagram we deduce that \eqref{theo:sjchar:2}, which is the kernel of the composition of the top two horizontal maps, vanishes on $\widetilde{\DH}^{I}_{d,J^{(n-2)}_d}\circ J^1_d (\iota_{J^{[n]}_d})$.

\eqref{theo:sjchar:1}$=$\eqref{theo:sjchar:4}: The only non-trivial cases occur for $m,h\ge 2$.
In this case, we obtain
\begin{equation}
\begin{split}
\bigcap_{j=1}^{n-1}\ker\left(J^{(n-j-1)}_d\widetilde{\DH}^I_{d,J^{(j-1)}_d}\right)
&=\bigcap_{j=n-h+1}^{n-1}\ker\left(J^{(n-j-1)}_d\widetilde{\DH}_{d,J^{(j-1)}_d}\right)\cap \bigcap_{j=1}^{m-1}\ker\left(J^{(n-j-1)}_d\widetilde{\DH}_{d,J^{(j-1)}_d}\right)\\
&=\bigcap_{j=1}^{h-1}\ker\left(J^{(h-j)}_d\widetilde{\DH}_{d,J^{(j-1)}_d}\right)\circ J^{(n-h)}_d \cap J^{(n-m)}_d\circ \bigcap_{j=1}^{m-1}\ker\left(J^{(m-j-1)}_d\widetilde{\DH}_{d,J^{(j-1)}_d}\right)\\
&=\left(J^{[h]}_d\circ J^{(n-h)}_d \right)\cap \left(J^{(n-m)}_d \circ J^{[m]}_d\right).
\end{split}
\end{equation}
The second equality is a consequence of the fact that limits are computed object-wise and the fact that $J^1_d$ preserves finite limits.
The last equality follows from the inductive hypothesis.

Notice that we have also proved that for $h,m\ge 2$, the subfunctors $\left(J^{[h]}_d\circ J^{(n-h)}_d \right)\cap \left(J^{(n-m)}_d \circ J^{[m]}_d\right)$ coincide.

\eqref{theo:sjchar:4}$\subseteq$\eqref{theo:sjchar:2}: There is nothing to prove unless $h,m\ge 2$, in which case we have observed that for all such $h,m$, the intersections of the form \eqref{theo:sjchar:4} coincide, so it is enough to prove that
\begin{equation}
\left(J^{[2]}_d\circ J^{(n-2)}_d \right)\cap \left(J^{1}_d \circ J^{[n-1]}_d\right)
\end{equation}
vanishes if we apply $\widetilde{\DH}^I_{d,J^{[n-2]}_d}\circ J^1_d (l^{[n-1]}_d)$.
Since $\Omega^1_d(\iota_{J^{[n-2]}_d})$ is a monomorphism by right flatness of $\Omega^1_d$, it is equivalent to check the vanishing of
\begin{equation}
\Omega^1_d(\iota_{J^{[n-2]}_d}) \circ \widetilde{\DH}^I_{d,J^{[n-2]}_d}\circ J^1_d (l^{[n-1]}_d)
=\widetilde{\DH}^I_{d,J^{(n-2)}_d}\circ J^1_d(\iota_{J^{[n-1]}_d}),
\end{equation}
(cf.\ \eqref{eq:edhInat}).
If we compose this map with the inclusion
\begin{equation}
\left(J^{[2]}_d\circ J^{(n-2)}_d \right)\cap \left(J^{1}_d \circ J^{[n-1]}_d\right)\longhookrightarrow J^1_d\circ J^{[n-1]}_d,
\end{equation}
which is itself the pullback of the inclusion $J^{[2]}_d\circ J^{(n-2)}_d\hookrightarrow J^{(n)}_d$ via $J^{1}_d \circ J^{[n-1]}_d\hookrightarrow J^{(n)}_d$, we obtain zero by commutativity.
This can be better seen in the following commutative diagram.
\begin{equation}\label{eq:edhInat1}
\begin{tikzcd}[column sep=40pt]
\left(J^{[2]}_d\circ J^{(n-2)}_d \right)\cap \left(J^{1}_d \circ J^{[n-1]}_d\right)\ar[r,hookrightarrow]\ar[d,hookrightarrow]&J^1_d\circ J^{[n-1]}_d \ar[d,hookrightarrow,"J^1_d (\iota_{J^{[n-1]}_d})"']\ar[r,hookrightarrow,"J^1_d (l^{[n-1]}_d)"]&J^{(2)}_d\circ J^{[n-2]}_d\ar[r,"\widetilde{\DH}^{I}_{d,J^{[n-2]}_d}"]\ar[d,hookrightarrow,"J^{(2)}_d \iota_{J^{[n-2]}_d}"]&\Omega^{1}_d\circ J^{[n-2]}_d\ar[d,hookrightarrow,"\Omega^1_d (\iota_{J^{[n-2]}_d})"]\\
J^{[2]}_d\circ J^{(n-2)}_d\ar[r,hookrightarrow,"(\iota_{J^{[2]}_{d}})_{J^{(n-2)}_d}"']&J^{(n)}_d \ar[r,equals]&J^{(n)}_d \ar[r,"\widetilde{\DH}^{I}_{d,J^{(n-2)}_d}"']&\Omega^1_d \circ J^{(n-2)}_d
\end{tikzcd}
\end{equation}
Here, in fact, the bottom horizontal composition vanishes.
\end{proof}
\begin{rmk}
Notice that \eqref{theo:sjchar:2} of Theorem \ref{theo:sjchar} enables us to define the semiholonomic jet functor inductively once we set $J^{[1]}_d=J^1_d$ and $J^{[0]}_d=\id_{\AMod}$.
\end{rmk}

We are now ready to describe the short exact sequence for semiholonomic jets.
\begin{theo}[Semiholonomic jet short exact sequence]\label{theo:Jshes}
Let $A$ be a $\bk$-algebra endowed with a first order differential calculus $\Omega^1_d$ which is flat in $\ModA$.
For $n\ge 1$, the following is a short exact sequence
\begin{equation}\label{es:shJet}
\begin{tikzcd}[column sep=40pt]
0\ar[r]& T^n_d\ar[r,hookrightarrow,"\iota^{[n]}_d"]& J^{[n]}_d\ar[r,twoheadrightarrow,"\pi^{[n,n-1]}_d"]& J^{[n-1]}_d\ar[r]&0
\end{tikzcd}
\end{equation}
Moreover, for $n\ge 2$, the natural transformation \eqref{eq:feq} is an epimorphism.
\end{theo}
\begin{proof}
We prove this proposition by induction on $n\ge 1$.
Notice that for $n=1$, the sequence \eqref{es:shJet} coincides with the $1$-jet exact sequence \eqref{es:natJ1}.

Now suppose the statement is true for $n-1$.
The sequence is obtained as the kernel of a morphism of short exact sequences, as shown in the following diagram.
\begin{equation}\label{es:shJetcd}
\begin{tikzcd}[column sep=50pt]
&T^n_d=\Omega^1_d\circ T^{n-1}_d\ar[r,"\iota^{[n]}_d"]\ar[d,hookrightarrow,"\Omega^1_d(\iota^{[n-1]}_d)"']& J^{[n]}_d\ar[r,"\pi^{[n,n-1]}_d"] \ar[d,hookrightarrow,"l^{[n]}_d"]& J^{[n-1]}_d\ar[d,equals]\\
0\ar[r]&\Omega^1_d\circ J^{[n-1]}_d\ar[r,hookrightarrow,"\iota^1_{d,J^{[n-1]}_d}"]\ar[d,twoheadrightarrow,"\Omega^1_d(\pi^{[n-1,n-2]}_d)"']& J^1_d\circ J^{[n-1]}_d\ar[r,twoheadrightarrow,"\pi^{1,0}_{d,J^{[n-1]}_d}"]\ar[d,"\widetilde{\DH}^I_{d,J^{[n-2]}}\circ J^1_d(l^{[n-1]}_d)"]& J^{[n-1]}_d\ar[d]\ar[r]&0\\
0\ar[r]&\Omega^1_d\circ J^{[n-2]}_d\ar[r,equals]& \Omega^1_d\circ J^{[n-2]}_d \ar[r]& 0\ar[r]&0
\end{tikzcd}
\end{equation}
The two bottom rows are exact, being the $1$-jet exact sequence at $J^{[n-1]}_d$ and an equality, respectively.
Notice also that the right bottom square trivially commutes and the left square commutes by naturality (cf.\ \eqref{DHIcom}).
More explicitly, we have
\begin{equation}
\begin{split}
&\iota^{(n-1)}_d\circ\Omega^1_d (\iota_{J^{[n-2]}_d})\circ \widetilde{\DH}^I_{d,J^{[n-2]}_d}\circ J^1_d(l^{[n-1]}_d)\circ \iota^1_{d,J^{[n-1]}_d}\\
&\qquad\qquad= \iota^{(n-1)}_d\circ\widetilde{\DH}^I_{d,J^{(n-2)}_d}\circ J^1_d(\iota_{J^{[n-1]}_d})\circ \iota^1_{d,J^{[n-1]}_d}\\
&\qquad\qquad= \iota^{(n-1)}_d\circ\widetilde{\DH}^I_{d,J^{(n-2)}_d}\circ \iota^1_{d,J^{(n)}_d}\circ \Omega^1_d(\iota_{J^{[n-1]}_d})\\
&\qquad\qquad= \left( \pi^{(n,n-1;n-1)}_d- \pi^{(n,n-1;n)}_d\right)\circ \iota^{(n)}_d\circ \Omega^1_d(\iota_{J^{[n-1]}_d})\\
&\qquad\qquad= \pi^{(n,n-1;n-1)}_d \circ \iota^{(n)}_d\circ \Omega^1_d(\iota_{J^{[n-1]}_d}).
\end{split}
\end{equation}
We now have the following commutative diagram
\begin{equation}
\begin{tikzcd}[column sep=50pt]
\Omega^1_d\circ J^{[n-1]}_d\ar[r,hookrightarrow,"\Omega^1_d(\iota_{J^{[n-1]}_d})"]\ar[d,twoheadrightarrow,"\Omega^1_d(\pi^{[n-1,n-2]}_d)"']& \Omega^1_d\circ J^{(n-1)}_d\ar[r,hookrightarrow,"\iota^{(n)}_d"]\ar[d,twoheadrightarrow,"\Omega^1_d(\pi^{(n-1,n-2)}_d)"]& J^{(n)}_d\ar[d,twoheadrightarrow,"\pi^{(n,n-1;n-1)}_d"]\\
\Omega^1_d\circ J^{[n-2]}_d\ar[r,hookrightarrow,"\Omega^1_d(\iota_{J^{[n-2]}_d})"']& \Omega^1_d\circ J^{(n-2)}_d \ar[r,hookrightarrow,"\iota^{(n-1)}_d"']& J^{(n-1)}_d
\end{tikzcd}
\end{equation}
This implies
\begin{equation}
\begin{split}
&\iota^{(n-1)}_d\circ\Omega^1_d (\iota_{J^{[n-2]}_d})\circ \widetilde{\DH}^I_{d,J^{[n-2]}}\circ J^1_d(l^{[n-1]}_d)\circ \iota^1_{d,J^{[n-1]}_d}\\
&\qquad\qquad= \pi^{(n,n-1;n-1)}_d \circ \iota^{(n)}_d\circ \Omega^1_d(\iota_{J^{[n-1]}_d})\\
&\qquad\qquad=\iota^{(n-1)}_d\circ \Omega^1_d (\iota_{J^{[n-2]}_d})\circ \Omega^1_d(\pi^{[n-1,n-2]}_d).
\end{split}
\end{equation}
The fact that $\iota^{(n-1)}_d\circ\Omega^1_d (\iota_{J^{[n-2]}_d})$ is a monomorphism implies the commutativity of the bottom left square in \eqref{es:shJetcd}.

Now that we have proven that the lower part of \eqref{es:shJetcd} is a map of short exact sequences, we will prove that the first row is its kernel.
The left column forms a short exact sequence, since it is obtained by applying $\Omega^1_d$ to the semiholonomic jet exact sequence for $n-1$.
The central column is left exact, as shown in Theorem \ref{theo:sjchar}, and the third column is visibly exact.
Commutativity of the top right square in \eqref{es:shJetcd} can be seen by composing both terms with the monomorphism $\iota_{J^{[n-1]}_d}$ and considering the following commutative diagram.
\begin{equation}
\begin{tikzcd}[column sep=40pt]
J^{[n]}_d \ar[dr,hookrightarrow,"\iota_{J^{[n]}_d}"']\ar[r,hookrightarrow,"l^{[n-1]}_d"]&J^1_d\circ J^{[n-1]}_d \ar[r,"\pi^{1,0}_{d,J^{[n-1]}_d}"]\ar[d,hookrightarrow,"J^{1}_d( \iota_{J^{[n-1]}_d})"]&J^{[n-1]}_d\ar[d,hookrightarrow,"\iota_{J^{[n-1]}_d}"]\\
&J^{(n)}_d \ar[r,"\pi^{(n,n-1)}_d"']& J^{(n-1)}_d
\end{tikzcd}
\end{equation}
Since the composition $\pi^{1,0}_{d,J^{[n-1]}_d} \circ l^{[n-1]}_d$ extends to $\pi^{(n,n-1)}_d$, by Proposition \ref{prop:piwd} it must be $\pi^{[n,n-1]}_d$.

Applying the snake lemma, we deduce that the top horizontal sequence is exact and that the vertical central sequence is also right exact, concluding the proof.
\end{proof}
\begin{rmk}
Let $E$ be a left $A$-module, and consider an element $\xi\in J^{[n]}_d E$.
A projection $\pi^{(n,n-1;m)}_{d,E}$ vanishes on $\xi$ if and only if it vanishes on all the other projections.
Thus, under the hypothesis that $\Omega^1_d$ is flat in $\ModA$, we have that the intersection between $J^{(n-m)}_d\circ \Omega^1_d\circ J^{(m-1)}_d$ and $J^{[n]}_d$ is precisely $T^n_d$.
\end{rmk}
\begin{rmk}[Semiholonomic jet functors on bimodules]\label{rem:semiholjetfunctoronbimodules}
By Remark \ref{rmk:nJbi}, we have nonholonomic jet bundles $J^{(m)}_d$ as endofunctors of $\AMod_B$.
We can also compatibly lift all nonholonomic projections $\pi^{(n,n-1;m)}_{d}$ to natural transformations between endofunctors $\AMod_B$.
Since $J^{[n]}_d$ is defined as a limit and the forgetful functor $\AMod_B\to \AMod$ creates limits, we can define $J^{[n]}_d$ object-wise as an equalizer in $\AMod_B$.
The semiholonomic jet functor thus lifts to an endofunctor on $\AMod_B$.
The same can be done for the whole semiholonomic sequence.
In this setting, the semiholonomic prolongation is a natural transformation in the category of functors $\AMod_B\to \Mod_B$.
\end{rmk}

\subsubsection{Stability and exactness}
From Theorem \ref{theo:Jshes}, we can deduce the following result regarding the stability of certain subcategories with respect to the functor $J^{[n]}_d$.
\begin{prop}\label{prop:Jshstable}
Let $A$ be a $\bk$-algebra endowed with a first order differential calculus $\Omega^1_d$ which is flat in $\ModA$.
Consider the functor
\begin{equation}
J^{[n]}_d\colon \AMod\longrightarrow\AMod.
\end{equation}
\begin{enumerate}
\item\label{prop:Jshstable:1} If $\Omega^1_d$ is in $\AFlat$, then $J^{[n]}_d$ preserves the subcategory $\AFlat$;
\item\label{prop:Jshstable:2} If $\Omega^1_d$ is in $\AProj$, then $J^{[n]}_d$ preserves the subcategory $\AProj$;
\item\label{prop:Jshstable:3} If $\Omega^1_d$ is in $\AFG$, then $J^{[n]}_d$ preserves the subcategory $\AFG$;
\item\label{prop:Jshstable:4} If $\Omega^1_d$ is in $\AFGP$, then $J^{[n]}_d$ preserves the subcategory $\AFGP$.
\end{enumerate}
\end{prop}
\begin{proof}
For $n=0$ there is nothing to prove.
Now consider $n>0$ and assume \eqref{prop:Jshstable:1} true for $n-1$.
Since $\Omega^1_d$ is flat in $\ModA$, we can apply Proposition \ref{theo:Jshes}, giving us the semiholonomic jet short exact sequence.

If $\Omega^1_d$ is flat in $\AMod$, then given $E$ in $\AFlat$, we know that $T^n_d(E)$ is in $\AFlat$ by Proposition \ref{prop:Tstable}, and $J^{[n-1]}_d (E)$ is in $\AFlat$ by inductive hypothesis.
Lemma \ref{lemma:2-3} implies that $J^{[n]}_d(E)$ is in $\AFlat$, whence \eqref{prop:Jshstable:1}.

The proof for \eqref{prop:Jshstable:2} and \eqref{prop:Jshstable:3} is analogous, and \eqref{prop:Jshstable:4} follows from them.
\end{proof}
The following proposition discusses exactness of $J^{[n]}_d$ as a functor
\begin{prop}\label{prop:shJregular}
Let $A$ be a $\bk$-algebra endowed with a first order differential calculus $\Omega^1_d$ which is flat in $\ModA$.
Then $J^{[n]}_d$ is an exact functor.
\end{prop}
\begin{proof}
We first prove left exactness.
Consider in $\AMod$ the following short exact sequence
\begin{equation}
\begin{tikzcd}
0\ar[r]&M\ar[r,hookrightarrow]&N\ar[r,twoheadrightarrow]&Q\ar[r]&0.
\end{tikzcd}
\end{equation}
Since $J^{(n)}_d$ is an exact functor for all $n$ by Remark \ref{rmk:nhJexact}, and $\pi^{(n,n-1;k)}_d$ for all $1\le k\le n$ is a natural transformation, we obtain for each $k$ a diagram of the form
\begin{equation}
\begin{tikzcd}
0\ar[r]&J^{(n)}_d M\ar[r,hookrightarrow]\ar[d,"\pi^{(n,n-1;k)}_{d,M}"]&J^{(n)}_d N\ar[r,twoheadrightarrow]\ar[d,"\pi^{(n,n-1;k)}_{d,N}"]&J^{(n)}_d Q\ar[r]\ar[d,"\pi^{(n,n-1;k)}_{d,Q}"]&0\\
0\ar[r]&J^{(n-1)}_d M\ar[r,hookrightarrow]&J^{(n-1)}_d N\ar[r,twoheadrightarrow]&J^{(n-1)}_d Q\ar[r]&0
\end{tikzcd}
\end{equation}
Combined, they form a diagram of exact sequences, that is an exact sequence of diagrams.
The limit of this diagram, which in this case is an equalizer, will give us a left exact sequence, because every limit functor is a right adjoint, and hence left exact.
By definition, we know that the equalizer of the projections is precisely $J^{[n]}_d$, so we have the following left exact sequence
\begin{equation}
\begin{tikzcd}
0\ar[r]&J^{[n]}_d M\ar[r,hookrightarrow]&J^{[n]}_d N\ar[r]&J^{[n]}_d Q,
\end{tikzcd}
\end{equation}
whence $J^{[n]}_d$ is left exact.

We prove right exactness by induction on $n\ge 0$.
For $n=0$, $J^{[0]}_d$ is the identity functor, which is trivially exact.
Now, let $n>0$ and assume the functor $J^{[n-1]}_d$ is exact.
Applying the semiholonomic short exact sequence of Proposition \ref{theo:Jshes}, we obtain the following double complex
\begin{equation}
\begin{tikzcd}[column sep=50pt]
0\ar[r]&T^n_d M\ar[r,hookrightarrow,"\iota^{[n]}_{d,M}"]\ar[d,hookrightarrow]&J^{[n]}_d M\ar[r,twoheadrightarrow,"\pi^{[n,n-1]}_{d,M}"]\ar[d,hookrightarrow]&J^{[n-1]}_d M\ar[r]\ar[d,hookrightarrow]&0\\
0\ar[r]&T^n_d N\ar[r,hookrightarrow,"\iota^{[n]}_{d,N}"]\ar[d,twoheadrightarrow]&J^{[n]}_d N\ar[r,twoheadrightarrow,"\pi^{[n,n-1]}_{d,N}"]\ar[d]&J^{[n-1]}_d N\ar[r]\ar[d,twoheadrightarrow]&0\\
0\ar[r]&T^n_d Q\ar[r,hookrightarrow,"\iota^{[n]}_{d,Q}"]&J^{[n]}_d Q\ar[r,twoheadrightarrow,"\pi^{[n,n-1]}_{d,Q}"]&J^{[n-1]}_d Q\ar[r]&0
\end{tikzcd}
\end{equation}
The three rows are exact by Proposition \ref{theo:Jshes}, the first column is exact because $T^n_d$ is an exact functor (cf.\ Proposition \ref{prop:Texact}), the second row is left exact by previous considerations, and finally the last row is exact by inductive hypothesis.
By applying the snake lemma on the last two rows, we obtain that the central column is also forced to be exact.
Thus, $J^{[n]}_d$ is an exact functor.
\end{proof}

Observe that the decomposition of Theorem \ref{theo:nhJdec} is not well-behaved with respect to the semiholonomic jet, as its components are not necessarily in the semiholonomic jet itself.
For instance, the component corresponding to $j^1_d\Omega^1_d$ is of the form $[1\otimes 1]\otimes_A \alpha_1\in J^{(2)}_d$, so the projection on the first component will give $\alpha_1$, whereas the projection on the second will give zero.

\subsection{The classical nonholonomic and semiholonomic jet functors}
As the $1$-jet functor $J^1_d$ was our main building block for this section, we can easily extend Theorem \ref{theo:classical1jet} to cover the nonholonomic and semiholonomic jet functors.
We first prove the following technical lemma concerning the preservation of limits by the global section functor realizing the equivalence of Serre-Swan Theorem (cf.\ \cite[§12.33, p.~191]{nestruev2020smooth}).
\begin{lemma}\label{lemma:gammacontinuous}
The global section functor $\Gamma(M,-)$, from the category of vector bundles over a smooth manifold $M$ to the category of $\smooth{M}$-modules, preserves all limits.
\end{lemma}
\begin{proof}
There is a bijection between sections of a vector bundle $E\to M$ and vector bundle maps $M\times \R\to E$.
This bijection is natural in $E$, which implies the natural isomorphism $\Gamma(M,-)\simeq \Hom(M\times \R,-)$, i.e.\ that $\Gamma(M,-)$ is representable.
Furthermore, this isomorphism is $\smooth{M}$-linear with the induced module structure.
Fiber-wise linearity of vector bundle maps and the universal property of the limit allow us to prove that the $\smooth{M}$-module valued $\Hom$ functor preserves limits in its second component.
Thus, given a diagram $E_{\bullet}$ of shape $I$, if its limit $\lim_I E_i$ exists, we have the following isomorphism of $\smooth{M}$-modules
\begin{equation}
\Gamma(M,E)
\cong \Hom(M\times \R,\lim_I E_i)
\cong \lim_I \Hom(M\times \R,E_i)
\cong \lim_I \Gamma(M,E_i).
\end{equation}
Which ends the proof by naturality of the isomorphisms.
\end{proof}
\begin{theo}\label{theo:classicalnonsemi}
	Let $M$ be a smooth manifold, $A=\smooth{M}$ its algebra of smooth functions, and $E=\Gamma(M,N)$ the space of smooth sections of a vector bundle $N\rightarrow M$.
	Let $J^n_{\text{non}}N$ and $J^n_{\text{semi}}N$ denote the classical nonholonomic and classical semiholonomic jet bundles of $N$, respectively.
	Then the following hold.
	\begin{itemize}
		\item $J^{(n)}_dE \cong \Gamma(M,J^n_{\text{non}}N)$ in $\AMod$.
		\item $J^{[n]}_dE \cong \Gamma(M,J^n_{\text{semi}}N)$ in $\AMod$.
	\end{itemize}
	The isomorphisms take our prolongations to the classical prolongations.
\end{theo}
\begin{proof}
	The classical nonholonomic jet functor is just the $n$th iterate of the classical $1$-jet functor, which coincides with $J^1_d$ by Theorem \ref{theo:classical1jet}.
	Hence, the nonholonomic jet functors are isomorphic.
	The semiholonomic jets are classically characterized by being the equalizer of the nonholonomic projections, exactly as in our definition.
	The proof follows from the fact that the functor $\Gamma(M,-)$ preserves limits (cf.\ Lemma \ref{lemma:gammacontinuous}), and hence equalizers.
	
	The prolongations are compatible with these isomorphisms as they are simply the iterates of $j^1_d$, which coincides with the classical prolongation map.
\end{proof}

\section{Generalization of symmetric forms}
\label{s:quantumsymmform}
\subsection{Exterior algebra}
\label{ss:extalg}
We will now construct a noncommutative analogue of the bundle of symmetric tensors with values in a vector bundle $E$, that is, $S^{k}(M,E)=S^k T^\ast M \otimes E$.
We do so by induction, using an exterior algebra \cite[Definition 1.30, p.~22]{BeggsMajid}, which is a differential graded algebra satisfying a certain surjectivity condition.

\begin{defi}\label{def:extalgebra}
An \emph{exterior algebra} $\Omega^\bullet_d$ over a $\bk$-algebra $A$, is an associative graded algebra $(\Omega^{\bullet}_d=\bigoplus_{n\ge 0}\Omega^n_d,\wedge)$ equipped with a map $d$ such that
\begin{enumerate}
\item $\Omega^0_d=A$;
\item $d$ is a \emph{differential map}, that is a $\bk$-linear map $d\colon \Omega^{\bullet}_d\to\Omega^{\bullet}_d$ such that $d(\Omega^n_d)\subseteq \Omega^{n+1}_d$ for all $n\ge 0$, which satisfies $d^2=0$ and
\begin{align}
d(\alpha\wedge\beta)=d\alpha\wedge \beta +(-1)^{n}\alpha\wedge d\beta,
&\hfill&
\forall\alpha\in\Omega^n_d, \beta\in\Omega^\bullet_d.
\end{align}
\item (\emph{surjectivity condition}) $A$, $dA$ generate $\Omega^{\bullet}_d$ via $\wedge$.
\end{enumerate}
\end{defi}
When confusion might arise, we will denote the restrictions of $\wedge$ and $d$ respectively as
\begin{align}
\wedge^{k,h}\colon \Omega^k_d\otimes \Omega^h_d\longrightarrow \Omega^{k+h}_d,
&\hfill&
d^k\colon \Omega^k_d\longrightarrow \Omega^{k+1}_d,
&\hfill&
\forall h,k\ge 0.
\end{align}
\begin{rmk}[Maximal exterior algebra and first order differential calculus]\label{rmk:maximalextalg}
Given an exterior algebra $\Omega^\bullet_d$, the first grade and $d\colon \Omega^0_d=A\to \Omega^1_d$ form a first order differential calculus for $A$.

Vice versa, given a first order differential calculus, $\Omega^1_d$, one can construct a maximal exterior algebra, $\Omega^{\bullet}_{d,\operatorname{max}}$, by quotienting the tensor algebra by the two-sided ideal, with respect to $\otimes_A$, generated by $\sum_i da_i\otimes_A db_i$ for all $\sum_i a_i\otimes b_i\in N_d$ (cf.\ \cite[Lemma 1.32, p.~23]{BeggsMajid}).
Any other exterior algebra with the same first order differential calculus $\Omega^1_d$ is a quotient of the universal exterior algebra, giving us the algebra epimorphisms
\begin{equation}
\Omega^\bullet_{u,\operatorname{max}}\longtwoheadrightarrow \Omega^\bullet_{d,\text{max}}\longtwoheadrightarrow \Omega^\bullet_d.
\end{equation}

\end{rmk}
\begin{eg}[Universal exterior algebra]
For the universal first order differential calculus, $\Omega^1_u$, the maximal exterior algebra coincides with the tensor algebra, i.e.\ $\Omega^{\bullet}_{u,\operatorname{max}}=T^{\bullet}_u$, where $\wedge=\otimes_A$ (cf.\ \cite[Theorem 1.33, p.~24]{BeggsMajid}).
\end{eg}
\begin{eg}[Classical exterior algebra]
	In the example of §\ref{ss:cman}, the maximal exterior algebra corresponds to the classical notion of exterior algebra, where it is realized as the algebra of skew symmetric forms, i.e.\ $\Omega^{\bullet}_d(\smooth{M})=\Omega^{\bullet}(M)$ equipped with the standard exterior derivative \cite[Example 1.35, p.~26]{BeggsMajid}.
\end{eg}
\begin{eg}[Infinitesimal first order differential calculus]
	In the example of §\ref{ss:K2}, $N_d=(tdt)\subset \Omega^1_u$, so the maximal exterior algebra is formed by quotienting the tensor algebra by the ideal generated by $(dt\otimes_A dt)$, that is $\bigoplus_{n\ge 2}T^{n}_d$.
It follows that $\Omega^\bullet_{d,\text{max}}=A\oplus \Omega^1_d\cong A\oplus \bk[0]$.
The maximal exterior algebra is therefore the only possible exterior algebra with this first order differential calculus.
\end{eg}

As we did for $\Omega^1_d$, for each degree $n$, we define a functor as follows
\begin{align}
\Omega^n_d\colon \AMod\longrightarrow\AMod,
&\hfill&
E\longmapsto \Omega^n_d\otimes_A E.
\end{align}
These functors in $\AMod\rightarrow \AMod$ give rise to a functor $\Omega^{\bullet}_d\colonequals \bigoplus_{n\ge 0}\Omega^n_d$.
\begin{rmk}
Since the tensor product is a right exact functor, it preserves colimits, which are computed object-wise, so we can explicitly describe the exterior algebra functor as the tensor product by the exterior algebra bimodule, i.e.\ $\Omega^\bullet_d= \Omega^{\bullet}_d\otimes_A -$.
\end{rmk}
In the functor interpretation, we can see the wedge as a natural epimorphism $\wedge\colon \Omega^\bullet_d\circ \Omega^\bullet_d\to\Omega^\bullet_d$ of functors of type $\AMod\to \AMod$.
In particular, we also get $\wedge\colon T^2_d\twoheadrightarrow \Omega^2_d$, and by associativity we can apply the wedge product simultaneously on each tensor component of $T^n_d$, obtaining the natural epimorphism $\wedge_n \colon T^n_d\twoheadrightarrow \Omega^n_d$.
In particular: $\wedge_2=\wedge$, $\wedge_1=\id_{\Omega^1_d}$, and $\wedge_0=\id_{A}$.
These maps induce a natural graded epimorphism $\wedge_{\bullet}\colon T^{\bullet}_d \twoheadrightarrow \Omega^{\bullet}_d$.

Given the notion of exterior algebra, we can generalize the following concept from differential geometry.
\begin{defi}
\label{defi:exterior_covariant_derivative}
A \emph{(left) exterior covariant derivative} with respect to the exterior algebra $\Omega^{\bullet}_d$ on a left $A$-module $E$ is a $\bk$-linear map
\begin{equation}
d_{\nabla}\colon \Omega^\bullet_d(E)\longrightarrow \Omega^\bullet_d(E)
\end{equation}
such that
\begin{enumerate}
\item $d_\nabla$ restricts to a map $d_\nabla\colon \Omega^n_d(E)\to \Omega^{n+1}_d(E)$;
\item (Leibniz rule) for all $\alpha\in\Omega^n_d$ and $\beta\in\Omega^\bullet_d(E)$, we have
\begin{equation}
d_{\nabla}(\alpha\wedge \beta)
=d\alpha\wedge \beta+(-1)^{n}\alpha\wedge d_{\nabla}\beta.
\end{equation}
\end{enumerate}

The \emph{curvature} for $d_\nabla$ is
\begin{equation}
R_\nabla \colonequals d_\nabla\circ d_\nabla \colon \Omega^\bullet_d(E)\to \Omega^\bullet_d(E).
\end{equation}
\end{defi}
As for the classical case, the restriction to degree zero of an exterior covariant derivative $d_{\nabla}\colon E\to \Omega^1_d(E)$ is a connection.
Vice versa, as in \cite[(4.2), p.~295]{BeggsMajid}, every connection $\nabla\colon E\to \Omega^1_d(E)$ can be uniquely prolonged to an exterior covariant derivative by defining for all $\alpha\otimes_A e\in\Omega^n_d(E)$
\begin{equation}
d_\nabla(\alpha\otimes_A e)
=d\alpha\otimes_A e+(-1)^n \alpha\wedge \nabla e.
\end{equation}
For this reason we will denote an exterior covariant derivative $d_\nabla$ by making the associated connection explicit.
\begin{rmk}
\label{rmk:ecdcurvature}
Concerning the curvature $R_\nabla$, if we restrict it to degree zero, we obtain a map $E\to\Omega^2_d(E)$, which coincides with the notion of curvature for $\nabla$ given in \cite[Definition 3.18, p.~219]{BeggsMajid}.
\end{rmk}
As for the classical case, the curvature of an exterior covariant derivative is fully determined by the information in degree zero.
We prove this fact in the following lemma, which also provides a slight generalization and an alternative proof of \cite[Lemma 3.19, p.~219]{BeggsMajid} proving that the curvature of a connection $\nabla$ is left $A$-linear.
\begin{lemma}
\label{lemma:ecdcurvaturelinear}
The curvature of an exterior covariant derivative $d_\nabla\colon \Omega^{\bullet}_d(E)\to \Omega^{\bullet}_d(E)$ can be described on all $\omega\otimes_A e\in\Omega^n_d(E)$ as
\begin{equation}\label{eq:charcurv}
R_\nabla(\omega\otimes_A e)=\omega\wedge R_\nabla(e).
\end{equation}

In particular $R_\nabla$ is left $A$-linear.
\end{lemma}
\begin{proof}
We compute the curvature as
\begin{equation}
\begin{split}
R_\nabla(\omega\otimes_A e)
&=d_\nabla^2(\omega \otimes_A e)\\
&=d_\nabla(d\omega \wedge e+(-1)^n\omega \wedge d_\nabla e)\\
&=0+(-1)^{n+1}d\omega \wedge d_\nabla e+(-1)^n d\omega \wedge d_\nabla e+(-1)^{2n}\omega \wedge d_\nabla^2 e\\
&=\omega\wedge R_\nabla (e).
\end{split}
\end{equation}

From this description it follows that $R_\nabla$ is left $A$-linear.
In particular, the left $A$-linearity in degree zero follows by taking $\omega=a\in A=\Omega^0_d$.
\end{proof}

As we did for connections, we can prove that exterior covariant derivatives are first order differential operators
\begin{prop}\label{prop:excoddiffop}
Let $A$ be a $\bk$-algebra and $\Omega^{\bullet}_d$ an exterior algebra over it.
Given a left $A$-module $E$, an exterior covariant derivative $d_{\nabla}\colon \Omega^{\bullet}_d(E)\to \Omega^{\bullet}_d(E)$ is a differential operator of order at most $1$ with lift
\begin{align}
\widetilde{d}_{\nabla}\colon J^1_d \Omega^\bullet_d (E)\longrightarrow \Omega^\bullet_d (E),
&\hfill&
[a\otimes \omega]\longmapsto ad_\nabla(\omega).
\end{align}

Its curvature $R_\nabla\colon \Omega^\bullet_d(E)\to \Omega^\bullet_d(E)$ is a differential operator of order zero.
\end{prop}
\begin{proof}
We apply the characterization provided by Proposition \ref{1storderwrtd}.
Let $\sum_i n_i\otimes \omega_i\in N_d(\Omega^{\bullet}_d(E))$, then
\begin{align}
\sum_i n_i d_\nabla (\omega_i)
=d_\nabla\left(\sum_i n_i \omega_i\right)-\sum_i dn_i\wedge \omega_i
=0.
\end{align}
The two terms vanish by the definition of $N_d$.

The second statement is a corollary of Lemma \ref{lemma:ecdcurvaturelinear}.
\end{proof}

\subsection{Functor of symmetric forms}
\begin{defi}[Functor of symmetric forms]\label{def:qsf}
Given a $\bk$-algebra $A$ endowed with an exterior algebra $\Omega^{\bullet}_d$, we define $S^0_d=\Omega^0_d=\id_{\AMod}$, $S^1_d=\Omega^1_d$, and $\iota^1_{\wedge}=\id_{\Omega^1_d}$.
For $n\ge 2$, by induction, $S^n_d$ is the kernel of
\begin{equation}
\wedge_{S^{n-2}_d} \circ \Omega^1_d(\iota^{n-1}_{\wedge})\colon \Omega^1_d\circ S^{n-1}_d\longrightarrow \Omega^2_d\circ S^{n-2}_d,
\end{equation}
and we denote the kernel inclusion by $\iota^n_{\wedge}\colon S^n_d\longhookrightarrow \Omega^1_d\circ S^{n-1}_d$.
We call $S^{\bullet}_d\colonequals \bigoplus_{n\ge 0} S^n_d$ the \emph{functor of symmetric forms}.

When no confusion arises, we denote $S^n_d (A)$ and $S^{\bullet}_d (A)$ also by $S^n_d$ and $S^{\bullet}_d$, respectively.
The latter will be termed \emph{bimodule of symmetric forms}.
\end{defi}
It is natural to consider the following composition, which maps $S^n_d$ to $T^n_d$.
\begin{equation}
\iota_{S^n_d}\colonequals T^{n-2}_d(\iota^2_{\wedge})\circ T^{n-3}_d(\iota^3_{\wedge})\circ \cdots \circ T^{1}_d(\iota^{n-1}_{\wedge})\circ \iota^n_{\wedge} \colon S^n_d\longrightarrow T^{n}_d.
\end{equation}
\begin{eg}[Classical functor of symmetric forms]
In the example of §\ref{ss:cman}, given a vector bundle $E\to M$, the functor of symmetric forms applied to the $\smooth{M}$-module $\Gamma(M,E)$ is $S^n_d(M,E)$.
\end{eg}
\begin{eg}[Universal functor of symmetric forms]\label{ex:unisymmetricforms}
In the universal case, $S^n_u=0$ for $n> 1$, so $S^{\bullet}_u=A \oplus \Omega^1_u$.
\end{eg}
Another important example is presented in more details in §\ref{sss:minfqsf}.

The following is an equivalent characterization of the functor of symmetric forms.
\begin{lemma}\label{lemma:pbqsf}
For all $n\ge 2$, the following is a pullback diagram
\begin{equation}\label{diag:pbqsf}
\begin{tikzcd}[column sep=40pt]
S^n_d\ar[r,hookrightarrow,"\iota^n_\wedge"]\ar[d]&\Omega^1_d\circ S^{n-1}_d\ar[d,"\Omega^1_d(\iota^{n-1}_\wedge)"]\\
S^2_d\circ S^{n-2}_d\ar[r,hookrightarrow,"\iota^2_{\wedge,S^{n-2}_d}"']&\Omega^1_d\circ\Omega^1_d\circ S^{n-2}_d
\end{tikzcd}
\end{equation}

Moreover, if $\Omega^1_d$ is flat in $\ModA$, then \eqref{diag:pbqsf} is an intersection in $\Omega^1_d\circ\Omega^1_d\circ S^{n-2}_d$ and in $T^{n}_d$.
Under this assumption, $\iota_{S^n_d}$ is a monomorphism.
\end{lemma}
\begin{proof}
Consider the following diagram
\begin{equation}
\begin{tikzcd}[column sep=50pt]
S^n_d\ar[r,hookrightarrow,"\iota^n_\wedge"]\ar[d,dashed]&\Omega^1_d\circ S^{n-1}_d\ar[d,"\Omega^1_d(\iota^{n-1}_\wedge)"]\ar[r,"\wedge_{S^{n-2}_d}\circ \Omega^1_d(\iota^{n-1}_\wedge)"]&\Omega^2_d\circ S^{n-2}_d\ar[d,equals]\\
S^2_d\circ S^{n-2}_d\ar[r,hookrightarrow,"\iota^2_{\wedge, S^{n-2}_d}"']&\Omega^1_d\circ\Omega^1_d\circ S^{n-2}_d\ar[r,twoheadrightarrow,"\wedge_{S^{n-2}_d}"']&\Omega^2_d\circ S^{n-2}_d
\end{tikzcd}
\end{equation}
The right square clearly commutes and the two horizontal maps on the left are kernel maps by definition of $S^n_d$ and of $S^2_d$ evaluated on $S^{n-2}_d$.
From the universal property of the kernel, we find the dashed map.
We then deduce the first statement via Lemma \ref{lemma:pb}.

Moreover, if $\Omega^1_d$ is flat in $\ModA$, then $\Omega^1_d(\iota^{n-1}_\wedge)$ is a mono (cf.\ Lemma \ref{lemma:derfun}).
Thus, by the same lemma, the square \eqref{diag:pbqsf} is an intersection of subfunctors.
The last statement follows from the fact that $\Omega^1_d\circ\Omega^1_d(\iota^{n-2}_\wedge)$ is a monomorphism, and thus the pullback does not change by composing $\Omega^1_d\circ\Omega^1_d(\iota^{n-2}_\wedge)$ in the bottom right corner of \eqref{diag:pbqsf}.
If we assume $\Omega^1_d$ to be flat in $\ModA$, then the maps $T^{n-k}_d(\iota^k_{\wedge})$ are injective for all $0\le k\le n$ (cf.\ Proposition \ref{prop:Texact}).
By composing them, we obtain the inclusion $\iota_{S^k_d}\colon S^k_d\to T^n_d$, which realizes $S^k_d$ as a submodule of $T^n_d$.
\end{proof}

Whereas the functors $\Omega^n_d$ are defined as tensor functors, $S^n_d$ is not.
However, this interpretation is still possible under certain circumstances, as shown by the following.
\begin{prop}\label{prop:Stensorcomparison}
There exists a canonical set of natural transformations $\tau^n_d\colon S^n_d A\otimes_A - \longrightarrow S^n_d$, that are compatible with $\iota^n_\wedge$ and $\iota_{S^n_d}$, i.e.
\begin{align}\label{eq:taucompatible}
\iota^n_{\wedge}\circ \tau^n_d=\Omega^1_d(\tau^{n-1}_d)\circ (\iota^n_{\wedge,A}\otimes_A \id),
&\hfill&
\iota_{S^n_d}\circ \tau^n_d=\iota_{S^{n}_d A}\otimes_A \id.
\end{align}
Moreover, the maps $\tau^0_d$ and $\tau^1_d$ are isomorphisms and $\tau^2_d$ is an epimorphism.

If we restrict the objects involved to the subcategory $\AFlat$, then $\tau^n_d$ is an isomorphism.
\end{prop}
\begin{proof}
Essentially by definition, $\tau^0_d$ and $\tau^1_d$ are isomorphisms without any further assumptions, and they trivially satisfy \eqref{eq:taucompatible}.

For $n>1$, assume such a $\tau^{n-1}_d$ exists.
We can apply $\Omega^1_d$, obtaining the following diagram
\begin{equation}\label{diag:Stensorcomparison}
\begin{tikzcd}[column sep=45pt]
S^n_d\otimes_A -\ar[r,"\iota^n_{\wedge,A}\otimes_A \id"]\ar[d,dashed,"\tau^n_d"']&\Omega^1_d\circ S^{n-1}_d\otimes_A -\ar[r,"\Omega^1_d(\iota^{n-1}_{\wedge,A})\otimes_A \id"]\ar[d,"\Omega^1_d(\tau^{n-1}_d)"]&T^{2}_d\circ S^{n-2}_d\otimes_A -\ar[r,twoheadrightarrow,"\wedge_{S^{n-2}_d A}\otimes_A \id"]\ar[d,"T^{2}_d(\tau^{n-2}_d)"]&\Omega^2_d \circ S^{n-2}_d\otimes_A -\ar[d,"\Omega^2_d(\tau^{n-2}_d)"']\\
S^n_d \ar[r,hookrightarrow,"\iota^n_{\wedge}"']&\Omega^1_d\circ S^{n-1}_d \ar[r,"\Omega^1_d(\iota^{n-1}_\wedge)"']&T^{2}_d\circ S^{n-2}_d \ar[r,twoheadrightarrow,"\wedge_{S^{n-1}_d }"']&\Omega^2_d \circ S^{n-2}_d 
\end{tikzcd}
\end{equation}
The central square commutes by inductive hypothesis and the fact that $\Omega^1_d$ is a tensor functor, and the right square commutes by the naturality of $\wedge$.
Thus, since the top row vanishes, by the universal property of the kernel, there exists a unique $\tau^n_d$ that makes the left square commute.
Furthermore,
\begin{equation}
\begin{split}
\iota_{S^n_d}\circ \tau^n_d
&=\Omega^1_d(\iota_{S^{n-1}_d})\circ \iota^n_d \circ \tau^n_d\\
&=\Omega^1_d(\iota_{S^{n-1}_d})\circ \Omega^1_d(\tau^{n-1}_d)\circ (\iota^n_{\wedge,A}\otimes_A \id)\\
&=\Omega^1_d(\iota_{S^{n-1}_d}\circ \tau^{n-1}_d)\circ (\iota^n_{\wedge,A}\otimes_A \id)\\
&=\Omega^1_d(\iota_{S^{n-1}_d}\otimes_A \id)\circ (\iota^n_{\wedge,A}\otimes_A \id)\\
&=\iota_{S^{n}_d}\otimes_A \id.
\end{split}
\end{equation}
If we consider \eqref{diag:Stensorcomparison} for $n=2$, the vertical arrows are isomorphisms, with the exception of $\tau^2_d$.
This implies that the kernels of the composition of the two rightmost horizontal maps in each row are isomorphic.
The top composition is surjective, so, by right exactness of the tensor product with a left $A$-module $E$, it follows that its kernel is the image of $\iota^2_{\wedge,A}\otimes_A \id_E$.
It follows that $\tau^2_d$ is surjective.

We now evaluate \eqref{diag:Stensorcomparison} at $E$ in $\AFlat$.
Since $-\otimes_A E$ is exact, it preserves limits, and in particular we get that $\iota^n_{\wedge,A}\otimes_A \id_E$ is the inclusion map of the kernel of $(\wedge_{S^n_d}\otimes_A \id_E)\circ (\Omega^1_d(\iota^{n-1}_{\wedge,A})\otimes_A \id_E)$.
In this context, the kernel can be seen as a functor $\ker\colon\AMod^{\bullet\to\bullet}\to \AMod$, where $\AMod^{\bullet \to \bullet}$ is the category of arrows in $\AMod$.
The composition of the two rightmost squares in \eqref{diag:Stensorcomparison} represents a morphism in $\AMod$ and by construction $\tau^n_d$ is the image of this map via the functor $\ker$.
Since $\tau^{n-1}_d$ and $\tau^{n-2}_d$ are isomorphisms by inductive hypothesis, and functors preserve isomorphisms, also $\Omega^1_d(\tau^{n-1}_d)$ and $\Omega^2_d(\tau^{n-2}_d)$ are isomorphisms.
Therefore, $\tau^n_d$ is an isomorphism, being the image of an isomorphism via $\ker$.
\end{proof}
\begin{rmk}\label{rmk:symm2surj}
The maps $\tau^0_d$ and $\tau^1_d$ are isomorphisms by definition, without further conditions.
For $n=2$, we have $\wedge_{S^{n-2}_d} \circ \Omega^1_d(\iota^{n-1}_{\wedge})=\wedge\colon T^2_d\longtwoheadrightarrow \Omega^2_d$, so $\tau^2_d$ is an epi by the surjectivity axiom of the exterior algebra.
It becomes an isomorphism if $\Tor^A_1(\Omega^2_d,E)=0$, as the sequence
\begin{equation}\label{es:wedge}
\begin{tikzcd}[column sep=40pt]
0 \arrow[r] &S_d^2 E \arrow[r,hookrightarrow,"\iota^2_{\wedge,E}"]& \Omega^1_d\otimes_A \Omega^1_d E \arrow[r,twoheadrightarrow,"\wedge_E"]& \Omega^2_d E \arrow[r]&0.
\end{tikzcd}
\end{equation}
can be obtained by applying $-\otimes_A E$ to the corresponding sequence with $E=A$.
\end{rmk}

The following lemma shows other equivalent descriptions of $S^n_d$.
\begin{lemma}\label{lemma:Schar}
If $\Omega^1_d$ and $\Omega^2_d$ are flat in $\ModA$, for all $n\ge 0$, the following subbimodules of $T^{n}_d$ coincide
\begin{enumerate}
\item\label{lemma:Schar:1} $S^n_d$;
\item\label{lemma:Schar:2} $\bigcap_{\substack{h\ge 2\\ 0\le k\le n-h}}\ker\left(T^k_d(\wedge_h)_{T^{n-k-h}_d}\right)$, where $\wedge_h\colon T^h_d\to\Omega^{h}_d$ is defined in §\ref{ss:extalg};
\item\label{lemma:Schar:3} $\bigcap_{k=0}^{n-2}\ker\left(T^k_d \wedge_{T^{n-k-2}_d} \right)$;
\item\label{lemma:Schar:4} $\left(S^h_d\circ T^{n-h}_d \right)\cap \left(T^{n-k}_d \circ S^k_d\right)$ for $0\le h,k\le n$ such that $h+k>n$.
\end{enumerate}
\end{lemma}
\begin{proof}
We will prove it by induction.
For $n=0$ and $n=1$ the definitions coincide as, other than in \eqref{lemma:Schar:1}, we get intersections of an empty collection of subfunctors of the tensor algebra functor, giving the tensor algebra functor itself.
For $n=2$ all conditions coincide.
Now let $n>2$ and assume the statement true for all $k<n$.

\eqref{lemma:Schar:1}$\subseteq$\eqref{lemma:Schar:2}: We prove that $S^n_d$ vanishes on all maps of the form $T^{k}_d(\wedge_h)_{T^{n-k-h}_d}$ for $h\ge 2$ and $0\le k\le n-h$.
For $k>0$, it vanishes by the inductive hypothesis, because $S^n_d\subseteq \Omega^1_d\circ S^{n-1}_d$, and we have $T^{k}_d(\wedge_h)_{T^{n-k-h}_d}=\Omega^1_d\circ T^{k-1}_d(\wedge_h)_{T^{n-k-h}_d}$ applied on
\begin{equation}
\Omega^1_d\circ S^{n-1}_d
=\Omega^1_d\circ\bigcap_{\substack{h\ge 2\\ 0\le k\le n-h-1}}\ker\left(T^{k-1}_d (\wedge_h)_{T^{n-k-h}_d} \right).
\end{equation}
For $k=0$, by associativity of $\wedge$ we can write $(\wedge_h)_{T^{n-h}_d}=\wedge^{2,n-2}\circ \Omega^2_d(\wedge_{h-2})_{T^{n-h}_d}\circ (\wedge_{T^{n-2}_d})$.
That is, we apply $\wedge$ on the first two components, followed by the remaining $h-2$ and by the wedge of the resulting $2$-form and $(h-2)$-form.
By definition, $S^n_d$ vanishes on $\wedge\otimes_A \id_{T^{n-2}_d}$.

\eqref{lemma:Schar:2}$\subseteq$\eqref{lemma:Schar:3}: Straightforward.

\eqref{lemma:Schar:3}$=$\eqref{lemma:Schar:4}: The only non-trivial cases occur for $k,h\ge 2$.
Using the inductive hypothesis and the fact that $T^{n-k}_d$ preserves limits (cf.\ Proposition \ref{prop:Texact}), we get
\begin{equation}
\begin{split}
\bigcap_{j=0}^{n-2}\ker\left(T^{j}_d\wedge_{T^{n-j-2}_d} \right)
&=\bigcap_{j=0}^{h-2}\ker\left(T^{j}_d\wedge_{T^{n-j-2}_d} \right) \cap \bigcap_{j=k}^{n-2}\ker\left(T^{j}_d\wedge_{T^{n-j-2}_d}\right)\\
&=\bigcap_{j=0}^{h-2}\ker\left(T^{j}_d\wedge_{T^{h-j-2}_d}\right)\circ T^{n-h}_d
\cap T^{n-k}_d \circ \bigcap_{j=0}^{n-k-2}\ker\left(T^{j}_d \wedge_{T^{n-k-2-j}_d} \right)\\
&=\left(S^h_d\circ T^{n-h}_d \right)\cap \left(T^{n-k}_d \circ S^k_d\right).
\end{split}
\end{equation}
Notice that we have also proved that for $h,k\ge 2$, the spaces $\left(S^h_d\circ T^{n-h}_d \right)\cap \left(T^{n-k}_d \circ S^k_d\right)$ coincide.

\eqref{lemma:Schar:4}$\subseteq$\eqref{lemma:Schar:1}: Again there is nothing to prove unless $h,k\ge 2$, in which case we have shown that for all such $h,k$ the intersections of the form \eqref{lemma:Schar:4} coincide, so it is enough to prove
\begin{equation}
S^n_d\subseteq\left(S^2_d\circ T^{n-2}_d \right)\cap \left(\Omega^{1}_d \circ S^{n-1}_d\right).
\end{equation}
Consider the following diagram, where the dashed map is obtained by the universal property of the kernel
\begin{equation}
\begin{tikzcd}[column sep=40pt]
S^n_d\ar[r,hookrightarrow,"\iota^n_{\wedge}"]\ar[d,dashed]&\Omega^1_d\circ S^{n-1}_d\ar[r,"\wedge_{S^{n-2}_d}"]\ar[d,hookrightarrow,"\Omega^1_d(\iota_{S^{k-1}_d})"]&\Omega^2_d\circ S^{n-2}_d\ar[d,hookrightarrow,"\Omega^2_d(\iota_{S^{k-2}_d})"]\\
S^2_d\circ T^{n-1}_d\ar[r,hookrightarrow,"(\iota^2_{\wedge})_{T^{n-2}_d}"']& T^{n}_d\ar[r,twoheadrightarrow,"\wedge_{T^{n-2}_d}"']&\Omega^2_d\circ T^{n-2}_d
\end{tikzcd}
\end{equation}
Using the fact that the horizontal maps form left exact sequences and $\Omega^2_d(\iota_{S^{k-2}_d})$ is a natural monomorphism, we can use Lemma \ref{lemma:pb} to prove that the left square is a pullback square.
Moreover, injectivity of $\Omega^{1}_d(\iota_{S^{k-1}_d})$ forces the left vertical natural transformation to be a monomorphism, and we realize $S^n_d$ as the claimed intersection, completing the proof.
\end{proof}
\begin{rmk}
The points \eqref{lemma:Schar:2} and \eqref{lemma:Schar:3} of Lemma \ref{lemma:Schar} give rise to alternative definitions of $S^n_d$ (cf.\ \cite[p.~11]{majid2023quantum}).
We chose instead the formulation of Definition \ref{def:qsf}, because it does not require extra assumptions to obtain the map $\iota^n_{\wedge}\colon S^n_d\to \Omega^1_d\circ S^{n-1}_d$.
Such a map will in fact be crucial for the treatment of Spencer $\delta$-cohomology (cf.\ §\ref{ss:Spencer}).

In order to obtain the same map for definitions \eqref{lemma:Schar:2} and \eqref{lemma:Schar:3}, one has to assume that $\Omega^1_d$ preserves the limit defining $S^{n-1}_d$, which happens, in particular, if $\Omega^1_d$ is in $\FlatA$.
\end{rmk}

The following result explores the behavior of the functor $S^n_d$ with limits and colimits.
\begin{prop}\label{prop:exactnessSn}
The functor $S^n_d$ preserves flat colimits of flat modules.

If $\Omega^1_d$ and $\Omega^2_d$ are flat in $\ModA$, the functor $S^n_d$ is left exact.
\end{prop}
\begin{proof}
Consider a diagram $F_\bullet \colon I\to \AFlat$ and let $F=\colim_{I}F_{\bullet}$ be flat.
By Proposition \ref{prop:Stensorcomparison}, on $\AFlat$, the functor $S^n_d$ is naturally isomorphic to $S^n_d\otimes_A -$, which is right exact.
Therefore, 
\begin{equation}
\colim_I S^n_d (F_\bullet)
=\colim_I S^n_d\otimes_A F_\bullet
=S^n_d\otimes_A \colim_I F_\bullet
=S^n_d \otimes_A F
=S^n_d(F).
\end{equation}

If $\Omega^1_d$ and $\Omega^2_d$ are flat in $\ModA$, then we prove that $S^n_d$ is left exact by induction on $n\ge 0$.
Exactness of $S^0_d=\id_{\AMod}$ and $S^1_d=\Omega^1_d$ is straightforward.
For $n>2$, we assume by inductive hypothesis that $S^{n-1}_d$ and $S^{n-2}_d$ are left exact.
Since $\Omega^2_d$ is also exact, $S^n_d$ is built as a limit of left exact functors, namely the kernel of $\Omega^1_d\circ S^{n-1}_d\to \Omega^2_d\circ S^{n-2}_d$.
Since limits commute with limits, $S^n_d$ is also left exact.
\end{proof}

\subsubsection{Minimal functor of symmetric forms}
\label{sss:minfqsf}
For every first order differential calculus $\Omega^1_d$ on $A$ there is a canonical choice of functor of symmetric forms, specifically the one corresponding to the maximal exterior algebra $\Omega^\bullet_{d,\text{max}}$ (cf.\ Remark \ref{rmk:maximalextalg}).
\begin{defi}
We call this functor the \emph{minimal functor of symmetric forms}, denoted by $S^\bullet_{d,\text{min}}$.
\end{defi}
The adjective minimal is justified by the following universal property.
\begin{prop}[Minimality of the minimal functor of symmetric forms]
	Let $A$ be a $\bk$-algebra, let $\Omega^\bullet_d$ be an exterior algebra over $A$ and let $\Omega^\bullet_{d,\text{max}}$ be the maximal exterior algebra generated by $\Omega^1_d$.
Let $S_d$ be the functor of symmetric forms associated to $\Omega^\bullet_d$, and let $S_{d,\text{min}}$ be the minimal functor of symmetric forms associated to $\Omega^1_d$.
There is a unique natural transformation $\nu\colon S^\bullet_{d,\text{min}}\to S^\bullet_d$ of graded functors of type $\AMod\to \AMod$ which commutes with the inclusion maps, i.e.\ the following diagrams commute for all $n$
\begin{equation}\label{eq:minqsffactornat}
\begin{tikzcd}
S^n_{d,\text{min}}\ar[r,hookrightarrow,"\iota^n_{\wedge_\text{max}}"]\ar[d,"\nu^n"']&\Omega^1_d\circ S^{n-1}_{d,\textrm{min}}\ar[d,"\Omega^1_d(\nu^{n-1})"]&&
S^n_{d,\text{min}}\ar[r,"\iota_{S^n_{d,\text{min}}}"]\ar[d,"\nu^n"']&T^{n}_{d}\ar[d,equals]\\
S^n_{d}\ar[r,hookrightarrow,"\iota^n_{\wedge}"]&\Omega^1_d\circ S^{n-1}_{d}&&
S^n_{d}\ar[r,"\iota_{S^n_{d}}"]&T^{n}_{d}
\end{tikzcd}
\end{equation}

If $\Omega^1_d$ is flat in $\ModA$, then $\nu$ is also a monomorphism.
\end{prop}
\begin{proof}
For $n=0,1$ the statement follows from the fact that $\nu^0=\id_{A}$ and $\nu^1=\id_{\Omega^1_d}$ are the only such maps.
For $n\ge 2$ we proceed by induction on $n$.
In particular, for all $m<n$ we have $\nu^{m}$ satisfying \eqref{eq:minqsffactornat}.
The two central squares in the following diagram commute by inductive hypothesis, functoriality of $\Omega^1_d$, and naturality of $\wedge_{\text{max}}$.
The bottom commutativity follows from the universal property of the maximal exterior algebra.
\begin{equation}\label{diag:canopener}
\begin{tikzcd}
S^n_{d,\text{min}}\ar[r,hookrightarrow,"\iota^n_{\wedge_\text{max}}"]\ar[d,dashed,"\nu^n"']&[20pt]\Omega^1_d\circ S^{n-1}_{d,\textrm{min}}\ar[d,"\Omega^1_d(\nu^{n-1})"]\ar[r,"\Omega^1_d(\iota^{n-1}_{\wedge_{\text{max}}})"]&[30pt]\Omega^1_d\circ \Omega^1_d\circ S^{n-2}_{d,\textrm{min}}\ar[d,"\Omega^1_d(\Omega^1_d(\nu^{n-2}))"']\ar[r,twoheadrightarrow,"(\wedge_{\text{max}})_{S^{n-2}_{d,\text{min}}}"]&[40pt]\Omega^2_{d,\text{max}}\circ S^{n-2}_{d,\textrm{min}}\ar[d,"\Omega^2_{d,\text{max}}(\nu^{n-2})"]\\
S^n_{d}\ar[r,hookrightarrow,"\iota^n_{\wedge}"']&\Omega^1_d\circ S^{n-1}_{d}\ar[r,"\Omega^1_d(\iota^{n-1}_{\wedge})"']&\Omega^1_d\circ \Omega^1_d\circ S^{n-2}_{d}\ar[r,twoheadrightarrow,"(\wedge_{\text{max}})_{S^{n-2}_{d}}"]\ar[rr,bend right=10pt,twoheadrightarrow,"\wedge_{S^{n-2}_d}"']&\Omega^2_{d,\text{max}}\circ S^{n-2}_d\ar[r,twoheadrightarrow]&\Omega^2_d\circ S^{n-2}_d
\end{tikzcd}
\end{equation}
The existence and uniqueness of $\nu^n$ then follows from the universal property of the kernel.
This automatically gives the left diagram in \eqref{eq:minqsffactornat}.
By composing diagrams of this form we also obtain the right one.

If $\Omega^1_d$ is flat in $\ModA$, it follows that $\Omega^1_d$ maps monomorphisms to monomorphisms.
If by inductive hypothesis $\nu^{n-1}$ is a mono, then also $\nu^n$ is a mono by the commutativity of \eqref{eq:minqsffactornat}.
\end{proof}

The second grade of the minimal functor of symmetric forms can be described explicitly.
We know by construction (cf.\ Remark \ref{rmk:maximalextalg}) that $S^2_{d,\text{min}}$ is the $A$-subbimodule of $T^2_d$ generated by elements of the form $\sum_i da_i\otimes_A db_i$ for $\sum_i a_i\otimes b_i\in N_d$.
By the Leibniz property and the properties defining $N_d$, we know that $d\otimes d\colon N_d\to S^2_{d,\text{min}}$ is $A$-bilinear.
Thus we obtain that
\begin{equation}\label{eq:S2mindefNd}
S^2_{d,\text{min}}=\left\{\sum_i da_i\otimes_A db_i\middle| \sum_i a_i\otimes b_i\in N_d\right\}.
\end{equation}
Furthermore, as left $A$-modules, right $A$-modules, or $A$-bimodules, $d\otimes d$ maps generators of $N_d$ to generators of $S^2_{d,\text{min}}$.
For $E$ in $\AMod$, by Proposition \ref{prop:Stensorcomparison}, we know that $\tau^2_{d,E}\colon S^2_{d,\text{min}}\otimes_A E\to S^2_{d,\text{min}}(E)$ is an epimorphism.
It follows that
\begin{equation}\label{eq:S2mindd}
S^2_{d,\text{min}}(E)=\left\{\sum_{i,j} da_{i,j}\otimes_A db_{i,j}\otimes_A e_j\in T^2_d(E)\middle| \forall j, \sum_i a_{i,j}\otimes b_{i,j}\in N_d, e_j\in E\right\}.
\end{equation}
\subsection{Spencer $\delta$-cohomology}
\label{ss:Spencer}
One of the fundamental tools in the classical formal theory of overdetermined PDE is the Spencer $\delta$-complex \cite{Spencer}.
This is a natural complex which exists on any manifold, and outside the context of a non-trivial differential equation, it is exact.
In our setting, we can build an analogous complex.
For all $k,h\ge 0$, consider the functor $\Omega^k_d\circ S^h_d$, and define $\delta^{h,k}_d$ as the following composition
\begin{equation}\label{eq:defdeltahk}
\begin{tikzcd}
\Omega^k_d\circ S^{h}_d\ar[r,"\Omega^k_d(\iota^h_{\wedge})"]\ar[rr,bend right=10pt,"\delta^{h,k}_d"']&[40pt]\Omega^{k}_d\circ\Omega^1_d\circ S^{h-1}_d\ar[r,twoheadrightarrow,"(-1)^k\wedge^{k,1}_{S^{h-1}_d}"]&[40pt]\Omega^{k+1}_d\circ S^{h-1}_d
\end{tikzcd}
\end{equation}
Notice that for all $h,k$, we have the following commuting diagram
\begin{equation}\label{diag:keel}
\begin{tikzcd}[row sep=30pt]
\Omega^k_d\circ S^h_d\ar[r,"\Omega^k_d(\iota^h_{\wedge})"]\ar[dr,"\delta^{h,k}_d"']&[50pt]\Omega^k_d\circ\Omega^1_d\circ S^{h-1}_d\ar[r,"\Omega^k_d\circ\Omega^1_d(\iota^h_{\wedge})"]\ar[d,twoheadrightarrow,"(-1)^{k}\wedge^{k,1}_{S^{h-1}_d}"]&[25pt]\Omega^k_d\circ\Omega^1_d\circ\Omega^1_d\circ S^{h-2}_d\ar[d,twoheadrightarrow,"(-1)^{k}\wedge^{k,1}_{\Omega^1_d\circ S^{h-2}_d}"']\ar[r,twoheadrightarrow,"\Omega^k_d \wedge_{S^{h-2}_d}"]&[15pt]\Omega^k_d\circ\Omega^2_d\circ S^{h-2}_d\ar[ddl,twoheadrightarrow,"-\wedge^{k,2}_{S^{h-2}_d}"]\\
&\Omega^{k+1}_d\circ S^{h-1}_d\ar[dr,"\delta^{h-1,k+1}_d"']\ar[r,"\Omega^{k+1}_d(\iota^{h-1}_{\wedge})"]&\Omega^{k+1}_d\circ\Omega^1_d\circ S^{h-2}_d\ar[d,twoheadrightarrow,"(-1)^{k+1}\wedge^{k+1,1}_{S^{h-2}_d}"']\\
&&\Omega^{k+2}_d\circ S^{h-2}_d&
\end{tikzcd}
\end{equation}
Here we consider $S^{-1}_d = S^{-2}_d = 0$, the constant functor with value $0$.
The two triangular diagrams on the left commute by definition of $\delta_d$, the central square diagram commutes by naturality of $\wedge^{k,1}$ with respect to $\iota^h_{\wedge}$, and the remain diagram on the right commutes by the associativity of $\wedge$.
Moreover, the top row composition vanishes, as it is obtained by applying $\Omega^k_d$ to the left exact sequence defining $S^h_d$.

We thus obtain a complex in the category of functors of type $\AMod\to\AMod$.
\begin{equation}\label{cpx:Spencer}
\begin{tikzcd}
0\ar[r]&S^n_d\ar[r,"\delta^{n,0}_d"]&\Omega^1_d\circ S^{n-1}_d\ar[r,"\delta^{n-1,1}_d"]&\Omega^2_d\circ S^{n-2}_d\ar[r,"\delta^{n-2,2}_d"]&\Omega^3_d\circ S^{n-3}_d\ar[r,"\delta^{n-3,3}_d"]&\cdots
\end{tikzcd}
\end{equation}
\begin{defi}[Spencer $\delta$-cohomology]
We call the complex \eqref{cpx:Spencer} the \emph{Spencer $\delta$-complex}, its cohomology the \emph{Spencer $\delta$-cohomology}, and we denote the cohomology at $\Omega^k_d\circ S^h_d$ by $H^{h,k}_{\delta_d}$.
\end{defi}
\begin{rmk}
Since the objects appearing in the Spencer complex are functors and the differentials are natural transformations, the cohomology is also functorial.
\end{rmk}
\begin{prop}\label{prop:lowspencercohom}
Given a $\bk$-algebra $A$ endowed with an exterior algebra $\Omega^{\bullet}_d$, we have:
\begin{enumerate}
\item $H^{n,0}_{\delta_d}=0$ for all $n\ge 1$;
\item $H^{n,1}_{\delta_d}=0$ for all $n\ge 0$;
\item $H^{n,2}_{\delta_d}=0$ for $n\le 0$.
\end{enumerate}
\end{prop}
\begin{proof}\
\begin{enumerate}
\item For $n=1$, the Spencer complex is
\begin{equation}
\begin{tikzcd}[column sep=30pt]
0\ar[r]&S^1_d\arrow{r}{\delta^{1,0}_d}[swap]{\sim}&\Omega^1_d\circ S^0_d\ar[r]&0
\end{tikzcd}
\end{equation}
which is exact.

For $n\ge 2$, consider the left exact sequence defining $S^{n}_d$, which up to sign can be written as
\begin{equation}
\begin{tikzcd}[column sep=30pt]
0\ar[r]&S^n_d\ar[r,hookrightarrow,"\delta^{n,0}_d"]&\Omega^1_d\circ S^{n-1}_d\ar[r,"\delta^{n-1,1}_d"]&\Omega^2_d\circ S^{n-2}_d
\end{tikzcd}
\end{equation}
This is the beginning of the Spencer sequence, implying the vanishing of $H^{n,0}_{\delta_d}$ and $H^{n-1,1}_{\delta_d}$.
\item Apply the same reasoning but for $S^{n+1}_d$.\qedhere
\item In position $(n,0)$, the Spencer complex vanishes for $n<0$.
For $n=2$, the Spencer complex becomes \eqref{es:wedge}, which is short exact.
\end{enumerate}
\end{proof}

With the following result, we observe that if we restrict our functors to the subcategory $\AFlat$, then the vanishing of the Spencer $\delta$-cohomology becomes a condition on the exterior algebra $\Omega^\bullet_d$.
\begin{prop}\label{prop:Spencertensorrep}
Let $\mathcal{S}^\bullet$ be the Spencer complex \eqref{cpx:Spencer}.
For all $E$ in $\AFlat$, we have $\mathcal{S}^\bullet(E)\cong\mathcal{S}^\bullet(A)\otimes_A E$.
Moreover,
\begin{equation}\label{eq:scohomtensor}
H^{\bullet,\bullet}_{\delta_d}(E)\cong H^{\bullet,\bullet}_{\delta_d}(A)\otimes_A E.
\end{equation}

Thus, the Spencer complex seen in the category of functors $\AFlat\to \AMod$ is naturally isomorphic to the functor $\mathcal{S}^\bullet(A)\otimes_A -$.
Here, the Spencer $\delta$-cohomology vanishes if it vanishes on $A$.
\end{prop}
\begin{proof}
Under these assumptions, the functor of symmetric forms can be seen as a tensor functor by Proposition \ref{prop:Stensorcomparison}.
The functors $\Omega^k_d$ are, by definition, tensor functors, and the maps used to define the Spencer complex are compatible with the tensor representations of the functors involved.
Thus we have $\mathcal{S}^\bullet(E)\cong\mathcal{S}^\bullet(A)\otimes_A E$.
The exactness of $-\otimes_A E$ yields \eqref{eq:scohomtensor}.
\end{proof}

\section{The $2$-jet functor}

\subsection{Construction of the $2$-jet module}

We build the (holonomic) $2$-jet module with the aim that the following sequence is exact
\begin{equation}\label{es:2jet}
\begin{tikzcd}
0 \arrow[r] &S_d^2 E \arrow[r,hookrightarrow,"\iota^2_{d,E}"]& J^2_d E \arrow[r,twoheadrightarrow,"\pi^{2,1}_{d,E}"]& J^1_d E \arrow[r]&0.
\end{tikzcd}
\end{equation}
Moreover, it is reasonable to assume (cf.\ the classical case in e.g.\ \cite{Spencer}) that the short exact sequence \eqref{es:2jet} maps naturally to the nonholonomic $2$-jet exact sequence as follows.
\begin{equation}\label{es:2in(2)}
\begin{tikzcd}[column sep=50pt]
0 \arrow[r] &S^2_d E \arrow[d,"\Omega^1_{d}(\iota^1_{d,E})\circ\iota^2_{\wedge,E}"']\arrow[r,hookrightarrow,"\iota^2_{d,E}"]& J^2_d E\ar[d,"l^2_{d,E}"] \arrow[r,twoheadrightarrow,"\pi^{2,1}_{d,E}"]& J^1_d E\arrow[d,equals] \arrow[r]&0\\
0 \arrow[r]& \Omega^1_{d}(J^1_d E) \arrow[r,hookrightarrow,"\iota^1_{d,J^1_d E}"]& J^{(2)}_d E \arrow[r,twoheadrightarrow,"\pi_{d,J^1_d E}^{1,0}"]& J^1_d E\arrow[r]&0
\end{tikzcd}
\end{equation}
In fact, we will show this mapping is an inclusion, and it factors through the semiholonomic $2$-jet exact sequence (cf.\ Proposition \ref{prop:semiholcpxfact}).
\begin{rmk}\label{rmk:TorOE}
If we assume $\Tor^A_1(\Omega^1_d,E)=0$, the map $\Omega^1_{d}(\iota^1_{d,E})\circ\iota^2_{\wedge,E}$ is a mono, as the functor $\Omega^1_d$ preserves the exactness of the $1$-jet exact sequence \eqref{es:jetd}.
Given this assumption, the existence of any map $l^2_{d,E}\colon J^2_d E\to J^{(2)}_d E$ commuting in \eqref{es:2in(2)} automatically implies that $l^2_{d,E}$ is an inclusion by the snake lemma.
This condition holds, in particular, if $\Omega^1_d$ is flat in $\ModA$ or $E$ flat in $\AMod$.
\end{rmk}
\subsubsection{An explicit presentation of the $2$-jet module}
In this section, we assume that $\Tor^A_1(\Omega^1_d,E)=0$.
In the classical case, the inclusion $l^2_{d,E}\colon J^2_d E\hookrightarrow J^{(2)}_d E$ also respects the jet prolongation \cite{GoldschmidtII}, meaning that
\begin{equation}\label{eq:prol2}
l^2_{d,E}\circ j^2_{d,E}
=j^{(2)}_{d,E}
=j^1_{d,J^1_d E}\circ j^1_{d,E}.
\end{equation}
We thus get that $J^2_d E$ must contain the left $A$-submodule of $J^{(2)}_d E$ generated by $\iota^2_{d,E}(S^2_d E)$ and $j^{(2)}_{d,E}(E)$, for brevity denoted $S^2_d E+Aj^2_{d}(E)\subseteq J^{(2)}_d E$.
\begin{rmk}\label{rmk:incl2Jdiff}
	Equation \eqref{eq:prol2} also shows that the inclusion $l^2_{d,E}$ is the left $A$-linear lift of the second order differential operator (cf.\ §\ref{s:differentialoperators}) $j^{(2)}_{d,E}$, i.e.\ $l^2_d=\widetilde{j^{(2)}_{d}}$.
\end{rmk}

\begin{prop}\label{prop:torsion2jetses}
If $\Tor^A_1(\Omega^1_d,E)=0$, the left $A$-module $S^2_d E+Aj^2_{d}(E)\subseteq J^{(2)}_dE$ satisfies the $2$-jet short exact sequence \eqref{es:2jet} and fits in \eqref{es:2in(2)} in place of $J^2_d E$.
\end{prop}
\begin{proof}
Each map is well-defined.
The map $\iota^2_{d,E}$ is a monomorphism, as $\iota^1_{d,J^1_d E}\circ \Omega^1_d(\iota^1_{d,E})\circ \iota^2_{\wedge,E}$ is a monomorphism (cf.\ Remark \ref{rmk:TorOE}).
The map $\pi^{1,0}_{d, J^1_d E}$ restricted to $S^2_d E+Aj^2_{d}(E)$ vanishes on the first summand and maps the second one to $A\pi^{1,0}_{d, J^1_d E}(j^{(2)}_{d, E}(E))=Aj^1_d(E)=J^1_d E$ (cf.\ Remark \ref{rmk:jpg}), implying the surjectivity of $\pi^{2,1}_d$.

We are left to prove that $\ker(\pi^{1,0}_{d,J^1_d E}|_{S^2_dE + Aj^2_d (E)})=S^2_d E$.
We know that the kernel contains $S^2_d E$, so it remains to show that the only elements of $A j_d^2(E)$ that vanish when applying $\pi^{1,0}_{d,J^1_d E}$ are contained in $S^2_{d}E$.
An element of $J^{(2)}_d E$ is generated by elements of the form $[a\otimes b]\otimes_A [c\otimes e]$ where $a,b,c\in A$ and $e\in E$.
An element in $Aj^2_d (E)$ is thus of the form $\sum_i [a_i\otimes 1]\otimes_A [1\otimes e_i]$ and $\pi^{1,0}_{d,J^1_d E}$ maps it to $\sum_i [a_i\otimes e_i]\in J^1_d E$, which is zero if and only if $\sum_i a_i\otimes e_i\in \ker(p_{d,E})=N_{d}(E)$.
By the right exactness of $-\otimes_A E$, there exists $\sum_j x_j\otimes y_j\otimes_A f_j\in N_d\otimes_A E$ such that $\sum_j x_j\otimes y_j f_j=\sum_i a_i\otimes e_i$.
The map $a\otimes e\mapsto a j^{(2)}_{d,E} e$ is well defined, since $j^{(2)}_{d,E}$ is $\bk$-linear.
It follows that $\sum_i [a_i\otimes 1]\otimes_A [1\otimes e_i]
=\sum_j [x_j\otimes 1]\otimes_A [1\otimes y_j f_j]$, and the latter corresponds to $\sum_j [x_j\otimes 1]\otimes_A [1\otimes y_j]\otimes_A f_j$ via the isomorphism $J^{(2)}_d E\cong J^{(2)}_d A \otimes_A E$.
Now we have
\begin{equation}
\sum_j [x_j\otimes 1]\otimes_A [1\otimes y_j]
=\sum_j [x_j\otimes 1-1\otimes x_j]\otimes_A [1\otimes y_j-y_j\otimes 1]
=\sum_j [-d_u x_j]\otimes_A [d_u y_j]
=-\sum_j d x_j\otimes_A d y_j.
\end{equation}
This element is in $S^2_d$ because applying $\wedge$ gives
\begin{equation}
\sum_j d x_j\wedge d y_j
=\sum_j d( x_j\wedge d y_j)
=d\left(\sum_j x_j dy_j\right)
=0,
\end{equation}
as $\sum_j x_j\otimes y_j\in N_{d}$.
It follows that $\sum_j [x_j\otimes 1]\otimes_A [1\otimes y_j]\otimes_A f_j$ belongs to $S^2_d\otimes_A E$ in $J^{(2)}_d E$, and thus it is in $S^2_d(E)$ as claimed.
\end{proof}

This proposition gives us a description of the $2$-jet module as
\begin{equation}\label{eq:exp2jet}
J^2_d E= S^2_d E+Aj^2_{d}(E)\subseteq J^{(2)}_d E.
\end{equation}
The inclusion of $S^2_d E$ in the $2$-jet module is induced by $\iota^1_{d,J^1_d E}$, and $\pi^{2,0}_{d,E}=\pi^{1,0}_{d,J^1_d E}\circ l^2_{d,E}$.
Finally, we see $j^2_{d,E}$ as the factorization of $j^{(2)}_{d, E}$ though the $2$-jet submodule, and it is compatible with the above notation $S^2_d E+Aj^2_{d}(E)\subseteq J^{(2)}_d E$.

If, in particular, $E=A$, then $S^2_d +Aj^2_{d}(A)$ is an $A$-subbimodule of $J^{(2)}_d A$ and all the maps in \eqref{es:2in(2)} are $A$-bimodule maps.

\subsection{Implicit construction of the $2$-jet module}\label{ss:implicit2jet}
We shall now construct the $2$-jet $J^2_d E$ implicitly as the kernel of an $A$-linear map $\widetilde{\DH}_{d,E}$, as that map will be used to construct the higher jet functors.
This map is then a generalization of the map $\rho$ from \cite[Proposition 3, p.~432]{Goldschmidt}.

Assuming $\Tor^A_1(\Omega^1_d,E)=0$, we compute the cokernel sequence obtained from the map of short exact sequences \eqref{es:2in(2)}.
From the nine lemma, the resulting sequence is exact.
\begin{equation}\label{es:2jetDt}
\begin{tikzcd}[column sep=40pt]
0 \arrow[r] &[-20pt]S^2_{d}(E)\arrow[d,"\Omega^1_d(\iota^1_{d,E})\circ\iota^2_{\wedge,E}"',hookrightarrow]\arrow[r,hookrightarrow,"\iota^2_{d,E}"]& J^2_d E\ar[d,hookrightarrow,"l^2_{d,E}"] \arrow[r,twoheadrightarrow,"\pi^{2,1}_{d,E}"]& J^1_d E\arrow[d,equals] \arrow[r]&[-20pt]0\\
0 \arrow[r]& \Omega^1_{d}( J^1_d E) \ar[d,twoheadrightarrow,"W_{d,E}"']\arrow[r,hookrightarrow,"\iota^1_{d,J^1_d E}"]& J^{(2)}_d E \ar[d,twoheadrightarrow,"\widetilde{\DH}_{d,E}"]\arrow[r,twoheadrightarrow,"\pi^{1,0}_{d,J^1_d E}"]& J^1_d E\ar[d]\arrow[r]&0\\
0\ar[r]&\coker\left(\Omega^1_d(\iota^1_{d,E})\circ\iota^2_{\wedge,E}\right)\ar[r,"\sim"]& \coker(l^2_{d,E})\ar[r]& 0 \ar[r]&0
\end{tikzcd}
\end{equation}
Thus, we can identify $\coker(l^2_{d,E})$ with $\coker(\Omega^1_d(\iota^1_{d,E})\circ\iota^2_{\wedge,E})$.

\subsubsection{The codomain of $\widetilde{\DH}_d$}
In order to describe $\coker(\Omega^1_d(\iota^1_{d,E})\circ\iota^2_{\wedge,E})$ explicitly, we consider the following cokernel of exact sequences which is a short exact sequence, again by the nine lemma.
\begin{equation}\label{es:W}
\begin{tikzcd}[column sep=40pt]
0 \arrow[r] &S^2_d(E)\arrow[d,"\iota^2_{\wedge,E}"',hookrightarrow]\arrow[r,equals]& S^2_d(E)\ar[d,hookrightarrow,"\Omega^1_d(\iota^1_{d,E})\circ\iota^2_{\wedge,E}"] \arrow[r]& 0\arrow[d] \arrow[r]&0\\
0 \arrow[r]& \Omega^1_d\otimes_A \Omega^1_{d}(E)\ar[d,twoheadrightarrow,"\wedge"']\arrow[r,hookrightarrow,"\Omega^1_d(\iota^1_{d,E})"']&\Omega^1_d\otimes_A J^{1}_d E \ar[d,twoheadrightarrow,"W_{d,E}"]\arrow[r,twoheadrightarrow,"\Omega^1_d(\pi^{1,0}_{d,E})"']& \Omega^1_d(E)\ar[d,equals]\arrow[r]&0\\
0\arrow[r]&\Omega^2_d(E) \arrow[r,hookrightarrow]&\coker\left(\Omega^1_d(\iota^1_{d,E})\circ\iota^2_{\wedge,E}\right)\arrow[r,twoheadrightarrow]& \Omega^1_d(E) \arrow[r]&0
\end{tikzcd}
\end{equation}
We can thus see $\coker(\Omega^1_d(\iota^1_{d,E})\circ\iota^2_{\wedge,E})$ as an extension of $\Omega^1_d(E)$ via $\Omega^2_d(E)$.
\begin{lemma}\label{lemma:spplsplit}
The bottom sequence in \eqref{es:W} is split exact in $\Mod$, and in $\Mod_B$ if $E$ is an $(A,B)$-bimodule.
More specifically, if $E=A$, we have $\coker(\Omega^1_d(\iota^1_{d,A})\circ\iota^2_{\wedge,A})=\Omega^1_d\oplusA\Omega^2_d$, and in general
\begin{equation}\label{eq:cokersp}
\coker(\Omega^1_d(\iota^1_{d,E})\circ\iota^2_{\wedge,E})
=\left(\Omega^1_d\oplusA\Omega^2_d\right)\otimes_A E
=\Omega^1_d(E)\oplus\Omega^2_d(E).
\end{equation}

Through this decomposition, we can see $W_{d,E}=(W^{I}_{d,E},W^{II}_{d,E})\colon \Omega^1\otimes_A J^1_d E\longrightarrow \Omega^1_d(E)\oplus\Omega^2_d(E)$, where
\begin{equation}\label{eq:WI-II}
\begin{tikzcd}
W^{I}_{d,E}\colon \Omega^1_d\otimes_A J^{1}_d E\ar[r,two heads]&\Omega^1_d(E)&W^{II}_{d,E}\colon \Omega^1_d\otimes_A J^{1}_d E\ar[r,two heads]&\Omega^2_d(E)\\[\vsfd]
\hphantom{W^{I}_{d,E}\colon } \alpha\otimes_A [x\otimes e]\ar[r,|->]&\alpha \otimes_A xe&\hphantom{W^{II}_{d,E}\colon} \alpha\otimes_A [x\otimes e]\ar[r,|->]&d\alpha\otimes_A xe-\alpha\wedge (dx)\otimes_A e.
\end{tikzcd}
\end{equation}

Via this description, the left $A$-action on $\Omega^1_d\oplus_A\Omega^2_d$ is given by
\begin{equation}\label{eq:laacSP}
f\smb (\alpha+\omega)=f\alpha+df\wedge\alpha+f\omega,
\end{equation}
for all $\alpha\in\Omega^1_d$, $\omega\in\Omega^2_d$, and $f\in A$.
On the right, $A$ acts component-wise.
\end{lemma}
\begin{proof}
Without loss of generality, we can prove the case where $E$ is an $(A,B)$-bimodule.
We start by proving that $W^{II}_{d,E}$ is well-defined.
Consider the map
\begin{equation}\label{eq:WIIexp}
d\otimes_A \pi^{1,0}_{d,E}+\id_{\Omega^1_d(E)}\wedge \rho_{d,E}\colon \Omega^1_d\otimes J^1_d E\longrightarrow \Omega^2_d (E),
\end{equation}
where $\rho_{d,E}$ is defined in \eqref{eq:rhon}.
This map factors through the quotient $\Omega^1_d\otimes J^1_d E\twoheadrightarrow \Omega^1_d\otimes_A J^1_d E$, since
\begin{equation}
\begin{split}
\left(d\otimes_A \pi^{1,0}_{d,E}+\id_{\Omega^1_d(E)}\wedge \rho_{d,E}\right)(\alpha\otimes \lambda\xi)
&=d\alpha\otimes_A \pi^{1,0}_{d,E}(\lambda \xi)+\alpha\wedge \rho_{d,E} (\lambda\xi)\\
&=d\alpha\otimes_A \lambda \pi^{1,0}_{d,E}(\xi)-\alpha\wedge d\lambda\otimes_A \rho_{d,E} (\xi)+\alpha\wedge \lambda\rho_{d,E} (\xi)\\
&=(d\alpha)\lambda\otimes_A \pi^{1,0}_{d,E}(\xi)-\alpha\wedge d\lambda \otimes_A \pi^{1,0}_{d,E}(\xi)+\alpha\lambda\wedge \rho_{d,E} (\xi)\\
&=d(\alpha\lambda)\otimes_A \pi^{1,0}_{d,E}(\xi)+\alpha\lambda\wedge \rho_{d,E} (\xi)\\
&=\left(d\otimes_A \pi^{1,0}_{d,E}+\id_{\Omega^1_d(E)}\wedge \rho_{d,E}\right)(\alpha\lambda\otimes \xi)
\end{split}
\end{equation}
for all $\alpha\in\Omega^1_d$, $\xi \in J^1_d E$, $\lambda\in A$ (cf.\ Remark \ref{rmk:rhodo}).
The factorization of this map is precisely $W^{II}_{d,E}$, which is in $\Mod_B$, as we can see from \eqref{eq:WI-II}.

If we restrict $W^{II}_{d,E}$ to $\Omega^1_d\otimes_A \Omega^1_d(E)$, we obtain
\begin{equation}\label{eq:Wrestrictstowedge}
\begin{split}
W^{II}_{d,E}\circ \Omega^1_d(\iota^1_{d,E})
&=\left(d\otimes_A \pi^{1,0}_{d,E}+\id_{\Omega^1_d(E)}\wedge \rho_{d,E}\right)\circ \Omega^1_d(\iota^1_{d,E})\\
&=d\otimes_A (\pi^{1,0}_{d,E}\circ \iota^1_{d,E})+\id_{\Omega^1_d(E)}\wedge \rho_{d,E}\circ \iota^1_{d,E}\\
&=0+ \id_{\Omega^1_d(E)}\wedge\id_{\Omega^1_{d}(E)}\\
&=\wedge.
\end{split}
\end{equation}
This implies that $W^{II}_{d,E}$ vanishes on $S^2_d(E)$, as
\begin{equation}
W^{II}_{d,E}\circ \left(\Omega^1_d(\iota^1_{d,E})\circ \iota^2_{\wedge,E}\right)
=\wedge\circ \iota^2_{\wedge,E}
=0.
\end{equation}
Therefore, $W^{II}_{d,E}$ factors through the cokernel of $\Omega^1_d(\iota^1_{d,E})\circ \iota^2_{\wedge,E}$.
This yields the existence of a compatible morphism $\varsigma_E\colon \coker(\Omega^1_d(\iota^1_{d,E})\circ\iota^2_{\wedge,E})\to \Omega^2_d(E)$ in $\Mod_B$.
Furthermore, letting $\inc^2$ denote the inclusion of $\Omega^2_d(E)$ in the cokernel gives
\begin{equation}
\varsigma_E\circ \inc^2\circ \wedge
=\varsigma_E\circ W_{d,E}\circ \Omega^1_d(\iota^1_{d,E})
=W^{II}_{d,E}\circ \Omega^1_d(\iota^1_{d,E})
=\wedge.
\end{equation}
Since $\wedge$ is an epi, it follows that $\varsigma_E\circ \inc^2=\id_{\Omega^2_d(E)}$, so $\varsigma_E$ is a right split for our sequence in $\Mod_B$.
We thus write the cokernel as $\Omega^1_d(E)\oplus_B\Omega^2_d(E)$.

We can now compute the explicit components of $W_{d,E}$ by applying the projections of $\Omega^1_d(E)$ and $\Omega^2_d(E)$, respectively.
The first is $\Omega^1_d(\pi^{1,0}_{d,E})=W^{I}_{d,E}$, by commutativity of \eqref{es:W}.
The second is $\varsigma_E \circ W_{d,E}$, i.e.\ $W^{II}_{d,E}$.

From the explicit description \eqref{eq:WI-II}, we see that $W^{II}_{d,E}=W^{II}_{d,A}\otimes_A E$.
We can now obtain \eqref{eq:cokersp} from the definition of $\Omega^1_d$ and $\Omega^2_d$ as tensor functors, applying the functor $-\otimes_A E$ to the cokernel split sequence for $E=A$.

The morphism $W_{d,A}$ is $A$-bilinear, being a cokernel morphism in $\AModA$.
Since $W_{d,A}$ is an epi, the actions are completely determined by the actions on $\Omega^1_d\otimes_A J^1_d A$.
By construction, the right action is component-wise.
For the left action, by linearity, we can compute separately the action on the two components.
For all $\alpha\in\Omega^1_d$, we have $\alpha=\sum_i (da_i)b_i$ by the surjectivity assumption, and $\alpha=W_{d,A}(\sum_i da_i\otimes_A [1\otimes b_i])$, and thus
\begin{equation}
\begin{split}
f\smb\alpha
&=f\smb W_{d,A}\left(\sum_i da_i\otimes_A [1\otimes b_i]\right)
=\sum_i W_{d,A}\left(f da_i\otimes_A [1\otimes b_i]\right)
=\sum_i \left(f da_i b_i+ d(f da_i) b_i-0\right)\\
&=f\sum_i da_i b_i+ df\wedge \sum_i (da_i) b_i
=f\alpha + df\wedge \alpha.
\end{split}
\end{equation}
Given $\omega\in\Omega^2_d$, by surjectivity we can write $\omega=\sum_i \alpha_i \wedge \beta_i$.
From the explicit description \eqref{eq:WIIexp} of $W^{II}_{d,A}$, we can see $\omega=W(\sum_i \alpha_i\wedge \iota^1_{d,A} (\beta_i))$.
Therefore,
\begin{equation}
\begin{split}
f\smb \omega
&=f\smb W\left(\sum_i \alpha_i\wedge \iota^1_{d,A} (\beta_i)\right)
=\sum_i W\left(f\alpha_i\wedge \iota^1_{d,A} (\beta_i)\right)
=\sum_i \left(0+f\alpha_i\wedge \beta_i\right)
=f\omega.
\end{split}
\end{equation}
\end{proof}
We denote by $\Omega^1_d(E)\ltimes\Omega^2_d(E)$ the $A$-module obtained by endowing $\Omega^1_d(E)\oplus\Omega^2_d(E)$ with the left action induced by \eqref{eq:laacSP}.
Similarly, for $E=A$ we write $\Omega^1_d\ltimes \Omega^2_d$ for the $A$-bimodule described in Lemma \ref{lemma:spplsplit}.
Moreover, for $E=A$ we will omit the subscript $A$ from all the morphisms involved.

\begin{rmk}\label{rmk:rightsa}
The following morphism provides a left $A$-linear right splitting of $\Omega^1_d\ltimes \Omega^2_d$.
\begin{align}
\Omega^1_d\longrightarrow \Omega^1_d\ltimes\Omega^2_d,
&\hfill&
\alpha\longmapsto \alpha+d\alpha.
\end{align}
It is left $A$-linear because
\begin{equation}
f\smb (\alpha+d\alpha)
=f\alpha+df\wedge\alpha+fd\alpha
=f\alpha+d(f\alpha).
\end{equation}

This realizes $\Omega^1_d\ltimes \Omega^2_d$ as $\Omega^1_d\Aoplus \Omega^2_d$, where $A$ acts component-wise on the left and on the right as
\begin{equation}
(\alpha+\omega)\smbr f
=\alpha f+\alpha\wedge df+\omega f,
\end{equation}
for all $\alpha\in\Omega^1_d$, $\omega\in\Omega^2_d$, and $f\in A$.
We denote this $A$-bimodule presentation as $\Omega^1_d\ltimes' \Omega^2_d$, but they are evidently isomorphic via the following $A$-bilinear isomorphism
\begin{align}
\Omega^1_d\ltimes' \Omega^2_d\longrightarrow \Omega^1_d\ltimes \Omega^2_d,
&\hfill&
\alpha+\omega\longmapsto \alpha+d\alpha+\omega.
\end{align}
The actions of Lemma \ref{lemma:spplsplit} and Remark \ref{rmk:rightsa} are reminiscent of the ones appearing in Remark \ref{rmk:semidirectjet}.
\end{rmk}

Before continuing, we show that, on its own, $W^{II}_{d,E}$ is not left $A$-linear.
For all $\alpha\in\Omega^1_d(E)$, $\xi\in J^1_d E$, and $\lambda\in A$, we have
\begin{equation}\label{eq:WIILeib}
\begin{split}
W^{II}_{d,E}(\lambda \alpha\otimes_A \xi)
&=d(\lambda\alpha)\otimes_A \pi^{1,0}_{d,E}(\xi)+\lambda \alpha\wedge \rho_{d,E}(\xi)\\
&=d\lambda\wedge\alpha\otimes_A \pi^{1,0}_{d,E}(\xi)+\lambda W^{II}_{d,E}(\alpha\otimes _A \xi)\\
&=\left(d\lambda\wedge W^{I}_{d,E}+\lambda W^{II}_{d,E}\right)(\alpha\otimes_A \xi).
\end{split}
\end{equation}

\subsubsection{The operator $\widetilde{\DH}_d$}
Now that we have $W_{d,E}$, we can prolong it to an $A$-linear map $\widetilde{\DH}_{d,E}$.
The extension is not necessarily unique, but every $A$-linear extension of $W_{d,E}$ through $\iota^1_{d,J^1_d E}$ mapping to $\Omega^1_d(E) \ltimes\Omega^2_d(E)$ and vanishing on $j^2_d(E)$ will give the same kernel.
Under these conditions, in fact, we would get an inclusion of $J^2_d E$ into the kernel of this operator, and said inclusion is part of an inclusion of short exact sequences.
More specifically, this is an inclusion of extensions of $J^1_d E$ via $S^2_d(E)$.
By the snake lemma, this inclusion is also surjective, and hence an isomorphism.
\begin{prop}\label{prop:defeth}
The map $\widetilde{\DH}_{d,E}\colon J^{(2)}_d E\to \Omega^1_d(E)\ltimes \Omega^2_d(E)=\coker(l^2_{d,E})$ is defined as $(\widetilde{\DH}^I_{d,E},\widetilde{\DH}^{II}_{d,E})$, where $\widetilde{\DH}^I_{d,E}$ is the map defined in \eqref{eq:DHI}, and $\widetilde{\DH}^{II}_{d,E}\colonequals W^{II}_{d,E}\circ \rho_{d,J^1_d E}$.
Explicitly, we can write
\begin{equation}\label{eq:FE-WE}
\begin{tikzcd}
\widetilde{\DH}^{I}_{d,E}\colon J^{(2)}_d E\ar[r]&\Omega^1_d(E)&[-20pt]\widetilde{\DH}^{II}_{d,E}\colon J^{(2)}_d E\ar[r]&\Omega^2_d(E)\\[\vsfd]
\hphantom{\widetilde{\DH}^{I}_d\colon } [a\otimes b]\otimes_A [c\otimes e]\ar[r,|->]&ad(bc) \otimes_A e&\hphantom{\widetilde{\DH}^{II}_d\colon} [a\otimes b]\otimes_A [c\otimes e]\ar[r,|->]&da\wedge d(bc)\otimes_A e.
\end{tikzcd}
\end{equation}
\end{prop}
\begin{proof}
We start by proving the following equality using Remark \ref{rmk:rhodo}, \eqref{eq:WIILeib}, and \eqref{DHIcom}.
For all $\eta\in J^1_d(E)$, $\xi\in J^1_d E$, and $\lambda\in A$, we have
\begin{equation}
\begin{split}
\widetilde{\DH}^{II}_{d,E}(\lambda \eta\otimes_A \xi)
&=W^{II}_{d,E}\left(\rho_{d,J^1_d E}(\lambda \eta\otimes_A \xi)\right)\\
&=W^{II}_{d,E}\left(-d\lambda\otimes_A \pi^{1,0}_{d,J^1_d E} (\eta\otimes_A \xi)+\lambda \rho_{d,J^1_d E}\left(\eta\otimes_A \xi\right)\right)\\
&=-d\lambda\wedge \rho_{d,E}\circ \pi^{1,0}_{d,J^1_d E} (\eta\otimes_A \xi)
+d\lambda \wedge W^{I}_{d,E}\circ \rho_{d,J^1_d E}\left(\eta\otimes_A \xi\right)
+\lambda W^{II}_{d,E}\circ\rho_{d,J^1_d E}\left(\eta\otimes_A \xi\right)\\
&=-d\lambda\wedge \rho_{d,E}\circ \pi^{1,0}_{d,J^1_d E} (\eta\otimes_A \xi)
+d\lambda \wedge \Omega^1_d(\pi^{1,0}_{d,E})\circ \rho_{d,J^1_d E}\left(\eta\otimes_A \xi\right)
+\lambda \widetilde{\DH}^{II}_{d,E}\circ\rho_{d,J^1_d E}\left(\eta\otimes_A \xi\right)\\
& = d\lambda \wedge \widetilde{\DH}^{I}_{d,E}\left(\eta\otimes_A \xi\right)
+\lambda \widetilde{\DH}^{II}_{d,E}\circ\rho_{d,J^1_d E}\left(\eta\otimes_A \xi\right)
\end{split}
\end{equation}
We can now prove that $\widetilde{\DH}_d$ is left $A$-linear, since for all $\zeta\in J^{(2)}_d E$ and $\lambda\in A$, we have
\begin{equation}
\begin{split}
\widetilde{\DH}_{d,E}(\lambda\zeta)
&=\widetilde{\DH}^{I}_{d,E}(\lambda\zeta)+\widetilde{\DH}^{II}_{d, E}(\lambda \zeta)\\
&=\lambda\widetilde{\DH}^{I}_{d,E}(\lambda\zeta)+df\wedge\widetilde{\DH}^{I}_{d,E}(\zeta)+\lambda \widetilde{\DH}^{II}_{d,E}(\zeta)\\
&=\lambda\smb \left(\widetilde{\DH}^{I}_{d,E}(\zeta)+\widetilde{\DH}^{II}_{d,E}(\zeta)\right)\\
&=\lambda\smb \widetilde{\DH}_{d,E}(\zeta).
\end{split}
\end{equation}

From \eqref{DHIcom} and \eqref{es:W}, we get
\begin{equation}
\widetilde{\DH}^I_{d,E}\circ \iota^1_{d,J^1_d E}
=\left(\Omega^1_d(\pi^{1,0}_{d,E})\circ \rho_{d,J^1_d E}-\rho_{d,E}\circ \pi^{1,0}_{d,J^1_d E}\right)\circ \iota^1_{d,J^1_d E}
=\Omega^1_d(\pi^{1,0}_{d,E})-0
=W^I_{d,E}.
\end{equation}
and by definition,
\begin{equation}
\widetilde{\DH}^{II}_{d,E}\circ \iota^1_{d,J^1_d E}
=W^{II}_{d,E}\circ \rho_{d,J^1_d E}\circ \iota^1_{d,J^1_d E}
=W^{II}_{d,E}.
\end{equation}
It follows that $\widetilde{\DH}_{d,E}\circ \iota^1_{d,J^1_d E}=W_{d,E}$.

It remains to show that $\widetilde{\DH}_d$ vanishes when precomposed with $j^2_{d}$, which is defined through the nonholonomic $2$-jet prolongation.
We can thus prove $\widetilde{\DH}_{d,E}\circ j^{(2)}_{d,E}=0$ on each component.
The first component vanishes because the nonholonomic $2$-jet prolongation factors through the semiholonomic jet (cf.\ Proposition \eqref{prop:nhjtosh}).
The second component vanishes because $\widetilde{\DH}^{II}_{d,E}\circ j^{1}_{J^1_d E}=W^{II}_{d,E}\circ \rho_{d,J^1_d E}\circ j^{1}_{J^1_d E}=0$.
Hence, $\ker(\widetilde{\DH}_d)$ contains, and thus coincides with, $J^2_d E=S^2_d(E)+Aj^2_d E$.
\end{proof}
It follows that $J^2_d E$ is contained in $J^{[2]}_d E$, as one of the two components of $\widetilde{\DH}_{d,E}$ is $\widetilde{\DH}^I_{d,E}$, whose kernel is the semiholonomic $2$-jet.

\begin{defi}\label{def:2-jet}
We define the \emph{$2$-jet module} of $E$ by $J^2_d E\colonequals \ker(\widetilde{\DH}_{d,E})$, with $\iota^2_{d,E}$ and $\pi^{2,1}_{d,E}$ as the induced maps in the kernel sequence in \eqref{es:2jetDt}.
Finally, let $j^2_{d,E}\colon E\to J^2_d E$ be the factorization of $j^{(2)}_{d,E}$ through $J^2_d E$.
\end{defi}
\begin{rmk}
The maps $\iota^2_d$ and $\pi^{2,1}_d$ are $A$-linear, and the map $j^2_d$ is only $\bk$-linear.

We can extend this construction to the category $\AMod_B$, in which case $\iota^2_d$ and $\pi^{2,1}_d$ are $(A,B)$-bilinear, and $j^2_d$ is only right $B$-linear.
\end{rmk}
\begin{rmk}\label{rmk:arrows2j}
By construction, we have
\begin{align}
l^2_d \circ\iota^2_d
=(\iota_d\otimes_A \iota_d)\circ \iota^2_{\wedge},
&\hfill&
\pi^{2,1}_d
=J^1_d (\pi^{1,0}_d)\circ l^2_d
=\pi^{1,0}_{J^1_d}\circ l^2_d.
\end{align}
The relation between the maps of the $2$-jet sequence and the naturality diagram for the $1$-jet functor applied to the $1$-jet sequence can be summarized in the following diagram
\begin{equation}\label{diag:fish}
\begin{tikzcd}
&&&&[6pt]\Omega^1_d(E)\ar[dr,hookrightarrow,"\iota^1_{d,E}"']&[18pt]&[28pt]&[10pt]\\
&&&J^1_d(\Omega^1_d (E))\ar[dr,"J^1_d (\iota^1_{d,E})"]\ar[ur,twoheadrightarrow,"\pi^{1,0}_{d,\Omega^1_d E}"']&&J^1_d E\ar[dr,twoheadrightarrow,"\pi^{1,0}_{d,E}"' near end]\ar[d,equals]&0\\
&S^2_d(E) \ar[r,hookrightarrow,"\iota^2_{\wedge,E}"]&\Omega^1_d(\Omega^1_d(E))\ar[ur,hookrightarrow,"\iota^1_{d,\Omega^1_d(E)}"' near start]\ar[dr,"\Omega^1_d(\iota^1_{d,E})"']&J^2_d E\ar[r,hookrightarrow,"l^2_{d,E}"]\ar[rr,twoheadrightarrow,bend right=40pt,"\pi^{2,1}_{d,E}"',line width=0.8pt]&J^{(2)}_d E\ar[dr,twoheadrightarrow,crossing over,"J^1_d(\pi^{1,0}_{d,E})", pos=0]\ar[ur,twoheadrightarrow,"\pi^{1,0}_{d,J^1_d E}"]&J^1_d E\ar[d,equals]\ar[ur, crossing over,line width=0.8pt,bend left=10pt]&E\\
0\ar[ur,line width=0.8pt]&&&\Omega^1_d(J^1_d E)\ar[dr,twoheadrightarrow,"\Omega^1_d(\pi^{1,0}_{d,E})"']\ar[ur,hookrightarrow,"\iota^1_{d,J^1_d E}"' near end,crossing over]&&J^1_d E\ar[ur,twoheadrightarrow,"\pi^{1,0}_{d,E}"']&\\
&&&&\Omega^1_d (E)\ar[ur,hookrightarrow,"\iota^1_{d,E}"']&&
\ar[from=3-2,to=3-4,bend left=38pt,"\iota^2_{d,E}",crossing over,line width=0.8pt]
\end{tikzcd}
\end{equation}
\end{rmk}
Finally, notice that since $\widetilde{\DH}_{d,E}\colon J^{(2)}_d E\to\Omega^1_d(E)\ltimes\Omega^2_d(E)$ is a left $A$-linear map, it can be interpreted as the extension of a first order differential operator $\DH_{d,E}=\widetilde{\DH}_{d,E}\circ j^2_{d,E}$ on $J^1_d E$
\begin{align}
\label{eq:DH_natural_DO}
\DH_{d,E}\colon J^1_d E\longrightarrow\Omega^1_d(E)\ltimes\Omega^2_d(E),
&\hfill&
[a\otimes e]\longmapsto(da)\otimes_A e.
\end{align}

\begin{rmk}
For $E=A$, we can write
\begin{align}\label{eq:MC}
W^{II}_{d,A}=d\otimes_A \pi^{1,0}_{d}+\id_{\Omega^1_d}\wedge \rho_d,
&\hfill&
\widetilde{\DH}^{II}_{d,A}=d\rho_d \circ \pi^{1,0}_{d,J^1_d A}+\rho_d\wedge\rho_d.
\end{align}
We note that this representation of $\widetilde{\DH}^{II}_{d,A}=0$ resembles the Maurer-Cartan equation for $\rho_d$.
Via this analogy, the equation $W^{II}_{d,A}=0$ could be interpreted as a torsion-freeness condition.
\end{rmk}
\subsubsection{Another interpretation of $\widetilde{\DH}_d$}
Let $E=A$, using the representation $\Omega^1_d\ltimes'\Omega^2_d$ (cf.\ Remark \ref{rmk:rightsa}) we can write an equivalent description of $\widetilde{\DH}_d$, that is
\begin{align}
\widetilde{\DH}_d'\colon J^{(2)}_d A\longrightarrow \Omega^1_d \ltimes'\Omega^2_d,
&\hfill&
[a\otimes b]\otimes_A [x\otimes y]\longmapsto {ad(bx)y+ad(bx)\wedge dy}.
\end{align}
For generic $E$, we have to define $\widetilde{\DH}'_{d,E}\colonequals \widetilde{\DH}_d'\otimes_A E$, being careful to use the correct right action $\smbr$.

Although $\widetilde{\DH}_d$ is more convenient for the presentation of the higher jet functors, the presentation $\widetilde{\DH}_d'$ shows the meaning of both components $(\widetilde{\DH}'^{I}_d,\widetilde{\DH}'^{II}_d)$, which for $\widetilde{\DH}_d$ is concealed within the action.

The interpretation of the first component vanishing is the same in both presentations, and is given in §\ref{ss:shj}.

For $\widetilde{\DH}'^{II}_d$, consider the differential acting on the first grade of the exterior algebra $d^1\colon \Omega^1_d\longrightarrow \Omega^2_d$.
This is a first order differential operator, so it can be lifted to a left $A$-linear map $\widetilde{d}^1\colon J^1_d \Omega^1_d\longrightarrow \Omega^2_d$.
Now apply the functor $J^1_d$ to the lifting $\widetilde{d}\colon J^1_d A\to \Omega^1_d$.
We obtain a map
\begin{align}
\widetilde{d}^1\circ J^1_d \widetilde{d}\colon J^{(2)}_d A\longrightarrow \Omega^2_d,
&\hfill&
[a\otimes b]\otimes_A [a'\otimes b']\longmapsto {\widetilde{d}^1([a\otimes b]\otimes_A a' d b')
=ad(ba'db')=ad(ba')\wedge db'}.
\end{align}
Notice that $\widetilde{d}^1\circ J^1_d \widetilde{d}=\widetilde{\DH}'^{II}_d$.
Furthermore, we have
\begin{equation}
\widetilde{\DH}'^{II}_d\circ j^{(2)}_d
=\widetilde{d}^1\circ J^1_d \widetilde{d}\circ j^1_{J^1_d A}\circ j^1_d
=\widetilde{d}^1\circ j^1_d\circ \widetilde{d}\circ j^1_d
=d^1\circ d.
\end{equation}
We will see in §\ref{s:differentialoperators} that this corresponds to the notion of nonholonomic differential operator which lifts $d^1\circ d=(d)^2=0$.
Therefore, the equation $\widetilde{\DH}'^{II}_d=0$ means that the canonical lift of the differential operator $d^2=0$ as composition of two differential operators lifts to the zero map on $J^2_d A$.

\subsection{Functoriality}\label{ss:functorialityof2jet}
In Lemma \ref{lemma:spplsplit}, we show that the sequence defining the extension $\Omega^1_d(E)\ltimes \Omega^2_d(E)$ can be obtained by applying $-\otimes_A E$ to the sequence defining the bimodule $\Omega^1_d\ltimes \Omega^2_d$.
We thus get the functor
\begin{align}
\Omega^1_d\ltimes \Omega^2_d\colon \AMod\longrightarrow \AMod,
&\hfill&
E\longmapsto (\Omega^1_d\ltimes \Omega^2_d)\otimes_A E\cong \Omega^1_d(E)\ltimes \Omega^2_d(E).
\end{align}
Since all of its entries are defined as tensor products, we obtain the following short exact sequence of endofunctors of $\AMod$
\begin{equation}\label{es:spsplit}
\begin{tikzcd}
0\ar[r]&\Omega^2_d\ar[r,hookrightarrow]&\Omega^1_d\ltimes \Omega^2_d \ar[r,twoheadrightarrow]&\Omega^1_d\ar[r]&0
\end{tikzcd}
\end{equation}
By Lemma \ref{lemma:spplsplit}, the sequence naturally splits in the category of functors of type $\AMod\to \Mod$.

\begin{lemma}\label{lemma:spplacederfun}
Let $L_{\bullet}(\Omega^1_d\ltimes \Omega^2_d)$ be the left derived functor of $\Omega^1_d\ltimes \Omega^2_d$.
As a functor $\AMod\to \Mod$, $L_{n}(\Omega^1_d\ltimes \Omega^2_d)$ is naturally isomorphic to $\Tor^A_n(\Omega^1_d,-)\oplus \Tor^A_n(\Omega^2_d,-)$.
\end{lemma}
\begin{proof}
The functor $\Omega^1_d\ltimes \Omega^2_d$ has as left derived functor $\Tor^A_{\bullet}(\Omega^1_d\ltimes \Omega^2_d,-)$.
As a functor $\AMod\to \Mod$, it coincides with the tensor by the right $A$-module $\Omega^1_d\oplusA \Omega^2_d$.
Since $\Tor^A_n$ preserves arbitrary direct sums \cite[Corollary 2.6.12, p.~56]{Weibel}, we have
\begin{equation}
\Tor^A_{n}(\Omega^1_d\ltimes \Omega^2_d,-)
=\Tor^A_{n}(\Omega^1_d\oplusA \Omega^2_d,-)
=\Tor^A_{n}(\Omega^1_d,-)\oplusA \Tor^A_{n}(\Omega^2_d,-).
\end{equation}
\end{proof}
We can now see $\widetilde{\DH}_d$ as a natural epimorphism
\begin{equation}
\widetilde{\DH}_d\colon J^{(2)}_d\longtwoheadrightarrow \Omega^1_d\ltimes \Omega^2_d.
\end{equation}
Naturality follows from the definitions of $J^{(2)}_d$ and $\Omega^1_d\ltimes \Omega^2_d$ as tensor functors, and each component is an epimorphism by construction, as $W_d$ is, cf.\ \eqref{es:W}.

In particular, $\widetilde{\DH}_d$ composed with the projection to $\Omega^1_d$ gives the natural transformation \eqref{eq:DHI}.
The map $\widetilde{\DH}^{II}_d$ induces a natural transformation of functors $\AMod\to \Mod$ that extends to $\widetilde{\DH}_d$.
\begin{rmk}\label{rmk:semiholethII}
If we restrict ourselves to $J^{[2]}_d$, then also $\widetilde{\DH}^{II}_d$ becomes a natural transformation of endofunctors of $\AMod$.
Therefore, by Theorem \ref{theo:sjchar}, if we restrict $\widetilde{\DH}_{d,J^{(n-2)}_d}$ to $J^{[n]}_d$, by \eqref{es:spsplit}, it factors to a left $A$-linear natural transformation as
\begin{equation}
\widetilde{\DH}^{II}_{d,J^{(n-2)}_d}\circ{\iota_{J^{[n]}_d}}\colon J^{[n]}_d\longrightarrow \Omega^2_d.
\end{equation}
\end{rmk}

By diagram \eqref{es:2jetDt}, if we now restrict $\widetilde{\DH}_d \colon J^{(2)}_d\to \Omega^1_d\ltimes \Omega^2_d$ to $\Omega^1_d\circ J^1_d$ by $\iota^1_{d,J^1_d}\colon \Omega^1_d\circ J^1_d\hookrightarrow J^{(2)}_d$, we obtain the natural transformation $W_d\colon \Omega^1_d\circ J^1_d\to \Omega^1_d\ltimes \Omega^2_d$.
\begin{rmk}
Diagram \eqref{es:2jetDt} also yields $S^2_d=\ker(W_d)$.
\end{rmk}
Moreover, if we precompose $W_d$ with $\Omega^1_d(\iota^1_d)$, then by \eqref{es:W} we obtain $\wedge\colon \Omega^1_d\circ\Omega^1_d\rightarrow \Omega^2_d$.

\begin{defi}
The \emph{(holonomic) $2$-jet functor}, denoted by $J^2_d$, is the kernel of $\widetilde{\DH}_d$ in the category of functors $\AMod\to \AMod$.
\end{defi}

\begin{prop}[Functorial $2$-jet exact sequence]\label{prop:functorial2jetseq}
If $\Omega^1_d$ is flat in $\ModA$, the $2$-jet functor satisfies the following short exact sequence in the category of functors $\AMod\to\AMod$
\begin{equation}\label{es:2jetfun}
\begin{tikzcd}
0 \arrow[r] &S_d^2 \arrow[r,hookrightarrow,"\iota_d^2"]& J^2_d \arrow[r,twoheadrightarrow,"\pi^{2,1}_d"]& J^1_d \arrow[r]&0.
\end{tikzcd}
\end{equation}
\end{prop}
\begin{proof}
All objects and arrows in the two bottom rows of \eqref{es:2jetDt} and \eqref{es:W} are natural, so we obtain the same diagrams on functors.
By flatness of $\Omega^1_d$, all rows in these diagrams are short exact sequences.
We take the kernel sequences for both, and they are exact by the nine lemma.
The first one gives us
\begin{equation}
\begin{tikzcd}
0 \arrow[r] &\ker(W) \arrow[r,hookrightarrow,"\iota_d^2"]& J^2_d \arrow[r,twoheadrightarrow,"\pi^{2,1}_d"]& J^1_d \arrow[r]&0,
\end{tikzcd}
\end{equation}
and the second gives us $\ker(W)=S^n_d$.
\end{proof}
The same relations shown for modules in \eqref{diag:fish} hold for functors if $\Omega^1_d$ is a flat right $A$-module.

\subsection{The classical $2$-jet functor}
Next, we extend Theorem \ref{theo:classical1jet} to cover the $2$-jet functor.
\begin{theo}\label{theo:classical2jet}
	Let $M$ be a smooth manifold, $A=\smooth{M}$ its algebra of smooth functions, and $E=\Gamma(M,N)$ the space of smooth sections of a vector bundle $N\rightarrow M$.
	Then $J^2_dE \cong \Gamma(M,J^2N)$ in $\AMod$, the classical module of $2$-jets of sections of $N$, and the isomorphism takes our prolongation to the classical prolongation.
\end{theo}
\begin{proof}
	By Theorem \ref{theo:classicalnonsemi}, the functor $J^{(2)}_d$ coincides with the classical one.
	Then, the uniqueness result \cite[Proposition 3]{Goldschmidt} identifies the map $\rho$ with our map $\widetilde{\DH}_d$.
	Thus, the kernel of $\widetilde{\DH}_d$ is the classical $2$-jet module, seen as a submodule of the nonholonomic jets.
	The prolongation maps coincide with the composition of the first jet prolongation with itself both for our construction and classically, and so by Theorem \ref{theo:classical1jet}, they coincide.
\end{proof}

\section{Holonomic and sesquiholonomic jet functors}
\subsection{The $n$-jet functor}
Similar to Definition \ref{def:qsf}, and in the spirit of the characterization \cite[Proposition 3, p.~432]{Goldschmidt}, we define the $n$-jet functor by induction on $n$.
\begin{defi}[Holonomic $n$-jet functor]\label{def:njet}
Let $A$ be a $\bk$-algebra endowed with an exterior algebra $\Omega^{\bullet}_d$.
We define $J^0_d\colonequals \id_{\AMod}$, $J^1_d$ as in §\ref{ss:funct1j}, $l^1_d\colonequals\id_{J^1_d}$, and by induction, $J^n_d$ as the kernel of the natural transformation
\begin{equation}
\widetilde{\DH}_{d,J^{n-2}_d} \circ J^1_d(l^{n-1}_d)\colon J^1_d\circ J^{n-1}_d \longrightarrow (\Omega^1_d\ltimes \Omega^2_d)\circ J^{n-2}_d,
\end{equation}
where we denote the natural inclusion by $l^n_d\colon J^n_d\longhookrightarrow J^1_d\circ J^{n-1}_d$.
We call $J^n_d$ the \emph{(holonomic) $n$-jet functor}.
\end{defi}
\begin{rmk}[Jet functors on bimodules]\label{rem:jetfunctorsonbimodules}
Analogously, we can define the functors $J^n_d\colon\AMod_B\to \AMod_B$.
Since the forgetful functor $\AMod_B\rightarrow\AMod$ creates limits, and since limits in functor categories are computed object-wise, the two definitions are compatible via the forgetful functor.
\end{rmk}
It is natural to consider the following composition
\begin{equation}\label{eq:iotajn}
\iota_{J^n_d}\colonequals J^{(n-2)}_d(l^2_d)\circ J^{(n-3)}_d(l^3_d)\circ \cdots \circ J^{(1)}_d(l^{n-1}_d)\circ l^n_d \colon J^n_d\longrightarrow J^{(n)}_d.
\end{equation}
\begin{rmk}\label{rmk:iotaholinj}
	The natural transformation $\iota_{J^n_d}$ provides a mapping from the holonomic to the semiholonomic jet, but in general it is not injective (as has been noted before in the setting of synthetic differential geometry, e.g.\ \cite[p.~87]{kock2010synthetic}).
However, if we assume $\Omega^1_d$ to be flat in $\ModA$, the functor $J^1_d$ is exact by Corollary \ref{cor:Jex}, and thus the maps $J^{(n-k)}_d(l^k_d)$ are injective for all $0\le k\le n$.
In this case, $\iota_{J^n_d}$ is a natural monomorphism.
\end{rmk}

From the definition and from the functoriality of the jet functors, it follows that for all $0\le m\le n$ we have the equality
\begin{equation}\label{eq:spinjn}
\iota_{J^n_d}
=J^{(m)}_d(\iota_{J^{n-m}_d})\circ J^{(m-1)}_d(l^{n-m+1}_d)\circ J^{(m-2)}_d(l^{n-m+2}_d)\circ \cdots \circ J^{(1)}_d(l^{n-1}_d)\circ l^n_d,
\end{equation}
which will be used later on in this section.

The following lemma shows an equivalent description of $J^n_d$.
\begin{lemma}\label{lemma:generalspencer}
For all $n\ge 2$, the following is a pullback diagram
\begin{equation}\label{diag:generalspencer}
\begin{tikzcd}[column sep=40pt]
J^n_d\ar[r,hookrightarrow,"l^n_d"]\ar[d]&J^1_d\circ J^{n-1}_d\ar[d,"J^1_d(l^{n-1}_d)"]\\
J^2_d\circ J^{n-2}_d\ar[r,hookrightarrow,"l^2_{d,J^{n-2}_d}"']&J^{(2)}_d\circ J^{n-2}_d
\end{tikzcd}
\end{equation}

Moreover, if $\Omega^1_d$ is flat in $\ModA$, then \eqref{diag:generalspencer} is an intersection in $J^{(2)}_d\circ J^{n-2}_d$ and in $J^{(n)}_d$.
\end{lemma}
\begin{proof}
The proof is essentially the same as that of Lemma \ref{lemma:pbqsf}, but using the following diagram
\begin{equation}
\begin{tikzcd}[column sep=50pt]
J^n_d\ar[r,hookrightarrow,"l^n_d"]\ar[d,dashed]&J^1_d\circ J^{n-1}_d\ar[d,"J^1_d(l^{n-1}_d)"]\ar[r,"\widetilde{\DH}_{d,J^{n-2}_d}\circ J^1_d(l^{n-1}_d)"]&(\Omega^1_d\ltimes\Omega^2_d)\circ J^{n-2}_d\ar[d,equals]\\
J^2_d\circ J^{n-2}_d\ar[r,hookrightarrow,"l^2_{d,J^{n-2}_d}"']&J^{(2)}_d\circ J^{n-2}_d\ar[r,twoheadrightarrow,"\widetilde{\DH}_{d,J^{n-2}_d}"']&(\Omega^1_d\ltimes\Omega^2_d)\circ J^{n-2}_d
\end{tikzcd}
\end{equation}

For the final statement one has to compose the pullback diagram with $J^{(2)}_d(l^{n-2}_d)$, which is a mono by flatness of $\Omega^1_d$.
\end{proof}

With more regularity on the exterior algebra, we get other equivalent descriptions of $J^n_d$.
\begin{lemma}\label{lemma:Jchar}
For all $n\ge 2$, if $\Omega^1_d$ and $\Omega^2_d$ are flat in $\ModA$, the following sub-functors of $J^{(n)}_d$ coincide
\begin{enumerate}
\item\label{lemma:Jchar:1} $J^n_d$;
\item\label{lemma:Jchar:2} $\bigcap_{k=0}^{n-2}\ker\left(J^{(k)}_d \widetilde{\DH}_{d,J^{(n-k-2)}_d} \right)$;
\item\label{lemma:Jchar:3} The intersection $J^h_d\circ J^{(n-h)}_d\cap J^{(n-k)}_d \circ J^k_d$ for $k,h\ge 0$ such that $h+k>n$.
\end{enumerate}
\end{lemma}
\begin{proof}
For $n=2$ all conditions coincide.
Now let $n>2$ and assume the statement true for all $k<n$.

\eqref{lemma:Jchar:1}$\subseteq$\eqref{lemma:Jchar:2}: We prove that $J^{(k)}_d\circ \widetilde{\DH}_{d,J^{(n-k-2)}_d}\circ \iota_{J^n_d}=0$ for $0\le k\le n-2$.
We apply \eqref{eq:spinjn} to decompose $\iota_{J^n_d}$ at the level $k+2$, obtaining
\begin{equation}
\begin{split}
&J^{(k)}_d \widetilde{\DH}_{d,J^{(n-k-2)}_d}\circ \iota_{J^n_d}\\
&\qquad=J^{(k)}_d \widetilde{\DH}_{d,J^{(n-k-2)}_d}\circ J^{(k+2)}_d(\iota_{J^{n-k-2}_d})\circ J^{(k+1)}_d(l^{n-k-1}_d)\circ J^{(k)}_d(l^{n-k}_d)\circ \cdots \circ J^{(1)}_d(l^{n-1}_d)\circ l^n_d\\
&\qquad=J^{(k)}_d \left(\widetilde{\DH}_{d,J^{(n-k-2)}_d}\circ J^{(2)}_d(\iota_{J^{n-k-2}_d})\circ J^1_d(l^{n-k+1}_d)\circ l^{n-k}_d\right)\circ J^{(k-1)}_d(l^{n-k+1}_d)\circ \cdots \circ J^{(1)}_d(l^{n-1}_d)\circ l^n_d.
\end{split}
\end{equation}
Then we see that
\begin{equation}
\widetilde{\DH}_{d,J^{(n-k-2)}_d}\circ J^{(2)}_d(\iota_{J^{n-k-2}_d})\circ J^1_d(l^{n-k+1}_d)\circ l^{n-k}_d
=(\Omega^1_d\ltimes\Omega^2_d)(\iota_{J^{n-k-2}_d})\circ\widetilde{\DH}_{d,J^{n-k-2}_d}\circ J^1_d(l^{n-k+1}_d)\circ l^{n-k}_d,
\end{equation}
which follows from the naturality of $\widetilde{\DH}_d$:
\begin{equation}
\begin{tikzcd}[column sep=50pt]
J^{(2)}_d\circ J^{n-k-2}_d\ar[r,two heads,"\widetilde{\DH}_{d,J^{n-k-2}_d}"]\ar[d,"J^{(2)}_d(\iota_{J^{n-k}_d})"']&(\Omega^1_d\ltimes \Omega^2_d)\circ J^{n-k-2}_d\ar[d,"(\Omega^1_d\ltimes\Omega^2_d)(\iota_{J^{n-k-2}_d})"]\\
J^{(n-k)}_d\ar[r,two heads,"\widetilde{\DH}_{d,J^{(n-k-2)}_d}"]&(\Omega^1_d\ltimes \Omega^2_d)\circ J^{(n-k-2)}_d
\end{tikzcd}
\end{equation}
The definition of $J^{n-k}_d$ yields $J^{(k)}_d \widetilde{\DH}_{d,J^{(n-k-2)}_d}\circ \iota_{J^n_d}=0$.

\eqref{lemma:Jchar:2}$\subseteq$\eqref{lemma:Jchar:3}: The only non-trivial cases occur for $k,h\ge 2$.
Using the inductive hypothesis, we get
\begin{equation}
\begin{split}
\bigcap_{j=0}^{n-2}\ker\left(J^{(j)}_d\widetilde{\DH}_{d,J^{(n-j-2)}_d}\right)
&=\bigcap_{j=0}^{h-2}\ker\left(J^{(j)}_d\widetilde{\DH}_{d,J^{(n-j-2)}_d}\right) \cap \bigcap_{j=n-k}^{n-2}\ker\left(J^{(j)}_d\widetilde{\DH}_{d,J^{(n-j-2)}_d}\right)\\
&=\bigcap_{j=0}^{h-2}\ker\left(J^{(j)}_d\widetilde{\DH}_{d,J^{(h-j-2)}_d}\right)\circ J^{(n-h)}_d \cap J^{(n-k)}_d\circ \bigcap_{j=0}^{k-2}\ker\left(J^{(j)}_d\widetilde{\DH}_{d,J^{(k-j-2)}_d}\right)\\
&=\left(J^h_d\circ J^{(n-h)}_d \right)\cap \left(J^{(n-k)}_d \circ J^k_d\right).
\end{split}
\end{equation}
The second equality is a consequence of the fact that limits are computed object-wise and the fact that $J^1_d$ preserves finite limits.
The last equality follows from the inductive hypothesis.

Notice that we have also proved that for $h,k\ge 2$, the sub-functors $\left(J^h_d\circ J^{(n-h)}_d \right)\cap \left(J^{(n-k)}_d \circ J^k_d\right)$ coincide.

\eqref{lemma:Jchar:3}$\subseteq$\eqref{lemma:Jchar:1}: There is nothing to prove unless $h,k\ge 2$, in which case we have observed that for all such $h,k$, the intersections of the form \eqref{lemma:Jchar:3} coincide, so it is enough to prove
\begin{equation}
J^n_d\subseteq\left(J^2_d\circ J^{(n-2)}_d \right)\cap \left(J^{1}_d \circ J^{n-1}_d\right).
\end{equation}
Consider the following diagram, where the left vertical natural transformation is obtained via the universal property of the kernel
\begin{equation}\label{dg:pbJn}
\begin{tikzcd}[column sep=40pt]
J^n_d\ar[r,hookrightarrow,"l^n_d"]\ar[d]&J^1_d\circ J^{n-1}_d\ar[r,"\widetilde{\DH}_{d,J^{n-2}_d}"]\ar[d,hookrightarrow,"J^1_d(\iota_{J^{n-1}_d})"]&(\Omega^1_d\ltimes\Omega^2_d)\circ J^{n-2}_d\ar[d,hookrightarrow,"(\Omega^1_d\ltimes\Omega^2_d)(\iota_{J^{n-2}_d})"]\\
J^2_d\circ J^{(n-2)}_d\ar[r,hookrightarrow,"l^2_{d,J^{(n-2)}_d}"']&J^{(n)}_d\ar[r,twoheadrightarrow,"\widetilde{\DH}_{d,J^{(n-2)}_d}"']&(\Omega^1_d\ltimes\Omega^2_d)\circ J^{(n-2)}_d
\end{tikzcd}
\end{equation}
Notice that the top horizontal maps form a left exact sequence by definition of $J^n_d$, and the bottom ones form the short exact sequence defining $J^2_d$ evaluated on the image of $J^{(n-2)}_d$.
Notice also that $(\Omega^1_d\ltimes\Omega^2_d)(\iota_{J^{k-2}_d})$ is a natural monomorphism, since $\Omega^1_d\ltimes \Omega^2_d$ is an exact functor (cf.\ Lemma \ref{lemma:spplacederfun}).
From these facts we can use Lemma \ref{lemma:pb} to prove that the left square in \eqref{dg:pbJn} is a pullback square, and since all maps in it are natural inclusions, we realize $J^n_d$ as the desired intersection, thus completing the proof.
\end{proof}
\begin{rmk}
Lemma \ref{lemma:Jchar} provides alternative definitions for the jet functor.
\end{rmk}

\subsubsection{Relation to $J^n_d A\otimes_A -$}
We apply the functor $J^n_d$ to the algebra $A$, obtaining the bimodule $J^n_d A$.
\begin{prop}\label{prop:tensorcomparison}
There exists a canonical natural transformation $\gamma^n_d\colon J^n_d A\otimes_A - \longrightarrow J^n_d$, which is compatible with $l^n_d$ and $\iota_{J^n_d}$, i.e.
\begin{align}\label{eq:gammacompatible}
l^n_d\circ \gamma^n_d=J^1_d(\gamma^{n-1}_d)\circ (l^n_{d,A}\otimes_A \id),
&\hfill&
\iota_{J^n_d}\circ \gamma^n_d=\iota_{J^{n}_d A}\otimes_A \id.
\end{align}

If we restrict the domain category of the functors involved to the subcategory $\AFlat$, then $\gamma^n_d$ is a natural isomorphism.
\end{prop}
\begin{proof}
The proof is analogous to that of Proposition \ref{prop:Stensorcomparison}, via the following diagram instead of \eqref{diag:Stensorcomparison}.
\begin{equation}\label{diag:tensorcomparison}
\begin{tikzcd}[column sep=45pt]
J^n_d A\otimes_A -\ar[r,"l^n_{d,A}\otimes_A \id"]\ar[d,dashed,"\gamma^n_d"']&J^1_d\circ J^{n-1}_d A\otimes_A -\ar[r,"J^1_d(l^{n-1}_{d,A})\otimes_A \id"]\ar[d,"J^1_d(\gamma^{n-1}_d)"]&J^{(2)}_d\circ J^{n-2}_d A\otimes_A -\ar[r,twoheadrightarrow,"\widetilde{\DH}_{d,J^{n-2}_d A}\otimes_A \id"]\ar[d,"J^{(2)}_d(\gamma^{n-2}_d)"]&\Omega^1_d\ltimes \Omega^2_d \circ J^{n-2}_d A\otimes_A -\ar[d,"(\Omega^1_d\ltimes \Omega^2_d)(\gamma^{n-2}_d)"']\\
J^n_d \ar[r,hookrightarrow,"l^n_{d}"']&J^1_d\circ J^{n-1}_d \ar[r,"J^1_d(l^{n-1}_d)"']&J^{(2)}_d\circ J^{n-2}_d \ar[r,twoheadrightarrow,"\widetilde{\DH}_{d,J^{n-2}_d }"']&\Omega^1_d\ltimes \Omega^2_d \circ J^{n-2}_d 
\end{tikzcd}
\end{equation}
\end{proof}

\subsubsection{Holonomic jet sequence}
The aim of this section is to construct the holonomic $n$-jet sequence.

Extending the cases $n=0,1,2$, we give the following definition.
\begin{defi}\label{def:piiotan}
We define the \emph{(holonomic) $n$-jet projection} as the natural transformation
\begin{equation}
\pi^{n,n-1}_d\colonequals \pi^{1,0}_{d,J^{n-1}_d}\circ l^n_d\colon J^n_d\longrightarrow J^{n-1}_d.
\end{equation}
More generally, by composing them, we get, for all $0\le m\le n$,
\begin{equation}
\pi^{n,m}_d\colonequals \pi^{n,n-1}_{d}\circ \pi^{n-1,n-2}_{d}\circ \cdots\circ \pi^{m+1,m}_{d}
\colon J^n_d\longrightarrow J^{m}_d.
\end{equation}

The natural map $\iota^n_d$ is defined by induction, for $n\ge 2$ as the unique (cf.\ Lemma \ref{lemma:uniqiota}) morphism that commutes in the following diagram
\begin{equation}\label{diag:defiiotand}
\begin{tikzcd}
S^n_d\ar[r,hookrightarrow,"\iota^n_\wedge"]\ar[d,dashed,"\iota^n_d"]&\Omega^1_d\circ S^{n-1}_d\ar[r,"\Omega^1_d(\iota^{n-1}_d)"]&[40pt]\Omega^1_d\circ J^{n-1}_d\ar[d,hookrightarrow,"\iota^1_{d,J^{n-1}_d}"]\\
J^n_d\ar[rr,hookrightarrow,"l^n_d"]&&J^1_d\circ J^{n-1}_d
\end{tikzcd}
\end{equation}
\end{defi}
Observe that $\iota^1_d$ and $\iota^2_d$ can also be seen as arising in this way from $\iota^0_d=\id_A$.
\begin{rmk}\label{rmk:gammacomppi}
The natural transformations $\gamma^n_d$ of Proposition \ref{prop:tensorcomparison} are also compatible with the jet projections, as shown in the following commutative diagram
\begin{equation}
\begin{tikzcd}[column sep=60pt, row sep=25]
J^n_d A\otimes_A -\ar[rr,bend left=10pt,"\pi^{n,n-1}_{d,A}\otimes_A \id"]\ar[r,"l^n_{d,A}\otimes_A \id"']\ar[d,"\gamma^n_d"]&J^1_d\circ J^{n-1}_d A\otimes_A - \ar[r,twoheadrightarrow,"\pi^{1,0}_{d,J^{n-1}_d A}\otimes_A \id"']\ar[d,"J^1_d(\gamma^{n-1}_d)"]&J^{n-1}_d A\otimes_A - \ar[d,"\gamma^{n-1}_d"]\\
J^n_d \ar[rr,bend right=10pt,"\pi^{n,n-1}_d"']\ar[r,hookrightarrow,"l^n_d"]&J^1_d\circ J^{n-1}_d\ar[r,twoheadrightarrow,"\pi^{1,0}_{d,J^{n-1}_d}"]&J^{n-1}_d
\end{tikzcd}
\end{equation}
\end{rmk}

\begin{lemma}\label{lemma:uniqiota}
The morphism $\iota^n_d$ of Definition \ref{def:piiotan} exists and is unique.
Moreover, if $\Omega^1_d$ is flat in $\ModA$, then $\iota^n_d$ is a monomorphism.
\end{lemma}
\begin{proof}
We prove it by induction on $n$.
Uniqueness for $\iota^2_d$ follows from the definition.

For $n>2$, we have the following commutative diagram
\begin{equation}\label{diag:telescopicstair}
\begin{tikzcd}[column sep=50pt]
S^n_d\ar[ddd,dashed,"\iota^n_d"]\ar[r,hookrightarrow,"\iota^n_{\wedge}"]&\Omega^1_d\circ S^{n-1}_d\ar[dd,"\Omega^1_d(\iota^{n-1}_d)"]\ar[r,"\Omega^1_d(\iota^{n-1}_\wedge)"]&\Omega^1_d\circ\Omega^1_d\circ S^{n-2}_d\ar[d,"T^2_d(\iota^{n-2}_d)"']\ar[r,twoheadrightarrow,,"\wedge_{S^{n-2}_d}"]&\Omega^2_d\circ S^{n-2}_d\ar[d,"\Omega^2_d(\iota^{n-2}_d)"]\\
&&\Omega^1_d\circ\Omega^1_d\circ J^{n-2}_d\ar[d,"\Omega^1_d\left(\iota^1_{d,J^{n-2}_d}\right)"']\ar[r,twoheadrightarrow,"\wedge_{J^{n-2}_d}"]&\Omega^2_d\circ J^{n-2}_d\ar[d,hookrightarrow,"\inc^2_{J^{n-2}_d}"]\\
&\Omega^1_d\circ J^{n-1}_d\ar[d,hookrightarrow,"\iota^1_{d,J^{n-1}_d}"]\ar[r,"\Omega^1_d(l^{n-1}_d)"]&\Omega^1_d\circ J^1_d\circ J^{n-2}_d\ar[d,hookrightarrow,"\iota^1_{d,\Omega^1_d\circ J^{n-2}_d}"']\ar[r,twoheadrightarrow,"W_{d,J^{n-2}_d}"]&\Omega^1_d\ltimes \Omega^2_d\circ J^{n-2}_d\ar[d,equals]\\
J^n_d\ar[r,hookrightarrow,"l^{n}_d"]&J^1_d\circ J^{n-1}_d\ar[r,"J^1_d(l^{n-1}_d)"]&J^{(2)}_d\circ J^{n-2}_d\ar[r,twoheadrightarrow,"\widetilde{\DH}_{d,J^{n-2}_d}"]&\Omega^1_d\ltimes \Omega^2_d\circ J^{n-2}_d
\end{tikzcd}
\end{equation}
The composition on the top row vanishes by definition of $S^n_d$, so $\iota^n_d$ exists, and it is unique by the kernel universal property of the bottom row.
If $\Omega^1_d$ is flat in $\ModA$, then $\Omega^1_d(\iota^{n-1}_d)$ is a monomorphism, and thus so is $\iota^n_d$.
\end{proof}

\begin{prop}\label{prop:hjc}
The following is a complex in the category of functors of type $\AMod\to \AMod$
\begin{equation}\label{es:hjn}
\begin{tikzcd}[column sep=40pt]
0\ar[r]&S^{n}_d\ar[r,"\iota^{n}_d"]&J^{n}_d\ar[r,"\pi^{n,n-1}_d"]&J^{n-1}_d\ar[r]&0.
\end{tikzcd}
\end{equation}
Moreover, the following is a map of complexes
\begin{equation}\label{es:hjninnhj}
\begin{tikzcd}[column sep=40pt]
0\ar[r]&S^{n}_d\ar[d,"\Omega^1_d\left(\iota_{J^{n-1}_d}\circ\iota^{n-1}_d\right) \circ\iota^n_{\wedge}"']\ar[r,"\iota^{n}_d"]&J^{n}_d\ar[d,"\iota_{J^n_d}"]\ar[r,"\pi^{n,n-1}_d"]&J^{n-1}_d\ar[r]\ar[d,"\iota_{J^{n-1}_d}"]&0\\
0\ar[r]&\Omega^1_d \circ J^{(n-1)}_d\ar[r,hookrightarrow,"\iota^{(n)}_d"]&J^{(n)}_d\ar[r,twoheadrightarrow,"\pi^{(n,n-1)}_d"]& J^{(n-1)}_d\ar[r]&0
\end{tikzcd}
\end{equation}
which becomes an inclusion of complexes if $\Omega^1_d$ is flat in $\ModA$.
\end{prop}
\begin{proof}
The fact that \eqref{es:hjn} is a complex follows from the definition of $\pi^{n,n-1}_d$ and the commutativity of \eqref{diag:defiiotand}, as we have
\begin{equation}
\pi^{n,n-1}_d\circ \iota^n_d
=\pi^{1,0}_{d,J^{n-1}_d}\circ l^{n}_d\circ \iota^n_{d,J^{n-1}_d}
=\pi^{1,0}_{d,J^{n-1}_d}\circ \iota^1_{d,J^{n-1}_d}\circ \Omega^1_d(\iota^{n-1}_d)\circ \iota^n_{\wedge}
=0.
\end{equation}

Now we prove the commutativity of \eqref{es:hjninnhj}.
The right square commutes by definition of $\pi^{n,n-1}_d$.
The left square commutes by the commutativity of \eqref{diag:defiiotand} and the naturality of $\iota^1_d$, as showed in the following diagram.
\begin{equation}
\begin{tikzcd}[column sep=50pt]
S^n_d\ar[d,"\iota^n_d"']\ar[r,"\Omega^1_d(\iota^{n-1}_d)\circ \iota^n_{\wedge}"]&\Omega^1_d\circ J^{n-1}_d\ar[d,hookrightarrow,"\iota^1_{d,J^{n-1}_d}"]\ar[r,"\Omega^1_d(\iota_{J^{n-1}_d})"]&\Omega^1_d\circ J^{(n-1)}_d\ar[d,hookrightarrow,"\iota^1_{d,J^{(n-1)}_d}"]\\
J^n_d\ar[rr,bend right=10pt,"\iota_{J^n_d}"']\ar[r,"l^n_d"]&J^1_d\circ J^{n-1}_d\ar[r,"J^1_d(\iota_{J^{n-1}_d})"]& J^{(n)}_d
\end{tikzcd}
\end{equation}

Finally, \eqref{es:hjninnhj} becomes an inclusion of complexes if $\Omega^1_d$ is flat in $\AMod$ because $\iota_{J^n_d}$ becomes injective (cf.\ Remark \ref{rmk:iotaholinj}) and $\Omega^1_d$ maps injective maps in injective maps.
\end{proof}
\begin{defi}
We call \eqref{es:hjn} the \emph{(holonomic) $n$-jet sequence}.
\end{defi}

We will now extend the notion of jet prolongation for higher jet functors.
\begin{lemma}\label{lemma:holprol}
There exists a unique family of natural transformations $j^n_d\colon \id_{\AMod}\to J^n_d$ in the category of functors of type $\AMod\to \Mod$ such that $j^0_d=\id$, $j^1_d$ is \eqref{eq:1jp}, and for $n\ge 1$ we have
\begin{equation}\label{eq:factjl}
j^1_{d,J^{n-1}_d}\circ j^{n-1}_d=l^n_d\circ j^n_d.
\end{equation}
\end{lemma}
\begin{proof}
We prove the statement by induction.
Suppose such a $j^{n-1}_d$ exists, then we show that there exists a unique $j^n_d$ satisfying \eqref{eq:factjl}.
Consider the following diagram
\begin{equation}
\begin{tikzcd}
&&[40pt]J^{n-2}_d\ar[d,"j^1_{d,J^{n-2}_d}"]\ar[r]&0\ar[dd]\\
\id_{\AMod}\ar[d,dashed,"j^n_d"']\ar[rru,"j^{n-2}_d"]\ar[r,"j^{n-1}_d"']&J^{n-1}_d\ar[d,"j^1_{d,J^{n-1}_d}"]\ar[r,hookrightarrow,"l^{n-1}_d"']&J^1_d\circ J^{n-2}_d\ar[d,"j^1_{d,J^1_d\circ J^{n-2}_d}"]\\
J^n_d\ar[r,"l^{n}_d"']&J^1_d\circ J^{n-1}_d\ar[r,hookrightarrow,"J^1_d(l^{n-1}_d)"']&J^{(2)}_d\circ J^{n-2}_d\ar[r,twoheadrightarrow,"\widetilde{\DH}_{d,J^{n-2}_d}"']&(\Omega^1_d\ltimes \Omega^2_d)\circ J^{n-2}_d
\end{tikzcd}
\end{equation}
The central square commutes by naturality of $j^1_d$ with respect to $l^{n-1}_d$, the top triangle commutes by inductive hypothesis, and the right square commutes because $\widetilde{\DH}_d\circ j^{(2)}_d=0$.
It follows by the kernel universal property that there exists a unique map $j^n_d$ satisfying \eqref{eq:factjl}.
\end{proof}
\begin{defi}
We call \emph{(holonomic) $n$-jet prolongation} the unique natural transformation $j^n_d\colon \id_{\AMod}\hookrightarrow J^n_d$ provided by Lemma \ref{lemma:holprol}.
\end{defi}
This definition is consistent with the definition of $j^2_d$ given in Definition \ref{def:2-jet}, as the latter satisfies \eqref{eq:factjl}.
\begin{rmk}\label{rmk:holprolpi}
The $n$-jet prolongation $j^n_d$ is a section of $\pi^{n,0}_d$.
More generally, for $0\le m\le n$, we get
\begin{equation}
\pi^{n,m}_d\circ j^n_d
=j^m_d.
\end{equation}
This can be proven by induction on $n-m\ge 0$.
If $n=m$ there is nothing to prove, and if $n>m$, then by definition of jet projections and prolongations, and by inductive hypothesis, we have
\begin{equation}
\pi^{n,m}_d\circ j^n_d
=\pi^{n-1,m}_d\circ \pi^{n,n-1}_d\circ j^n_d
=\pi^{n-1,m}_d\circ \pi^{1,0}_{d,J^{n-1}_d}\circ l^n_d\circ j^n_d
=\pi^{n-1,m}_d\circ \pi^{1,0}_{d,J^{n-1}_d}\circ j^1_{d,J^{n-1}_d}\circ j^{n-1}_d
=\pi^{n-1,m}_d\circ j^{n-1}_d
=j^m_d.
\end{equation}
\end{rmk}

The following lemma relates the holonomic and nonholonomic $n$-jet prolongations.
\begin{lemma}\label{lemma:holjfact}
The prolongation $j^{(n)}_d\colon \id_{\AMod}\to J^{(n)}_d$ factors through $\iota_{J^n_d}\colon J^n_d\to J^{(n)}_d$ via $j^n_d\colon \id_{\AMod}\to J^{n}_d$ in the category of functors of type $\AMod\to \Mod$.
That is
\begin{equation}
j^{(n)}_d=\iota_{J^n_d}\circ j^n_d.
\end{equation}
\end{lemma}
\begin{proof}
We proceed by induction on $n$.
For $n=0$, the result is straightforward.
For $n>0$, and assume that we have $j^{(n-1)}_d=\iota_{J^{n-1}_d}\circ j^{n-1}_d$.
By \eqref{eq:spinjn}, \eqref{eq:factjl}, and the naturality of $j^1_d$ we obtain
\begin{equation}
\iota_{J^n_d}\circ j^n_d
=J^1_d(\iota_{J^{n-1}_d})\circ l^n_d\circ j^n_d
=J^1_d(\iota_{J^{n-1}_d})\circ j^1_{d,J^{n-1}_d}\circ j^{n-1}_d
=j^1_{d,J^{(n-1)}_d}\circ\iota_{J^{n-1}_d}\circ j^{n-1}_d
=j^1_{d,J^{(n-1)}_d}\circ j^{(n-1)}_d
=j^{(n)}_d
\end{equation}
\end{proof}

We can now refine the point \eqref{lemma:Jchar:3} of Lemma \ref{lemma:Jchar} with the following result.
\begin{lemma}\label{lemma:smM}
For all $m,n\ge 0$, there exists a natural transformation $l^{m,n}_d\colon J^{m+n}_d\to J^m_d\circ J^n_d$ such that $l^{0,n}_d=\id_{J^n_d}$ and the following square commutes
\begin{equation}\label{diag:inddeflnm}
\begin{tikzcd}[column sep=40pt]
J^{m+n}_d\ar[r,hookrightarrow,"l^{m+n}_d"]\ar[d,"l^{m,n}_d"']&J^1_d\circ J^{m+n-1}_d\ar[d,"J^1_d (l^{m-1,n}_d)"]\\
J^m_d\circ J^{n}_d\ar[r,hookrightarrow,"l^m_{d,J^n_d}"]&J^1_d\circ J^{m-1}_d\circ J^{n}_d
\end{tikzcd}
\end{equation}
Moreover,
\begin{equation}\label{eq:commjlmn}
l^{m,n}_d\circ j^{m+n}_d=j^{m}_{d,J^n_d}\circ j^n_d.
\end{equation}

In particular, if $\Omega^1_d$ and $\Omega^2_d$ are flat in $\ModA$, then $l^{m,n}_d$ is a mono and \eqref{diag:inddeflnm} is a pullback square, implying
\begin{equation}
J^{m+n}_d=J^m_d\circ J^n_d\cap J^1_d\circ J^{m+n-1}_d\subseteq J^1_d\circ J^{n-1}_d\circ J^n_d\subseteq J^{(m+n)}_d.
\end{equation}
\end{lemma}
\begin{proof}
For $m=0$, equation \eqref{eq:commjlmn} trivially holds.
For $m\ge 1$, we proceed by induction.
For $m=1$, we see that we can only choose $l^{1,n}_d=l^{n+1}_d$ in order for \eqref{diag:inddeflnm} to commute.
Notice that it is a mono.
Furthermore, \eqref{eq:commjlmn} becomes \eqref{eq:factjl}.

For $m> 1$, consider the following diagram, obtained from the two rightmost squares by taking the kernel of the top and bottom row compositions
\begin{equation}\label{diag:smMpb}
\begin{tikzcd}[column sep=40pt]
J^{m+n}_d\ar[r,hookrightarrow,"l^{m+n}_d"]\ar[d,dashed,"l^{m,n}_d"']&J^1_d\circ J^{m+n-1}_d\ar[d,"J^1_d (l^{m-1,n}_d)"]\ar[r,"J^1_d(l^{m+n-1}_d)"]&J^{(2)}_d\circ J^{m+n-2}_d\ar[d,"J^{(2)}_d (l^{m-2,n}_d)"]\ar[r,twoheadrightarrow,"\widetilde{\DH}_{d,J^{n+m-2}_d}"]&(\Omega^1_d\ltimes \Omega^2_d)\circ J^{m+n-2}_d\ar[d,"\Omega^1_d\ltimes \Omega^2_d (l^{m-2,n}_d)"]\\
J^m_d\circ J^{n}_d\ar[r,hookrightarrow,"l^m_{d,J^n_d}"]&J^1_d\circ J^{m-1}_d\circ J^{n}\ar[r,"J^1_d(l^{m-1}_{d,J^{n}_d})"]&J^{(2)}_d\circ J^{m-2}_d\circ J^n_d \ar[r,twoheadrightarrow,"\widetilde{\DH}_{d,J^{m-2}_d\circ J^{n}_d}"]&(\Omega^1_d\ltimes \Omega^2_d)\circ J^{m-2}_d\circ J^{n}_d
\end{tikzcd}
\end{equation}
The central square commutes by inductive hypothesis and the right square commutes by naturality of $\widetilde{\DH}_d$.
We thus obtain, from the universal property of the kernel, the unique dashed function $l^{m,n}_d$ which makes the left square commute.

In order to prove \eqref{eq:commjlmn}, consider the following cube diagram
\begin{equation}\label{diag:cube}
\begin{tikzcd}[column sep=40pt]
&J^{m+n-1}_d\ar[dd,hookrightarrow,near start,"l^{m-1,n}_d"]\ar[rr,hookrightarrow,"j^1_{d,J^{m+n-1}_d}"]&&J^1_d\circ J^{m+n-1}_d\ar[dd,"J^1_d(l^{m-1,n}_d)"]\\
\id_{\AMod}\ar[dd,hookrightarrow,"j^n_d"']\ar[rr,crossing over,hookrightarrow,near end,"j^{m+n}_d"']\ar[ru,hookrightarrow,"j^{m+n-1}_d"]&&J^{m+n}_d\ar[ru,hookrightarrow,"l^{m+m}_d"]\\
&J^{m-1}_d\circ J^n_d\ar[rr,hookrightarrow, near end,"j^1_{d,J^{m-1}_d\circ J^n_d}"]&&J^1_d\circ J^{m-1}_d\circ J^n_d\\
J^n_d\ar[rr,hookrightarrow,"j^{m}_{d,J^n_d}"']\ar[ru,hookrightarrow,"j^{m-1}_{d,J^n_d}"]&&J^{m}_d\circ J^n_d\ar[uu,<-,crossing over,near start,"l^{m,n}_d"']\ar[ru,hookrightarrow,"l^{m}_{d,J^n_d}"']
\end{tikzcd}
\end{equation}
Here, the left square commutes by inductive hypothesis, the right square by \eqref{eq:commjlmn}, and the back square by naturality of $j^1_d$.
The top and bottom squares commute because of Lemma \ref{lemma:holprol}.
Since $l^m_{d,J^n_d}$ is a mono, it follows that the front square commutes, and hence \eqref{eq:commjlmn} is satisfied.

Now, if $\Omega^1_d$ and $\Omega^2_d$ are flat in $\ModA$, then, by inductive hypothesis, $l^{m-1,n}_d$ and $l^{m-1,n}_d$ are monos, whence $J^1_d(l^{m-1,n}_d)$ and $\Omega^1_d\ltimes \Omega^2_d(l^{m-1,n}_d)$ are also monos, by Corollary \ref{cor:Jex} and Lemma \ref{lemma:spplacederfun}.
The final statement of the lemma can thus be deduced by applying Lemma \ref{lemma:pb} to \eqref{diag:smMpb}.
\end{proof}

\subsection{Relation to the semiholonomic jet functor}
\begin{prop}\label{prop:holinsemi}
The natural transformation $\iota_{J^n_d}\colon J^n_d\to J^{(n)}_d$ factors through $\iota_{J^{[n]}_d}\colon J^{[n]}_d\hookrightarrow J^{(n)}_d$ via the natural transformation $h^n_d\colon J^{n}_d\to J^{[n]}_d$.

In particular, if $\Omega^1_d$ is flat in $\ModA$, then we have the functor inclusions $J^{n}_d\subseteq J^{[n]}_d\subseteq J^{(n)}_d$.
\end{prop}
\begin{proof}
For $n=0, 1$, all jet functors coincide.
For higher $n$, as observed in Proposition \ref{prop:defeth}, the operator $\widetilde{\DH}^I_d$ is the projection of $\widetilde{\DH}_d$ to $\Omega^1_d$.
This implies that $\ker\left(J^{(k)}_d \widetilde{\DH}_{d,J^{(n-k-2)}_d} \right)$ factors through $\ker\left(J^{(k)}_d \widetilde{\DH}^{I}_{d,J^{(n-k-2)}_d} \right)$, and so $\bigcap_{k=0}^{n-2}\ker\left(J^{(k)}_d \widetilde{\DH}_{d,J^{(n-k-2)}_d} \right)$ factors through $\bigcap_{k=0}^{n-2}\ker\left(J^{(k)}_d \widetilde{\DH}^{I}_{d,J^{(n-k-2)}_d} \right)$.
By Theorem \ref{theo:sjchar}, we get \eqref{theo:sjchar:3}.
Furthermore, in proving Lemma \ref{lemma:Jchar}, we do not use $\Omega^1_d$ to be flat in $\AMod$ to prove \eqref{lemma:Jchar:1}$\Rightarrow$\eqref{lemma:Jchar:2}.
Thus, we have a factorization of the map $\iota_{J^n_d}\colon J^n_d\to J^{(n)}_d$ as
\begin{align}
J^{n}_d\longrightarrow
 \bigcap_{k=0}^{n-2}\ker\left(J^{(k)}_d \widetilde{\DH}_{d,J^{(n-k-2)}_d} \right)
\longhookrightarrow
\bigcap_{k=0}^{n-2}\ker\left(J^{(k)}_d \widetilde{\DH}^{I}_{d,J^{(n-k-2)}_d} \right)
=
 J^{[n]}_d
\longhookrightarrow
 J^{(n)}_d
\end{align}

Finally, the fact that vertical maps become inclusion follows from the fact that $\iota_{J^n_d}$ is injective when $\Omega^1_d$ is flat in $\ModA$.
\end{proof}

\begin{prop}\label{prop:semiholcpxfact}
The functor complex morphism \eqref{es:hjninnhj} factors through \eqref{es:shjci} as follows
\begin{equation}\label{es:hjninshj}
\begin{tikzcd}[column sep=40pt]
0\ar[r]&S^{n}_d\ar[d,"\iota_{S^n_d}"']\ar[r,"\iota^{n}_d"]&J^{n}_d\ar[d,"h^n_d"]\ar[r,"\pi^{n,n-1}_d"]&J^{n-1}_d\ar[r]\ar[d,"h^{n-1}_d"]&0\\
0\ar[r]&T^{n}_d\ar[d,"\Omega^1_d(\iota_{T^{n-1}_d})"']\ar[r,"\iota^{[n]}_d"]&J^{[n]}_d\ar[d,hookrightarrow,"\iota_{J^{[n]}_d}"]\ar[r,twoheadrightarrow,"\pi^{[n,n-1]}_d"]&J^{[n-1]}_d\ar[r]\ar[d,hookrightarrow, "\iota_{J^{[n-1]}_d}"]&0\\
0\ar[r]&\Omega^1_d \circ J^{(n-1)}_d\ar[r,hookrightarrow,"\iota^{(n)}_d"]&J^{(n)}_d\ar[r,twoheadrightarrow,"\pi^{(n,n-1)}_d"]& J^{(n-1)}_d\ar[r]&0
\end{tikzcd}
\end{equation}
Moreover, if $\Omega^1_d$ is flat in $\ModA$, all the vertical arrows become inclusions and respect the inclusions $l^n_d$ and $l^{[n]}_d$, i.e.\ the following diagram commutes
\begin{equation}
\begin{tikzcd}
J^n_d\ar[r,hookrightarrow,"l^n_d"]\ar[d,"h^n_d"']&J^1_d\circ J^{n-1}_d\ar[d,"J^1_d(h^{n-1}_d)"]\\
J^{[n]}_d\ar[r,hookrightarrow,"l^{[n]}_d"]&J^1_d\circ J^{[n-1]}_d
\end{tikzcd}
\end{equation}
\end{prop}
\begin{proof}
From Proposition \ref{prop:holinsemi} we find that $\iota_{J^{n}_d}$ and $\iota_{J^{n-1}_d}$ factor through the corresponding semiholonomic functors.
Let us write $\iota_{J^{n}_d}=\iota_{J^{[n]}_d}\circ h^n_d$ for all $n$.
By Proposition \ref{prop:hjc}, we know that $\pi^{n,n-1}_d$ is compatible with $\pi^{(n.n-1)}_d$, so we have
\begin{equation}
\iota_{J^{[n-1]}_d}\circ h^{n-1}_d\circ \pi^{n,n-1}_d
=\iota_{J^{n-1}_d}\circ \pi^{n,n-1}_d
=\pi^{(n,n-1)}_d\circ \iota_{J^{[n]}_d}\circ h^n_d
=\iota_{J^{[n-1]}_d}\circ \pi^{[n,n-1]}_d\circ h^n_d
\end{equation}
Since $\iota_{J^{[n-1]}_d}$ is a mono, we obtain the commutativity of the top right square of \eqref{es:hjninshj}.

We will now prove that the top left square commutes, and since $\iota_{J^{[n]}_d}$ is a mono, it is enough to prove that $\iota_{J^{n}_d}\circ \iota^n_d=\iota_{J^{[n]}_d}\circ h^n_d\circ \iota^n_d$ is equal to $\iota_{T^n_d}\circ \iota_{S^n_d}=\iota_{J^{[n]}_d}\circ \iota^{[n]}_d\circ \iota_{S^n_d}$.
Thus, we prove $\iota_{J^{n}_d}\circ \iota^n_d=\iota_{T^n_d}\circ \iota_{S^n_d}$ by induction on $n\ge 1$.
For $n=1$ the statement is true, for $S^1_d=T^1_d=\Omega^1_d$ and all maps involved are either identities or the inclusion $\iota^1_d$.
Now let $n>1$ and assume the statement true for $n-1$, then we have the following diagram
\begin{equation}\label{diag:Ark}
\begin{tikzcd}[column sep=60pt]
S^n_d\ar[rr,"\iota^n_d"]\ar[d,hookrightarrow,"\iota^n_{\wedge}"]\ar[dd,bend right=60pt,"\iota_{S^n_d}"']&&J^n_d\ar[d,hookrightarrow,"l^n_d"']\ar[dd,bend left=60pt,"\iota_{J^n_d}"]\\
\Omega^1_d\circ S^{n-1}_d\ar[d,"\Omega^1_d(\iota_{S^{n-1}_d})"]\ar[r,"\Omega^1_d(\iota^{n-1}_d)"]&\Omega^1_d\circ J^{n-1}_d\ar[r,hookrightarrow,"\iota^1_{d,J^{n-1}_d}"]\ar[d,"\Omega^1_d(\iota_{J^{n-1}_d})"]&J^1_d\circ J^{n-1}_d\ar[d,"J^1_d(\iota_{J^{n-1}_d})"']\\
T^{n}_d\ar[rr,bend right=10pt,"\iota_{T^n_d}"']\ar[r,near end,"\Omega^1_d(\iota_{T^{n-1}_d})"]&\Omega^1_d\circ J^{(n-1)}_d\ar[r,"\iota^{(n)}_{d}"]&J^1_d\circ J^{(n-1)}_d
\end{tikzcd}
\end{equation}
Commutativity of \eqref{diag:Ark} follows from \eqref{diag:defiiotand}, from the naturality of $\iota^1_d$, and from the inductive hypothesis.

We are left to prove that the left vertical composition in \eqref{es:hjninshj} is the left vertical map in \eqref{es:hjninnhj}, that is $\Omega^1_d\left(\iota_{J^{n-1}_d}\circ\iota^{n-1}_d\right) \circ\iota^n_{\wedge}=\Omega^1_d(\iota_{T^{n-1}_d})\circ \iota_{S^n_d}$.
From \eqref{diag:Ark} we have that 
\begin{equation}
\iota^{(n)}_d\circ\Omega^1_d\left(\iota_{J^{n-1}_d}\circ\iota^{n-1}_d\right) \circ\iota^n_{\wedge}
=\iota_{J^{n}_d}\circ \iota^n_d
=\iota^{(n)}_d\circ\Omega^1_d(\iota_{T^{n-1}_d})\circ\iota_{S^n_d}.
\end{equation}
The statement thus follows since $\iota^{(n)}_d$ is a mono, thereby completing the proof of the factorization.

Now let $\Omega^1_d$ be flat in $\ModA$.
The injectivity of the vertical maps follows from the last statements of Proposition \ref{prop:hjc} and Proposition \ref{prop:complexinclusionseminon}.
Consider the following diagram
\begin{equation}\label{diag:leatherback}
\begin{tikzcd}[column sep=30pt]
J^n_d\ar[r,hookrightarrow,"l^n_d"]\ar[d,"h^n_d"]\ar[dd,bend right=60pt,"\iota_{J^n_d}"']&J^1_d\circ J^{n-1}_d\ar[d,"J^1_d(h^{n-1}_d)"']\ar[dd,bend left=60pt,"J^1_d(\iota_{J^{n-1}_d})"]\\
J^{[n]}_d\ar[r,hookrightarrow,"l^{[n]}_d"]\ar[d,"\iota_{J^{[n]}_d}"]&J^1_d\circ J^{[n-1]}_d\ar[d,"J^1_d(\iota_{J^{[n-1]}_d})"']\\
J^{(n)}_d\ar[r,equals]&J^{(n)}_d
\end{tikzcd}
\end{equation}
The left triangle is the factorization, and the right triangle commutes by functoriality of $J^1_d$.
The external square commutes by \eqref{eq:spinjn}, and the bottom square because of Theorem \ref{theo:sjchar}.
The top square then commutes because $J^1_d(\iota_{J^{[n-1]}_d})$ is a monomorphism.
\end{proof}
\begin{rmk}\label{rmk:lowdegh}
In general, $h^0=\id$ and $h^1=\id_{J^1_d}$.
Since $\iota_{J^2_d}$ is a monomorphism by definition, also $h^2$ is a monomorphism.
\end{rmk}
\subsection{Sesquiholonomic and holonomic jet exact sequences}\label{s:sesquihol}
We will now introduce another type of jet functor, which we will only use as a tool to address the exactness of the jet sequence, although its classical counterpart has already been studied \cite{libermann1997}.
\begin{defi}\label{def:sesqui}
	The {\em sesquiholonomic $n$-jet functor}, denoted by $J^{\{n\}}_d$, is defined as $J^{\{0\}}_d=\id_{\AMod}$, $J^{\{1\}}_d=J^1_d$, and for $n\ge 3$, as the subfunctor of $J^1_d\circ J^{n-1}_d$ defined by the kernel of
\begin{equation}\label{eq:sesqud}
\widetilde{\DH}^I_{d,J^{n-2}_d}\circ J^1_d(l^{n-1}_d)\colon J^1_d\circ J^{n-1}_d\longrightarrow \Omega^1_d\circ J^{n-2}_d.
\end{equation}
We denote the inclusion of $J^{\{n\}}_d$ into $J^1_d\circ J^{n-1}_d$ by $l^{\{n\}}_d$, and set $l^{\{1\}}_d=\id_{J^1_d}$.
\end{defi}
\begin{lemma}
There are unique natural transformations $\pi^{\{n,n-1\}}_d$ and $\iota^{\{n\}}_d$ making the following diagram commute.
\begin{equation}\label{diag:sesqui}
\begin{tikzcd}[column sep=40pt]
0\ar[r]&\Omega^1_d\circ S^{n-1}_d\ar[d,"\Omega^1_d(\iota^{n-1}_d)"']\ar[r,"\iota^{\{n\}}_d"]&J^{\{n\}}_d\ar[d,hookrightarrow,"l^{\{n\}}_d"]\ar[r,"\pi^{\{n,n-1\}}_d"]&J^{n-1}_d\ar[d,equals]\ar[r]&0\\
0\ar[r]&\Omega^1_d\circ J^{n-1}_d \ar[r,hookrightarrow,"\iota^1_{d,J^{n-1}_d}"]&J^1_d\circ J^{n-1}_d\ar[r,twoheadrightarrow,"\pi^{1,0}_{d,J^{n-1}_d}"]&J^{n-1}_d\ar[r]&0
\end{tikzcd}
\end{equation}
In particular the horizontal maps form complexes and the vertical maps form a morphism of complexes.
\end{lemma}
\begin{proof}
The right square commutes if and only if we define $\pi^{\{n,n-1\}}_d\colonequals \pi^{1,0}_{d,J^{n-1}_d}\circ l^{\{n\}}_d$.
 
In order to find $\iota^{\{n\}}_d$, consider the map
\begin{equation}\label{eq:inclOmSn}
\iota^1_{d,J^{n-1}_d}\circ \Omega^1_d(\iota^{n-1}_d)\colon \Omega^1_d\circ S^{n-1}_d\longrightarrow J^1_d\circ J^{n-1}_d.
\end{equation}
If we compose it with $J^1_d(l^{n-1}_d)=J^1_d\circ J^{n-1}_d\to J^{(2)}_d\circ J^{n-2}_d$, we can see that, composing both $\pi^{1,0}_{d,J^1_d\circ J^{n-2}_d}$ and $J^1_d \pi^{1,0}_{d,J^{n-2}_d}$, we obtain the zero map.
This implies that the composition of \eqref{eq:inclOmSn} with \eqref{eq:sesqud} gives zero.
Therefore, by the universal property of the kernel, we obtain a unique factorization, which we call $\iota^{\{n\}}_d$.

Finally, by the commutativity of \eqref{diag:sesqui}, we get $\pi^{\{n,n-1\}}_d\circ \iota^{\{n\}}_d=\pi^{1,0}_{d,J^{n-1}_d}\circ \iota^1_{d,J^{n-1}_d}\circ \Omega^1_d(\iota^{n-1}_d)=0$.
\end{proof}
\begin{defi}\label{defi:sesquisequence}
We call the top row of \eqref{diag:sesqui} the \emph{sesquiholonomic $n$-jet sequence}, and the natural transformation $\pi^{\{n,n-1\}}_d\colonequals \pi^{1,0}_{d,J^{n-1}_d}\circ l^{\{n\}}_d\colon J^{\{n\}}_d\to J^{n-1}_d$ the \emph{sesquiholonomic jet projection}.
\end{defi}

If the exterior algebra $\Omega^\bullet_d$ is sufficiently well-behaved, we can offer the following characterization of the sesquiholonomic jet functor.
\begin{lemma}\label{lemma:charsesqui}
If $\Omega^1_d$ is flat in $\ModA$, then $J^{\{n\}}_d=J^{[n]}_d\cap (J^1_d\circ J^{n-1}_d)\subseteq J^{(n)}_d$.

Furthermore, if both $\Omega^1_d$ and $\Omega^2_d$ are flat in $\ModA$, then $J^{\{n\}}_d$ is the submodule of $J^{[n]}_d$ obtained as
\begin{equation}
\bigcap_{k=1}^{n-2}\ker\left(J^{(k)}_d \widetilde{\DH}^{II}_{d,J^{[n-k-2]}_d}\circ \iota_{J^{[n]}_d} \right).
\end{equation}
\end{lemma}
\begin{proof}
If $\Omega^1_d$ is flat in $\ModA$, we have the following diagram
\begin{equation}
\begin{tikzcd}[column sep=50pt]
J^{\{n\}}_d\ar[r,hookrightarrow,"l^{\{n\}}_d"]\ar[d,dashed]&J^1_d\circ J^{n-1}_d\ar[r,hookrightarrow,"J^1_d(l^{n-1}_d)"]\ar[d,hookrightarrow,"J^1_d(h^{n-1}_d)"]&J^{(2)}_d\circ J^{n-2}_d\ar[r,twoheadrightarrow,"\widetilde{\DH}^I_{d,J^{n-2}_d}"]\ar[d,hookrightarrow,"J^{(2)}_d(h^{n-2})"]&\Omega^1_d\circ J^{n-2}_d\ar[d,hookrightarrow,"\Omega^1_d(h^{n-2})"]\\
J^{[n]}_d\ar[r,hookrightarrow,"l^{[n]}_{d}"']&J^1_d\circ J^{[n-1]}_d\ar[r,hookrightarrow,"J^1_d(l^{[n-1]}_d)"']&J^{(2)}_d\circ J^{[n-2]}_d\ar[r,twoheadrightarrow,"\widetilde{\DH}^I_{d,J^{[n-2]}_d}"']&\Omega^1_d\circ J^{[n-2]}_d
\end{tikzcd}
\end{equation}
Here, the right square commutes by the naturality of $\widetilde{\DH}^I_d$ with respect to the factorization map $h^n_d\colon J^n_d\hookrightarrow J^{[n]}_d$ provided by Proposition \ref{prop:holinsemi}.
The central square commutes by Proposition \ref{prop:semiholcpxfact}.
By definition, in each row the leftmost arrow is the kernel inclusion of the composition of the two to its right.
By Lemma \ref{lemma:pb} we obtain that $J^{{n}}_d=J^1_d\circ J^{n-1}_d\cap J^{[n]}_d\subseteq J^1_d\circ J^{[n]}_d\subseteq J^{(n)}_d$.

If, additionally, $\Omega^2_d$ is flat in $\ModA$, we can use Lemma \ref{lemma:Jchar} to write $J^{n-1}_d=\bigcap_{k=0}^{n-3}\ker\left(J^{(k)}_d \widetilde{\DH}_{d,J^{(n-k-2)}_d} \right)$.
Since $J^1_d$ preserves limits, and limits commute with limits, we have
\begin{equation}
\begin{split}
J^{\{n\}}_d
&=J^{[n]}_d\cap J^1_d\circ J^{n-1}_d\\
&=J^{[n]}_d\cap \bigcap_{k=0}^{n-3}\ker\left(J^1_d\circ J^{(k)}_d \widetilde{\DH}_{d,J^{(n-k-3)}_d} \right)\\
&=J^{[n]}_d\cap \bigcap_{k=1}^{n-2}\ker\left(J^{(k)}_d \widetilde{\DH}_{d,J^{(n-k-2)}_d} \right)\\
&=\bigcap_{k=1}^{n-2}J^{[n]}_d\cap \ker\left(J^{(k)}_d \widetilde{\DH}_{d,J^{(n-k-2)}_d} \right)
\end{split}
\end{equation}
Via Lemma \ref{lemma:pb}, we can prove that the kernel of $J^{(k)}_d \widetilde{\DH}_{d,J^{(n-k-2)}_d}$ intersected with $J^{[n]}_d$ equals the kernel of $J^{(k)}_d \widetilde{\DH}_{d,J^{(n-k-2)}_d}\circ \iota_{J^{[n]}_d}$.
Consider $J^{(k)}_d \pj^1_{J^{(n-k-2)}_d}$, where $\pj^1$ is the natural transformation $\pj^1\colon \Omega^1_d\ltimes \Omega^2_d\to \Omega^1_d$ in \eqref{es:spsplit}.
From Theorem \ref{theo:sjchar} we have
\begin{equation}
\left(J^{(k)}_d \pj^1_{J^{(n-k-2)}_d}\circ \iota_{J^{[n]}_d}\right)
\circ
\left(J^{(k)}_d \widetilde{\DH}_{d,J^{(n-k-2)}_d}\circ \iota_{J^{[n]}_d}\right)
=J^{(k)}_d \widetilde{\DH}^{I}_{d,J^{(n-k-2)}_d}\circ \iota_{J^{[n]}_d}
=0
\end{equation}
Thus, $J^{(k)}_d \widetilde{\DH}_{d,J^{(n-k-2)}_d}\circ \iota_{J^{[n]}_d}$ factors through $J^{(k)}_d\circ \Omega^2_d\circ J^{(n-k-2)}_d$ via the corresponding inclusion, as $J^{(k)}_d \widetilde{\DH}^{II}_{d,J^{(n-k-2)}_d}\circ \iota_{J^{[n]}_d}$.
This map is also $A$-linear (cf.\ Remark \ref{rmk:semiholethII}).
It follows that
\begin{equation}
J^{\{n\}}_d
=\bigcap_{k=1}^{n-2} \left(J^{(k)}_d \widetilde{\DH}_{d,J^{(n-k-2)}_d}\circ \iota_{J^{[n]}_d}\right)
=\bigcap_{k=1}^{n-2} \left(J^{(k)}_d \widetilde{\DH}^{II}_{d,J^{(n-k-2)}_d}\circ \iota_{J^{[n]}_d}\right).
\end{equation}
\end{proof}
\begin{prop}\label{prop:hol=>sesqui+1}
Let $A$ be a $\bk$-algebra equipped with an exterior algebra $\Omega^\bullet_d$.
If $\pi^{n-1,n-2}_d$ is an epi, then $\pi^{\{n,n-1\}}_d$ is an epi.

If moreover $\Omega^1_d$ is flat in $\ModA$ and the holonomic $(n-1)$-jet sequence
\begin{equation}\label{es:hjn-1}
\begin{tikzcd}[column sep=40pt]
0\ar[r]&S^{n-1}_d\ar[r,hookrightarrow,"\iota_d^{n-1}"]&J^{n-1}_d\ar[r,twoheadrightarrow,"\pi^{n-1,n-2}_d"]&J^{n-2}_d\ar[r]&0
\end{tikzcd}
\end{equation}
is exact, then the sesquiholonomic $n$-jet sequence is also exact
\begin{equation}
\begin{tikzcd}[column sep=40pt]
0\ar[r]&\Omega^1_d\circ S^{n-1}_d\ar[r,hookrightarrow,"\iota_d^{\{n\}}"]&J^{\{n\}}_d\ar[r,twoheadrightarrow,"\pi^{\{n,n-1\}}_d"]&J^{n-1}_d\ar[r]&0.
\end{tikzcd}
\end{equation}

Moreover, the natural transformation \eqref{eq:sesqud} is an epimorphism.
That is, $\widetilde{\DH}^I_{d,J^{n-2}_d}$ remains surjective even when restricted to $J^1_d\circ J^{n}_d$.
\end{prop}
\begin{proof}
Consider the following diagram
\begin{equation}
\begin{tikzcd}[column sep=50pt]
0\ar[r]&\Omega^1_d\circ S^{n-1}_d\ar[r,hookrightarrow,"\iota_d^{\{n\}}"]\ar[d,"\Omega^1_d(\iota^{n-1}_d)"']&J^{\{n\}}_d\ar[r,"\pi^{\{n,n-1\}}_d"]\ar[d,hookrightarrow,"l^{\{n\}}_d"]&J^{n-1}_d\ar[r]\ar[d,equals]&0\\
0\ar[r]&\Omega^1_d\circ J^{n-1}_d\ar[r,hookrightarrow,"\iota^1_{d,J^{n-1}_d}"]\ar[d,twoheadrightarrow,"\Omega^1_d(\pi^{n-1,n-2}_d)"']&J^1_d\circ J^{n-1}_d\ar[r,twoheadrightarrow,"\pi^{1,0}_{d,J^{n-1}_d}"]\ar[d,"\widetilde{\DH}^I_{d,J^{n-2}_d}\circ J^1_d(l^{n-1}_d)"]&J^{n-1}_d\ar[r]\ar[d]&0\\
0\ar[r]&\Omega^1_d\circ J^{n-2}_d\ar[r,equals]&\Omega^1_d\circ J^{n-2}_d\ar[r]&0\ar[r]&0
\end{tikzcd}
\end{equation}
The bottom right square trivially commutes.
The bottom left square commutes, since it is the restriction of the bottom left square in \eqref{es:shJetcd} via the natural inclusions $h^n_d$ and $h^{n-1}_d$ (cf.\ Proposition \ref{prop:semiholcpxfact}).
The central column is left exact by definition, and the right column is trivially exact.
Since $\Omega^1_d$ is a tensor functor, and hence right exact, we have that the bottom left vertical map is an epi.
Thus, applying the snake lemma we obtain the surjectivity of $\pi^{\{n,n-1\}}_d$.

If moreover we have the holonomic $(n-1)$-jet short exact sequence and $\Omega^1_d$ flat in $\ModA$, then also the left column is a short exact sequence.
Applying the snake lemma, we obtain the sesquiholonomic $n$-jet exact sequence as the kernel sequence, and we obtain the surjectivity of $\widetilde{\DH}^I_{d,J^{n-2}_d}\circ J^1_d(l^{n-1}_d)$ from the cokernel sequence.
\end{proof}
\begin{lemma}\label{lemma:imageD}
Let $A$ be a $\bk$-algebra equipped with an exterior algebra $\Omega^{\bullet}_d$, then $\widetilde{\DH}_{d,J^{n-2}_d}$ restricted to $J^{\{n\}}_d$, i.e.\ 
\begin{equation}\label{eq:ethonsesqui}
\widetilde{\DH}_{d,J^{n-2}_d}\circ J^1_d(l^{n-1}_d)\circ l^{\{n\}}_d\colon J^{\{n\}}_d \longrightarrow (\Omega^1_d\ltimes \Omega^2_d)\circ J^{n-2}_d,
\end{equation}
factors through the inclusion of $\Omega^2_d\circ J^{n-2}_d$ as $\widetilde{\DH}^{II}_{d,J^{n-2}_d}\circ J^1_d(l^{n-1}_d)\circ l^{\{n\}}_d$.

Furthermore, if $n=2$ or if $n>2$ and if we also assume that $\Omega^2_d$ is flat in $\ModA$ and that (for $n>3$) the holonomic $(n-2)$-jet sequence is left exact, then \eqref{eq:ethonsesqui} has image contained in $\Omega^2_d\circ S^{n-2}_d$.
\end{lemma}
\begin{proof}
By definition, projecting \eqref{eq:ethonsesqui} to the component $\Omega^1_d\circ J^{n-2}_d$ yields
\begin{equation}
\widetilde{\DH}^{I}_{d,J^{n-2}_d}\circ J^1_d(l^{n-1}_d)\circ l^{\{n\}}_d,
\end{equation}
which vanishes by definition of sesquiholonomic jet.
From \eqref{es:spsplit}, it follows that \eqref{eq:ethonsesqui} has image contained in $\Omega^2_d\circ J^{n-2}_d$, and it coincides with $\widetilde{\DH}^{II}_{d,J^{n-2}_d}\circ J^1_d(l^{n-1}_d)\circ l^{\{n\}}_d$ (cf.\ Remark \ref{rmk:semiholethII}).

If $n=2$, we have that $J^0_d=S^0_d=\id_{\AMod}$, thus the statement is trivially true.
For $n=3$, the $(n-2)$-jet sequence is exact, so the same conditions that apply for the case $n>3$ hold and thus we can prove this case together with the remaining ones.
If $n>2$, consider the following diagram
\begin{equation}\label{eq:sesquieth}
\begin{tikzcd}[column sep=50pt]
J^{\{n\}}_d\ar[r,hookrightarrow,"l^{\{n\}}_d"]\ar[d,"\pi^{\{n,n-1\}}_d"']&J^1_d\circ J^{n-1}_d\ar[r,"J^1_d(l^{n-1}_d)"]\ar[d,"J^1_d(\pi^{n-1,n-2}_d)"]\ar[dl,twoheadrightarrow,"\pi^{1,0}_{d,J^{n-1}_d}"']&J^{(2)}_d\circ J^{n-2}_d\ar[r,two heads,"\widetilde{\DH}^{II}_{d,J^{n-2}_d}"]\ar[d,"J^{(2)}_d(\pi^{n-2,n-3}_d)"]&\Omega^2_d\circ J^{n-2}_d\ar[d,"\Omega^2_d(\pi^{n-2,n-3}_d)"]\\
J^{n-1}_d\ar[r,hookrightarrow,"l^{n-1}_d"']&J^1_d\circ J^{n-2}_d\ar[r,"J^1_d(l^{n-2}_d)"']&J^{(2)}_d\circ J^{n-3}_d\ar[r,two heads,"\widetilde{\DH}^{II}_{d,J^{n-3}_d}"']&\Omega^2_d\circ J^{n-3}_d
\end{tikzcd}
\end{equation}
The top left triangle commutes by definition of $\pi^{\{n,n-1\}}_d$.
The bottom left triangle does not commute, but if we consider the difference of the two maps of which it is made, we obtain $\widetilde{\DH}^{I}_{d,J^{n-2}_d}\circ J^1_d(l^{n-2}_d)$, which vanishes on $J^{\{n\}}_d$.
Thus, the left square commutes.
Similarly, the central square commutes because the difference of the two compositions forming it coincides with $J^1_d(\widetilde{\DH}^{I}_{d,J^{n-3}_d})\circ J^1_d(l^{n-1}_d)=0$.
The right square commutes by the naturality of $\widetilde{\DH}^{II}_d$.

The composition of the bottom maps in \eqref{eq:sesquieth} is zero, by definition of $J^{n-1}_d$.
The top horizontal composition equals the map \eqref{eq:ethonsesqui}, so the commutativity of \eqref{eq:sesquieth} implies that \eqref{eq:ethonsesqui} factors through the kernel of $\Omega^2_d (\pi^{n-2,n-3}_d)$.
By the flatness assumption, the functor $\Omega^2_d$ is exact.
We apply it to the holonomic $(n-2)$-jet sequence, obtaining that the kernel of $\Omega^2_d(\pi^{n-2,n-3}_d)$ is $\Omega^2_d\circ S^{n-2}_d$.
Therefore, the image of $\widetilde{\DH}^{II}_{d,J^{n-2}_d}\circ J^1_d(l^{n-1}_d)\circ l^{\{n\}}_d$ is contained in $\Omega^2_d\circ S^{n-2}_d$.
\end{proof}

We can actually deduce a stronger result, that is, the image of $J^{\{n\}}_d$ via $\widetilde{\DH}_{d,J^{n-2}_d}$ is contained in the kernel of $\delta^{n-2,2}_d=\wedge^{2,1}_{S^{n-3}_d}\circ \iota^{n-2}_\wedge \colon \Omega^2_d\otimes_A S^{n-2}_d\to \Omega^3_d\otimes_A S^{n-3}_d$ (cf.\ §\ref{ss:Spencer}).
However, in order to compute $\delta^{n-2,2}_d$ on the image of $\widetilde{\DH}_{d,J^{n-2}_d}\circ J^1_d(l^{n-1}_d)\circ l^{\{n\}}_d$, it will be more convenient to compute a map that extends $\delta^{n-2,2}_d$ to $\Omega^2_d\circ J^{n-2}_d$.
For this purpose, we prove the following result.
\begin{lemma}\label{lemma:extwedge2}
In the category of functors $\AMod\to \Mod$, there is a natural transformation
\begin{align}
\mu\colonequals -d\otimes_A \pi^{1,0}_d+\id_{\Omega^2_d}\wedge^{2,1} \rho_d\colon \Omega^2_d\circ J^1_d\longrightarrow \Omega^3_d,
&\hfill&
\omega\otimes_A [x\otimes e]\longmapsto -(d\omega)x\otimes_A e-\omega\wedge dx\otimes_A e.
\end{align}
such that $\mu\circ \Omega^2_d(\iota^1_d)=\wedge^{2,1}\colon \Omega^2_d\circ \Omega^1_d\twoheadrightarrow \Omega^3_d$.

For $n>2$, restricting $\mu_{S^{n-3}_d}$ and $\mu_{J^{n-3}_d}$ to $\Omega^2_d\circ S^{n-2}_d$ gives
\begin{align}
\mu_{S^{n-3}_d}\circ \Omega^2_d(\iota^{1}_{d,S^{n-2}_d})\circ \Omega^2_d(\iota^{n-2}_{\wedge})
=\delta^{n-2,2}_d,
&\hfill&
\mu_{J^{n-3}_d}\circ \Omega^2_d(l^{n-2}_d)\circ \Omega^2_d(\iota^{n-2}_d)
=\Omega^{3}_d(\iota^{n-3}_d)\circ \delta^{n-2,2}_d.
\end{align}
\end{lemma}
\begin{proof}
We prove the first result for a generic $E$ in $\AMod$.
Consider the following well-defined $\bk$-linear map
\begin{equation}
-d\otimes_A \pi^{1,0}_{d,E}+\id_{\Omega^2_d}\wedge^{2,1} \rho_{d,E}\colon \Omega^2_d\otimes J^1_d(E)\longrightarrow \Omega^3_d(E).
\end{equation}
We want to prove that it factors through the map $\otimes_A\colon \Omega^2_d\otimes J^1_d(E)\twoheadrightarrow \Omega^2_d\otimes_A J^1_d(E)$.
For all $f\in A$, $\omega\in\Omega^1_d$, and $[x\otimes e]\in J^1_d E$ we have
\begin{equation}
\begin{split}
\left(-d\otimes_A \pi^{1,0}_{d,E}+\id_{\Omega^2_d}\wedge^{2,1} \rho_{d,E}\right)(\omega f\otimes [x\otimes e])
&=-d(\omega f)x\otimes_A e-\omega f\wedge dx \otimes_A e\\
&=-(d\omega) fx\otimes_A e-\omega\wedge (df)x\otimes_A e-\omega \wedge f dx \otimes_A e\\
&=-(d\omega)\otimes_A fxe-\omega\wedge d(fx)\otimes_A e\\
&=\left(-d\otimes_A \pi^{1,0}_{d,E}+\id_{\Omega^2_d}\wedge^{2,1} \rho_{d,E}\right)(\omega \otimes f[x\otimes e])
\end{split}
\end{equation}
We call the factorization $\mu_E$.
Since all the maps used to define $\mu_E$ are natural in $E$, we get that $\mu$ is a natural transformation.

If we precompose $\mu$ with $\Omega^2_d(\iota^1_{d})$, we obtain
\begin{equation}
\begin{split}
\mu\circ \Omega^2_d(\iota^1_{d})
&=\left(-d\otimes_A \pi^{1,0}_{d}+\id_{\Omega^2_d}\wedge^{2,1} \rho_{d} \right)\circ \left(\id_{\Omega^2_d}\otimes_A \iota^1_{d}\right)\\
&=-d\otimes_A \pi^{1,0}_{d}\circ \iota^1_{d}+\id_{\Omega^2_d}\wedge^{2,1} \rho_{d}\circ \iota^1_{d}\\
&=0+\id_{\Omega^2_d}\wedge^{2,1}\id_{\Omega^1_d}\\
&=\wedge^{2,1}.
\end{split}
\end{equation}

The last statements of the lemma follow from the next diagram.
\begin{equation}\label{diag:dunes}
\begin{tikzcd}[column sep=50pt]
\Omega^2_d\circ S^{n-2}_d\ar[rrr,bend left=15pt,"\delta^{n-2,2}_d"']\ar[r,"\Omega^2_d(\iota^{n-2}_{\wedge})"']\ar[d,"\Omega^2_d(\iota^{n-2}_d)"']&\Omega^2_d\circ \Omega^1_d\circ S^{n-3}_d\ar[rr,bend left=10pt,"\wedge"]\ar[r,"\Omega^2_d (\iota^1_{d,S^{n-1}_d})"']\ar[d,"\Omega^2_d(\iota^{\{n-2\}}_{d})"']&\Omega^2_d\circ J^1_d\circ S^{n-3}_d\ar[r,"\mu_{S^{n-3}_d}"']\ar[d,"\Omega^2_d\circ J^1_d(\iota^{n-3}_d)"]&\Omega^3_d \circ S^{n-3}_d\ar[d,"\Omega^3_d(\iota^{n-3}_d)"]\\
\Omega^2_d\circ J^{n-2}_d\ar[r]\ar[rr,bend right=10pt,"\Omega^2_d(l^{n-2}_{d})"']&\Omega^2_d\circ J^{\{n-2\}}_d\ar[r,"\Omega^2_d(l^{\{n-2\}}_d)"]&\Omega^2_d\circ J^1_d\circ J^{n-3}_d\ar[r,"\mu_{J^{n-3}_d}"]&\Omega^3_d \circ J^{n-3}_d
\end{tikzcd}
\end{equation}
The right square in \eqref{diag:dunes} commutes by naturality of $\mu$.
The central square commutes by \eqref{diag:sesqui} and the naturality of $\iota^1_d$.
The commutativity of the left square descends from the following diagram
\begin{equation}
\begin{tikzcd}[column sep=40pt]
S^{n-2}_d\ar[d,"\iota^{n-2}_d"']\ar[r,hookrightarrow,"\iota^{n-2}_\wedge"]&\Omega^1_d\circ S^{n-3}_d\ar[d,"\iota^{\{n-2\}}_d"]\ar[r,"\Omega^1_d(\iota^{n-3}_d)"]&\Omega^1_d\circ J^{n-3}_d\ar[d,hookrightarrow,"\iota^1_{d,J^{n-3}_d}"]\\
J^{n-2}_d\ar[rr,hookrightarrow,bend right=10pt,"l^{n-2}_d"']\ar[r,hookrightarrow]&J^{\{n-2\}}_d\ar[r,hookrightarrow,"l^{\{n-2\}}_d"]&J^1_d\circ J^{n-3}_d
\end{tikzcd}
\end{equation}
Here in fact, the exterior square commutes by \eqref{diag:defiiotand}, and the right square commutes by \eqref{diag:sesqui}, so the commutativity of the left square follows from the fact that $l^{\{n-2\}}_d$ is a monomorphism.
\end{proof}
\begin{lemma}\label{lemma:edhinkerspencer}
If $n=2$, or if $n>2$ and $\Omega^1_d$, $\Omega^2_d$, and $\Omega^3_d$ are flat in $\ModA$ and (for $n>3$) the $(n-2)$-jet sequence is left exact, then we have $\im\left(\widetilde{\DH}_{d,J^{n-2}_d}\circ J^1_d(l^{n-1}_d)\circ l^{\{n\}}_d\right)\subseteq \ker(\delta^{n-2,2}_d)$.
\end{lemma}
\begin{proof}
The case $n=2$ is trivially satisfied by Lemma \ref{lemma:imageD} and $\delta^{n-2,2}_d=\delta^{0,2}_d=0$.
For $n>2$, by Lemma \ref{lemma:imageD}, we know that $\im\left(\widetilde{\DH}_{d,J^{n-2}_d}\circ J^1_d(l^{n-1}_d)\circ l^{\{n\}}_d\right)$ maps to $\Omega^2_d\circ S^{n-2}_d$.
Therefore, by Lemma \ref{lemma:extwedge2} we have that
\begin{equation}
\Omega^3_d(\iota^{n-3}_d)\circ \wedge\circ\widetilde{\DH}_{d,J^{n-2}_d}\circ J^1_d(l^{n-1}_d)\circ l^{\{n\}}_d
=\mu_{J^{n}_d}\circ\widetilde{\DH}_{d,J^{n-2}_d}\circ J^1_d(l^{n-1}_d)\circ l^{\{n\}}_d.
\end{equation}
From flatness of $\Omega^3_d$ follows that $\Omega^3_d(\iota^{n-3}_d)$ is injective.

It remains to show that $\mu_{J^{n-3}_d}\circ \widetilde{\DH}_{d,J^{n-2}_d}\circ J^1_d(l^{n-1}_d)\circ l^{\{n\}}_d=0$.
We check this on a generic element of $J^{\{n\}}_d E$, which can always be written as
\begin{equation}
\sum_i[a^1_i\otimes b^1_i]\otimes_A\underbrace{\sum_j [a^2_{i,j}\otimes b^2_{i,j}]\otimes_A \overbrace{\sum_k[a^3_{i,j,k}\otimes b^3_{i,j,k}]\otimes_A \underbrace{\xi_{i,j,k}}_{\in J^{n-3}_d E}}^{\in J^{n-2}_d E}}_{\in J^{n-1}_d E}\in J^{\{n\}}_d\subseteq J^1_dJ^{n-1}_d E\subseteq J^{(2)}_dJ^{n-2}_d E\subseteq J^{(3)}_dJ^{n-3}_d E.
\end{equation}
By applying $\widetilde{\DH}_{d,J^{n-3}_d E}$, we obtain
\begin{equation}
\sum_{i,j,k}d a^1_i\wedge d( b^1_i a^2_{i,j})\otimes_A [ b^2_{i,j} a^3_{i,j,k}\otimes b^3_{i,j,k}]\otimes_A \xi_{i,j,k}
\end{equation}
We can now apply $\mu_{J^{n-3}_d}$, and we obtain
\begin{equation}
\begin{split}
&-\sum_{i,j,k}d\left(d a^1_i\wedge d( b^1_i a^2_{i,j})\right)\otimes_A b^2_{i,j} a^3_{i,j,k} b^3_{i,j,k}\xi_{i,j,k}
-\sum_{i,j,k}d a^1_i\wedge d( b^1_i a^2_{i,j})\wedge d \left(b^2_{i,j} a^3_{i,j,k}\right)\otimes_A b^3_{i,j,k} \xi_{i,j,k}\\
&\qquad=0-\sum_{i} d a^1_i\wedge \left(\sum_{j,k} d( b^1_i a^2_{i,j})\wedge d \left(b^2_{i,j} a^3_{i,j,k}\right)\otimes_A b^3_{i,j,k} \xi_{i,j,k}\right)\\
&\qquad=-\sum_{i} d a^1_i\wedge \widetilde{\DH}^{II}_{d,J^{n-3}_d E}\left(
\sum_j [b^1_ia^2_{i,j}\otimes b^2_{i,j}]\otimes_A \sum_k[a^3_{i,j,k}\otimes b^3_{i,j,k}]\otimes_A \xi_{i,j,k}\right)\\
&\qquad=0,
\end{split}
\end{equation}
since $\widetilde{\DH}^{II}_{d,J^{n-3}_d E}$ vanishes on $J^{n-1}_d E$.
\end{proof}
\begin{rmk}\label{rmk:weakO3flat}
The flatness of $\Omega^3_d$ in $\ModA$ is not the minimal condition for Lemma \ref{lemma:edhinkerspencer} to hold.
In fact, we only need $\Omega^3_d(\iota^{n-3}_d)$ to be a monomorphism.
This can happen for instance if $\iota^{n-3}_d$ is a left $A$-linear split mono.
In particular, for $n\le 3$, flatness of $\Omega^3_d$ can be removed, as $\iota^{n-3}_d=\id$.
For $n=4$, for it to hold at $E$ in $\AMod$ it is sufficient to have $\Tor^A_1(\Omega^3_d,E)$.
This weakening of hypotheses can be applied to every result which applies Lemma \ref{lemma:edhinkerspencer}, including the following theorem.
\end{rmk}
\begin{theo}[Holonomic jet exact sequence]\label{theo:higherwolves}
Let $A$ be a $\bk$-algebra endowed with an exterior algebra $\Omega^{\bullet}_d$ such that $\Omega^1_d$, $\Omega^2_d$ and $\Omega^3_d$ are flat in $\ModA$.
Suppose the holonomic $(n-2)$-jet sequence is left exact, and the holonomic $(n-1)$-jet sequence is exact.
Then the following sequence is also exact,
\begin{equation}\label{es:hjsesi}
\begin{tikzcd}
0\ar[r]&S^n_d\ar[r,hookrightarrow,"\iota^n_{d}"]& J^{n}_d\ar[r,"\pi^{n,n-1}_{d}"]&J^{n-1}_d\ar[r,"\dh^{n-1}_d"]&H^{n-2,2}_{\delta_d},
\end{tikzcd}
\end{equation}
where $H^{n-2,2}_{\delta_d}$ is the Spencer $\delta$-cohomology.

Therefore, we obtain a short exact sequence
\begin{equation}\label{es:hjsesn}
\begin{tikzcd}
0\ar[r]&S^n_d\ar[r,hookrightarrow,"\iota^n_{d}"]& J^{n}_d\ar[r,two heads,"\pi^{n,n-1}_{d}"]&J^{n-1}_d\ar[r]&0
\end{tikzcd}
\end{equation}
if and only if $\dh^{n-1}_d$ is the zero map.
\end{theo}
\begin{proof}
We already have the short jet exact sequence for $0$, $1$, and $2$-jet functors.

Now let $n\ge 3$.
Since we have the $(n-1)$-jet exact sequence and the $(n-2)$-jet sequence is left exact, by Proposition \ref{prop:hol=>sesqui+1} the sesquiholonomic $n$-jet exact sequence holds.
Since \eqref{cpx:Spencer} is a complex, the image of $\wedge=-\delta^{n-1,1}_d$ is contained in the kernel of $\delta^{n-2,2}_d$.
Lemma \ref{lemma:edhinkerspencer} tells us that also $\widetilde{\DH}_{d,J^{n-3}_d}$ restricted to the sesquiholonomic jet functor $J^{\{n\}}_d$ has image in $\ker(\delta^{n-2,2}_d)$.

Thus we have the following diagram by applying the snake lemma to the two central rows
\begin{equation}\label{diag:eshwsnake}
\begin{tikzcd}[column sep=50pt]
0\ar[r]&[-20pt]S^n_d\ar[r,hookrightarrow,"\iota^n_{d}"]\ar[d,hookrightarrow,"\iota^n_{\wedge}"]& J^{n}_d\ar[r,"\pi^{n,n-1}_{d}"]\ar[d]\ar[draw=none]{ddd}[name=X, anchor=center]{}&J^{n-1}_d\ar[d,equals]
\ar[rounded corners=12pt,
	to path={ -- ([xshift=50pt]\tikztostart.east)
		|- ([yshift=5pt]X.center) \tikztonodes
		-| ([xshift=-60pt]\tikztotarget.west)
		-- (\tikztotarget)},near start,"\dh^{n-1}_d"]{dddll}
&[-20pt]\\
0\ar[r]&\Omega^1_d\circ S^{n-1}_d\ar[r,hookrightarrow,"\iota_d^{\{n\}}"]\ar[d,near end,"\wedge"]&J^{\{n\}}_d\ar[r,twoheadrightarrow,"\pi^{\{n,n-1\}}_d"]\ar[d,near end, "\widetilde{\DH}_{d,J^{n-2}_d}\circ J^1_d(l^{n-1}_d)\circ l^{\{n\}}_d"]&J^{n-1}_d\ar[r]\ar[d]&0\\[10pt]
0\ar[r]&\ker(\delta^{n-2,2}_d)\ar[d]\ar[r,equals]&\ker(\delta^{n-2,2}_d)\ar[r]\ar[d]&0\ar[d]\\
&H^{n-2,2}_{\delta_d}\ar[r]&C\ar[r]&0
\end{tikzcd}
\end{equation}
The snake lemma provides the long exact sequence, 
\begin{equation}
\begin{tikzcd}
0\ar[r]&S^n_d\ar[r,hookrightarrow,"\iota^n_{d}"]& J^{n}_d\ar[r,"\pi^{n,n-1}_{d}"]&J^{n-1}_d\ar[r,"\dh^{n-1}_d"]&H^{n-2,2}_{\delta_d}\ar[r]&C\ar[r]&0.
\end{tikzcd}
\end{equation}
Thence follows \eqref{es:hjsesi}.
If $\dh^{n-1}_d=0$, we further obtain \eqref{es:hjsesn}.
\end{proof}
\begin{cor}\label{cor:spencerdelta_jes}
Suppose $\Omega^1_d$, $\Omega^2_d$, and $\Omega^3_d$ are flat in $\ModA$.
If $H^{m,2}_{\delta_d}=0$ for all $1\le m\le n-2$, then the $n$-jet sequence is exact.
\end{cor}
\begin{rmk}
Classically, the vanishing of the Spencer $\delta$-cohomology is associated to the involutivity and formal integrability of systems of partial differential equations \cite{Lychagin1990}.
Our complex \eqref{cpx:Spencer} is associated to the empty set of equations, so it is reasonable to expect a sound noncommutative generalization of the notion of vector bundle to have vanishing Spencer $\delta$-cohomology.

If we restrict our attention to $\AFlat$, the vanishing of $H^{\bullet,2}_{\delta_d}$ becomes a condition on the exterior algebra $\Omega^\bullet_d$ rather than on the particular $A$-module $E$, by Proposition \ref{prop:Spencertensorrep}.
Therefore, if the exterior algebra is flat in $\ModA$ in the grades $1$, $2$, and $3$, and $H^{\bullet,2}_{\delta_d}(A)=0$, then the theory of noncommutative jets developed in the present work will be closely aligned with the classical theory.
\end{rmk}
\begin{prop}\label{prop:dhdiffop}
The map $\dh^{n-1}_d$ of Theorem \ref{theo:higherwolves} is a natural lift of $0\colon \id_{\AMod}\to H^{n-2,2}_{\delta_d}$ in the category of functors of type $\AMod\to \AMod$.
That is $\dh^{n-1}_d\circ j^1_d=0$.

In particular, for all $E$, $\dh^{n-1}_{d,E}$ lifts the zero map $0\colon E\to H^{n-2,2}_{\delta_d}(E)$.
\end{prop}
\begin{proof}
We have to prove that $\dh^{n-1}_d\circ j^{n-1}_d=0$, and we know that $j^{n-1}_d=\pi^{n,n-1}_d\circ j^{n}_d$ (cf.\ Remark \ref{rmk:holprolpi}).
From the long exact sequence in \eqref{diag:eshwsnake} being a complex, we get
\begin{equation}
\dh^{n-1}_d\circ j^{n-1}_d
=\dh^{n-1}_d\circ \pi^{n,n-1}_d\circ j^{n}_d
=0.
\end{equation}
\end{proof}
\begin{cor}
Given the assumptions of Theorem \ref{theo:higherwolves}, if differential operators of order at most $n-1$ (cf.\ Definition \ref{def:differentialoperators}) lift uniquely to $J^{n-1}_d$, then $\dh^{n-1}_d=0$ and the $n$-jet sequence \eqref{es:hjsesn} is exact.
\end{cor}
\begin{proof}
Immediate consequence of Proposition \ref{prop:dhdiffop} and Theorem \ref{theo:higherwolves}.
\end{proof}

For the $3$-jet we can compute an explicit expression for the map $\dh^{2}_d$
\begin{prop}
Let $A$ be a $\bk$-algebra endowed with an exterior algebra $\Omega^{\bullet}_d$ such that $\Omega^1_d$ and $\Omega^2_d$ are flat in $\ModA$.
Then the following sequence is exact
\begin{equation}
\begin{tikzcd}
0\ar[r]&S^3_d\ar[r,hookrightarrow,"\iota^3_{d,A}"]& J^{3}_d A\ar[r,"\pi^{n,n-1}_{d,A}"]&J^{n-1}_d A\ar[r,"\dh^{2}_{d,A}"]&H^{1,2}_{\delta_d}(A).
\end{tikzcd}
\end{equation}

Explicitly, consider an element in $J^2_d A$ written (cf.\ \eqref{eq:exp2jet}) in the form
\begin{equation}\label{eq:repS2forJ3}
\xi=\sum_i a_ij^2_d(b_i)+\sum_j dx^2_j\otimes_A dx^{1}_j\otimes_A x^0_j.
\end{equation}
Then $\dh^2_{d,A}(\xi)=\left[\sum_j dx^2_j\wedge dx^{1}_j\otimes_A dx^0_j\right]$, and this expression is independent from the representation \eqref{eq:repS2forJ3} of $\xi$.
\end{prop}
\begin{proof}
The first step is to prove Theorem \ref{theo:higherwolves} for $n=3$ and $E=A$ without assuming $\Omega^3_d$ to be flat.
This property is used in particular in Lemma \ref{lemma:edhinkerspencer} and Remark \ref{rmk:weakO3flat} shows that flatness of $\Omega^3_d$ is superfluous.
In this case we can also provide an explicit alternative proof that the image of $\widetilde{\DH}_{d,J^{1}_d}\circ J^1_d(l^{2}_d)\circ l^{\{3\}}_d$ is contained in $\ker(\delta^{1,2}_d)$ without using any assumption on $\Omega^3_d$.
Let $\eta = \sum_{i,j} [a^3_i \otimes b^3_i] \otimes_A [a^2_{i,j} \otimes b^2_{i,j}] \otimes_A [a^1_{i,j} \otimes b^1_{i,j}]\in J^{\{3\}}_d A\subseteq J^1_d\circ J^{2}_d$ be such that $\sum_j [a^2_{i,j} \otimes b^2_{i,j}] \otimes_A [a^1_{i,j} \otimes b^1_{i,j}]\in J^2_d$ for all $i$.
All the projections of $\eta$ coincide, in particular $\pi^{(n,n-1;1)}_{d,A}(\eta) = \sum_{i,j} [a^3_i \otimes b^3_i] \otimes_A [a^2_{i,j} \otimes b^2_{i,j}a^1_{i,j} b^1_{i,j}]$ is holonomic.
Applying $\widetilde\DH_d$, we get
	\begin{equation}
		0
		=\widetilde{\DH}_d (\pi^{(n,n-1;1)}_{d,A} (\eta))
		=\sum_{i,j}\sum_{i,j} a^3_i d( b^3_i a^2_{i,j}) b^2_{i,j}a^1_{i,j} b^1_{i,j}
		+\sum_{i,j}\sum_{i,j} da^3_i\wedge d( b^3_i a^2_{i,j}) b^2_{i,j}a^1_{i,j} b^1_{i,j}.
	\end{equation}
	Next, we apply $\widetilde{\DH}_{d,J^1_d A}$ to $\eta$, and obtain
	\begin{equation}
	\begin{split}
	\widetilde{\DH}_{d,J^1_d A}(\eta)
	&=\sum_{i,j} da^3_i \wedge d(b^3_i a^2_{i,j}) \otimes_A [b^2_{i,j}a^1_{i,j} \otimes b^1_{i,j}]\\
	&=\sum_{i,j} da^3_i \wedge d(b^3_i a^2_{i,j}) \otimes_A \left([b^2_{i,j}a^1_{i,j} b^1_{i,j}\otimes 1]+b^2_{i,j}a^1_{i,j} db^1_{i,j}\right)\\
	&=\sum_{i,j} da^3_i \wedge d(b^3_i a^2_{i,j})b^2_{i,j}a^1_{i,j} b^1_{i,j} \otimes_A [1\otimes 1]
	+\sum_{i,j} da^3_i \wedge d(b^3_i a^2_{i,j}) b^2_{i,j}a^1_{i,j} \otimes_A db^1_{i,j}\\
	&=\sum_{i,j} da^3_i \wedge d(b^3_i a^2_{i,j}) b^2_{i,j}a^1_{i,j} \otimes_A db^1_{i,j}.
	\end{split}
	\end{equation}
	We now apply $\delta^{1,2}_d$, obtaining
	\begin{equation}
	\begin{split}
		\delta^{1,2}_d(\widetilde{\DH}_{d,J^1_d A})(\eta)
		&=\sum_{i,j} da^3_i \wedge d(b^3_i a^2_{i,j}) b^2_{i,j}a^1_{i,j} \wedge db^1_{i,j}\\
		&=\sum_{i,j} d\left( da^3_i \wedge d(b^3_i a^2_{i,j}) b^2_{i,j}a^1_{i,j} b^1_{i,j}\right)
		-\sum_{i,j} da^3_i \wedge d(b^3_i a^2_{i,j}) \wedge d(b^2_{i,j}a^1_{i,j}) b^1_{i,j}\\
		&=0-\sum_{i} da^3_i\wedge\widetilde{\DH}_d\left(\sum_j [b^3_i a^2_{i,j} \otimes b^2_{i,j}] \otimes_A [a^1_{i,j} \otimes b^1_{i,j}]\right)\\
		&=0.
	\end{split}
	\end{equation}
	The last equality follows from the fact that $\sum_j [b^3_i a^2_{i,j} \otimes b^2_{i,j}] \otimes_A [a^1_{i,j} \otimes b^1_{i,j}]\in J^2_d A$.

Consider an element
\begin{equation}
\xi=\sum_i a_i j^2_d(b_i)+\sum_j dx^2_j\otimes_A dx^{1}_j\otimes_A x^0_j.
\end{equation}

By Proposition \ref{prop:dhdiffop} and $A$-linearity of $\dh^{2}_d$, we know that $\dh^2_d$ vanishes on $AJ^1_d A$.
We thus compute $\dh^2_{d,A}$ on the $S^2_d$ part alone.
By Remark \ref{rmk:TnEstd}, we can write any element in $S^2_d$ in the form $\xi=\sum_{j} dx^2_{j}\otimes_A dx^{1}_{j}\otimes_A x^0_{j}$ for $x^h_{j}\in A$.
We now compute $\dh^2_{d,A}$ as prescribed by the snake lemma.
We start by lifting this element to $J^{\{n\}}_d A$ by taking
\begin{equation}
\widetilde{\xi}
=\sum_{j}[1\otimes 1]\otimes_A dx^2_{j}\otimes_A dx^{1}_{j}\otimes_A x^0_{j}
+\sum_{j} dx^2_{j}\otimes_A [1\otimes 1]\otimes_A dx^{1}_{j}\otimes_A x^0_{j}
+\sum_{j} dx^2_{j}\otimes_A dx^{1}_{j}\otimes_A [1\otimes x^0_{j}].
\end{equation}
Since every projection $\pi^{(3,2;m)}_d$ gives $\xi$, the element $\widetilde{\xi}$ belongs to $J^{[3]}_d A$.
Furthermore, applying $J^1_d \widetilde{\DH}^{II}_d$ gives
\begin{equation}
J^1_d \widetilde{\DH}^{II}_d(\widetilde{\xi})
=\sum_{j}[1\otimes 1]\otimes_A dx^2_{j}\wedge dx^{1}_{j}\otimes_A x^0_{j}
+0
+\sum_{j} dx^2_{j}\otimes_A d(d x^{1}_{j}) x^0_{j}
=0.
\end{equation}
Thus, by Lemma \ref{lemma:charsesqui}, $\widetilde{\xi}\in J^{\{3\}}_d A$.
We now apply $\widetilde{\DH}_{d,J^1_d}$, which will coincide with $\widetilde{\DH}^{II}_{d,J^1_d}$, obtaining
\begin{equation}
\widetilde{\DH}_{d,J^1_d}(\widetilde{\xi})
=0+0+\sum_{j} dx^2_{j}\wedge dx^{1}_{j}\otimes_A [1\otimes x^0_{j}]
=\sum_{j} dx^2_{j}\wedge dx^{1}_{j}\otimes_A d x^0_{j}
\end{equation}
where the last equality follows from the fact that $\sum_{j} (dx^2_{j}\wedge dx^{1}_{j} )x^0_{j}=0$, as $\xi\in S^2_d$.
The element $\sum_{j} dx^2_{j}\wedge dx^{1}_{j}\otimes_A d x^0_{j}$ is in $\ker(\delta^{1,2}_d)$, and its class in $H^{1,2}_{\delta_d}$ is the image of $\xi$ via $\dh^2_{d,A}$.
The map $\dh^2_{d,A}$ is canonical, so it is independent of the chosen representation.
\end{proof}
\subsection{Stability and exactness}
We can deduce the following result regarding the stability of certain subcategories with respect to the functor $J^{n}_d$.
\begin{prop}\label{prop:Jhstable}
Let $A$ be a $\bk$-algebra and let $\Omega^\bullet_d$ be an exterior algebra over it.
Consider the functor
\begin{equation}
J^{n}_d\colon \AMod\longrightarrow\AMod.
\end{equation}
\begin{enumerate}
\item\label{prop:Jhstable:1} If the $m$-jet sequence is exact and $S^m_d$ is in $\AFlat$ for all $1\le m \le n$, then $J^{n}_d$ preserves $\AFlat$;
\item\label{prop:Jhstable:2}If the $m$-jet sequence is exact and $S^m_d$ is in $\AProj$ for all $1\le m \le n$, then $J^{n}_d$ preserves $\AProj$;
\item\label{prop:Jhstable:3} If the $m$-jet sequence is exact and $S^m_d$ is in $\AFGP$ for all $1\le m \le n$, then $J^{n}_d$ preserves $\AFGP$.
\end{enumerate}
\end{prop}
\begin{proof}
We proceed by induction on $n$.
For $n=1$, all points are true by Proposition \ref{prop:OmJstab}.
Now consider $n> 1$, and assume \eqref{prop:Jhstable:1} holds for $n-1$.
In the $n$ case, the hypotheses required contain the hypotheses for the $n-1$ case, and thus we can conclude that $J^{n-1}_d$ preserves the subcategory $\AFlat$.
Now we evaluate the $n$-jet short exact sequence at $E$ in $\AFlat$.
By Proposition \ref{prop:Spencertensorrep}, we can write $S^n_d(E)=S^n_d\otimes_A E$, and by Lemma \ref{lemma:flpr} we get that $S^n_d(E)$ is in $\AFlat$.
It follows by Lemma \ref{lemma:2-3} that also $J^n_d(E)$ is in $\AFlat$, ending the proof of \eqref{prop:Jhstable:1}.
The proof for \eqref{prop:Jhstable:2} and \eqref{prop:Jhstable:3} is analogous.
\end{proof}

The following proposition discusses exactness of $J^{n}_d$ as a functor.
\begin{prop}\label{prop:hJexact}
Let $\Omega^{\bullet}_d$ an exterior algebra over a $\bk$-algebra $A$.
The functor $J^n_d$ preserves flat colimits of flat modules.
If $\Omega^1_d$ and $\Omega^2_d$ are flat in $\ModA$, then $J^n_d$ is left exact.

Moreover, if the $m$-jet sequence is exact and $S^m_d$ is exact for all $3\le m\le n$, then $J^n_d$ is exact.
\end{prop}
\begin{proof}
The proof of the first two statements is analogous to that of Proposition \ref{prop:exactnessSn} via Corollary \ref{cor:Jex} and Lemma \ref{lemma:spplacederfun}.

We prove the last statement by induction on $n$.
For $n=0,1$, we already know $J^n_d$ is exact.
For $n\ge 2$, we proceed by induction, and assuning the $n-1$ case gives us the exactness of $J^{n-1}_d$.
Further, we also have the $n$-jet short exact sequence \eqref{es:hjsesn}.
If we now take a short exact sequence $0\to M\to N\to Q\to 0$ and apply the $n$-jet short exact sequence to it, we obtain
\begin{equation}
\begin{tikzcd}[column sep=50pt]
0\ar[r]&S^n_d M\ar[r,hookrightarrow,"\iota^{n}_{d,M}"]\ar[d,hookrightarrow]&J^{n}_d M\ar[r,twoheadrightarrow,"\pi^{n,n-1}_{d,M}"]\ar[d,hookrightarrow]&J^{n-1}_d M\ar[r]\ar[d,hookrightarrow]&0\\
0\ar[r]&S^n_d N\ar[r,hookrightarrow,"\iota^{n}_{d,N}"]\ar[d,twoheadrightarrow]&J^{n}_d N\ar[r,twoheadrightarrow,"\pi^{n,n-1}_{d,N}"]\ar[d]&J^{n-1}_d N\ar[r]\ar[d,twoheadrightarrow]&0\\
0\ar[r]&S^n_d Q\ar[r,hookrightarrow,"\iota^{n}_{d,Q}"]&J^{n}_d Q\ar[r,twoheadrightarrow,"\pi^{n,n-1}_{d,Q}"]&J^{n-1}_d Q\ar[r]&0
\end{tikzcd}
\end{equation}
The three rows are exact, and by hypothesis, so it the left column.
The right column is exact by inductive hypothesis.
By the nine lemma, also the central column is exact, proving the exactness of $J^n_d$.
\end{proof}

\subsection{Classical holonomic jet functors}
We make our final extension of Theorem \ref{theo:classical1jet} to cover the holonomic jet functors.
\begin{theo}\label{theo:classicalnjet}
	Let $M$ be a smooth manifold, $A=\smooth{M}$ its algebra of smooth functions, endowed with the associated de Rham exterior algebra $\Omega^\bullet(M)$, and $E=\Gamma(M,N)$ the space of smooth sections of a vector bundle $N\rightarrow M$.
	Then $J^n_dE \simeq \Gamma(M,J^nN)$ in $\AMod$, the classical bundle of $n$-jets of sections of $N\to M$, and the isomorphism takes our prolongation to the classical prolongation.
\end{theo}
\begin{proof}
	Suppose the result holds for order $n-1$ and $n-2$.
	By Lemma \ref{lemma:generalspencer}, we have that $J^n_d E = J^1_d(J^{n-1}_d E) \cap J^2_d(J^{n-2}_d E)$.
	Then, by Theorem \ref{theo:classical1jet} and Theorem \ref{theo:classical2jet}, we get that the two components of the intersection coincide with the classical analogues.
	By Lemma \ref{lemma:gammacontinuous}, intersections of $\smooth{M}$-modules of sections correspond to the module of sections of the intersection of the respective bundles.
	Therefore, by the classical result \cite[Lemma 1.2.1, p.~184]{Spencer}, we obtain the desired isomorphism.
	The base of the induction is provided by $J^2_d E$ and $J^1_d E$ via Theorem \ref{theo:classical1jet} and Theorem \ref{theo:classical2jet}.
	The prolongation map coincides with the $n$th iterate of the first prolongation map both classically and for our construction, and so by Theorem \ref{theo:classical1jet}, they coincide.
\end{proof}

\section{Infinity jet functors}\label{s:infinityjets}
\subsection{The $\infty$-jet functors}
Let $A$ be a $\bk$-algebra and $\Omega^\bullet_d$ an exterior algebra over it.
Consider the following diagram in the abelian category of functors $\AMod\to \AMod$, constructed using the jet projections.
\begin{equation}\label{diag:jettower}
\begin{tikzcd}
\cdots J^n_d\ar[r,"\pi^{n,n-1}_d"]&J^{n-1}_d\cdots \ar[r,"\pi^{3,2}_d"]& J^2_d \ar[r,"\pi^{2,1}_d"] & J^1_d \ar[r,twoheadrightarrow,"\pi^{1,0}_d"] & \id_{\AMod}.
\end{tikzcd}
\end{equation}
We can do the same for the semiholonomic and nonholonomic jets and their respective projections, using solely the data of a $\bk$-algebra $A$ and a first order differential calculus $\Omega^1_d$ over it.
\begin{defi}[$\infty$-jet functor]
We call the diagram \eqref{diag:jettower} the \emph{(holonomic) jet tower}, and its limit in the category of functors $\AMod\to \AMod$ the \emph{(holonomic) $\infty$-jet functor}, denoted
\begin{equation}
J^{\infty}_d\colonequals \lim_{n\in\N} J^n_d.
\end{equation}

We term the corresponding diagrams constructed with the semiholonomic and nonholonomic jet functors and projections the \emph{semiholonomic} and \emph{nonholonomic jet tower}, respectively.
We call the respective limits the \emph{semiholonomic} and \emph{nonholonomic $\infty$-jet functor}, denoted by $J^{[\infty]}_d$ and $J^{(\infty)}_d$ respectively.
\end{defi}

By construction, we obtain the following maps for all $n\ge 0$
\begin{align}
\pi^{\infty,n}_d\colon J^{\infty}_d\longrightarrow J^n_d,
&\hfill&
\pi^{[\infty,n]}_d\colon J^{[\infty]}_d\longrightarrow J^{[n]}_d,
&\hfill&
\pi^{(\infty,n)}_d\colon J^{(\infty)}_d\longrightarrow J^{(n)}_d.
\end{align}
\begin{defi}
We also call $\pi^{\infty,k}_d$, $\pi^{[\infty,k]}_d$, and $\pi^{(\infty,k)}_d$, \emph{(holonomic), semiholonomic, and nonholonomic jet projections}, respectively.
\end{defi}
By the universal property of the limit, for all $0\le m\le n$ we have
\begin{align}
\pi^{n,m}_d\circ \pi^{\infty,n}_d=\pi^{\infty,m}_d,
&\hfill&
\pi^{[n,m]}_d\circ \pi^{[\infty,n]}_d=\pi^{[\infty,m]}_d,
&\hfill&
\pi^{(n,m)}_d\circ \pi^{(\infty,n)}_d=\pi^{(\infty,m)}_d.
\end{align}

We now interpret the jet tower in the category of functors $\AMod\to \Mod$.
Since the forgetful functor $\AMod\to \Mod$ preserves limits, $J^\infty_d$ is the limit of the jet tower in this category as well.
Consider now the family of all natural prolongations $j^n_d\colon \id_{\AMod}\rightarrow J^n_d$.
This family is a cone over the jet tower (cf.\ Remark \ref{rmk:holprolpi}) in the category of functors $\AMod\to \Mod$.
By the universal property of the limit, there exists a unique map
\begin{equation}
j^\infty_d\colon \id_{\AMod}\longrightarrow J^\infty_d.
\end{equation}
such that for all $n\ge 0$
\begin{equation}\label{eq:compprojprol}
\pi^{\infty,n}_d\circ j^\infty_d=j^n_d.
\end{equation}
The same construction can be done for semiholonomic and nonholonomic jet functors using the semiholonomic and nonholonomic prolongation, respectively.
We thus obtain the unique natural transformations
\begin{align}
j^{[\infty]}_d\colon \id_{\AMod}\longrightarrow J^{[\infty]}_d,
&\hfill&
j^{(\infty)}_d\colon \id_{\AMod}\longrightarrow J^{(\infty)}_d,
\end{align}
such that for all $n\ge 0$
\begin{align}\label{eq:sncompprojprol}
\pi^{[\infty,n]}_d\circ j^{[\infty]}_d=j^{[n]}_d,
&\hfill&
\pi^{(\infty,n)}_d\circ j^{(\infty)}_d=j^{(n)}_d.
\end{align}
\begin{defi}
We call the natural transformations $j^\infty_d$, $j^{[\infty]}_d$, and $j^{(\infty)}_d$ the \emph{(holonomic)}, \emph{semiholonomic}, and \emph{nonholonomic $\infty$-jet prolongation}.
\end{defi}
It follows from \eqref{eq:compprojprol} and \eqref{eq:sncompprojprol} that for $n=0$,
\begin{align}
\pi^{\infty,0}_d\circ j^\infty_d=j^0_d=\id,
&\hfill&
\pi^{[\infty,0]}_d\circ j^{[\infty]}_d=\id,
&\hfill&
\pi^{(\infty,0)}_d\circ j^{(\infty)}_d=\id.
\end{align}
Hence the $\infty$-jet prolongations are sections of the projections from the $\infty$-jet to the $0$-jet, and as such are necessarily monomorphisms, while the maps $\pi^{\infty,0}_d$, $\pi^{[\infty,0]}_d$, and $\pi^{(\infty,0)}_d$ are epimorphisms.

\begin{rmk}[Infinity jet functors on bimodules]\label{rem:infinityjetfunctoronbimodules}
We can also define the $\infty$-jet functors on the category of functors $\AMod_B\to \AMod_B$.
Since the forgetful functor $\AMod_B\rightarrow\AMod$ creates limits, the functors $J^{\infty}_d\colon \AMod_B\to \AMod_B$ and $J^{\infty}_d\colon \AMod\to \AMod$ are compatible with the forgetful functor.

In this case, the morphism components of the jet projections are in $\AMod_B$ and morphism components of the jet prolongations are in $\Mod_B$.
These maps are also compatible with the forgetful functor, as all the $j^n_d$ are compatible with it.
\end{rmk}

\subsection{Natural maps between infinity jets}
We will now use Proposition \ref{prop:semiholcpxfact} to prove the existence of certain natural maps between the various types of $\infty$-jets.
\begin{prop}
The natural morphisms $\iota_{J^n_d}$ and $\iota_{J^{[n]}_d}$ induce natural transformations
\begin{align}
\iota_{J^{\infty}_d}\colon J^\infty_d\longrightarrow J^{(\infty)}_d,
&\hfill&
\iota_{J^{[\infty]}_d}\colon J^{[\infty]}_d\longhookrightarrow J^{(\infty)}_d.
\end{align}

Moreover, we have a unique factorization
\begin{equation}\label{diag:inftyjethsnfact}
\begin{tikzcd}
J^\infty_d\ar[rr,bend left,"\iota_{J^{\infty}_d}"]\ar[r,"h^{\infty}_d"']&J^{[\infty]}_d\ar[r,hookrightarrow,"\iota_{J^{[\infty]}_d}"']&J^{(\infty)}_d
\end{tikzcd}
\end{equation}

The morphism $\iota_{J^{[\infty]}_d}$ is a mono, and if $\Omega^1_d$ is flat in $\ModA$, then $\iota_{J^\infty_d}$ and $h^{\infty}_d$ are also monos.
\end{prop}
\begin{proof}
The maps $\iota_{J^n_d}$ and $\iota_{J^{[n]}_d}$ commute with the projection maps by Proposition \ref{prop:hjc} and Proposition \ref{prop:complexinclusionseminon}, respectively.
Hence, they form morphisms between the jet towers.
The universal property of the limit implies that there are induced maps between the respective limits, which yields
\begin{align}
\iota_{J^{\infty}_d}\colonequals \lim_{n\in \N} \iota_{J^n_d},
&\hfill&
\iota_{J^{[\infty]}_d}\colonequals \lim_{n\in \N} \iota_{J^{[n]}_d}.
\end{align}

Proposition \ref{prop:semiholcpxfact} provides the factorization \eqref{diag:inftyjethsnfact}, where $h^{\infty}_d=\lim_{n\in \N} h^n_d$.

Since the $\infty$-jets are defined as limits, the functor preserves limits.
In particular, it preserves kernels, and hence injective maps.
It follows that $\iota_{J^{[\infty]}_d}$ is always a monomorphism.
Analogously, $\iota_{J^{\infty}_d}$ is a mono if $\Omega^1_d$ is flat in $\ModA$ by Remark \ref{rmk:iotaholinj}.
In this case, $h^{\infty}_d$ is also a mono, as composing it with $\iota_{J^{[\infty]}_d}$ gives a mono.
\end{proof}

\begin{rmk}[Classical $\infty$-jet functor]
Although the $\infty$-jet bundles are not finite rank vector bundles, they are bundles that are limits of finite rank vector bundles.
If we apply Lemma \ref{lemma:gammacontinuous}, together with Theorem \ref{theo:classicalnjet}, we obtain
\begin{equation}
\Gamma(M,J^\infty E)
\cong \Gamma(M,\lim_{n\in \N} J^n E)
\cong \lim_{n\in \N} \Gamma(M,J^n E)
\cong \lim_{n\in \N} J^n_d(\Gamma(M,E))
\cong J^\infty_d(\Gamma(M,E)).
\end{equation}
Naturality in $E$ provides the equivalence between our notion of $\infty$-jet and the classical one, when realized as a limit in the category of $\smooth{M}$-modules.
The same holds analogously for $J^{(\infty)}_d$ and $J^{[\infty]}_d$.
\end{rmk}

\subsection{Exactness}
The following technical lemma provides sufficient conditions for the jet tower to satisfy the Mittag-Leffler condition (cf.\ \cite[Definition 5.3.5, p.~82]{Weibel}).
\begin{lemma}\label{lemma:M-L}
Let $A$ be a $\bk$-algebra.
\begin{enumerate}
\item Let $\Omega^1_d$ be a first order differential calculus for $A$, then the corresponding nonholonomic jet tower satisfies the Mittag-Leffler condition;
\item Let $\Omega^1_d$ be a first order differential calculus for $A$ which is flat in $\ModA$.
Then the corresponding semiholonomic jet tower satisfies the Mittag-Leffler condition;
\item Let $\Omega^\bullet_d$ be an exterior algebra over $A$, with $\Omega^1_d$, $\Omega^2_d$, and $\Omega^3_d$ flat in $\ModA$.
If $\dh^n_d=0$ for all $n\in\N$, then the holonomic jet tower satisfies the Mittag-Leffler condition.
\end{enumerate}
\end{lemma}
\begin{proof}
The proof follows by the definition of nonholonomic jet projections, Theorem \ref{theo:Jshes}, and Theorem \ref{theo:higherwolves}, respectively.
The conditions specified are in fact sufficient conditions for the jet projections to be surjective, which implies the Mittag-Leffler condition.
\end{proof}
We can now prove the following result concerning exactness of the $\infty$-jet functors.
\begin{prop}\label{prop:Jinfexact}
Let $A$ be a $\bk$-algebra.
\begin{enumerate}
\item\label{propJinfexact:1} Let $\Omega^{\bullet}_d$ an exterior algebra over $A$.
If $\Omega^1_d$ and $\Omega^2_d$ are flat in $\ModA$, then $J^{\infty}_d$ is left exact.
Moreover, if $\Omega^3_d$ is flat in $\ModA$, and $\dh^{n-1}_d=0$, and $S^n_d$ is exact for all $n\ge 3$, then $J^{\infty}_d$ is exact;
\item\label{propJinfexact:2} Let $\Omega^1_d$ be a first order differential calculus for $A$ which is flat in $\ModA$, then $J^{[\infty]}_d$ is exact;
\item\label{propJinfexact:3} Let $\Omega^1_d$ be a first order differential calculus for $A$ which is flat in $\ModA$, then $J^{(\infty)}_d$ is exact.
\end{enumerate}
\end{prop}
\begin{proof}\
\begin{enumerate}
\item Consider a short exact sequence
\begin{equation}\label{es:test}
\begin{tikzcd}
0\ar[r]&E\ar[r,hookrightarrow]&F\ar[r,twoheadrightarrow]&Q\ar[r]&0.
\end{tikzcd}
\end{equation}
Under the first set of conditions, each functor $J^n_d$ is left exact by Proposition \ref{prop:hJexact}.
Applying $J^n_d$ to \eqref{es:test}, we have the left exact sequences
\begin{equation}\label{es:Jntest}
\begin{tikzcd}
0\ar[r]&J^n_dE\ar[r,hookrightarrow]&J^n_dF\ar[r]&J^n_dQ.
\end{tikzcd}
\end{equation}
By naturality of the projection maps, it follows that we have a left exact sequence of jet towers.
Being a right adjoint, the limit functor is left exact.
Computing the limit of \eqref{es:Jntest}, we obtain
\begin{equation}
\begin{tikzcd}
0\ar[r]&J^\infty_dE\ar[r,hookrightarrow]&J^\infty_dF\ar[r]&J^\infty_dQ.
\end{tikzcd}
\end{equation}
This proves that $J^\infty_d$ is left exact.

Under the second set of assumptions, we can apply Theorem \ref{theo:higherwolves}, obtaining by induction that all jet sequences are exact.
The functors $J^n_d$ are exact by Proposition \ref{prop:hJexact}.
Applying them to \eqref{es:test}, we obtain a short exact sequence
\begin{equation}\label{es:Jnteste}
\begin{tikzcd}
0\ar[r]&J^n_dE\ar[r,hookrightarrow]&J^n_dF\ar[r,twoheadrightarrow]&J^n_dQ\ar[r]&0.
\end{tikzcd}
\end{equation}
Similarly to the previous case, this is a short exact sequence of jet towers.
By Lemma \ref{lemma:M-L}, the set of assumptions also allow us to say that the jet tower at $E$ in particular satisfies the Mittag-Leffler condition.
This condition implies that the first grade of the right derived functor of the limit functor vanishes on the jet tower (cf.\ \cite[Proposition 3.5.7, p.~83]{Weibel}).
This implies that the following is a short exact sequence
\begin{equation}
\begin{tikzcd}
0\ar[r]&J^\infty_dE\ar[r,hookrightarrow]&J^\infty_dF\ar[r,twoheadrightarrow]&J^\infty_dQ\ar[r]&0.
\end{tikzcd}
\end{equation}
It follows that $J^\infty_d$ is exact.

\item The proof is based on the same principle on which the last part of \eqref{propJinfexact:1} is proved.
It differs in that in place of Proposition \ref{prop:hJexact}, we use Proposition \ref{prop:shJregular}.
\item The proof is as for \eqref{propJinfexact:2}, but using Remark \ref{rmk:nhJexact}.\qedhere
\end{enumerate}
\end{proof}

\section{The category of differential operators}\label{s:differentialoperators}
\subsection{Linear differential operators of order at most $n$}
\begin{defi}
\label{def:differentialoperators}	
Let $E,F\in \AMod$.
A $\bk$-linear map $\Delta\colon E \rightarrow F$ is called a \textit{(holonomic) linear differential operator} of order at most $n$ with respect to the exterior algebra $\Omega^\bullet_d$, if it factors through the holonomic prolongation operator $j^{n}_d$, i.e.\ there exists an $A$-module map $\widetilde \Delta \in \AHom(J^n_d E,F)$ such that the following diagram commutes:
	\begin{equation}
	\begin{tikzcd}\label{diag:universaldifferentialoperator}
		J_d^nE \arrow[dr, "\widetilde\Delta"] & \\
		E \arrow[r,"\Delta"] \arrow[u,"j^n_{d,E}"] & F 
	\end{tikzcd}
	\end{equation}
	If $n$ is minimal, we say that $\Delta$ is a \emph{(holonomic) linear differential operator of order $n$} with respect to the exterior algebra $\Omega^\bullet_d$.
\end{defi}
Replacing the functor $J^n_d $ with either $J^{[n]}_d$ or $J^{(n)}_d$, and the natural transformation $j^{n}_d$ with either $j^{[n]}_d$ or $j^{(n)}_d$, in Definition \ref{def:differentialoperators} gives the definition of semiholonomic or nonholonomic differential operators, respectively.

We will denote by $\Diff_d(E,F)$, $\SDiff_d(E,F)$, and $\NDiff_d(E,F)$ the filtered $\bk$-modules of holonomic, semiholonomic, and nonholonomic differential operators of finite order from $E$ to $F$, and when $E = F = A$, we will shorten this to $\DO_d$, $\SDO_d$, or $\NDO_d$.
In particular, we will denote by $\Diff^n_d(E,F)\subseteq \Diff_d(E,F)$, ($\SDiff^n_d(E,F)\subseteq \SDiff_d(E,F)$, and $\NDiff^n_d(E,F)\subseteq \NDiff_d(E,F)$) the submodules of holonomic, (semiholonomic, and nonholonomic respectively) differential operators of order at most $n$.

\begin{prop}\label{prop:operatorsalsohigherorder}
Let $n\le m$, then a differential operator of order at most $n$ is also a differential operator of order at most $m$.
Similarly for semiholonomic and nonholonomic differential operators.
\end{prop}
\begin{proof}
Given $\Delta\in\Diff^n_d(E,F)$, let $\widetilde{\Delta}\colon J^n_d E\to F$ be the corresponding $A$-linear lift of $\Delta$.
Then the map $\widetilde{\Delta}\circ \pi^{m,n}_{d,E}\colon J^n_d E\to F$ is also left $A$-linear, and
\begin{equation}
\widetilde{\Delta}\circ \pi^{m,n}_{d,E}\circ j^m_{d,E}
=\widetilde{\Delta}\circ j^n_{d,E}
=\Delta.
\end{equation}
It follows that $\Delta\in \Diff^m_d(E,F)$.

The proof is analogous for the semiholonomic and nonholonomic case, using the semiholonomic and nonholonomic projections respectively.
\end{proof}
This gives the following filtration
\begin{equation}\label{eq:DOfiltration}
\AHom(E,F)=\Diff^0_d(E,F)\subseteq \Diff^1_d(E,F)\subseteq \cdots \subseteq \Diff^n_d(E,F)\subseteq \cdots\subseteq \Diff_d(E,F)\subseteq \Hom(E,F).
\end{equation}
Analogously for $\SDiff_d$ and $\NDiff_d$.

\begin{prop}\label{prop:differentialoperatorcomposition}
	Let $\Delta_1\colon E \to F$ and $\Delta_2\colon F \to G$ be differential operators of order at most $n$ and $m$, respectively.
	Then the composition $\Delta_2 \circ \Delta_1\colon E \to G$ is a differential operator of order at most $n+m$.
	
	An analogous property holds for nonholonomic differential operators.
Moreover, if $\Omega^1_d$ is flat in $\ModA$, then the same is true for semiholonomic differential operators.
\end{prop}
\begin{proof}
	First, let $\widetilde{\Delta}_1\colon J^n_d E \to F$ be a lift of $\Delta_1$.
	Since this is $A$-linear, it is a morphism in $\AMod$, so we can apply $J^m_d$ to obtain
	\begin{equation}
		J^m_d(\widetilde{\Delta}_1)\colon J^m_d (J^n_d E) \to J^m_d F
	\end{equation}
	Next, let $\widetilde{\Delta}_2\colon J^m_dF \to G$ be a lift of $\Delta_2$.
	Finally, consider the left $A$-linear map $l^{n,m}_{d,E}\colon J^{n+m}_dE \to J^m_d(J^n_dE)$ from Lemma \ref{lemma:smM}, which is such that $l^{m,n}_d\circ j^{m+n}_d=j^{m}_{d,J^n_d}\circ j^n_d$.
	We combine these maps in the following commutative diagram.
\begin{equation}\label{operatorcomposition}
\begin{tikzcd}
J^{m+n}_d E\ar[r,"l^{m,n}_{d,E}"]&J^m_d(J^n_dE) \arrow[rd,near start, "J^m_d(\widetilde{\Delta}_1)"] \\
&J^n_dE \arrow[rd,near start, "\widetilde\Delta_1"] \arrow[u,"j^m_{d,J^n_d E}"]	&	J^m_dF \arrow[rd, near start, "\widetilde\Delta_2"]	&[20pt]\\
&E\ar[uul,"j^{m+n}_{d,E}"] \arrow[r, "\Delta_1"] \arrow[u,"j^n_{d,E}"']	&	F \arrow[r, "\Delta_2"] \arrow[u,"j^m_{d,F}"]	&	G
\end{tikzcd}
\end{equation}
The composition $\widetilde{\Delta}_2 \circ J^m_d(\widetilde{\Delta}_1)\circ l^{m,n}_{d,E}$ is left $A$-linear, and
\begin{equation}
\widetilde{\Delta}_2 \circ J^m_d(\widetilde{\Delta}_1)\circ l^{m,n}_{d,E}\circ j^{m+n}_{d,E}
=\Delta_2\circ \Delta_1.
\end{equation}
It follows that $\Delta_2\circ \Delta_1$ is a differential operator of order at most $m+n$, with lift given by $ \widetilde{\Delta}_2 \circ J^m_d(\widetilde{\Delta}_1)\circ l^{m,n}_{d,E}$.

The claims for semiholonomic and nonholonomic differential operators are proven similarly by finding appropriate
\begin{align}
l^{[m,n]}_d\colon J^{[n+m]}_d \longrightarrow J^{[m]}_d\circ J^{[n]}_d,
&\hfill&
l^{(m,n)}_d\colon J^{(n+m)}_d \longrightarrow J^{(m)}_d\circ J^{(n)}_d,
\end{align}
that are compatible with the corresponding prolongations.
In the nonholonomic case $J^{(m)}_d\circ J^{(n)}_d=J^{(m+n)}_d$, so we can take $l^{(m,n)}_d$ to be the identity.

For the semiholonomic case, the flatness hypothesis allows us to apply Theorem \ref{theo:sjchar}, which implies
\begin{equation}
J^{[m+n]}_d\subseteq J^{(m)}_d\circ J^{[n]}_d.
\end{equation}
We then have
\begin{equation}
J^{[m+n]}_d
\subseteq \bigcap_{k=1}^{m-1}\ker\left(J^{(n-k-1)}_d \widetilde{\DH}^I_{d,J^{(k-1)}_d\circ J^{[n]}_d} \right)
=J^{[m]}_d\circ J^{[n]}_d.
\end{equation}
We thus choose $l^{[m,n]}_d$ to be this inclusion map, which satisfies the desired property as a consequence of Proposition \ref{prop:nhjtosh}.
\end{proof}
\begin{rmk}
If we defined instead the semiholonomic jet functors by induction, in the style of point \eqref{theo:sjchar:2} of Theorem \ref{theo:sjchar}, we would get a result similar to Lemma \ref{lemma:smM} also for semiholonomic jets with an analogous proof.
By adopting that definition, this theorem would then hold without any assumption on $\Omega^1_d$.
\end{rmk}
The composition in $\Mod$ thus restricts to a map
\begin{equation}
\circ \colon \Diff^m_d(F,G)\times \Diff^n_d(E,F)\longrightarrow \Diff^{m+n}_d(E,G).
\end{equation}
\begin{rmk}
The composition of differential operators of order $n$ and $m$ respectively is not necessarily of order $n+m$, but only at most $n+m$.
A counterexample is given by Proposition \ref{prop:excoddiffop}, that shows how $R_\nabla=d_\nabla\circ d_\nabla$ is a differential operator of order zero, albeit in general $d_\nabla$ is a differential operator of order one.
\end{rmk}
\begin{cor}\label{cor:diffcat}
	There is a category $\Diff_d$ with the same objects as $\AMod$ and with maps between $E$ and $F$ in $\AMod$ given by $\Diff_d(E,F)$.
	The same holds for $\NDiff_d$.
	If $\Omega^1_d$ is flat in $\ModA$, the same holds for $\SDiff_d$.
\end{cor}
\begin{defi}
We call the category $\Diff_d$ described by Corollary \ref{cor:diffcat}, the \emph{category of (holonomic) finite order linear differential operators}.
Analogously for $\SDiff_d$ and $\NDiff_d$ with the adjectives semiholonomic and nonholonomic, respectively.
\end{defi}

\begin{cor}\label{cor:algebrasofDOs}
	The modules $\Diff_d(E,E)$ and $\DO_d$ form filtered algebras with multiplication given by the composition of the category.

	The spaces $\Diff_d(E,F)$ are filtered $(\Diff_d(F,F),\Diff_d(E,E))$-bimodules with actions given by the composition.
Moreover, each $\Diff^n_d(E,F)$ and $\Diff_d(E,F)$ have an induced structure of $(\AHom(F,F),\AHom(E,E))$-bimodule.
	
	The analogous statements hold for nonhomolonomic differential operators and, if $\Omega^1_d$ is flat in $\ModA$, for semiholonomic differential operators. 
\end{cor}

Another consequence of Proposition \ref{prop:differentialoperatorcomposition} is the following.
\begin{cor}\label{cor:Aopzeroorder}
	The algebra $A^{\op}$ embeds into each of $\DO_d$, $\SDO_d$, and $\NDO_d$ as zero order differential operators via the right action.
\end{cor}
\begin{proof}
	Follows from Proposition \ref{prop:zeroorder}.
\end{proof}
\begin{prop}\label{prop:inclusiondiffopreg}
	For all $n$, the maps $h^n_d$ (cf.\ Proposition \ref{prop:semiholcpxfact}), and $\iota_{J^{[n]}_d}$ (cf.\ Definition \ref{def:shj}) induce the following natural inclusions
\begin{equation}
\AHom(E,F)\subseteq \NDiff^n_d(E,F)\subseteq \SDiff^n_d(E,F)\subseteq \Diff^n_d(E,F)\subseteq \Hom(E,F).
\end{equation}
\end{prop}
\begin{proof}
All holonomic linear differential operators are $\bk$-linear maps, so $\Diff_d(E,F)\subseteq \Hom(E,F)$.

Given a nonholonomic differential operator $\Delta\in\NDiff_d(E, F)$, we have an $A$-linear lift $\widetilde{\Delta}\colon J^{(n)}_d E\to F$.
Precomposing it with $\iota_{J^{[n]}_d}\colon J^{[n]}_d\to J^{(n)}_d$, gives us a semiholonomic lift of $\Delta$.

The inclusion $\SDiff^n_d(E,F)\subseteq \Diff^n_d(E,F)$ can be obtained similarly using $h^n_d\colon J^n_d\to J^{[n]}_d$.

The final inclusion follows from the fact that $\AHom(E,F)=\Diff^0_d(E,F) = \NDiff^0_d(E,F)$.
\end{proof}
\begin{rmk}
When $E=F=A$, the corresponding inclusion maps $\NDO_d\rightarrow \SDO_d\rightarrow \DO_d$ are algebra maps.
\end{rmk}

\begin{rmk}\label{rem:DOcatareenrichedinMod}
When they are defined (cf.\ Proposition \ref{prop:differentialoperatorcomposition}), the categories $\Diff_d$, $\SDiff_d$, and $\NDiff_d$ are enriched in $\Mod$.
The inclusions in Proposition \ref{prop:inclusiondiffopreg} are inclusions of $\bk$-modules.
\end{rmk}
Proposition \ref{prop:inclusiondiffopreg} provides the following faithful inclusions: $\NDiff_d\to \Diff_d$, $\Diff_d\to \Mod$, and their composition.
When $\Omega^1_d$ is flat in $\ModA$, and hence $\SDiff_d$ forms a category, the first inclusion factors through this category, providing two more faithful inclusions: $\SDiff_d\to \Diff_d$ and $\NDiff_d\to \SDiff_d$.
Moreover, the functor $\AMod\to \Mod$ factors through $\Diff_d\to \Mod$.

\begin{prop}
	Suppose there exists a morphism $\Omega^1_{d} \twoheadrightarrow \Omega^1_{d'}$ in $\CalcA$.
	Let $\Diff^1_d(E,F)$ and $\Diff^1_{d'}(E,F)$ be the differential operators of order at most $1$ from $E$ to $F$ with respect to $\Omega^1_d$ and $\Omega^1_{d'}$, respectively.
	Then $\Diff^1_{d'}(E,F) \subseteq \Diff^1_d(E,F)$.
\end{prop}
\begin{proof}
	Let $\Delta \in \Diff^1_{d'}(E,F)$.
	Then $\widetilde \Delta |_{N_{d'}(E)} = 0$.
	Since $N_{d}(E) \subseteq N_{d'}(E)$ by Proposition \ref{prop:Nd-inclusion}, it follows that $\widetilde \Delta |_{N_{d}(E)} = 0$.
\end{proof}
\begin{cor}
	Let $S\subset A$.
	For each calculus in $\CalcSA$, and each object $E$ in $\AMod$, there is a canonical subalgebra of differential operators generated by $\Diff^1_{S}(E,E)$, the linear differential operators of order at most $1$ with respect to $\Omega^1_S$.
\end{cor}

\subsection{Toy example on the quaternions}\label{ss:quaternions}
Let $\Hq$ denote the quaternions, i.e.\ the unital associative $\mathbb{R}$-algebra generated by $\{1,i,j,k\}$ subject to the relations $i^2 = j^2 = k^2 = ijk = -1$.
The universal calculus is given by
\begin{equation}
	\Omega^1_u(\Hq) = {}_{\Hq}\!\spn{ i \otimes i + 1 \otimes 1,\, j \otimes j + 1 \otimes 1,\, k \otimes k + 1 \otimes 1 }.
\end{equation}
First consider the $i$-terminal calculus, as in Definition \ref{defi:Sterminal}.
We find that 
\begin{equation}
	N_i = {}_{\Hq}\!\spn{ i \otimes i + 1 \otimes 1,\, j \otimes j -k\otimes k }.
\end{equation}
Similarly, we can determine $N_j$ and $N_k$, from which we deduce
\begin{equation}
	N_{\{i,j\}} \colonequals N_i\cap N_j = {}_{\Hq}\!\spn{ 1 \otimes 1 + i \otimes i + j \otimes j - k \otimes k }.
\end{equation}
Let $\Omega^1_d = \Omega^1_{\{i,j\}} \colonequals \Omega^1_u(\Hq)/N_{\{i,j\}}$.
This first order differential calculus is free and generated by $\{di, dj\}$.
Note that it is not canonical, but depends on a choice of two elements, which is unique up to algebra automorphism of $\Hq$.
We then compute the structure equation for the first order differential calculus as follows
\begin{equation}
\begin{split}
	dk & \colonequals [1\otimes k -k \otimes 1] 
	 = [-k(k\otimes k + 1\otimes 1)] 
	 = [-k (2\otimes 1 + i\otimes i + j \otimes j)]
	 = -j di + i dj.
\end{split}
\end{equation}
We also compute the bimodule structure
\begin{align}
\begin{split}
	idi &= -(di)i,\\
	jdi &= -(di)j,\\
	kdi &= (di)k,
	\end{split}
	\hfill
	\begin{split}
	idj &= -(dj)i,\\
	jdj &= -(dj)j,\\
	kdj &= (dj)k.
	\end{split}
\end{align}
\subsubsection{Differential operators}
Now we consider the algebra $\DO_d$.
The zero order operators form a subalgebra isomorphic to $\Hq^{\op}$, generated by the right multiplication maps $1, R_i, R_j, R_k$, i.e.\ $R_i \colonequals x\mapsto xi$, etc.
We also have the first order operators, which include the partial derivatives,
\begin{align}
	\partial_i\colon \begin{bmatrix}
		1 \\ i \\ j \\ k\\
	\end{bmatrix} \longmapsto \begin{bmatrix}
		0 \\ 1 \\ 0 \\ -j
	\end{bmatrix},
&\hfill&	
	 \partial_j\colon \begin{bmatrix}
		1 \\ i \\ j \\ k\\
	\end{bmatrix} \longmapsto \begin{bmatrix}
		0 \\ 0 \\ 1 \\ i
	\end{bmatrix},
\end{align}
where the column vectors encode the action on the basis of $\Hq$.
Since $\Omega^1_d$ is $\{i,j\}$-terminal, the left multiplication maps $L_i$ and $L_j$ by $i$ and $j$, respectively, are also first order differential operators (cf.\ Proposition \ref{prop:Sterminalcalc}).
To compute all higher order differential operators, we extend $\Omega^1_d$ to the maximal exterior algebra.
This gives
\begin{align}
	S^2_d = {}_{\Hq}\!\spn{ di \otimes_{\Hq} dj - dj\otimes_{\Hq} di},
&\hfill&
	S^n_d = 0,\text{ for } n\geq 3.
\end{align}
Moreover, this exterior algebra has vanishing Spencer $\delta$-cohomology.
Hence, by Corollary \ref{cor:spencerdelta_jes}, the jets stabilize beginning at order $2$.
\begin{prop}
	The algebra $\DO_d$ is generated by $1, R_i, R_j, R_k$ and $\partial_i, \partial_j$.
	The relations are $\partial_i^2 = \partial_j^2 = \llbracket\partial_i ,\,\partial_j\rrbracket= 0$ together with the $\Hq^{\op}$-bimodule structure 
	\begin{align}
	\begin{split}
		&\llbracket \partial_i, R_j \rrbracket = \llbracket \partial_j, R_i \rrbracket = 0,\\
		&\llbracket \partial_i, R_i \rrbracket = \llbracket \partial_j, R_j \rrbracket = 1,
		\end{split}
		\hfill
		\begin{split}
		&[ \partial_j, R_k ] = R_i,\\
		&[ \partial_i, R_k ] = -R_j.
		\end{split}
	\end{align}
	Here $\llbracket\cdot,\cdot\rrbracket$ and $[\cdot,\cdot]$ are the anticommutator and commutator of operators, respectively.
	An $\mathbb{R}$-basis is given by $R_q$, $\partial_p \circ R_q$, and $\partial_i \circ \partial_j \circ R_q$, where $q \in \{1,i,j,k\}$, $p\in \{i,j\}$, and $R_1 = 1$.
\end{prop}
\begin{rmk}
	$\Hq$ can be embedded into $\DO_d$ by its left action.
	$L_i$ and $L_j$ are first order, so $L_k = L_i \circ L_j$ is at most second order by Proposition \ref{prop:differentialoperatorcomposition}.
	Additionally, $N_d$ is not in the kernel of the universal lift $\widetilde{L_k}$, so $L_k$ is of second order.
\end{rmk}
Next we consider the notion of quantum metric from \cite[§1.3]{BeggsMajid}.
For general $A$, this is an element of $S^2_d$ satisfying some assumptions with respect to an $A$-bimodule inner product $(\cdot,\cdot)\colon \Omega^1_d \otimes_A \Omega^1_d \rightarrow A$ (cf.\ \cite[Definition 1.15, p.~14]{BeggsMajid}).
For our calculus $\Omega^1_{\{i,j\}}$, we have $S^2_d \cong \Hq$ as a bimodule, and so $\Omega^1_d$ admits a unique quantum metric of this kind, up to scale (and sign).
We choose as metric $di\otimes_{\Hq}dj-dj\otimes_{\Hq} di$.
Then the formula
\begin{equation}
	\Delta(ab) = \Delta(a)b +a\Delta(b) + 2(da,db)
\end{equation}
defines the quantum Laplacian $\Delta$ (cf.\ \cite[Definition 1.17, p.~15]{BeggsMajid}).
This is a second order operator, and in this case, we may compute
\begin{equation}
	\Delta = 2\partial_j \circ \partial_i = [\partial_j,\partial_i].
\end{equation}

\bibliography{../Bibliography}
\bibliographystyle{alpha}

\end{document}